\documentclass[10pt]{article}
\usepackage{geometry}
\usepackage{hyperref}
\usepackage[utf8]{inputenc}
\usepackage[T1]{fontenc}
\usepackage{amsmath}
\usepackage{amssymb}
\usepackage{mathrsfs}
\usepackage{amscd}
\usepackage{textcomp}

\usepackage{amsthm}
\theoremstyle{plain}
\newtheorem{theorem}{Theorem}[section]

\newtheorem{lemma}[theorem]{Lemma}

\newtheorem{corollary}[theorem]{Corollary}
\newtheorem{proposition}[theorem]{Proposition}

\theoremstyle{definition}
\newtheorem{definition}[theorem]{Definition}

\newtheorem{example}[theorem]{Example}
\newtheorem{remark}[theorem]{Remark}

\theoremstyle{remark}

\usepackage{xy}
\xyoption{all}

\newdir^{ (}{{}*!/-3pt/\dir^{(}}    
\newdir^{  }{{}*!/-3pt/\dir^{}}    
\newdir_{ (}{{}*!/-3pt/\dir_{(}}    
\newdir_{  }{{}*!/-3pt/\dir_{}}

\newcounter{zahl}

\def\theenumi{(\alph{enumi})}

\def\p@enumii{\theenumi}

\newcommand{\DS}{\displaystyle}
\newcommand{\TS}{\textstyle}
\newcommand{\SC}{\scriptstyle}
\newcommand{\SSC}{\scriptscriptstyle}

\DeclareMathOperator{\Aut}{Aut}

\DeclareMathOperator{\Cov}{Cov}

\DeclareMathOperator{\Ext}{Ext}
\DeclareMathOperator{\Frob}{Frob}
\DeclareMathOperator{\Gal}{Gal}
\DeclareMathOperator{\GL}{GL}

\DeclareMathOperator{\Koh}{H}

\DeclareMathOperator{\Hom}{Hom}

\DeclareMathOperator{\Id}{Id}
\DeclareMathOperator{\Ind}{Ind}
\DeclareMathOperator{\Int}{Int}
\DeclareMathOperator{\Isom}{Isom}
\DeclareMathOperator{\Lie}{Lie}

\DeclareMathOperator{\PGL}{PGL}

\DeclareMathOperator{\Quot}{Frac}

\DeclareMathOperator{\Rep}{Rep}

\DeclareMathOperator{\SL}{SL}
\DeclareMathOperator{\SpBerk}{{\tt BSpec}}
\DeclareMathOperator{\Spm}{Sp}

\DeclareMathOperator{\Spec}{Spec}
\DeclareMathOperator{\Spf}{Spf}

\DeclareMathOperator{\Var}{V}

\newcommand{\ad}{{\rm ad}}
\newcommand{\alg}{{\rm alg}}
\newcommand{\an}{{\rm an}}
\DeclareMathOperator{\charakt}{char}

\DeclareMathOperator{\coker}{coker}

\newcommand{\cont}{{\rm cont}}
\newcommand{\cris}{{\rm cris}}

\newcommand{\dom}{{\rm dom}}
\DeclareMathOperator{\diag}{diag}
\newcommand{\dR}{{\rm dR}}
\newcommand{\et}{{\rm \acute{e}t}}
\newcommand{\fpqc}{{\it fpqc}}
\newcommand{\fppf}{{\it fppf}}

\DeclareMathOperator{\id}{\,id}
\DeclareMathOperator{\im}{im}

\renewcommand{\mod}{{\rm\;mod\;}}
\DeclareMathOperator{\ord}{ord}
\newcommand{\opp}{{\rm opp}}

\newcommand{\rig}{{\rm rig}}
\DeclareMathOperator{\rk}{rk}

\newcommand{\sep}{{\rm sep}}

\newcommand{\univ}{{\rm univ}}

\renewcommand{\phi}{\varphi}
\renewcommand{\epsilon}{\varepsilon}

\usepackage{amsfonts}
\newcommand{\BOne} {{\mathchoice{\hbox{\rm1\kern-2.7pt l\kern.9pt}}
                              {\hbox{\rm1\kern-2.7pt l\kern.9pt}}
                              {\hbox{\scriptsize\rm1\kern-2.3pt l\kern.4pt}}
                              {\hbox{\scriptsize\rm1\kern-2.4pt l\kern.5pt}}}}

\newcommand{\BA}{{\mathbb{A}}}

\newcommand{\BD}{{\mathbb{D}}}

\newcommand{\BF}{{\mathbb{F}}}
\newcommand{\BG}{{\mathbb{G}}}
\newcommand{\BH}{{\mathbb{H}}}

\newcommand{\BM}{{\mathbb{M}}}
\newcommand{\BN}{{\mathbb{N}}}

\newcommand{\BP}{{\mathbb{P}}}
\newcommand{\BQ}{{\mathbb{Q}}}
\newcommand{\BR}{{\mathbb{R}}}

\newcommand{\BZ}{{\mathbb{Z}}}

\newcommand{\bB}{{\mathbf{B}}}

\newcommand{\CC}{{\cal{C}}}

\newcommand{\CE}{{\cal{E}}}
\newcommand{\CF}{{\cal{F}}}
\newcommand{\CG}{{\cal{G}}}
\newcommand{\CH}{{\cal{H}}}
\newcommand{\CI}{{\cal{I}}}

\newcommand{\CM}{{\cal{M}}}
\newcommand{\CN}{{\cal{N}}}
\newcommand{\CO}{{\cal{O}}}

\newcommand{\CQ}{{\cal{Q}}}

\newcommand{\CS}{{\cal{S}}}
\newcommand{\CT}{{\cal{T}}}

\newcommand{\CV}{{\cal{V}}}

\newcommand{\CX}{{\cal{X}}}
\newcommand{\CY}{{\cal{Y}}}

\newcommand{\scrG}{{\mathscr{G}}}

\newcommand{\scrS}{{\mathscr{S}}}

\newcommand{\FG}{{\mathfrak{G}}}

\newcommand{\Fa}{{\mathfrak{a}}}
\newcommand{\Fb}{{\mathfrak{b}}}

\newcommand{\Fm}{{\mathfrak{m}}}

\newcommand{\Fp}{{\mathfrak{p}}}
\newcommand{\Fq}{{\mathfrak{q}}}

\let\setminus\smallsetminus

\newcommand{\es}{\enspace}
\newcommand{\open}{^\circ}

\newcommand{\dual}{^{\SSC\lor}}

\newcommand{\mal}{^{\SSC\times}}
\newcommand{\fdot}{{\,{\SSC\bullet}\,}}
\newcommand{\ul}[1]{{\underline{#1}}}
\newcommand{\ol}[1]{{\overline{#1}}}
\newcommand{\wh}[1]{{\widehat{#1}}}
\newcommand{\wt}[1]{{\widetilde{#1}}}
\DeclareMathOperator{\whtimes}{\mathchoice
            {\wh{\raisebox{0ex}[0ex]{$\DS\times$}}}
            {\wh{\raisebox{0ex}[0ex]{$\TS\times$}}}
            {\wh{\raisebox{0ex}[0ex]{$\SC\times$}}}
            {\wh{\raisebox{0ex}[0ex]{$\SSC\times$}}}}

\usepackage{ifthen}

\newcommand{\invlim}[1][]{\ifthenelse{\equal{#1}{}}
{\DS \lim_{\longleftarrow}}
{\DS \lim_{\underset{#1}{\longleftarrow}}}
}

\newcommand{\dirlim}[1][]{\ifthenelse{\equal{#1}{}}
{\DS \lim_{\longrightarrow}}
{\DS \lim_{\underset{#1}{\longrightarrow}}}
}

\newcommand{\dbl}{{\mathchoice{\mbox{\rm [\hspace{-0.15em}[}}
                              {\mbox{\rm [\hspace{-0.15em}[}}
                              {\mbox{\scriptsize\rm [\hspace{-0.15em}[}}
                              {\mbox{\tiny\rm [\hspace{-0.15em}[}}}}
\newcommand{\dbr}{{\mathchoice{\mbox{\rm ]\hspace{-0.15em}]}}
                              {\mbox{\rm ]\hspace{-0.15em}]}}
                              {\mbox{\scriptsize\rm ]\hspace{-0.15em}]}}
                              {\mbox{\tiny\rm ]\hspace{-0.15em}]}}}}
\newcommand{\dpl}{{\mathchoice{\mbox{\rm (\hspace{-0.15em}(}}
                              {\mbox{\rm (\hspace{-0.15em}(}}
                              {\mbox{\scriptsize\rm (\hspace{-0.15em}(}}
                              {\mbox{\tiny\rm (\hspace{-0.15em}(}}}}
\newcommand{\dpr}{{\mathchoice{\mbox{\rm )\hspace{-0.15em})}}
                              {\mbox{\rm )\hspace{-0.15em})}}
                              {\mbox{\scriptsize\rm )\hspace{-0.15em})}}
                              {\mbox{\tiny\rm )\hspace{-0.15em})}}}}

\usepackage{color}
\newcounter{commentcounter}

\def\?{\ 
{\bf\color{red}???}\ 
\immediate\write16{}
\immediate\write16{Warning: There was still a question mark . . . }
\immediate\write16{}}

\long\def\forget#1{}

\def\longto{\longrightarrow}
\def\into{\hookrightarrow}
\def\onto{\twoheadrightarrow}
\def\isoto{\stackrel{}{\mbox{\hspace{1mm}\raisebox{+1.4mm}{$\SC\sim$}\hspace{-3.5mm}$\longrightarrow$}}}

\def\leftmapsto{\gets\joinrel\shortmid}

\newbox\mybox
\def\arrover#1{\mathrel{
       \setbox\mybox=\hbox spread 1.4em{\hfil$\scriptstyle#1$\hfil}
       \vbox{\offinterlineskip\copy\mybox
             \hbox to\wd\mybox{\rightarrowfill}}}}

\newcommand{\ancon}[1][]{{\mathchoice
           {\TS\langle\frac{z}{\zeta^{#1}},z^{-1}\}}
           {\TS\langle\frac{z}{\zeta^{#1}},z^{-1}\}}
           {\SC\langle\frac{z}{\zeta^{#1}},z^{-1}\}}
           {\SSC\langle\frac{z}{\zeta^{#1}},z^{-1}\}}}}
\newcommand{\tminus}{t_{\SSC -}}
\newcommand{\tplus}{t_{\SSC +}}
\newcommand{\tplusminus}{t}

\DeclareMathOperator{\QIsog}{QIsog}
\DeclareMathOperator{\Sets}{\CS \!{\it ets}}
\DeclareMathOperator{\Gr}{Gr}
\newcommand{\Flag}{\CF\ell}

\def\olB{{\,\overline{\!B}}}
\def\olE{{\,\overline{\!E}}}
\def\ulD{{\underline{D\!}\,}}
\def\ulM{{\underline{M\!}\,}}
\def\ulN{{\underline{N\!}\,}}
\def\olM{{\,\overline{\!M}}}

\def\olS{{\,\overline{\!S}}}
\def\olT{{\,\overline{\!T}}}

\def\olY{{\overline{Y}}}

\def\ulCE{{\underline{\CE\!}\,}}
\def\ulCF{{\underline{\CF\!}\,}}
\def\olCS{{\,\overline{\!\scrS}}}
\def\ulCT{{\underline{\CT\!}\,}}
\def\ulCV{{\underline{\CV\!}\,}}

\def\olDelta{{\,\overline{\!\Delta\!}\,}}

\def\s{\sigma^\ast}

\newcommand{\Test}{X}
\DeclareMathOperator{\Triv}{Triv}
\def\onto{\mbox{$\kern2pt\to\kern-8pt\to\kern2pt$}}
\newcommand{\BaseFld}{\BF}
\newcommand{\OtherField}{k}

\DeclareMathOperator{\Nilp}{\CN\!\it ilp}
\newcommand{\Darst}{\rho}
\def\Mod{{\rm Mod}}
\def\FMod{{\rm FMod}}
\newcommand{\ALoc}{A\mbox{-}\ul{\rm Loc}}
\newcommand{\BLoc}{\BF_q\dbl z\dbr\mbox{-}\ul{\rm Loc}}
\newcommand{\PLoc}{\BF_q\dpl z\dpr\mbox{-}\ul{\rm Loc}}
\newcommand{\HeckeTower}{{\breve\CE}}
\newcommand{\RZ}{{\CM}_{\ul\BG_0}^{\raisebox{-0.7ex}{$\SC\hat{Z}^{-1}$}}}
\newcommand{\breveRZ}{{\breve\CM}_{\ul\BG_0}^{\raisebox{-0.7ex}{$\SC\hat{Z}^{-1}$}}}
\newcommand{\Trivializer}{f}

\begin{document}


\author{Urs Hartl and Eva Viehmann}

\title{The generic fiber of moduli spaces of bounded local $G$-shtukas
}
\date{}

\maketitle

\begin{abstract}
Moduli spaces of bounded local $G$-shtukas are a group-theoretic generalization of the function field analog of Rapoport and Zink's moduli spaces of $p$-divisible groups. In this article we generalize some very prominent concepts in the theory of Rapoport-Zink spaces to our setting. More precisely, we define period spaces, as well as the period map from a moduli space of bounded local $G$-shtukas to the corresponding period space, and we determine the image of the period map. Furthermore, we define a tower of coverings of the generic fiber of the moduli space that is equipped with a Hecke action and an action of a suitable automorphism group. Finally we consider the $\ell$-adic cohomology of these towers.\\

Les espaces de modules de $G$-chtoucas locaux born\'es sont une g\'en\'eralisation des espaces de modules de groupes $p$-divisibles de Rapoport-Zink, au cas d'un corps de fonctions local, pour des groupes plus g\'en\'eraux et des copoids pas n\'ecessairement minuscules. Dans cet article nous d\'efinissons les espaces de p\'eriodes et l'application de p\'eriodes associ\'es \`a un tel espace, et nous calculons son image. Nous \'etudions la tour au-dessus de la fibre g\'en\'erique de l'espace de modules, \'equip\'ee d'une action de Hecke ainsi que d'une action d'un groupe d'automorphismes. Enfin, nous d\'efinissons la cohomologie $\ell$-adique de ces tours.

\end{abstract}

\bigskip

%
%

\section{Introduction}
\setcounter{equation}{0}

Towers of moduli spaces of $p$-divisible groups with additional structure as defined by Drinfeld \cite{Drinfeld76} and Rapoport and Zink \cite{RZ} have become a central topic in the study of the geometric realization of local Langlands correspondences. These towers consist of covering spaces of the generic fiber of moduli spaces of $p$-divisible groups with EL or PEL structure. They carry a Hecke action and an action of an associated automorphism group of the defining $p$-divisible group with extra structure, and possess a period morphism to a $p$-adic period space. Recently, generalizations of these moduli spaces to groups of unramified Hodge type (instead of PEL type) have been defined by Kim \cite{KimRZHodge} and Howard and Pappas \cite{HoP}. Conjecturally, in all these cases, the cohomology of the tower realizes local Langlands correspondences. Several cases of these conjectures have been shown so far, compare for example \cite{Fargues}, \cite{Chen}. However, in general, still very little is known.

In the present paper we define the analogous towers, cohomology groups and period spaces in the function field case, and study their basic properties. This generalizes Drinfeld's work \cite{Drinfeld76}. We thus provide the foundations for a similar theory as the one initiated by Drinfeld, Rapoport and Zink. It is conceivable that the cohomology of our towers likewise realizes local Langlands correspondences. For Drinfeld's towers \cite{Drinfeld76} this was conjectured by Carayol~\cite{Carayol90} and proved by Boyer~\cite{Boyer99} and Hausberger~\cite{Hausberger} building on work of Laumon, Rapoport and Stuhler~\cite{LRS}. One major difference in our context is that instead of being restricted to groups of PEL or Hodge type, there is a natural and group-theoretic way to define moduli spaces of local $G$-shtukas for any reductive group $G$. Furthermore, one can also define more general boundedness conditions than minuscule bounds which would be the direct analog of the number field situation.

To give an overview of their definition let $\BF_q$ be a finite field with $q$ elements, let $\BaseFld$ be a fixed algebraic closure of $\BF_q$, and let $\BF_q\dbl z\dbr$ and $\BF_q\dbl\zeta\dbr$ be the power series rings over $\BF_q$ in the (independent) variables $z$, resp.\ $\zeta$. As base schemes we will consider the category $\Nilp_{\BF_q\dbl\zeta\dbr}$ consisting of schemes over $\Spec\BF_q\dbl\zeta\dbr$ on which $\zeta$ is locally nilpotent. Let $G$ be a parahoric group scheme over $\Spec\BF_q\dbl z\dbr$ in the sense of \cite[D\'efinition 5.2.6]{BT72} and \cite{AnhangPR} with connected reductive generic fiber. (One may ask whether the assumptions on $G$ can be relaxed, but we crucially use the ind-projectivity of $\Flag_G$ in the central Propositions~\ref{PropBound} and \ref{PropLiftOfIsog}; see the beginning of Section~\ref{sec2} for more explanations.)

Let $S\in \Nilp_{\BF_q\dbl\zeta\dbr}$ and let $H$ be a sheaf of groups on $S$ for the \fpqc{} topology. By a (\emph{right}) \emph{$H$-torsor} on $S$ we mean a sheaf $\CH$ for the \fpqc{} topology on $S$ together with a (right) action of the sheaf $H$ such that $\CH$ is isomorphic to $H$ on an \fpqc{} covering of $S$. Here $H$ is viewed as an $H$-torsor by right multiplication. Let $LG$ and $L^+G$ be the loop group and the group of positive loops associated with $G$; compare Section \ref{sec2}. Let $\CG$ be an $L^+G$-torsor on $S$. Via the inclusion of sheaves $L^+G\subset LG$ we can associate an $LG$-torsor $L\CG$ with $\CG$. Also for an $LG$-torsor $\CG$ on $S$ we denote by $\s\CG$ the pullback of $\CG$ under the $q$-Frobenius morphism $\sigma:=\Frob_q\colon S\to S$.

\begin{definition}\label{DefLocSht}
A \emph{local $G$-shtuka} over some $S\in\Nilp_{\BF_q\dbl\zeta\dbr}$ is a pair $\ul\CG = (\CG,\tau_\CG)$ consisting of an $L^+G$-torsor $\CG$ on $S$ and an isomorphism of the associated $LG$-torsors $\tau_\CG\colon\sigma^\ast L\CG \isoto L\CG$.

A \emph{quasi-isogeny} $g\colon(\CG'\!,\tau_{\CG'})\to(\CG,\tau_\CG)$ between local $G$-shtukas over $S$ is an isomorphism $g\colon L\CG'\isoto L\CG$ of the associated $LG$-torsors with $g\circ\tau_{\CG'}=\tau_\CG\circ\s g$.
\end{definition}

Local $G$-shtukas were introduced and studied in \cite{HV1,HV2} in the case where $G$ is a constant split reductive group over $\BF_q$, and in \cite{GL,HartlPSp} for $G=\GL_r$. The general case was considered in \cite{AH_Local}. For a local $G$-shtuka $\ul\CG$ over $S$ there exists an \emph{\'etale covering} $S'\to S$ and a trivialization $\ul\CG\times_S S'\cong\bigl((L^+G)_{S'},b\s\bigr)$ with $b\in LG(S')$; see \cite[Proposition~2.2(c)]{HV1} and \cite[Proposition~2.4]{AH_Local}. 

Note that we may view $\Spf\BF_q\dbl\zeta\dbr$ as an ind-scheme. By $\Flag_G$ we denote the affine flag variety of $G$ over $\BF_q$, compare Section \ref{sec2}. We may form the fiber product $\wh\Flag_G:=\Flag_G\whtimes_{\BF_q}\Spf\BF_q\dbl\zeta\dbr$ in the category of ind-schemes. By \cite[Theorem~4.4]{AH_Local} it represents the functor on $\Nilp_{\BF_q\dbl\zeta\dbr}$ with
\begin{eqnarray}\label{EqFunctorFlag}
\wh\Flag_G(S) & = & \Bigl\{\,\text{Isomorphism classes of pairs }(\CG,\delta)\text{ where }\CG \text{ is an $L^+G$-torsor on $S$ and} \nonumber \\
&&\qquad \delta\colon L\CG\isoto LG_S\text{ is an isomorphism of the associated $LG$-torsors}\,\Bigr\}\,.
\end{eqnarray}

We consider local $G$-shtukas that satisfy an additional boundedness condition. Similarly to \cite[\S\,4.2]{AH_Local} we introduce the notion of a bound $\hat Z$ and its reflex ring $R_{\hat Z}$, which is a finite extension of $\BF_q\dbl\zeta\dbr$. Our bounds are defined as closed ind-subschemes $\hat Z\subset \wh\Flag_{G,R}=\Flag_G\whtimes_{\BF_q}\Spf R$ satisfying certain additional properties. Here $R$ is a finite extension of $\BF_q\dbl\zeta\dbr$. In particular, we allow more general bounds than usual, in the sense that they do not have to correspond directly to some coweight $\mu$ (that in the classical context even had to be minuscule). We consider local $G$-shtukas $\ul\CG=(\CG,\tau_\CG)$ over schemes in $\Nilp_{R_{\hat Z}}$ such that the singularities of the morphism $\tau_\CG^{-1}$ are bounded by $\hat Z$, compare Definition \ref{DefBDLocal}. In this case we say that \emph{$\ul\CG$ is bounded by $\hat{Z}^{-1}$}; see Remark~\ref{RemBdLocal}(a) for a comment on this terminology. These bounded local $G$-shtukas can be seen as the function field analogs of $p$-divisible groups with extra structure. We write $R_{\hat Z}=\kappa\dbl\xi\dbr$, let $E:=E_{\hat Z}:=\kappa\dpl\xi\dpr$ be its fraction field, and let $\breve R_{\hat Z}=\BaseFld\dbl\xi\dbr$ and $\breve E:=\breve E_{\hat Z}:=\BaseFld\dpl\xi\dpr$ be the completions of the maximal unramified extensions. 

One can then consider the usual Rapoport-Zink type moduli space representing the following functor: Let $\ul{\BG}_0$ be a local $G$-shtuka over $\BaseFld$, and consider the functor $\breveRZ\colon (\Nilp_{\breve R_{\hat Z}})^o \;\longto\; \Sets$, 
\begin{eqnarray*}
S&\longmapsto & \Bigl\{\,\text{Isomorphism classes of }(\ul\CG,\bar\delta)\colon\text{ where }\ul{\CG}\text{ is a local $G$-shtuka over $S$ }\\ 
&&~~~~ \text{bounded by $\hat{Z}^{-1}$ and }\bar{\delta}\colon  \ul{\CG}_{\bar{S}}\to \ul{\BG}_{0,\bar{S}}~\text{is a quasi-isogeny  over $\bar{S}$ }\Bigr\}.
\end{eqnarray*}
Here $\bar{S}:=\Var_S(\zeta)$ is the zero locus of $\zeta$ in $S$. The functor $\breveRZ$ is ind-representable by a formal scheme over $\Spf\breve R_{\hat Z}$ that is locally formally of finite type and separated; see \cite[Theorem~4.18]{AH_Local}. The group $J=\QIsog_{\BaseFld}(\ul{\BG}_0)$ of self-quasi-isogenies of $\ul{\BG}_0$ acts on $\breveRZ$ via $g\colon(\ul\CG,\bar\delta)\mapsto(\ul\CG,g\circ\bar\delta)$ for $g\in\QIsog_{\BaseFld}(\ul{\BG}_0)$. 

We consider the generic fiber $(\breveRZ)^{\an}$ of this moduli space, as a strictly $\breve E$-analytic space in the sense of Berkovich~\cite{Berkovich1,Berkovich2}. Using the fully faithful functors \cite[\S\,1.6]{Berkovich2} and \cite[(1.1.11)]{Huber96} from strictly $\breve E$-analytic spaces to rigid analytic spaces over $\breve E$, respectively from rigid analytic spaces to Huber's analytic adic spaces, many of the results below can be formulated likewise in terms of rigid analytic, respectively analytic adic spaces. However, as we want to use \'etale fundamental groups and local systems on these spaces, we prefer in this work the Berkovich point of view for which such a theory exists in the literature.

As in \cite[Definition~3.5]{AH_Local} we consider the rational (dual) Tate module of the universal local $G$-shtuka over each connected component $Y$ of $(\breveRZ)^{\an}$, see Section \ref{SectTower}. It is a tensor functor 
$$\check V_{\ul{\CG},\fdot}:\Rep_{\BF_q\dpl z\dpr}G\rightarrow \Rep^\cont_{\BF_q\dpl z\dpr }\bigl(\pi_1^\et(Y,\bar s)\bigr).$$
Here $\Rep_{\BF_q\dpl z\dpr}G$ denotes the Tannakian category of $\BF_q\dpl z\dpr $-rational representations of $G$, and $\Rep^\cont_{\BF_q\dpl z\dpr }\bigl(\pi_1^\et(Y,\bar s)\bigr)$ denotes the category of finite-dimensional $\BF_q\dpl z\dpr $-vector spaces with a continuous action of de Jong's \cite[\S\,2]{dJ95a} \'etale fundamental group $\pi_1^\et(Y,\bar s)$ where $\bar s$ is a fixed base point in the given component. 

Trivializing the rational Tate module up to the action of $K$ for each compact open subgroup $K\subset G\bigl(\BF_q\dpl z\dpr\bigr)$ we obtain a tower $({\breve\CM}^K)_K$ of analytic spaces. Each of the spaces is equipped with an action of the group $J=\QIsog_{\BaseFld}(\ul{\BG}_0)$ that is induced by the action on the moduli space $\breveRZ$ itself. Furthermore, the group $G\bigl(\BF_q\dpl z\dpr\bigr)$ acts vertically on the tower via Hecke operators, i.e. for $g\in G\bigl(\BF_q\dpl z\dpr\bigr)$ we have compatible isomorphisms $g:{\breve\CM}^K\isoto {\breve\CM}^{g^{-1}Kg}$.

In the last section we consider the $\ell$-adic cohomology with compact support of the spaces ${\breve\CM}^K$ and their limit over $K$ together with induced actions of $J$, of $G\bigl(\BF_q\dpl z\dpr\bigr)$, and of the Weil group $W_E$. We provide basic finiteness properties of these cohomology groups and representations. Note that Tate modules, towers of moduli spaces of local $G$-shtukas and their cohomology are also considered in a similar, but slightly different context by Neupert in \cite{N16}. There, the relation to moduli spaces of global $G$-shtukas and their cohomology is studied.

Besides this construction of the tower of moduli spaces, our second main topic is the definition of the associated period space and the properties of the period morphism. Period spaces are strictly $\BF_q\dpl\zeta\dpr$-analytic spaces in the sense of Berkovich~\cite{Berkovich1,Berkovich2}. Since we allow more general bounds than those associated with minuscule coweights, these period spaces have to be defined as subspaces of an affine Grassmannian instead of a (classical) flag variety. To define them we consider the group scheme $G\times_{\BF_q\dbl z\dbr}\Spec \BF_q\dpl\zeta\dpr\dbl z-\zeta\dbr$ under the homomorphism $\BF_q\dbl z\dbr\to \BF_q\dpl\zeta\dpr\dbl z-\zeta\dbr,\,z\mapsto z=\zeta+(z-\zeta)$. Note that as this induces an inclusion $\BF_q\dpl z\dpr\to \BF_q\dpl\zeta\dpr\dbl z-\zeta\dbr$, this group is reductive. The associated affine Grassmannian $\Gr_G^{\bB_\dR}$ is the sheaf of sets for the \fpqc{} topology on $\Spec \BF_q\dpl\zeta\dpr$ associated with the presheaf
\begin{equation}\label{EqH_G}
X \;\longmapsto\; G\bigl(\CO_X\dpl z-\zeta\dpr\bigr)/G\bigl(\CO_X\dbl z-\zeta\dbr\bigr).
\end{equation}
$\Gr_G^{\bB_\dR}$ is an ind-scheme over $\Spec\BF_q\dpl\zeta\dpr$, that is ind-projective by \cite[Theorem~1.4]{PR2} and \cite[Theorem~A]{Richarz13}. Here, the notation $\bB_\dR$ refers to the fact that if $C$ is the completion of an algebraic closure of $\BF_q\dpl\zeta\dpr$, then $C\dpl z-\zeta\dpr$ is the function field analog of Fontaine's $p$-adic period field $\bB_\dR,$ compare \cite[\S\,2.9]{HartlDict}.

For our fixed bound $\hat Z$ we call the associated $E$-analytic space $\CH_{G,\hat Z}^\an:=\hat Z^{\an}$ the \emph{space of Hodge-Pink $G$-structures bounded by $\hat Z$}. It is the $E$-analytic space associated with a projective variety $\CH_{G,\hat Z}$ over $E=E_{\hat Z}$ by Proposition~\ref{PropBound}\ref{PropBound_D} and is a closed subscheme of $\Gr_G^{\bB_\dR}\otimes_{\BF_q\dpl\zeta\dpr} E$. Let $\ul{\BG}_0$ be the local $G$-shtuka over $\BaseFld$ from above and fix a trivialization $\ul{\BG}_0\cong(L^+G_\BaseFld,b\s)$, where $b\in LG(\BaseFld)$ represents the Frobenius morphism. The period space $\breve\CH_{G,\hat Z,b}^{wa}$ is then defined as the set of all $\gamma\in \breve\CH_{G,\hat Z}^\an:=\CH_{G,\hat Z}^\an\widehat\otimes_{E}\breve E$ such that $(b,\gamma)$ is weakly admissible. For the usual condition of weak admissibility (checked on all representations of $G$) we refer to Definition~\ref{DefWA}. Likewise one defines the admissible locus $\breve\CH_{G,\hat Z,b}^{a}$ in $\breve\CH_{G,\hat Z}^\an$ as the subset over which the universal $\sigma$-bundle has slope zero. In Theorem~\ref{ThmWAOpen} we show that $\breve\CH_{G,\hat Z,b}^{wa}$ and $\breve\CH_{G,\hat Z,b}^{a}$ are open paracompact strictly $\breve E$-analytic subspaces of $\breve\CH_{G,\hat Z}^\an$.

We prove that there is an \'etale period morphism
$$\breve\pi\colon(\breveRZ)^\an\;\longto\;\breve\CH_{G,\hat{Z},b}^a.$$ Very roughly, it is defined as follows: Consider the filtration on the universal local $G$-shtuka on $(\breveRZ)^\an$ induced by the image of the inverse of the universal Frobenius morphism $\tau_{\ul{\mathcal{G}}^{\rm univ}}$. This filtration is the function field analog of the Hodge filtration on the de Rham cohomology and is bounded by $\hat Z$. Using the universal quasi-isogeny, one can associate with it a natural filtration on the base change of $\ul{\BG}_0$ that is bounded by $\hat Z$. Strictly speaking, we carry out this construction rather with the \emph{Hodge-Pink $G$-structure} instead of the filtration; see Definition~\ref{DefHPG-Structure} and Remark~\ref{RemDIsTensorFunctor}(a). The reason for this is again that as we allow non-minuscule bounds, the Hodge-Pink $G$-structure contains more information than just the Hodge filtration. The former yields a point of $\breve\CH_{G,\hat{Z},b}^a$. This period morphism also induces compatible period morphisms for all elements ${\breve\CM}^K$ of the tower of coverings. In Theorem \ref{MainThm}\ref{MainThm_1} we show that the image of the period morphism is equal to a suitable union of connected components of $\breve\CH_{G,\hat{Z},b}^a$. 

There is an analogy between the theory of local $G$-shtukas and the theory of $p$-divisible groups \cite[\S\,3.9]{HartlDict}. In this sense, our results have natural counterparts in the theory of $p$-divisible groups, for particular cases by \cite{HartlRZ}, in general by Scholze and Weinstein \cite{ScholzeWeinstein,ScholzeBerkeley} using the Fargues-Fontaine curve \cite{FarguesFontaine}. One main difference is that in the function field case, the flag variety $\Flag_G$ is an honest ind-scheme. Also the Fargues-Fontaine curve is replaced by its role model, the Hartl-Pink curve~\cite{HP}. This allows us to consider non-minuscule Hodge-Pink structures and to work without Scholze’s theory of diamonds. One feature of our theory is the group-theoretic approach that makes the results automatically functorial in the group $G$; see Remarks~\ref{RemFunctoriality1}, \ref{RemFunctoriality2} and \ref{RemFunctoriality3}. Interestingly, the proofs for local $G$-shtukas in this work had to be largely different from the techniques used for $p$-divisible groups, and are technically quite involved.

\bigskip

\noindent{\it Acknowledgments.} We would like to thank the anonymous referee for careful reading and many good comments. We further thank Johannes Ansch\"utz, Bhargav Bhatt, Ofer Gabber, Tom Haines, Jack Hall, Jochen Heinloth, Roland Huber, Eike Lau, Brandon Levin, Stephan Neupert, Timo Richarz, Daniel Sch{\"a}ppi, Peter Scholze and Torsten Wedhorn for helpful discussions, and Johannes Ansch\"utz for pointing out an error in an earlier proof of Theorem~\ref{MainThm}\ref{MainThm_1}. The first author acknowledges support of the DFG (German Research Foundation) in form of SFB 878, Project-ID 427320536 -- SFB 1442, and Germany's Excellence Strategy EXC 2044--390685587 ``Mathematics M\"unster: Dynamics--Geometry--Structure''. The second author was partially supported by ERC starting grant 277889 ``Moduli spaces of local $G$-shtukas''.

\tableofcontents

\section{Bounded local $G$-shtukas}\label{sec2}
\setcounter{equation}{0}

Recall that we fixed a parahoric group scheme $G$ over $\Spec \BF_q\dbl z\dbr$. 

For an $\BF_q$-scheme $S$ we let $\CO_S\dbl z\dbr$ be the sheaf of $\CO_S$-algebras on $S$ for the \fpqc{} topology whose ring of sections on an $S$-scheme $Y$ is the ring of power series $\CO_S\dbl z\dbr(Y):=\Gamma(Y,\CO_Y)\dbl z\dbr$. This is indeed a sheaf being the countable direct product of $\CO_S$. A sheaf $M$ of $\CO_S\dbl z\dbr$-modules on $S$ that is finite free \fpqc{}-locally on $S$ is already finite free Zariski-locally on $S$ by \cite[Proposition~2.3]{HV1}. We call those modules \emph{locally free sheaves of $\CO_S\dbl z\dbr$-modules}. Let $\CO_S\dpl z\dpr$ be the \fpqc{} sheaf of $\CO_S$-algebras on $S$ associated with the presheaf $Y\mapsto\Gamma(Y,\CO_Y)\dbl z\dbr[\frac{1}{z}]$. If $Y$ is quasi-compact then $\CO_S\dpl z\dpr(Y)=\Gamma(Y,\CO_Y)\dbl z\dbr[\frac{1}{z}]$ by \cite[Tag 009F]{StacksProject}. The \emph{group of positive loops associated with $G$} is the infinite-dimensional affine group scheme $L^+G$ over $\BF_q$ whose $S$-valued points are $L^+G(S):=G\bigl(\CO_S\dbl z\dbr(S)\bigr)=G\bigl(\Gamma(S,\CO_S)\dbl z\dbr)$. The \emph{group of loops associated with $G$} is the ind-group-scheme $LG$ over $\BF_q$ that represents the \fpqc{} sheaf of groups $S\longmapsto LG(S):=G\bigl(\CO_S\dpl z\dpr(S)\bigr)$. A good reference for the theory of ind-schemes is \cite[\S\,7.11]{BeilinsonDrinfeld}. The \emph{affine flag variety $\Flag_G$ associated with $G$} is the \fpqc{} sheaf associated with the presheaf
$$
S\;\longmapsto\; LG(S)/L^+G(S)\;=\;G\bigl(\CO_S\dpl z \dpr(S) \bigr)/G\bigl(\CO_S\dbl z\dbr(S)\bigr)
$$ 
on the category of $\BF_q$-schemes. Pappas and Rapoport~\cite[Theorem~1.4]{PR2} and Richarz \cite[Theorem~A]{Richarz13} show that $\Flag_G$ is represented by an ind-scheme that is ind-projective over $\BF_q$, and that the natural morphism $LG\to\Flag_G$ admits sections locally for the \'etale topology. We crucially use the ind-projectivity of $\Flag_G$ in Propositions~\ref{PropBound} and \ref{PropLiftOfIsog}. By \cite[Theorem~0.1]{PR2}, after base change to $\BaseFld$ the connected components of $LG\wh\otimes_{\BF_q}\BaseFld$ and $\Flag_G\wh\otimes_{\BF_q}\BaseFld$ are in canonical bijection to the coinvariants $\pi_1(G)_I$. Here $\pi_1(G)$ is Borovoi's fundamental group \cite[Chapter~1]{Borovoi98} defined as $\pi_1(G):=X_*(T)/\text{(coroot lattice)}$ for a maximal torus $T$ of $G_{\BF_q\dpl z\dpr^\sep}$. Moreover, $\pi_1(G)_{I}$ denotes the group of coinvariants under the inertia subgroup $I$ of $\Gamma=\Gal\bigl(\BF_q\dpl z\dpr^\sep\!/\BF_q\dpl z\dpr\bigr)$. The bijection $\pi_0(\Flag_G\wh\otimes_{\BF_q}\BaseFld)=\pi_0(LG\wh\otimes_{\BF_q}\BaseFld)\cong \pi_1(G)_I$ is induced by the Kottwitz homomorphism $\kappa_G\colon LG(\BaseFld)=G(\BaseFld\dpl z\dpr)\rightarrow \pi_1(G)_{I}$ (introduced by Kottwitz in \cite{Kottwitz97}, for the reformulation used here compare \cite[2.a.2]{PR2}). It induces a bijection between the set $\pi_0(LG)=\pi_0(\Flag_G)$ and the set of $\langle \sigma\rangle$-orbits in $\pi_1(G)_I$ by \cite[Lemma~2.2.6]{N16}, a set that is in general no longer a group.

\begin{remark}\label{remcompgsv}
We will define bounds on local $G$-shtukas as (equivalence classes of) certain ind-subschemes of $\wh\Flag_{G,R}:=\Flag_G\whtimes_{\BF_q}\Spf R$ where $R$ is a finite extension of $\BF_q\dbl\zeta\dbr$. In order to define and consider also the generic fiber of the associated moduli spaces, one needs to bound the singularities with respect to $z-\zeta$ of the local $G$-shtukas. In particular, our definition is more restrictive than the one in \cite[Definitions~4.5 and 4.8]{AH_Local}. To encode this condition in our notion of bounds, we have to compare $\wh\Flag_{G,R}$ to the following closed ind-subschemes associated with a representation of $G$. 
 
If $A=\BF_q\dbl z\dbr$ or $A=\BF_q\dpl z\dpr$ we let $\Rep_{A}G$ denote the category of (algebraic) representations of $G$ on finite free $A$-modules. Here, we consider representations $\Darst\colon G\to\SL_r$ over $\BF_q\dbl z\dbr$ and the induced functor $\CG\mapsto\Darst_*\CG$ from $L^+G$-torsors to $L^+\SL_r$-torsors, which in turn yields a morphism $\Darst_*\colon\wh\Flag_G\to\wh\Flag_{\SL_r}$. Here, $\wh\Flag_{G}:=\Flag_{G,\BF_q\dbl\zeta\dbr}$. The category of $L^+\SL_r$-torsors on $S$ is equivalent to the category of pairs $(M,\alpha)$, where $M$ is a finite locally free $\CO_S\dbl z\dbr$-module of rank $r$ on $S$ and $\alpha\colon\wedge^r_{\CO_S\dbl z\dbr} M\isoto\CO_S\dbl z\dbr$ is an isomorphism of $\CO_S\dbl z\dbr$-modules, with isomorphisms as morphisms. We denote the $\CO_S\dbl z\dbr$-module associated with an $L^+\SL_r$-torsor $\CS$ by $M(\CS)$. For example, $M\bigl((L^+\SL_r)_S\bigr)=\CO_S\dbl z\dbr^{\oplus r}$. For a positive integer $n$ we consider the closed ind-subscheme of $\wh\Flag_{\SL_r}$ given by
\begin{eqnarray}\label{EqBDSL}
\wh\Flag^{(n)}_{\SL_r}(S) & := & \Bigl\{\,\bigl(\CS,\,\delta\colon L\CS\isoto (L\SL_r)_S\bigr)\in\wh\Flag_{\SL_r}(S)\colon \text{ for all }j=1,\ldots,r\text{ we have } \nonumber \\
& & \TS\bigwedge^j_{\CO_S\dbl z\dbr}M(\delta)\bigl(M(\CS)\bigr)\;\subset\;(z-\zeta)^{n(j^2-jr)}\cdot\bigwedge^j_{\CO_S\dbl z\dbr} M\bigl((L^+\SL_r)_S\bigr)\,\Bigr\}\,.
\end{eqnarray}
It is a $\zeta$-adic formal scheme over $\Spf\BF_q\dbl\zeta\dbr$ by \cite[Proposition~5.5]{HV1}; see Example~\ref{ExConstantG} for more explanations. Note that the compatibility with the isomorphism $\alpha\colon\wedge^r M(\CS)\isoto\CO_S\dbl z\dbr$ is equivalent to the assertion that the inclusion of the exterior powers in \eqref{EqBDSL} is an equality for $j=r$, because $\bigl(\CS,\,\delta\bigr)\in\wh\Flag_{\SL_r}(S)$ implies $\wedge^r M(\delta)=\alpha$. We then require the bounds to factor through some  $\wh\Flag^{(n)}_{\SL_r}$, cf. Condition \ref{DefBDLocal_A4}.

A different way to formulate such a condition would be to use the isomorphism
\begin{equation}\label{eqcompbdr}
\lim_{\longrightarrow_n} \wh\Flag_G \times_{\wh\Flag_{\SL_r}} \wh\Flag^{(n)}_{\SL_r}\cong {\rm Gr}(G_X,X) \times_X \Spec R
\end{equation}
 where the right hand side is the BD-Grassmannian associated with $G$ of \cite[Definition~3.3]{Richarz13}. Here we use that $G$ extends by \cite[Lemma~3.1]{Richarz13} to a smooth affine,  group scheme $G_X$ on a smooth connected curve $X$ over $\BF_q$ on which $\BF_q\dbl z\dbr$ is identified with the completion of the local ring at a point $x\in X$. The map $\Spec R\to X$ comes from the inclusion $\BF_q\dbl z\dbr\into R, z\mapsto\zeta$. The isomorphism~\eqref{eqcompbdr} also induces a comparison between specific bounds with global Schubert varieties of \cite{Richarz13}, compare Example \ref{exboundmu}. 
\end{remark}

We now define bounds by requiring minimal conditions needed to obtain the results of this article. In Remark \ref{DefBDLocal'} we will discuss further conditions that seem reasonable to impose, but that we do not need to assume in this article. We will then also describe more explicitly which bounds can arise, and in Examples~\ref{exboundmu} and \ref{ExConstantG} we will give a more specific class of bounds that depend on cocharacters of the generic fiber of $G$.

\begin{definition}\label{DefBDLocal}
\begin{enumerate}
\item We fix an algebraic closure $\BF_q\dpl\zeta\dpr^\alg$ of $\BF_q\dpl\zeta\dpr$, and consider pairs $(R,\hat Z_R)$, where $R/\BF_q\dbl\zeta\dbr$ is a finite extension of discrete valuation rings such that $R\subset\BF_q\dpl\zeta\dpr^\alg$, and where $\hat Z_R\subset \wh\Flag_{G,R}:=\Flag_G\whtimes_{\BF_q}\Spf R$ is a closed ind-subscheme. Two such pairs $(R,\hat Z_R)$ and $(R'\!,\hat Z'_{R'})$ are \emph{equivalent} if for some finite extension of discrete valuation rings $\wt R/\BF_q\dbl\zeta\dbr$ with $R,R'\subset \wt R$ the two closed ind-subschemes $\hat{Z}_R\whtimes_{\Spf R}\Spf\wt R$ and $\hat{Z}'_{R'}\whtimes_{\Spf R'}\Spf\wt R$ of $\wh\Flag_{G,\wt R}$ are equal. By \cite[Remark~4.6]{AH_Local} this then holds for all such rings $\wt R$.
\item \label{DefBDLocal_A}
A \emph{bound} is an equivalence class $\hat Z:=[(R,\hat Z_R)]$ of pairs $(R,\hat Z_R)$ as above satisfying the following properties.
\begin{enumerate}
\item \label{DefBDLocal_A1}
All $\hat{Z}_R\subset\wh{\Flag}_{G,R}$ are stable under the left $L^+G$-action.
\item \label{DefBDLocal_A2}
The \emph{special fiber} $Z_R:=\hat{Z}_R\whtimes_{\Spf R}\Spec \kappa_R$ is a  quasi-compact subscheme of $\Flag_G\whtimes_{\BF_q}\kappa_R$ where $\kappa_R$ is the residue field of $R$. (By \cite[Remark~4.10]{AH_Local} this implies that the $\hat Z_R$ are formal schemes in the sense of \cite[I$_{\rm new}$, \S\,10]{EGA}.)
\item \label{DefBDLocal_A3}
$\hat Z_R$ is a $\zeta$-adic formal scheme over $\Spf R$.
\item \label{DefBDLocal_A4}
There is a faithful representation $\Darst\colon G\into\SL_r$ over $\BF_q\dbl z\dbr$ and a positive integer $n$ such that all the induced morphisms $\Darst_*\colon\hat{Z}_R\to\wh\Flag_{\SL_r,R}$ factor through $\wh\Flag^{(n)}_{\SL_r,R}$.
\item \label{DefBDLocal_A11} 
Let $\hat{Z}_R^\an$ be the strictly $R[\tfrac{1}{\zeta}]$-analytic space associated with $\hat{Z}_R$. By Proposition~\ref{PropBound}\ref{PropBound_D} there is a closed subscheme $\hat Z_E$ of the affine Grassmannian $\Gr_G^{\bB_\dR}\times_{\BF_q\dpl\zeta\dpr}\Spec E_{\hat Z}$ from \eqref{EqH_G}, such that $\hat{Z}_R^\an$ arises by base change to $R[\tfrac{1}{\zeta}]$ from the strictly $E_{\hat Z}$-analytic space $(\hat Z_E)^\an$ associated with $\hat Z_E$. Then we require that $\hat Z_E$, and hence also all the $\hat{Z}_R^\an$ are invariant under the left multiplication of $G\bigl(\fdot\dbl z-\zeta\dbr\bigr)$  on $\Gr_G^{\bB_\dR}$.
\end{enumerate}

\item \label{DefBDLocal_B}
The \emph{reflex ring} $R_{\hat Z}$ of a bound $\hat Z=[(R,\hat Z_R)]$ is the intersection of the fixed field of $\{\gamma\in\Aut_{\BF_q\dbl\zeta\dbr}(\BF_q\dpl\zeta\dpr^\alg)\colon \gamma(\hat{Z})=\hat{Z}\,\}$ in $\BF_q\dpl\zeta\dpr^\alg$ with all the finite extensions $R\subset\BF_q\dpl\zeta\dpr^\alg$ of $\BF_q\dbl\zeta\dbr$ over which a representative $\hat{Z}_R$ of $\hat{Z}$ exists. We write $R_{\hat Z}=\kappa\dbl\xi\dbr$ and call its fraction field $E:=E_{\hat Z}=\kappa\dpl\xi\dpr$ the \emph{reflex field} of $\hat{Z}$. We let $\breve R_{\hat Z}:=\BaseFld\dbl\xi\dbr$ and $\breve E:=\breve E_{\hat Z}:=\BaseFld\dpl\xi\dpr$ be the completions of their maximal unramified extensions, where $\BaseFld$ is an algebraic closure of the finite field $\kappa$.
 
\item \label{DefBDLocal_C}
Let $\hat Z=[(R,\hat Z_R)]$ be a bound with reflex ring $R_{\hat Z}$. Let $\CG$ and $\CG'$ be $L^+G$-torsors over a scheme $S\in \Nilp_{R_{\hat Z}}$ and let $\delta\colon L\CG\isoto L\CG'$ be an isomorphism of the associated $LG$-torsors. We consider an \'etale covering $S'\to S$ over which trivializations $\alpha\colon\CG\isoto(L^+G)_{S'}$ and $\alpha'\colon\CG'\isoto(L^+G)_{S'}$ exist. Then the automorphism $\alpha'\circ\delta\circ\alpha^{-1}$ of $(LG)_{S'}$ corresponds to a morphism $S'\to LG\whtimes_{\BF_q}\Spf R_{\hat Z}$. We say that $\delta$ is \emph{bounded by $\hat{Z}$} if for every such trivialization and for every finite extension $R$ of $\BF_q\dbl\zeta\dbr$ over which a representative $\hat Z_R$ of $\hat Z$ exists the induced morphism
\[
S'\whtimes_{R_{\hat Z}}\Spf R\longto LG\whtimes_{\BF_q}\Spf R\longto \wh{\Flag}_{G,R}
\]
factors through $\hat{Z}_R$. Furthermore, we say that a local $G$-shtuka $\ul\CG=(\CG, \tau_\CG)$ is \emph{bounded by $\hat{Z}$} if $\tau_\CG$ is bounded by $\hat{Z}$, and, even more importantly, that $\ul\CG$ is \emph{bounded by $\hat{Z}^{-1}$} if the \emph{inverse} $\tau_\CG^{-1}$ of its Frobenius is bounded by $\hat{Z}$, compare the remark below. 
\end{enumerate}
\end{definition}

Let us explain the conditions of this definition in more detail.

\begin{remark}\label{RemBdLocal}
(a) The definition of a bound in Definition~\ref{DefBDLocal}\ref{DefBDLocal_A} is more restrictive than the one in \cite[Definition~4.8]{AH_Local} where only conditions \ref{DefBDLocal_A1} and \ref{DefBDLocal_A2} were required. The reason is that in \cite{AH_Local} the content of Proposition~\ref{PropBound} was not needed and the $R[\tfrac{1}{\zeta}]$-analytic spaces $(\hat{Z}_R)^\an$ were not considered.

In this article we will mainly consider local $G$-shtukas that are bounded by $\hat{Z}^{-1}$. This definition coincides with the notion of boundedness from \cite[Definition~4.8(b)]{AH_Local} in the following way. If $\hat{Z}$ is a bound in the sense of \cite[Definition~4.8]{AH_Local}, like for example our bound $\hat Z$, then by Lemma~\ref{LemmaInverseBound}, there is a bound $\hat{Z}^{-1}$ in the sense of \cite[Definition~4.8]{AH_Local} and $\tau_\CG^{-1}$ is bounded by $\hat{Z}$ if and only if $\tau_\CG$ is bounded by $\hat{Z}^{-1}$. 

\medskip\noindent 
(b) The reflex ring  in Definition~\ref{DefBDLocal}\ref{DefBDLocal_B} is always the ring of integers of a finite extension of $\BF_q\dpl\zeta\dpr$. For a detailed explanation of the definition of the reflex ring and a comparison with the number field case see \cite[Remark~4.7]{AH_Local}.  We do not know whether in general $\hat Z$ has a representative over the reflex ring. In contrast, the equivalence class of the $Z_R:=\hat{Z}_R\whtimes_{\Spf R}\Spec \kappa_R$ always has a representative $Z\subset\Flag_G\whtimes_{\BF_q}\Spec\kappa$ over the residue field $\kappa$ of the reflex ring $R_{\hat Z}$, because the Galois descent for closed ind-subschemes of $\Flag_G$ is effective. We call $Z$ the \emph{special fiber} of $\hat Z$. It is a projective scheme over $\kappa$ by \cite[Lemma~5.4]{HV1} because $\Flag_G$ is ind-projective.

\medskip\noindent 
(c) The condition of Definition~\ref{DefBDLocal}\ref{DefBDLocal_C} is satisfied for \emph{all} trivializations $\alpha$ and $\alpha'$ and for \emph{all} such finite extensions $R$ of $\BF_q\dbl\zeta\dbr$ if and only if it is satisfied for \emph{one} trivialization and for \emph{one} such finite extension. Indeed, by the $L^+G$-invariance of $\hat Z$ the definition is independent of the trivializations. That one finite extension suffices follows from \cite[Remark~4.6]{AH_Local}.

\medskip\noindent 
(d) At first glance one might think that conditions \ref{DefBDLocal_A1} and \ref{DefBDLocal_A11} of Definition~\ref{DefBDLocal} are related. However, in Example \ref{ExNotInvariant} we show that we really need to impose both of them.
\end{remark}

Before we discuss properties of bounds we recall the following well-known lemma.

\begin{lemma}\label{LemmaClosedImmersion}
Let $f\colon X\to Y$ be a morphism of locally noetherian adic formal schemes. Then $f$ is a closed immersion in the sense of \cite[I$_{\rm new}$, Definition~10.14.2]{EGA} if and only if $f$ is adic and an ind-closed immersion of ind-schemes.
\end{lemma}

\begin{proof}
By definition $f$ is a closed immersion if and only if there is a covering of $Y$ by open affine formal subschemes $\Spf B$ such that $X\times_Y\Spf B\cong\Spf B/\Fa$ for an ideal $\Fa\subset B$. In particular, if $I\subset B$ is a finitely generated ideal of definition of $\Spf B$ then $I\cdot B/\Fa$ is an ideal of definition of $\Spf B/\Fa$ and so $f$ is adic. Moreover, $\Spf B=\invlim\Spec B/I^n$ and $\Spf B/\Fa=\invlim\Spec B/(\Fa+I^n)$ and so $f$ is an ind-closed immersion of ind-schemes.

To prove the converse let $\CI\subset\CO_Y$ be an ideal sheaf of definition. Since $f$ is adic, $\CI\cdot\CO_X$ is an ideal sheaf of definition of $X$. That $f$ is an ind-closed immersion means that $X_n:=(X,\CO_X/\CI^n)\into Y_n:=(Y,\CO_Y/\CI^n)$ is a closed immersion of schemes. So there is a sheaf of ideals $\Fa_n\subset\CO_Y/\CI^n$ defining $X_n$. Moreover, $\Fa_n=\Fa_{n+1}\cdot\CO_Y/\CI^n$, because $X_n=X_{n+1}\times_{Y_{n+1}}Y_n$. Let $\Fa:=\invlim\Fa_n\subset\invlim\CO_Y/\CI^n=\CO_Y$. Since $Y$ is locally noetherian $\Fa$ is a coherent sheaf of $\CO_Y$-modules by \cite[I$_{\rm new}$, Theorem~10.10.2]{EGA}. Then $X_n=(X,\bigl(\CO_Y/(\Fa+\CI^n)\bigr)|_X)$ and $X=\bigl(X,(\CO_Y/\Fa)|_X\bigr)$. This proves that $f$ is a closed immersion in the sense of \cite[I$_{\rm new}$, Definition~10.14.2]{EGA}.
\end{proof}

\begin{remark}
Without the assumption that $f$ is adic the conclusion of the lemma is false as the following example shows. Let $Y=\dirlim Y_n$ with $Y_n=\Spec\BF_q\dbl\zeta\dbr[x]/(\zeta^n)$ and $X=\dirlim X_n$ with $X_n=\Spec\BF_q\dbl\zeta\dbr[x]/(\zeta,x)^n$. Then $X=\Spf\BF_q\dbl\zeta,x\dbr\to Y=\Spf\BF_q\dbl\zeta\dbr\langle x\rangle$, with the notation of \eqref{EqFormalTateAlgebra}, is an ind-closed immersion of ind-schemes, but not a closed immersion of formal schemes.
\end{remark}

In the next proposition we associate with a bound $\hat Z$ a strictly $\BF_q\dpl\zeta\dpr$-analytic space in the sense of Berkovich~\cite{Berkovich1,Berkovich2}. On the category of $\BF_q\dpl\zeta\dpr$-analytic spaces we consider the \'etale topology; see \cite[\S\,4.1]{Berkovich2}.

\begin{proposition}\label{PropBound}
Let $\hat Z=[(R,\hat Z_R)]$ be a bound with reflex ring $R_{\hat{Z}}$, and let $E:=E_{\hat Z}:=R_{\hat Z}[\tfrac{1}{\zeta}]$ be its field of fractions. We only assume that $Z$ satisfies conditions \ref{DefBDLocal_A1} -- \ref{DefBDLocal_A4} from Definition~\ref{DefBDLocal}, but not condition \ref{DefBDLocal_A11}, whose formulation uses the results of the present proposition.
\begin{enumerate}
\item \label{PropBound_A}
Then for every representation $\Darst\colon G\to\SL_{r}$ over $\BF_q\dbl z\dbr$ there is a positive integer $n$ such that all the induced morphisms $\Darst_*\colon\hat{Z}_R\into\wh\Flag_{G,R}\to\wh\Flag_{\SL_{r},R}$ factor through $\wh\Flag^{(n)}_{\SL_{r},R}$.
\item \label{PropBound_C}
$\wh\Flag^{(n)}_{\SL_{r}}$ is a $\zeta$-adic formal scheme over $\Spf\BF_q\dbl\zeta\dbr$. It is the $\zeta$-adic completion of a projective scheme over $\Spec\BF_q\dbl\zeta\dbr$ that we also denote by $\wh\Flag^{(n)}_{\SL_{r}}$. The corresponding strictly $\BF_q\dpl\zeta\dpr$-analytic space $\bigl(\wh\Flag^{(n)}_{\SL_{r}}\bigr)^\an$ is the analytification of the projective scheme $\wh\Flag^{(n)}_{\SL_{r}}\times_{\BF_q\dbl\zeta\dbr}\Spec\BF_q\dpl\zeta\dpr$ over $\Spec\BF_q\dpl\zeta\dpr$ that represents the sheaf of sets for the \'etale topology associated with the presheaf
\begin{eqnarray}\label{EqFlSLrAn}
X & \longmapsto & \Bigl\{\;g\mod\SL_r\bigl(\CO_X\dbl z-\zeta\dbr\bigr)\;\in\;\SL_r\bigl(\CO_X\dpl z-\zeta\dpr\bigr)/\SL_r\bigl(\CO_X\dbl z-\zeta\dbr\bigr)\colon\\
& & \quad\text{all $j\times j$-minors of $g$ lie in $(z-\zeta)^{n(j^2-jr)}\CO_X\dbl z-\zeta\dbr$ for all $j$ }\Bigr\}. \nonumber
\end{eqnarray}
The scheme $\wh\Flag^{(n)}_{\SL_{r}}\times_{\BF_q\dbl\zeta\dbr}\Spec\BF_q\dpl\zeta\dpr$ is a closed subscheme of the affine Grassmannian $\Gr_{\SL_r}^{\bB_\dR}$ from \eqref{EqH_G}.
\item \label{PropBound_B}
If $n$ and $\Darst$ are as in \ref{PropBound_A} such that $\Darst$ is faithful with quasi-affine quotient $\SL_{r}/G$ then all $\Darst_*\colon\hat{Z}_R\into\wh{\Flag}_{\SL_{r},R}^{(n)}$ are closed immersions of formal schemes over $\Spf R$ in the sense of \cite[I$_{\rm new}$, Definition~10.14.2]{EGA}.
\item \label{PropBound_D}
All $\hat{Z}_R$ are $\zeta$-adic formal schemes, projective over $\Spf R$. All their associated $R[\tfrac{1}{\zeta}]$-analytic spaces $(\hat{Z}_R)^\an$ arise by base change to $R[\tfrac{1}{\zeta}]$ from one strictly $E_{\hat Z}$-analytic space $\hat Z^\an:=(\hat Z_E)^\an$ associated with a projective scheme $\hat Z_E$ over $\Spec E_{\hat Z}$. The latter is a closed subscheme of the affine Grassmannian $\Gr_G^{\bB_\dR}\times_{\BF_q\dpl\zeta\dpr}\Spec E_{\hat Z}$ from \eqref{EqH_G}.
\end{enumerate}
\end{proposition}

\noindent
{\it Remark.} The proof of statements \ref{PropBound_B} and \ref{PropBound_D} uses the ind-projectivity of the affine flag variety $\Flag_G$. If $G$ is not parahoric, it is not clear to us whether $\Darst_*\colon\hat{Z}_R\into\wh{\Flag}_{\SL_{r},R}^{(n)}$ is a locally closed immersion of $\zeta$-adic formal schemes.

\begin{proof}[{Proof of Proposition~\ref{PropBound}}]
\ref{PropBound_A} Let $\Darst'\colon G\into\SL_{r'}$ and $n'$ be the representation and the integer from Definition~\ref{DefBDLocal}\ref{DefBDLocal_A} for which all $\Darst'_*\colon\hat{Z}_R\to\wh\Flag_{\SL_{r'},R}$ factor through $\wh\Flag^{(n')}_{\SL_{r'},R}$. Let $\wh{L\SL}_{r'\!,R}:=L\SL_{r'}\whtimes_{\BF_q}\Spf R$ and define $\wh{L\SL}_{r'\!,R}^{(n')}:=\wh{L\SL}_{r'\!,R}\whtimes_{\wh\Flag_{\SL_{r'},R}}\wh\Flag_{\SL_{r'},R}^{(n')}$. Then $\wh{L\SL}_{r'\!,R}^{(n')}(S)\;=\;$
\[
\bigl\{\,g\in \wh{L\SL}_{r'\!,R}(S)\colon \text{ all $j\times j$-minors of $g$ lie in $(z-\zeta)^{n'(j^2-j{r'})}\CO_S(S)\dbl z\dbr\es\forall\;j=1,\ldots,{r'}$}\,\bigr\}.
\]
This implies that $\wh{L\SL}_{r'\!,R}^{(n')}$ is an infinite-dimensional affine formal scheme over $\Spf R$. Thus its closed subscheme $\wh{LG}^{(n')}_R:=(LG\whtimes_{\BF_q}\Spf R)\whtimes_{\wh{L\SL}_{r'\!,R}}\wh{L\SL}_{r'\!,R}^{(n')}$ is also affine. By \cite[Remark~4.10]{AH_Local} the ind-schemes $\hat{Z}_R$ are in fact formal schemes over $\Spf R$ in the sense of \cite[I$_{\rm new}$]{EGA}. Since the morphism $LG\to\Flag_G$ has sections \'etale locally there is an \'etale covering of formal schemes $\hat{Z}_R'\to\hat{Z}_R$ such that the morphism $\hat{Z}_R'\to\wh\Flag_{G,R}$ factors through $\wh{LG}^{(n')}_R$. Let $\Spf A\subset\hat{Z}_R'$ be an affine open formal subscheme with $\Spf A=\dirlim \Spec A_i$ for some $A_i$. The induced compatible collection of morphisms $\Spec A_i\to\wh{LG}^{(n')}_R\into\wh{L\SL}_{r'\!,R}^{(n')}$ corresponds to a compatible collection of ring homomorphisms $\CO(\wh{L\SL}_{r'\!,R}^{(n')})\onto\CO(\wh{LG}^{(n')}_R)\to A_i$ and thus to a homomorphism $\CO(\wh{L\SL}_{r'\!,R}^{(n')})\onto\CO(\wh{LG}^{(n')}_R)\to A$. We view the latter as an element $b\in\SL_{r'}\bigl(A\dbl z\dbr[\tfrac{1}{z-\zeta}]\bigr)\cap\bigl((z-\zeta)^{-n'({r'}-1)}A\dbl z\dbr\bigr)^{r'\times r'}$. It actually lies in $G\bigl(A\dbl z\dbr[\tfrac{1}{z-\zeta}]\bigr)$, because the closed ind-subscheme $LG\into L\SL_{r'}$ is defined by the equations that applied to the entries of a matrix in $\SL_{r'}$ cut out the closed subgroup $\Darst'\colon G\into\SL_{r'}$.

If now $\Darst\colon G\to\SL_{r}$ is any representation over $\BF_q\dbl z\dbr$ then we claim that there is a positive integer $n$ that only depends on $\Darst$ and $n'$ such that $\Darst(b)\in\SL_{r}\bigl(A\dbl z\dbr[\tfrac{1}{z-\zeta}]\bigr)$ and all $j\times j$ minors of $\Darst(b)$ lie in $(z-\zeta)^{n(j^2-jr)}A\dbl z\dbr$ for all $j$. Indeed, equality for $j=r$ always holds as the image of $\Darst$ is in $\SL_{r}$. For the other $j$ we realize $\Darst$ as a subquotient of $\bigoplus_{i=1}^{i_0} (\Darst')^{\otimes l_i}\otimes (\Darst'{}\dual)^{\otimes m_i}$ for suitable $i_0$, $l_i$ and $m_i$. Then it is enough to show the claim for this direct sum. Here we can bound all minors by bounds only depending on $n'$, $i_0$, and the $l_i$ and $m_i$. The claim follows. Thus $\Darst_*\colon\hat{Z}_R\to\wh\Flag_{\SL_{r},R}$ factors through $\wh\Flag^{(n)}_{\SL_{r},R}$ because the equations defining the closed ind-subscheme $\wh\Flag^{(n)}_{\SL_{r},R}\subset\wh\Flag_{\SL_{r},R}$ vanish on the \'etale covering $\hat{Z}'_R\to\hat{Z}_R$. 

\medskip\noindent
\ref{PropBound_C} To show that $\wh\Flag^{(n)}_{\SL_r}(S)$ is projective over $\Spf\BF_q\dbl\zeta\dbr$ we use the equivalence between $L^+\SL_r$-torsors over $S$ and pairs $(M,\alpha)$ where $M$ is a locally free $\CO_S\dbl z\dbr$-module on $S$ and $\alpha\colon\wedge^r M\isoto\CO_S\dbl z\dbr$ is an isomorphism of $\CO_S\dbl z\dbr$-modules. Under this equivalence and using that $\CO_S\dbl z\dbr=\CO_S\dbl z-\zeta\dbr$ for all $S\in\Nilp_{\BF_q\dbl\zeta\dbr}$, we may identify $\wh\Flag^{(n)}_{\SL_r}(S)$ with the set
\begin{eqnarray}\label{EqPropBound_C-Proof}
& \Bigl\{\text{ locally free $\CO_S\dbl z-\zeta\dbr$-submodules $M\subset(z-\zeta)^{-n(r-1)}\CO_S\dbl z-\zeta\dbr^{\oplus r}$ such that for all $j=1,\ldots,r$ } \nonumber \\
& \TS\bigwedge^j_{\CO_S\dbl z-\zeta\dbr}M\;\subset\;(z-\zeta)^{n(j^2-jr)}\cdot\bigwedge^j_{\CO_S\dbl z-\zeta\dbr} \CO_S\dbl z-\zeta\dbr^{\oplus r}\text{ with equality for $j=r$ }\Bigr\}\,.
\end{eqnarray}
Note that the quotient $(z-\zeta)^{-n(r-1)}\CO_S\dbl z-\zeta\dbr^{\oplus r}/M$ is finite locally free as $\CO_S$-module by \cite[Lemma~4.3]{HV1}. From Cramer's rule (e.g.~\cite[III.8.6, Formulas (21) and (22)]{BourbakiAlgebra}) one sees that the above condition for $j=r-1$ implies that $M\supset(z-\zeta)^{n(r-1)}\CO_S\dbl z-\zeta\dbr^{\oplus r}$. By considering the image $\olM$ of $(z-\zeta)^{n(r-1)}M$ in $\bigl(\CO_S\dbl z-\zeta\dbr/(z-\zeta)^{2n(r-1)}\bigr)^{\oplus r}$ and using arguments similar to \cite[Lemma~2.7]{HartlHellmann}, see also \cite[Proposition~2.4.6]{SchauchDiss}, we obtain a closed embedding of $\wh\Flag_{\SL_r}^{(n)}$ into the formal $\zeta$-adic completion $$({\rm Quot}_{\CO^r\,|\,\Spec\BF_q\dbl\zeta\dbr[z]/(z-\zeta)^{2n(r-1)}\,|\,\Spec\BF_q\dbl\zeta\dbr})\times_{\Spec\BF_q\dbl\zeta\dbr}\Spf\BF_q\dbl\zeta\dbr$$ of Grothendieck's Quot-scheme whose points over a $\Spec\BF_q\dbl\zeta\dbr$-scheme $S$ are 
\begin{eqnarray*}
& & ({\rm Quot}_{\CO^r\,|\,\Spec\BF_q\dbl\zeta\dbr[z]/(z-\zeta)^{2n(r-1)}\,|\,\Spec\BF_q\dbl\zeta\dbr})(S)= \\[2mm]
& & \Bigl\{\text{ finitely presented $\CO_S[z]/(z-\zeta)^{2n(r-1)}$-submodules $\olM\subset\bigl(\CO_S[z]/(z-\zeta)^{2n(r-1)}\bigr)^{\oplus r}$} \\
& & \qquad\qquad\qquad\qquad \text{ whose quotient is finite locally free over $\CO_S$ }\Bigr\}\,;
\end{eqnarray*}
see \cite[n$\open$221, Theorem 3.1]{FGA} or \cite[Theorem 2.6]{AltmanKleiman}. By the projectivity of the Quot-scheme, $\wh\Flag_{\SL_r}^{(n)}$ is projective over $\Spf\BF_q\dbl\zeta\dbr$. From \eqref{EqPropBound_C-Proof} also the description of the sheaf represented by $\bigl(\wh\Flag^{(n)}_{\SL_r}\bigr)^\an$ follows.

\medskip\noindent
\ref{PropBound_B} If $\Darst\colon G\into\SL_{r}$ is a faithful representation with quasi-affine quotient $\SL_{r}/\Darst(G)$ then Pappas and Rapoport~\cite[Theorem~1.4]{PR2} show that the induced morphism $\Darst_*\colon\Flag_G\into\Flag_{\SL_{r}}$ is a locally ind-closed immersion of ind-schemes. Since $\Flag_G$ is ind-proper by \cite[Theorem~A]{Richarz13} this is even an ind-closed immersion. Since $\hat Z_R$ is a $\zeta$-adic formal scheme by  Definition~\ref{DefBDLocal}\ref{DefBDLocal_A3} all $\Darst_*\colon\hat{Z}_R\into\wh{\Flag}_{\SL_{r},R}^{(n)}$ are adic ind-closed immersions of formal schemes over $\Spf R$, and hence closed immersions by Lemma~\ref{LemmaClosedImmersion}.

\medskip\noindent
\ref{PropBound_D} By \cite[Proposition~1.3]{PR2} there is a faithful representation $\Darst\colon G\into\SL_{r}$ as in \ref{PropBound_A} with quasi-affine quotient $\SL_{r}/\Darst(G)$. Therefore all $\hat Z_R$ are projective over $\Spf R$ by \ref{PropBound_C} and \ref{PropBound_B}, and the associated strictly $R[\tfrac{1}{\zeta}]$-analytic space $(\hat Z_R)^\an$ is Zariski-closed in the projective $R[\tfrac{1}{\zeta}]$-analytic space $(\wh\Flag_{\SL_{r},R}^{(n)})^\an$. By analytic GAGA \cite[Theorem~2.8]{LuboFRG}, $(\hat Z_R)^\an$ is the analytification of a closed subscheme $\hat Z_{R[\frac{1}{\zeta}]}$ of the projective scheme $\wh\Flag^{(n)}_{\SL_{r}}\times_{\BF_q\dbl\zeta\dbr}\Spec R[\tfrac{1}{\zeta}]$. By \cite[Remark~4.7(f)]{AH_Local} there is an $R$ over which a representative $\hat Z_R$ exists, such that $R[\tfrac{1}{\zeta}]$ is Galois over $E_{\hat Z}$ and $\hat Z_R$ is invariant under $\Gal(R[\tfrac{1}{\zeta}]/E_{\hat Z})$. Since the Galois descent for projective $E_{\hat Z}$-schemes is effective by \cite[Chapitre~VIII, Corollaire~7.7]{SGA1}, the scheme $\hat Z_{R[\frac{1}{\zeta}]}$ and its analytification $(\hat Z_R)^\an$ descend to a projective scheme $\hat Z_E$ over $\Spec E_{\hat Z}$ and its associated strictly $E_{\hat Z}$-analytic space $(\hat Z_E)^\an$.

By our proof of \ref{PropBound_A} above there is an \'etale covering of $\hat{Z}_R$ formed by formal schemes $\Spf A$ on which a lift of the inclusion $\hat Z_R\into\wh\Flag_{G,R}$ to a point in $G\bigl(A\dbl z\dbr[\tfrac{1}{z-\zeta}]\bigr)$ exists, that is unique up to multiplication by $G(A\dbl z\dbr)$ on the right. Under the map $A\dbl z\dbr\into A[\tfrac{1}{\zeta}]\dbl z-\zeta\dbr$, $z\mapsto z=\zeta+(z-\zeta)$ this point gives rise to a point in $G\bigl(A[\tfrac{1}{\zeta}]\dpl z-\zeta\dpr\bigr)$. The latter induces a morphism $\Spec A[\tfrac{1}{\zeta}]\to \Gr_G^{\bB_\dR}\times_{\BF_q\dpl\zeta\dpr}\Spec R[\tfrac{1}{\zeta}]$ and we consider its analytification $(\Spf A)^\an\to \bigl(\Gr_G^{\bB_\dR}\times_{\BF_q\dpl\zeta\dpr}\Spec R[\tfrac{1}{\zeta}]\bigr)^\an$, that now descends to a morphism $(\hat{Z}_R)^\an\to \bigl(\Gr_G^{\bB_\dR}\times_{\BF_q\dpl\zeta\dpr}\Spec R[\tfrac{1}{\zeta}]\bigr)^\an$. By Galois descent it descends further to $(\hat Z_E)^\an$ and provides the closed immersion into $\Gr_G^{\bB_\dR}\times_{\BF_q\dpl\zeta\dpr}\Spec E_{\hat Z}$.
\end{proof}

\begin{example}\label{exboundmu}
We describe bounds that arise from a conjugacy class of coweights. Let $E_0$ be a finite field extension of $\mathbb{F}_q\dpl\zeta\dpr$ and let $\mu\colon\BG_{m,E_0}\to G_{E_0}:=G\times_{\BF_q\dbl z\dbr,z\mapsto\zeta}{E_0}$ be a coweight over ${E_0}$ of the generic fiber of $G$. Since the following construction only depends on the conjugacy class of $\mu$ we may assume that ${E_0}$ is separable over $\BF_q\dpl\zeta\dpr$ by \cite[\S\,8.11, Corollary~1]{Borel91}. Recall the affine Grassmannian $\Gr_G^{\bB_\dR}$ from \eqref{EqH_G}. We first define $\hat{Z}_{\preceq\mu,{E_0}}$ as the scheme theoretic closure in $\Gr_{G,{E_0}}^{\bB_\dR}:=\Gr_G^{\bB_\dR}\times_{\BF_q\dpl\zeta\dpr}\Spec {E_0}$ of 
\[
G({E_0}\dbl z-\zeta\dbr)\cdot\mu(z-\zeta)\cdot G({E_0}\dbl z-\zeta\dbr)\,\big/\,G({E_0}\dbl z-\zeta\dbr)\;\subset\;\Gr_{G,{E_0}}^{\bB_\dR}.
\]
That is, $\hat{Z}_{\preceq\mu,{E_0}}$ is the (reduced) closed Schubert variety associated with $\mu$. Choose a faithful representation $\Darst\colon G\into\SL_{r}$ over $\BF_q\dbl z\dbr$. Then there is a positive integer $n$ such that the induced morphisms $\Darst_*\colon\hat{Z}_{\preceq\mu,{E_0}}\into\Gr_{G,{E_0}}^{\bB_\dR}\into\Gr_{\SL_r,{E_0}}^{\bB_\dR}$ factors through the closed subscheme $\wh\Flag^{(n)}_{\SL_{r},{E_0}}:=\wh\Flag^{(n)}_{\SL_{r}}\times_{\BF_q\dbl\zeta\dbr}\Spec {E_0}\subset\Gr_{\SL_r,{E_0}}^{\bB_\dR}$ from \eqref{EqFlSLrAn} which is projective over $\Spec {E_0}$. We let $R$ be the integral closure of $\BF_q\dbl\zeta\dbr$ in ${E_0}$, and we let $\hat{Z}_R$ be a closed subscheme of the projective $R$-scheme
\begin{equation}\label{EqBDGrass}
\wh\Flag_{G,R}\underset{\wh\Flag_{\SL_r,R}}{\times} (\wh\Flag^{(n)}_{\SL_{r}}\times_{\BF_q\dbl\zeta\dbr}\Spec R)
\end{equation}
with $\hat{Z}_R\times_R\Spec {E_0}=\hat{Z}_{\preceq\mu,{E_0}}$. For example, one could take $\hat{Z}_R$ as the reduced closure of $\hat{Z}_{\preceq\mu,{E_0}}$, which is flat over $\BF_q\dbl\zeta\dbr$, and which we call $\hat{Z}_{\preceq\mu,R}$. In this case, it coincides with the global Schubert variety of \cite[Definition~3.5]{Richarz13}, compare Remark~\ref{remcompgsv}. However, this is not the only possible choice for $\hat{Z}_R$. 

By our assumption that $G$ is parahoric, $\wh\Flag_{G,R}$ is ind-projective by \cite[Theorem~A]{Richarz13} and the schemes \eqref{EqBDGrass} and $\hat{Z}_R$ are indeed projective over $\Spec R$. Therefore, the formal $\zeta$-adic completion of $\hat{Z}_R$ defines a bound $\hat{Z}$ whose associated strictly ${E_0}$-analytic space $(\hat{Z}_R)^\an$ arises as the analytification of $\hat{Z}_{E_0}$. If $\hat{Z}_R=\hat{Z}_{\preceq\mu,R}$ we denote the associated bound by $\hat{Z}_{\preceq\mu}$.

\medskip\noindent
{\it Claim.} The reflex field $E_{\hat{Z}_{\preceq\mu}}$ of $\hat{Z}_{\preceq\mu}$ is equal to the \emph{reflex field $E_\mu$ of the conjugacy class of $\mu$}, and is, in particular, separable over $\BF_q\dpl\zeta\dpr$. 

\medskip
Indeed, the field $E_\mu$ is defined as the fixed field in a separable closure ${E_0}^\sep$ of the group
\begin{equation}\label{EqReflexmu}
\bigl\{\,\gamma\in\Gal\bigl({E_0}^\sep\!/\BF_q\dpl\zeta\dpr\bigr)\colon \exists\;g\in G({E_0}^\sep)\text{ with } {}^\gamma\!\mu = \Int_g\circ\mu\,\bigr\}\,.
\end{equation}
The field $E_\mu$ is contained in ${E_0}$ and more generally, in every field over which a representative of the conjugacy class of $\mu$ exists, but these inclusions may be strict. To show that $E_{\hat{Z}_{\preceq\mu}}\subset E_\mu$ we note that $E_{\hat{Z}_{\preceq\mu}}\subset {E_0}\subset {E_0}^\sep$ and that every element of the group \eqref{EqReflexmu} satisfies $\gamma(\hat{Z}_{\preceq\mu})=\hat{Z}_{\preceq\mu}$, because $\hat{Z}_{\preceq\mu}$ is defined as the closure of $\hat{Z}_{\preceq\mu,{E_0}}$ and ${}^\gamma\!\mu(z-\zeta)=g\cdot\mu(z-\zeta)\cdot g^{-1}$ implies $\gamma(\hat{Z}_{\preceq\mu,{E_0}})=\hat{Z}_{\preceq\mu,{E_0}}$. For the opposite inclusion $E_\mu\subset E_{\hat{Z}_{\preceq\mu}}$ note that every $\gamma\in\Gal\bigl({E_0}^\sep\!/\BF_q\dpl\zeta\dpr\bigr)$ with $\gamma(\hat{Z}_{\preceq\mu})=\hat{Z}_{\preceq\mu}$ satisfies $\hat{Z}_{\preceq{}^\gamma\!\mu,{E_0}}=\gamma(\hat{Z}_{\preceq\mu,{E_0}})=\hat{Z}_{\preceq\mu,{E_0}}$. Since the Schubert varieties in $\Gr_{G,{E_0}^\sep}^{\bB_\dR}$ are in bijection with the $G({E_0}^\sep)$-conjugacy classes of cocharacters $\BG_{m,{E_0}^\sep}\to G_{{E_0}^\sep}$ we conclude that ${}^\gamma\!\mu$ is conjugate to $\mu$, and $\gamma$ lies in the group \eqref{EqReflexmu}. Therefore, $E_\mu\subset E_{\hat{Z}_{\preceq\mu}}$, and our claim is proved.

We describe the generic fiber of $\hat{Z}$. Let $L$ be the completion of an algebraic closure of ${E_0}$. Choose a maximal torus $T$ of $G_L$ through which $\mu$ factors, a Borel subgroup $B\supset T$ with respect to which $\mu$ is dominant, and consider all dominant $\mu'\in X_*(T)_\dom$ with $\mu'\preceq\mu$ in the Bruhat order. Then the $L$-valued points of $\hat Z^\an$ lie in the union of the $G(L\dbl z-\zeta\dbr)$-cosets
\[
\bigcup_{\mu'\preceq\mu}G(L\dbl z-\zeta\dbr)\cdot\mu'(z-\zeta)\cdot G(L\dbl z-\zeta\dbr)\,\big/\,G(L\dbl z-\zeta\dbr)\,.
\]
Compare also Remark~\ref{DefBDLocal'}\ref{DefBDLocal_A6}.

The special fiber of $\hat{Z}_{\preceq\mu}$ is discussed in \cite[p.~3739 ff.]{Richarz13} as a certain union of Schubert varieties.

Finally, since $\hat{Z}_{\preceq\mu}$ is defined as the closure of $\hat{Z}_{\preceq\mu,{E_0}}$, the bound $(\hat{Z}_{\preceq\mu})^{-1}$ from Lemma~\ref{LemmaInverseBound} equals $\hat{Z}_{\preceq(-\mu)}$ where the cocharacter $-\mu\colon\BG_{m,{E_0}}\to G_{E_0}$ is obtained from $\mu$ by precomposing with the inversion on $\BG_{m,{E_0}}$.
\end{example}

\begin{example}\label{ExConstantG}
We explain the relation of boundedness in Definition~\ref{DefBDLocal} to the definition from \cite{HV1}. Consider a split reductive group $G_0$ over $\BF_q$, and set $G:=G_0\times_{\BF_q}\Spec\BF_q\dbl z\dbr$. Let $T\subset G_0$ be a maximal split torus over $\BF_q$. Let $B$ be a Borel subgroup containing $T$ and $\olB$ its opposite Borel. Let $\mu\in X_*(T)_\dom$ be a coweight that is dominant with respect to $B$. In \cite[Definition~3.5]{HV1} we define ``boundedness by $(\mu,z-\zeta)$'' as follows. We consider a finite generating system $\Lambda$ of the monoid of dominant weights $X^*(T)_\dom$, and for all $\lambda\in\Lambda$ the Weyl module $V_\lambda:=\bigl(\Ind_\olB^{G_0}(-\lambda)_\dom\bigr)\dual$. Here $(-\lambda)_\dom$ is the dominant representative in the Weyl group orbit of $-\lambda$. Let $\CG$ and $\CG'$ be $L^+G$-torsors over a scheme $S\in \Nilp_{\BF_q\dbl\zeta\dbr}$ and let $\delta\colon L\CG\isoto L\CG'$ be an isomorphism of the associated $LG$-torsors. For the representation $\Darst_\lambda\colon G\to\GL(V_\lambda)$ in $\Rep_{\BF_q\dbl z\dbr}G$ we consider the sheaves of $\CO_{S}\dbl z\dbr$-modules $\Darst_{\lambda*}\CG$ and $\Darst_{\lambda*}\CG'$ associated with the $L^+G$-torsors $\CG$ and $\CG'$ over $S$. The isomorphism $\delta$ induces an isomorphism $\Darst_{\lambda*}\delta\colon\Darst_{\lambda*}\CG\otimes_{\CO_S\dbl z\dbr}\CO_S\dpl z\dpr\isoto\Darst_{\lambda*}\CG'\otimes_{\CO_S\dbl z\dbr}\CO_S\dpl z\dpr$. After choosing trivializations of $\CG$ and $\CG'$ over an \'etale covering $S'\to S$, the isomorphism $\delta$ is given by multiplication with an element $\delta_{S'}\in LG(S')$. The latter corresponds to a morphism $S'\to LG$. Let $LG_{\mu^\#}\subset LG$ and $\CF_{\mu^\#}\subset\wh\Flag_G$ be the connected components corresponding to the image $\mu^{\#}\in\pi_1(G)_I=\pi_1(G)=\pi_1(G_0)=\pi_0(\wh\Flag_G)=\pi_0(LG)$ of $\mu$. According to \cite[Definition~3.5]{HV1} ``$\delta$ is bounded by $(\mu,z-\zeta)$'' if
\begin{itemize}
\item the morphism $S'\to LG$ factors through $LG_{\mu^\#}$ and 
\item $\Darst_{\lambda*}\delta\,(\Darst_{\lambda*}\CG)\;\subset\;(z-\zeta)\;^{-\langle\,(-\lambda)_\dom,\mu\rangle}\cdot\Darst_{\lambda*}\CG'$ for all $\lambda\in\Lambda$.
\end{itemize}
In terms of Definition~\ref{DefBDLocal} this can be described as follows. Consider the ind-scheme
\[
Y_\lambda\;:=\;L\GL(V_\lambda)\whtimes_{\BF_q}\Spf\BF_q\dbl\zeta\dbr\;=\;\dirlim Y_{\lambda,m}\qquad \text{for}\qquad Y_{\lambda,m}\;:=\;L\GL(V_\lambda)\times_{\BF_q}\Spec\BF_q\dbl\zeta\dbr/(\zeta^m)
\]
and let $M_\lambda\in L\GL(V_\lambda)(Y_\lambda)$ be the universal element over $Y_{\lambda}$. Since $\CO_{Y_{\lambda,m}}\dbl z\dbr[z^{-1}]=\CO_{Y_{\lambda,m}}\dbl z-\zeta\dbr[\tfrac{1}{z-\zeta}]$ we can write $M_\lambda=(M_{\lambda,m})_m$ with
\[
M_{\lambda,m}\;\in\;L\GL(V_\lambda)(Y_{\lambda,m})\;=\;\GL(V_\lambda)\bigl(\CO_{Y_{\lambda,m}}\dbl z\dbr[z^{-1}]\bigr)\;=\;\GL(V_\lambda)\bigl(\CO_{Y_{\lambda,m}}\dbl z-\zeta\dbr[\tfrac{1}{z-\zeta}]\bigr)\,.
\]
Let $\olY_\lambda\subset Y_{\lambda}$ be the closed ind-subscheme where the matrix $(z-\zeta)^{\langle\,(-\lambda)_\dom,\mu\rangle}M_\lambda$ has entries in $\CO_{Y_\lambda}\dbl z-\zeta\dbr=\CO_{Y_{\lambda,m}}\dbl z\dbr$ for all $m$, and let $\olY'_\lambda\subset Y_{\lambda}$ be the closed ind-subscheme where the matrix $(z-\zeta)^{\langle\,(-\lambda)_\dom,\mu\rangle}M_\lambda^{-1}$ has entries in $\CO_{Y_{\lambda,m}}\dbl z-\zeta\dbr=\CO_{Y_{\lambda,m}}\dbl z\dbr$ for all $m$. Set
\begin{alignat*}{2}
& \hat Z^{{\SSC\rm Weyl}}_{\preceq\mu,\lambda} && \;:=\; \olY_\lambda/(L^+\GL(V_\lambda)\whtimes_{\BF_q}\Spf\BF_q\dbl\zeta\dbr)\;\subset\;\wh\Flag_{\GL(V_\lambda)} \qquad\text{and} \\[2mm]
& \hat Z^{{\SSC\rm Weyl},-1}_{\preceq\mu,\lambda} && \;:=\; \olY'_\lambda/(L^+\GL(V_\lambda)\whtimes_{\BF_q}\Spf\BF_q\dbl\zeta\dbr)\;\subset\;\wh\Flag_{\GL(V_\lambda)}\,.
\end{alignat*}
Write $\Lambda=\{\lambda_1,\ldots,\lambda_m\}$ and for each $\lambda_i$ consider the morphism $\Darst_{\lambda_i*}\colon \wh\Flag_G\to\wh\Flag_{\GL(V_{\lambda_i})}$ induced from $\Darst_{\lambda_i}$. Let $\hat Z^{{\SSC\rm Weyl}}_{G_0,\preceq\mu}\subset\CF_{\mu^\#}$ be the base change of the closed ind-subscheme
$$
\hat Z^{{\SSC\rm Weyl}}_{\preceq\mu,\lambda_1}\whtimes_{\Spf\BF_q\dbl\zeta\dbr}\ldots\whtimes_{\Spf\BF_q\dbl\zeta\dbr}\hat Z^{{\SSC\rm Weyl}}_{\preceq\mu,\lambda_m}
$$
under the morphism $\prod_i\Darst_{\lambda_i*}\colon \CF_{\mu^\#}\to\wh\Flag_{\GL(V_{\lambda_1})}\whtimes_{\BF_q\dbl\zeta\dbr}\ldots\whtimes_{\BF_q\dbl\zeta\dbr}\wh\Flag_{\GL(V_{\lambda_m})}$. Likewise, let $\hat Z^{{\SSC\rm Weyl},-1}_{G_0,\preceq\mu}\subset\CF_{-\mu^\#}$ be the base change of the closed ind-subscheme
$$
\hat Z^{{\SSC\rm Weyl},-1}_{\preceq\mu,\lambda_1}\whtimes_{\Spf\BF_q\dbl\zeta\dbr}\ldots\whtimes_{\Spf\BF_q\dbl\zeta\dbr}\hat Z^{{\SSC\rm Weyl},-1}_{\preceq\mu,\lambda_m}
$$
under the morphism $\prod_i\Darst_{\lambda_i*}\colon \CF_{-\mu^\#}\to\wh\Flag_{\GL(V_{\lambda_1})}\whtimes_{\BF_q\dbl\zeta\dbr}\ldots\whtimes_{\BF_q\dbl\zeta\dbr}\wh\Flag_{\GL(V_{\lambda_m})}$. The ind-subscheme $\hat Z^{{\SSC\rm Weyl},-1}_{G_0,\preceq\mu}\subset\CF_{-\mu^\#}$ was denoted $\wh\Gr^{\preceq(\mu,z-\zeta)}$ in \cite[Definition~5.5]{HV1}. We will show that both $\hat Z^{{\SSC\rm Weyl}}_{G_0,\preceq\mu}$ and $\hat Z^{{\SSC\rm Weyl},-1}_{G_0,\preceq\mu}$ define bounds in the sense of Definition~\ref{DefBDLocal}\ref{DefBDLocal_A} with reflex ring $\BF_q\dbl\zeta\dbr$ and are representatives defined over the reflex ring. In terms of Lemma~\ref{LemmaInverseBound} they satisfy $\hat Z^{{\SSC\rm Weyl},-1}_{G_0,\preceq\mu}=\bigl(\hat Z^{{\SSC\rm Weyl}}_{G_0,\preceq\mu}\bigr)^{-1}$ and $\hat Z^{{\SSC\rm Weyl}}_{G_0,\preceq\mu}=\bigl(\hat Z^{{\SSC\rm Weyl},-1}_{G_0,\preceq\mu}\bigr)^{-1}$. Conditions~\ref{DefBDLocal_A1}, respectively \ref{DefBDLocal_A11} are satisfied because all the $\hat Z^{{\SSC\rm Weyl}}_{\preceq\mu,\lambda}$ and $\hat Z^{{\SSC\rm Weyl},-1}_{\preceq\mu,\lambda}$ are invariant by multiplication on the left with $L^+G$, respectively with $G\bigl(\fdot\dbl z-\zeta\dbr\bigr)$. Conditions~\ref{DefBDLocal_A2} and \ref{DefBDLocal_A3} for $\hat Z^{{\SSC\rm Weyl},-1}_{G_0,\preceq\mu}$ follow from \cite[Proposition~5.5]{HV1}, and for $\hat Z^{{\SSC\rm Weyl}}_{G_0,\preceq\mu}$ from Lemma~\ref{LemmaInverseBound}. This proposition also says that the underlying reduced subscheme of $\hat Z^{{\SSC\rm Weyl},-1}_{G_0,\preceq\mu}$ is the closed Schubert variety associated with $(-\mu)_\dom$ in $\Flag_G$ and the underlying reduced subscheme of $\hat Z^{{\SSC\rm Weyl}}_{G_0,\preceq\mu}$ is the closed Schubert variety associated with $\mu$. If $V:=\bigoplus_{\lambda\in\Lambda}V_\lambda$ and $\Darst$ is the representation of $G_0$ on $V\oplus(\det V)\dual$ that is faithful by \cite[Proposition~3.14]{HV1} and factors through $\SL_{1+\dim V}$, then $\hat Z^{{\SSC\rm Weyl}}_{G_0,\preceq\mu}$ is contained in $\wh\Flag^{(n)}_{\SL_{1+\dim V}}$ for $n=\max\bigl\{\langle (-\lambda)_\dom,\mu\rangle\colon\lambda\in\Lambda\bigr\}$, that is, Condition~\ref{DefBDLocal_A4} holds for $\hat Z^{{\SSC\rm Weyl}}_{G_0,\preceq\mu}$. Then it also holds for $\hat Z^{{\SSC\rm Weyl},-1}_{G_0,\preceq\mu}$ by Lemma~\ref{LemmaInverseBound}.

Now $\delta$ ``is bounded by $(\mu,z-\zeta)$'' in the sense of \cite[Definition~3.5]{HV1} if and only if the morphism $S'\to LG\whtimes_{\BF_q}\Spf\BF_q\dbl\zeta\dbr\onto\wh\Flag_G$ given by $\delta_{S'}\in LG(S')$ factors through $\hat Z^{{\SSC\rm Weyl}}_{G_0,\preceq\mu}$, that is if and only if $\delta$ is bounded by $\hat Z^{{\SSC\rm Weyl}}_{G_0,\preceq\mu}$ in the sense of Definition~\ref{DefBDLocal}\ref{DefBDLocal_C}. This is the case if and only if the morphism $S'\to LG\whtimes_{\BF_q}\Spf\BF_q\dbl\zeta\dbr\onto\wh\Flag_G$ given by $\delta_{S'}^{-1}\in LG(S')$ factors through $\hat Z^{{\SSC\rm Weyl},-1}_{G_0,\preceq\mu}$, that is if and only if $\delta^{-1}$ is bounded by $\hat Z^{{\SSC\rm Weyl},-1}_{G_0,\preceq\mu}$. In particular, a local $G$-shtuka is bounded by $\hat Z^{{\SSC\rm Weyl}}_{G_0,\preceq\mu}$ in the sense of Definition~\ref{DefBDLocal}\ref{DefBDLocal_C} if and only if it is ``bounded by $(\mu,z-\zeta)$''. 

For the constant split group $G=G_0\times_{\BF_q}\Spec\BF_q\dbl z\dbr$ the double cosets occurring in the description of the sets of closed points of the generic and of the special fiber of bounds are both parametrized by the set $X_*(T)_{\dom}$, and for our bound $\hat Z^{{\SSC\rm Weyl}}_{G_0,\preceq\mu}$, both the generic fiber and the special fiber correspond to the union of all double cosets for $\mu'\preceq \mu$. That is, the reduced generic, respectively the reduced special fiber of $\hat Z^{{\SSC\rm Weyl}}_{G_0,\preceq\mu}$ equals the  closed Schubert variety associated with $\mu$ in $\Gr_{G}^{\bB_\dR}$, respectively in $\Flag_G$. So the underlying reduced structure of $\hat Z^{{\SSC\rm Weyl}}_{G_0,\preceq\mu}$ coincides with the bound $\hat Z_{\preceq\mu}$ defined in Example~\ref{exboundmu} above, and in terms of Remark~\ref{DefBDLocal'} condition \ref{DefBDLocal_A5} and $N_0=N_\an$ from \ref{DefBDLocal_A6} are satisfied. The nilpotent structure as discussed in Remark~\ref{DefBDLocal'}\ref{DefBDLocal_A7} is in general not so clear in this case. For example, if $G_0=\GL_r$ and $\mu=2n\rho\dual$ where $2\rho\dual=(r-1,\ldots,1-r)$ is the sum of the positive coroots of $G_0$, then $(-\mu)_\dom=\mu$ and all bounds $\hat Z^{{\SSC\rm Weyl}}_{\GL_r,\preceq\mu}, \hat Z^{{\SSC\rm Weyl},-1}_{\GL_r,\preceq\mu},\hat Z^{{\SSC\rm Weyl}}_{\GL_r,\preceq(-\mu)_\dom}\subset\wh\Flag_{\GL_r}$ are equal to $\wh\Flag_{\SL_r}^{(n)}\subset\wh\Flag_{\SL_r}\subset\wh\Flag_{\GL_r}$ by \cite[Lemma~4.3]{HV1}. Note however, that in general $\bigl(\hat Z^{{\SSC\rm Weyl}}_{G_0,\preceq\mu}\bigr)^{-1}=\hat Z^{{\SSC\rm Weyl},-1}_{G_0,\preceq\mu}$ coincides with $\hat Z^{{\SSC\rm Weyl}}_{G_0,\preceq(-\mu)_\dom}$ only on the underlying reduced structure, but not in the nilpotent structure; see Example~\ref{Ex-mu_dom}.
\end{example}

\begin{example}\label{ExNotInvariant}
We give an example of an ind-scheme $\hat Z$ satisfying all properties of a bound in Definition~\ref{DefBDLocal} except for \ref{DefBDLocal_A11} to show that this condition is not implied by \ref{DefBDLocal_A1}. Let $\hat Z\subset\wh\Flag^{(1)}_{\SL_2}$ be the ind-closure of $\hat Y\subset\wh\Flag^{(1)}_{\SL_2}$ given by
\[
\hat Y(B)\;:=\;L^+\SL_2(B)\cdot\bigl\{\left(\begin{smallmatrix} a & b \\ c & d \end{smallmatrix}\right)\in L\SL_2(B)\colon a,b,d\in B\dbl z\dbr, c\in \tfrac{\zeta}{z-\zeta}B\dbl z\dbr\,\bigr\}\cdot L^+\SL_2(B)/L^+\SL_2(B)
\]
for any $\BF_q\dbl\zeta\dbr$-algebra $B$. This satisfies all conditions in the definition of bounds except for possibly \ref{DefBDLocal_A11}. Notice further that the special fiber of $\hat Z$ consists just of the one point $L^+\SL_2/L^+\SL_2$. 

We have that $x=\left(\begin{smallmatrix} 1 & 0 \\ \tfrac{\zeta}{z-\zeta} & 1 \end{smallmatrix}\right)$ and $y=\left(\begin{smallmatrix} z-\zeta & 0 \\ 0 & \tfrac{1}{z-\zeta} \end{smallmatrix}\right)$ are elements of $L\SL_2(B)$ for all $B$ such that $\zeta$ is nilpotent in $B$. Therefore, $x\in\hat Z(\Spf \BF_q\dbl\zeta\dbr)$ as one can see by reducing modulo $(\zeta^i)$ for all $i$. However, already considering the reduction modulo $\zeta$ shows that $y$ is not an element of $\hat Z(\BF_q\dbl\zeta\dbr)$. On the other hand, $$y\;=\;\left(\begin{smallmatrix} 1 & -\zeta^{-1}(z-\zeta) \\ 0 & 1 \end{smallmatrix}\right)\cdot x\cdot\left(\begin{smallmatrix} z-\zeta & \zeta^{-1} \\ -\zeta & 0 \end{smallmatrix}\right)\;\in\; \SL_2(\BF_q\dpl\zeta\dpr\dbl z-\zeta\dbr)\cdot x\cdot\SL_2(\BF_q\dpl\zeta\dpr\dbl z-\zeta\dbr).$$ This shows that $\hat Z^\an$ is not invariant under multiplication with $\SL_2(\fdot\dbl z-\zeta\dbr)$ on the left.
\end{example}

We need condition \ref{DefBDLocal_A11} of Definition~\ref{DefBDLocal} mainly in form of the following lemma.

\begin{lemma}\label{LemBoundConjInv}
Let $\hat Z$ be a bound. Then $\hat Z^{\an}\otimes_{E_{\hat Z}}\breve E_{\hat Z}$ is invariant under left and right multiplication  with $LG(\BaseFld)$.  
\end{lemma}
Here, as always, we use the morphism $\BaseFld\dpl z\dpr\to \CO_X\dbl z-\zeta\dbr,\,z\mapsto z=\zeta+(z-\zeta)$ for an $\breve E_{\hat Z}$-scheme $X$ and the induced homomorphism $LG(\BaseFld)=G\bigl(\BaseFld\dpl z\dpr\bigr)\to G\bigl(\CO_X\dbl z-\zeta\dbr\bigr)$ to define the actions. 
\begin{proof}
This follows directly from the right invariance of $\hat Z^{\an}$ as a subscheme of the affine Grassmannian, and the left invariance imposed in condition \ref{DefBDLocal_A11} of Definition~\ref{DefBDLocal}.
\end{proof}

\begin{remark}\label{DefBDLocal'} 
We discuss some possible additional assumptions on bounds.
\begin{enumerate}
\item \label{DefBDLocal_A5}
The set $\pi_0(\wh\Flag_{G,R})$ coincides with the set of $\Gamma$-orbits in $\pi_1(G)_I$ by \cite[Lemma~2.2.6]{N16}. Every bound is a disjoint union of its intersections with the various connected components of $\wh\Flag_{G,R}$, and thus we also obtain similar decompositions of base schemes for bounded local $G$-shtukas, and later of the corresponding moduli spaces. If one wants to consider only one of these disjoint parts at a time, one has to assume that there is a $\Gamma$-orbit of elements $\xi\in\pi_1(G)_I$ such that the $\hat Z_R$ are contained in the corresponding connected component of $\wh\Flag_{G,R}$. 

Also compare the non-emptiness condition of Theorem \ref{ThmWAOpen} that is in terms of $\pi_1(G)_{\Gamma}$ instead of $\pi_1(G)_{I}/\Gamma$. The natural projection map $\pi_1(G)_{I}/\Gamma\rightarrow\pi_1(G)_{\Gamma}$ is surjective but in general not injective. Thus there may be several connected components of a given bound that lead to non-empty parts of the period domain.

\item \label{DefBDLocal_A6}
If one wants to compare properties of the generic and the special fiber of a moduli space of local $G$-shtukas bounded by $\hat Z^{-1}$ (as defined in Section \ref{secrz}) it might be useful to consider bounds $\hat Z$ satisfying certain flatness or extension properties. 

Assumptions \ref{DefBDLocal_A1} and \ref{DefBDLocal_A11} of Definition~\ref{DefBDLocal} imply that the closed points of the special fiber $Z$ from Remark~\ref{RemBdLocal}(b), resp.\ of the analytic space $\hat Z^\an$ of $\hat Z$ from Proposition~\ref{PropBound}\ref{PropBound_D}, consist of the points of a finite set $N_0$ of $L^+G$-cosets, resp.~a finite set $N_{\an}$ of $G\bigl(\fdot\dbl z-\zeta\dbr\bigr)$-cosets. The former cosets are parametrized by some quotient $\bar W$ of the extended affine Weyl group $\widetilde W$ of $G$ with respect to a chosen maximal torus $T$ of $G_{\BF_q\dpl z\dpr}$. Let $L_0\supset\BF_q\dpl\zeta\dpr$ be a finite separable field extension over which $T_{L_0}:=T\times_{\BF_q\dpl z\dpr,z\mapsto\zeta}\Spec L_0$ (and therefore also $G_{L_0}$) splits. Because $G\bigl(L_0\dbl z-\zeta\dbr\bigr)\subset G\bigl(L_0\dpl z-\zeta\dpr\bigr)$ is a hyperspecial maximal bounded open subgroup, every $G\bigl(L_0\dbl z-\zeta\dbr\bigr)$-coset is of the form $G\bigl(L_0\dbl z-\zeta\dbr\bigr)\mu(z-\zeta)G\bigl(L_0\dbl z-\zeta\dbr\bigr)\big/G\bigl(L_0\dbl z-\zeta\dbr\bigr)$ for a uniquely determined $\mu\in X_*(T_{L_0})_{\dom}= X_*(T)_{\dom}$ where dominance is with respect to some chosen Borel subgroup.  It would thus be interesting to see if such flatness conditions can be formulated in terms of the associated sets $N_0\subset \bar W$ and $N_{\an}\subset X_*(T)_{\dom}$. However, for the present paper we do not need any such condition. 

\item \label{DefBDLocal_A7} Given the discussion above, one might ask if bounds should just be defined by the corresponding sets of double cosets of their geometric points. Even under assumptions as discussed above, our definition gives some more freedom in the sense that the nilpotent structure of the bound may still vary.
\end{enumerate}
\end{remark}

For the sake of completeness we want to add the following lemma that was used in Remark~\ref{RemBdLocal}(a).

\begin{lemma}\label{LemmaInverseBound}
Let $\hat{Z}=[(R,\hat Z_R)]$ be a bound in the sense of \cite[Definition~4.8]{AH_Local}, that is, satisfying conditions \ref{DefBDLocal_A1} and \ref{DefBDLocal_A2} from Definition~\ref{DefBDLocal}. Consider the subsheaves $\hat{Z}_R^{-1}\subset\wh\Flag_{G,R}$ defined on schemes $S$ in $\Nilp_R$ as the subset $\hat{Z}_R^{-1}(S)$ of $\wh\Flag_{G,R}(S)$ given by 
\begin{eqnarray*}
 & \Bigl\{\;x\in\wh\Flag_{G,R}(S)\colon\text{ there is an \'etale covering $f\colon S'\to S$ and an element $g\in LG(S')$} \\
& \quad \text{ such that }f^*x = g\cdot L^+G(S')\text{ in }\wh\Flag_{G,R}(S')\text{ and } g^{-1}\cdot L^+G(S')\in\hat{Z}_R(S')\;\Bigr\}\,.
\end{eqnarray*}
Then $\hat Z^{-1}=[(R,\hat Z_R^{-1})]$ is a bound in the sense of \cite[Definition~4.8]{AH_Local} with the same reflex ring $R_{\hat Z}$ as $\hat{Z}$. If in addition, $\hat{Z}$ satisfies condition \ref{DefBDLocal_A3} and \ref{DefBDLocal_A4} from Definition~\ref{DefBDLocal}, then the same is true for $\hat{Z}^{-1}$. Moreover, for two $L^+G$-torsors $\CG$ and $\CG'$ over a scheme $S\in \Nilp_{R_{\hat Z}}$, an isomorphism $\delta\colon L\CG\isoto L\CG'$ between the associated $LG$-torsors is bounded by $\hat{Z}$ if and only if $\delta^{-1}$ is bounded by $\hat{Z}^{-1}$.
\end{lemma}

We do not know whether the same is true for condition \ref{DefBDLocal_A11} from Definition~\ref{DefBDLocal} in general. It is true, however, for the bounds in Example~\ref{exboundmu}.

\begin{proof}[Proof of Lemma~\ref{LemmaInverseBound}]
The subset is well defined, because $\hat{Z}_R$ is by Definition~\ref{DefBDLocal}\ref{DefBDLocal_A1} invariant under multiplication with $L^+G$ on the left. We fix a faithful representation $\Darst\colon G\into\SL_r$ and consider the induced morphism $\rho_*\colon\Flag_G\into\Flag_{\SL_r}$. The ind-scheme structure on $\Flag_{G,R}$ is given as the inductive limit of the schemes 
\[
X_n\;:=\;X_{n,n}\quad\text{with}\quad X_{n,m}\;:=\;\wh\Flag_G\,\whtimes_{\wh\Flag_{\SL_r}}\,\wh\Flag^{(n)}_{\SL_r}\,\whtimes_{\BF_q\dbl\zeta\dbr}\,\Spec R/(\zeta^m),
\]
where $\wh\Flag^{(n)}_{\SL_r}$ was defined in \eqref{EqBDSL}. We must show that $\hat{Z}_R^{-1}\times_{\wh\Flag_{G,R}}X_n$ is representable by a closed subscheme of $X_n$. Let $Y_n:=X_n\whtimes_{\wh\Flag_{G,R}}\bigl(LG\whtimes_{\BF_q}\Spec R/(\zeta^n)\bigr)$. Since $X_n\whtimes_{\wh\Flag_{G,R}}\hat{Z}_R\subset X_n$ is a closed subscheme, the base change $Z_n:=Y_n\whtimes_{X_n}X_n\whtimes_{\wh\Flag_{G,R}}\hat{Z}_R$ is a closed subscheme of $Y_n$, on which $L^+G_{R/(\zeta^n)}:=L^+G\whtimes_{\BF_q}\Spec R/(\zeta^n)$ acts by multiplication on the left. By Cramer's rule (e.g.~\cite[III.8.6, Formulas (21) and (22)]{BourbakiAlgebra}) the inversion $g\mapsto g^{-1}$ on $LG$ induces isomorphisms of $X_n$ with the quotient sheaf $L^+G_{R/(\zeta^n)}\backslash Y_n$ and of $\hat{Z}_R^{-1}\times_{\wh\Flag_{G,R}}X_n$ with $L^+G_{R/(\zeta^n)}\backslash Z_n$. Therefore it suffices to show that the morphism $L^+G_{R/(\zeta^n)}\backslash Z_n\to L^+G_{R/(\zeta^n)}\backslash Y_n$ of sheaves is representable by a closed immersion. The latter follows by \fpqc{} descent \cite[\S\,6.1, Theorem~6]{BLR} and \cite[IV$_2$, Proposition~2.7.1]{EGA} from the fact that the diagram
\[
\xymatrix @C=0.4pc @R=0.4pc {
Z_n \ar@{^{ (}->}[rr] \ar[dd] & & Y_n \ar[dd] \\
& \TS\qed\quad & \\
L^+G_{R/(\zeta^n)}\backslash Z_n \ar[rr] & & L^+G_{R/(\zeta^n)}\backslash Y_n
}
\]
is cartesian and the right vertical arrow is an affine faithfully flat morphism of schemes. This proves that $\hat{Z}^{-1}_R\subset \wh\Flag_{G,R}$ is representable by a closed ind-subscheme. By construction it satisfies the invariance under left multiplication by $L^+G$ from Definition~\ref{DefBDLocal}\ref{DefBDLocal_A1}.

To show that $\hat{Z}^{-1}_R$ satisfies Definition~\ref{DefBDLocal}\ref{DefBDLocal_A2} let $S'$ be an \'etale covering of the special fiber $Z_R$ of $\hat{Z}_R$, such that the closed immersion $Z_R\into\wh\Flag_{G,R}\whtimes_R\Spec\kappa_R$ lifts to a morphism $S'\to LG\whtimes_{\BF_q}\Spec\kappa_R$, which we view as an element $g^{-1}\in LG(S')$. Since $Z_R$ is quasi-compact, we may choose $S'$ to be quasi-compact. Then the element $g\cdot L^+G(S')$ in $\wh\Flag_{G,R}(S')$ corresponds to a morphism $S'\to \wh\Flag_{G,R}\whtimes_R\Spec\kappa_R$ that factors through $Z_R^{-1}$. The morphism $L^+G\whtimes_{\BF_q}S'\onto Z_R^{-1}$, $(h,g)\mapsto hg\cdot L^+G$ is a surjective morphism of ind-schemes. Therefore $Z_R^{-1}$ is a quasi-compact scheme.

If $\hat{Z}$ satisfies conditions \ref{DefBDLocal_A3} and \ref{DefBDLocal_A4} from Definition~\ref{DefBDLocal} for some $\Darst\colon G\into\SL_r$ and some $n$, then also $\hat{Z}^{-1}$ satisfies \ref{DefBDLocal_A4} for the same $\Darst$ and $n$ by Cramer's rule (e.g.~\cite[III.8.6, Formulas (21) and (22)]{BourbakiAlgebra}). Then $\hat{Z}_R^{-1}=\lim\limits_{\longto m} \hat{Z}_R^{-1}\times_{\wh\Flag_{G,R}}X_{n,m}$ for constant $n$ and variable $m$. In particular, $\hat{Z}_R^{-1}=\lim\limits_{\longto m} \hat{Z}_R^{-1}\times_R\Spec R/(\zeta^m)$ is $\zeta$-adic, that is, satisfies condition \ref{DefBDLocal_A3}.

The equality of reflex rings follows from the fact that $\gamma(\hat{Z})=\hat{Z}$ if and only if $\gamma(\hat{Z}^{-1})=\hat{Z}^{-1}$ for $\gamma\in\Aut_{\BF_q\dbl\zeta\dbr}(\BF_q\dpl\zeta\dpr^\alg)$. Finally the statement about the boundedness of $\delta$ and $\delta^{-1}$ is clear from the definition of the $\hat{Z}_R^{-1}$.
\end{proof}

\begin{example}\label{Ex-mu_dom}
We revert to Example~\ref{ExConstantG}. If $G_0=\GL_r$ or $G_0=\SL_r$ then $$\hat{Z}^{\SSC\rm Weyl}_{G_0,\,\preceq\mu}\;=\;(\hat{Z}^{{\SSC\rm Weyl},-1}_{G_0,\,\preceq\mu})^{-1}\;\stackrel{!}{=}\;\hat{Z}^{{\SSC\rm Weyl},-1}_{G_0,\preceq(-\mu)_\dom}$$ by Cramer's rule (e.g.~\cite[III.8.6, Formulas (21) and (22)]{BourbakiAlgebra}). However, this is not true for general $G_0$. For example let $G_0=\PGL_2$ and $\charakt(\BF_q)=2$. We choose $\Lambda$ to consist of the only positive root $\alpha$. The corresponding Weyl module $V_\alpha$ is the dual of the adjoint representation. With respect to the decomposition $V_\alpha=(\Lie T\oplus\Lie U_\alpha\oplus\Lie U_{-\alpha})\dual$ it is given by 
\[
\Darst_\alpha\colon\PGL_2\to\GL(V_\alpha)\;,\quad g=\left(\begin{matrix}a&b\\c&d\end{matrix}\right)\mapsto\left(\begin{array}{ccc}
1&\frac{ac}{\det g}&\frac{bd}{\det g}\\[1mm]
0&\frac{a^2}{\det g}&\frac{b^2}{\det g}\\[1mm]
0&\frac{c^2}{\det g}&\frac{d^2}{\det g}\end{array}\right).
\]
We let $\mu\in X_*(T)_\dom$ be the dominant coweight with $\mu(a)=\left(\begin{smallmatrix} a & 0 \\ 0 & 1 \end{smallmatrix}\right)$. Then $(-\mu)_\dom=\mu$ and $\langle\mu,\alpha\rangle=1$. Over the $\BF_q\dbl\zeta\dbr/(\zeta)$-algebra $B=\BF_q[\epsilon]/(\epsilon^2)$ the element $g=\left(\begin{smallmatrix} 1 & \frac{\epsilon}{z} \\ 0 & z \end{smallmatrix}\right)\in LG(B)=\PGL_2\bigl(B\dpl z\dpr\bigr)$ lies in $\hat{Z}^{{\SSC\rm Weyl},-1}_{\PGL_2,\,\preceq\mu}$, but $g^{-1}=\left(\begin{smallmatrix} z & -\frac{\epsilon}{z} \\ 0 & 1 \end{smallmatrix}\right)$ does not belong to $\hat{Z}^{{\SSC\rm Weyl},-1}_{\PGL_2,\,\preceq\mu}$ because
\[
\Darst_\alpha(g)= \left(\begin{array}{ccc}
1&0&\frac{\epsilon}{z}\\[1mm]
0&\frac{1}{z}&0\\[1mm]
0&0&z\end{array}\right)
\qquad\text{and}\qquad
\Darst_\alpha(g^{-1})= \left(\begin{array}{ccc}
1&0&-\frac{\epsilon}{z^2}\\[1mm]
0&z&0\\[1mm]
0&0&\frac{1}{z}\end{array}\right).
\]
So $\hat{Z}^{\SSC\rm Weyl}_{\PGL_2,\,\preceq\mu}\ne\hat{Z}^{{\SSC\rm Weyl},-1}_{\PGL_2,\preceq(-\mu)_\dom}$ in this case. Note that nevertheless, the underlying topological spaces of these two bounds coincide by Remark~\ref{DefBDLocal'}\ref{DefBDLocal_A6}. So the difference lies in the nilpotent structure. 
\end{example}

\section{Rapoport-Zink spaces for bounded local $G$-shtukas}\label{secrz}
\setcounter{equation}{0}

To recall the definition of Rapoport-Zink spaces for local $G$-shtukas let $\ul{\BG}_0$ be a local $G$-shtuka over $\BaseFld$. Since $\BaseFld$ has no non-trivial \'etale coverings, we may fix a trivialization $\ul\BG_0\cong\bigl((L^+G)_\BaseFld,b\s\bigr)$ where $b\in LG(\BaseFld)$ represents the Frobenius morphism. In all that follows we may replace $\ul\BG_0$ by a quasi-isogenous local $G$-shtuka $\ul\BG'_0\cong\bigl((L^+G)_\BaseFld,b'\s\bigr)$. In terms of the trivializations this means that there is an $h\in LG(\BaseFld)$ with $b'=h^{-1}b\s(h)$. In this case we say that $b$ and $b'$ are \emph{$\sigma$-conjugate under $LG(\BaseFld)$} or \emph{$LG(\BaseFld)$-$\sigma$-conjugate}. We write $[b]$ for the $LG(\BaseFld)$-$\sigma$-conjugacy class of $b$.

\begin{definition}\label{DefRZSpace}
Let $\hat{Z}=[(R,\hat Z_R)]$ be a bound with reflex ring $R_{\hat Z}=\kappa\dbl\xi\dbr$, set $\breve R_{\hat Z}:=\BaseFld\dbl\xi\dbr$, and consider the functor\\[2mm]
$\breveRZ\colon (\Nilp_{\breve R_{\hat Z}})^o \;\longto \;\Sets$
\begin{eqnarray*}
S&\longmapsto & \Bigl\{\,\text{Isomorphism classes of }(\ul\CG,\bar\delta)\colon\text{ where }\ul{\CG}\text{ is a local $G$-shtuka over $S$ }\\ 
&&~~~~ \text{bounded by $\hat{Z}^{-1}$ and }\bar{\delta}\colon  \ul{\CG}_{\bar{S}}\to \ul{\BG}_{0,\bar{S}}~\text{is a quasi-isogeny  over $\bar{S}$ }\Bigr\}
\end{eqnarray*}
Here $\bar{S}:=\Var_S(\zeta)$ is the zero locus of $\zeta$ in $S$. 

The group $\QIsog_{\BaseFld}(\ul{\BG}_0)$ of quasi-isogenies of $\ul{\BG}_0$ acts on the functor $\breveRZ$ via $j\colon(\ul\CG,\bar\delta)\mapsto(\ul\CG,j\circ\bar\delta)$ for $j\in\QIsog_{\BaseFld}(\ul{\BG}_0)$.
\end{definition}

\begin{remark}\label{RemGroupJ}
Since $\ul\BG_0\cong\bigl((L^+G)_\BaseFld,b\s\bigr)$ we can identify $\QIsog_\BaseFld(\ul\BG_0)\cong J_b\bigl(\BF_q\dpl z\dpr\bigr)$ where $J_b$ is the connected algebraic group over $\BF_q\dpl z\dpr$ that is defined by its functor of points that assigns to an $\BF_q\dpl z\dpr$-algebra $A$ the group
\begin{equation}\label{EqJ}
J_b(A):=J_b^G(A):=\big\{g \in G(A\otimes_{\BF_q\dpl z\dpr} {\BaseFld\dpl z\dpr})\colon g^{-1}\,b\,\s(g)=b\big\}\,;
\end{equation}
see \cite[Remark~4.16]{AH_Local}.
\end{remark}
\begin{remark}\label{defweildes}
As in the arithmetic case (compare \cite[3.48]{RZ}) we have a Weil descent datum $\alpha$ on the functor $\breveRZ$. To define it, let $S\in \Nilp_{\breve R_{\hat Z}}$ and let $f\colon S\rightarrow \Spf\breve R_{\hat Z}$ and $\bar f\colon\bar S\rightarrow \Spec\breve R_{\hat Z}/(\zeta)$ be the structure morphisms. Let $\phi$ be the Frobenius of $\breve R_{\hat Z}=\BaseFld\dbl\xi\dbr$ over $R_{\hat Z}=\kappa\dbl\xi\dbr$ and let $S_\phi$ be the scheme $S$ together with the structure morphism $\phi\circ f\colon S\rightarrow \Spf\breve R_{\hat Z}$. If $\#\kappa=q^e$ the inclusion $\BaseFld\into\breve R_{\hat Z}$ is equivariant for the action of $\phi$ on $\breve R_{\hat Z}$ and the action of $\sigma^e$ on $\BaseFld$. To define $\alpha\colon\breveRZ(S)\isoto(\phi^*\breveRZ)(S):=\breveRZ(S_{\phi})$ let $(\ul\CG,\bar\delta)$ be a point in the first set. To make the definition more clear we write $\ul{\BG}_{0,\bar{S}}=\bar f^*(\ul{\BG}_0)$. Then we set $\alpha(\ul\CG,\bar\delta)=(\ul\CG,\bar f^*(\tau_{\BG_0}^e)^{-1}\circ\bar\delta)$, i.e. we replace the quasi-isogeny $\bar\delta$ by the composite $\ul{\CG}_{\bar{S}}\to \bar f^*(\ul{\BG}_0)\to\bar f^*\sigma^{e*}(\ul{\BG}_0)=(\bar\phi\bar f)^*(\ul{\BG}_0)$. Although this Weil descent datum is in general not effective, $\breveRZ$ always descends to a finite unramified extension of $\breve R_{\hat Z}$; see Remark~\ref{RemKottwitz}(b).
\end{remark}

\begin{remark}\label{RemRVConj4.16}
By its definition, $\breveRZ$ only depends on the triple $(G,[b],\hat{Z})$, where $[b]$ is the $LG(\BaseFld)$-$\sigma$-conjugacy class of $b$ corresponding to the isogeny class of $\ul\BG_0\cong\bigl((L^+G)_\BaseFld,b\s\bigr)$. In the arithmetic analog the corresponding independence of the Rapoport-Zink spaces on additional choices made in their definition is far from obvious. It was conjectured in \cite[Conjecture~4.16]{RV} and proved in \cite[Corollary~23.4.3 and thereafter]{ScholzeBerkeley}.
\end{remark}

As in Remark~\ref{RemBdLocal}(a) and (b) we let $\hat Z^{-1}$ be the bound from Lemma~\ref{LemmaInverseBound}, and we let $Z^{-1}$ be the special fiber of $\hat Z^{-1}$ over $\kappa$. We define the associated \emph{affine Deligne-Lusztig variety} as the reduced closed ind-subscheme $X_{Z^{-1}}(b)\subset\Flag_G$ whose $K$-valued points (for any field extension $K$ of $\BaseFld$) are given by
\begin{equation}\label{defadlv}
X_{Z^{-1}}(b)(K):=\big\{ g\in \Flag_G(K)\colon g^{-1}\,b\,\s(g) \in Z^{-1}(K)\big\}.
\end{equation}
In \cite[Theorem~4.18 and Corollary~4.26]{AH_Local} the following theorem was proved.

\begin{theorem} \label{ThmRRZSp}
The functor $\breveRZ\colon (\Nilp_{\breve R_{\hat Z}})^o\to\Sets$ from Definition~\ref{DefRZSpace} is ind-repre\-sen\-ta\-ble by a formal scheme over $\Spf \breve R_{\hat Z}$ that is locally formally of finite type and separated. It is an ind-closed ind-subscheme of $\Flag_G\whtimes_{\BF_q}\Spf\breve R_{\hat Z}$. Its underlying reduced subscheme equals $X_{Z^{-1}}(b)$, which is a scheme locally of finite type and separated over $\BaseFld$, all of whose irreducible components are projective.
\end{theorem}

 The formal scheme representing $\breveRZ$ is called a \emph{Rapoport-Zink space for bounded local $G$-shtukas}. Recall that a formal scheme over $\breve R_{\hat Z}$ in the sense of \cite[I$_{\rm new}$, 10]{EGA} is called \emph{locally formally of finite type} if it is locally noetherian and adic and its reduced subscheme is locally of finite type over $\BaseFld$.

\begin{remark}\label{RemKottwitz}
(a) $LG(\BaseFld)$-$\sigma$-conjugacy classes in $LG(\BaseFld)$ are in bijection to isogeny classes of local $G$-shtukas over $\BaseFld$. To classify them, Kottwitz associates with every element $b\in LG(\BaseFld)$ a slope homomorphism $\nu_b\colon \BD_{\BaseFld\dpl z\dpr}\to G_{\BaseFld\dpl z\dpr}$, called \emph{Newton point} (or \emph{Newton polygon}) of $b$; see \cite[4.2]{Kottwitz85}. Here $\BD$ is the diagonalizable pro-algebraic group over $\BaseFld\dpl z\dpr$ with character group $\BQ$. The slope homomorphism is characterized by assigning the slope filtration of 
\[
\bigl(V\otimes_{\BF_q\dpl z\dpr}{\BaseFld\dpl z\dpr},\,\Darst(b)\s\bigr)
\]
to any $(V,\Darst)$ in $\Rep_{\BF_q\dpl z\dpr}G$; see \cite[Section 4]{Kottwitz85}. Furthermore, he shows that $\sigma$-conjuga\-ting $b$ amounts to conjugating $\nu_b$ by the corresponding element. One can thus associate with the $LG(\BaseFld)$-$\sigma$-conjugacy class $[b]$ a well-defined $G\bigl(\BaseFld\dpl z\dpr\bigr)$-conjugacy class $\{\nu_b\}$ in $\Hom(\BD_{\BaseFld\dpl z\dpr},G_{\BaseFld\dpl z\dpr})$, which moreover is invariant under $\sigma$ by \cite[4.4]{Kottwitz85}.

The second important invariant of $[b]$ is defined as follows. Consider again the Kottwitz homomorphism $\kappa_G\colon LG(\BaseFld)\rightarrow \pi_1(G)_{I}$ as explained in \cite[2.a.2]{PR2}. We compose $\kappa_G$ with the projection to $\pi_1(G)_{\Gamma}$. This then yields a well-defined map (again denoted $\kappa_G$) from the set of $\sigma$-conjugacy classes $B(G)$ to $\pi_1(G)_{\Gamma}$ (see \cite{Kottwitz97}). Together, $\{\nu_b\}$ and $\kappa_G([b])$ determine $[b]$ uniquely. We denote by $[b]^{\#}$ the images of $\kappa_G(b)$ in $\pi_1(G)_{\Gamma}$ and in $\pi_1(G)_{\Gamma,\BQ}:=\pi_1(G)_{\Gamma}\otimes_\BZ\BQ$. The latter equals the image of the conjugacy class $\{\nu_b\}$ in $\pi_1(G)_{\Gamma,\BQ}$; see \cite[Theorem~1.15(iii)]{RapoportRichartz}.

\medskip\noindent
(b) Since $\BaseFld$ is algebraically closed and the generic fiber of $G$ is connected reductive we may replace $b$ by $h^{-1}b\s(h)$ and assume that $b\in LG(\BaseFld)$ satisfies a \emph{decency equation for a positive integer $s$}, that is, $s\nu_b:\BD_{\BaseFld\dpl z\dpr}\to G_{\BaseFld\dpl z\dpr}$ factors through $\BG_m$, and
\begin{equation}\label{EqDecency}
(b\sigma)^s\;=\;s\nu_b(z)\,\sigma^s\qquad\text{in}\quad LG(\BaseFld)\rtimes \langle\sigma\rangle;
\end{equation}
see \cite[Section 4]{Kottwitz85}. Let $\BF_{q^s}\subset\BaseFld$ be the finite field extension of $\BF_q$ of degree $s$. Then $b\in LG(\BF_{q^s})$, $\nu_b$ is defined over $\BF_{q^s}\dpl z\dpr$ (\cite[Corollary 1.9]{RZ}) and $J_b\times_{\BF_q\dpl z\dpr}\BF_{q^s}\dpl z\dpr$ is an inner form of the centralizer of the 1-parameter subgroup $s\nu_b$ of $G$, a Levi subgroup of $G_{\BF_{q^s}\dpl z\dpr}$; see \cite[Corollary 1.14]{RZ}. In particular $J_b\bigl(\BF_q\dpl z\dpr\bigr)\subset G\bigl(\BF_{q^s}\dpl z\dpr\bigr)= LG(\BF_{q^s})$. Moreover, in this case $\breveRZ$ descends by \cite[Theorem~4.18]{AH_Local} to a formal scheme locally formally of finite type over $(\BF_{q^s}\cdot\kappa)\dbl\xi\dbr$ where $\BF_{q^s}\cdot\kappa$ is the compositum inside $\BaseFld$.

\medskip\noindent
(c) If more generally we start with a local $G$-shtuka $\ul\BG_0$ over any field $\OtherField$ in $\Nilp_{ R_{\hat Z}}$ we can define the Rapoport-Zink functor $\RZ$ as in Definition~\ref{DefRZSpace} on the category $\Nilp_{\OtherField\dbl\xi\dbr}$. Then $\ul\BG_0\otimes_\OtherField \OtherField^\alg$ is trivial and decent, and hence $\RZ\otimes_\OtherField \OtherField^\alg$ is ind-representable by a formal scheme locally formally of finite type over $\OtherField^\alg\dbl\xi\dbr$. By an unpublished result of Eike Lau on Galois descent of formal schemes locally formally of finite type, already $\RZ$ is ind-representable by a formal scheme locally formally of finite type over $\OtherField\dbl\xi\dbr$. However, we will not use this in the rest of this work.
\end{remark}

\begin{remark}\label{RemFunctoriality1}
The constructions described above are functorial with respect to the group scheme $G$. To explain this let $\epsilon\colon G\to G'$ be a morphism of parahoric group schemes over $\BF_q\dbl z\dbr$. Important examples are closed immersions, epimorphisms or the change of the parahoric model, that is morphisms that are generically isomorphisms. Then $\epsilon$ induces a functor
\[
\epsilon_*\colon\; \{\,\text{local $G$-shtukas}\,\}\;\longto\;\{\,\text{local $G'$-shtukas}\,\}\,,\quad\ul\CG\;\longmapsto\;\epsilon_*\ul\CG\;:=\;\ul\CG\overset{G}{\times} G'\,.
\]
The morphism $\epsilon$ also induces morphisms $\epsilon\colon L^+G\to L^+G'$ and $\epsilon\colon LG\to LG'$ and $\epsilon\colon \Flag_G\to \Flag_{G'}$. We say that $\epsilon$ is \emph{compatible} with two given bounds $\hat{Z}=[(R,\hat{Z}_R)]$ and $\hat{Z}'=[(R,\hat{Z}'_R)]$ with $\hat{Z}_R\subset\wh\Flag_{G,R}$ and $\hat{Z}'_R\subset\wh\Flag_{G'\!,R}$ if $\epsilon(\hat{Z}_R)\subset\hat{Z}'_R$ for all suitable $R$. If $\epsilon$ is compatible with $\hat{Z}$ and $\hat{Z}'$ and $\ul\CG$ is bounded by $\hat{Z}^{-1}$ then $\epsilon_*\ul\CG$ is bounded by $(\hat{Z}')^{-1}$.

Let $\ul\BG_0\cong\bigl((L^+G)_\BaseFld,b\s\bigr)$ be a local $G$-shtuka over $\BaseFld$ with $b\in LG(\BaseFld)$. Then $\ul\BG'_0:=\epsilon_*\ul\BG_0\cong\bigl((L^+G')_\BaseFld,b'\s\bigr)$ with $b'=\epsilon(b)$. Kottwitz's classification of isogeny classes is functorial in the sense that $\nu_{b'}=\epsilon\circ\nu_b$ and $\epsilon\circ\kappa_G=\kappa_{G'}\circ\epsilon$ for the induced $\Gamma$-equivariant morphism $\epsilon\colon\pi_1(G)\to\pi_1(G')$. We also obtain a morphism of Rapoport-Zink spaces
\begin{equation}\label{EqFunctRZ}
\epsilon_*\colon\;\breveRZ\;\longto\;{\breve\CM}_{\ul\BG'_0}^{\raisebox{-0.7ex}{$\SC\hat{Z}'{}^{-1}$}},\quad(\ul\CG,\bar\delta)\;\longmapsto\;(\epsilon_*\ul\CG,\epsilon_*\bar\delta)\,,
\end{equation}
that is equivariant for the group $J_b^G$ that acts on the target via the morphism of algebraic groups over $\BF_q\dpl z\dpr$
\begin{equation}\label{EqFunctJ}
J_b^G\;\longto\;J_{b'}^{G'}\,,\quad g\;\longmapsto\;\epsilon(g)\,.
\end{equation}
\end{remark}

\section{Period spaces for bounded local $G$-shtukas}
\setcounter{equation}{0}

In this section we construct period spaces. These will be strictly $\BF_q\dpl\zeta\dpr$-analytic spaces in the sense of Berkovich~\cite{Berkovich1,Berkovich2}. We equip the category of $\BF_q\dpl\zeta\dpr$-schemes and the category of $\BF_q\dpl\zeta\dpr$-analytic spaces with the \'etale topology; see \cite[\S\,4.1]{Berkovich2}. Recall the group scheme $G\times_{\BF_q\dbl z\dbr}\Spec \BF_q\dpl\zeta\dpr\dbl z-\zeta\dbr$, which is reductive because $\BF_q\dbl z\dbr\to \BF_q\dpl\zeta\dpr\dbl z-\zeta\dbr,\,z\mapsto z=\zeta+(z-\zeta)$ factors through $\BF_q\dpl z\dpr$, and recall its affine Grassmannian $\Gr_G^{\bB_\dR}$ from \eqref{EqH_G}. For $G=\GL_r$, Hilbert 90 for loop groups \cite[Proposition~2.3]{HV1} shows that
\[
\Gr_{\GL_r}^{\bB_\dR}(L)\;=\;\GL_r\bigl(L\dpl z-\zeta\dpr\bigr)\big/\GL_r\bigl(L\dbl z-\zeta\dbr\bigr)
\]
for any field extension $L/\BF_q\dpl\zeta\dpr$.

Again for all $G$ as above, the morphism of sheaves of sets on $\Spec \BF_q\dpl\zeta\dpr$ $$G\bigl(\CO_X\dpl z-\zeta\dpr\bigr)\to \Gr_G^{\bB_\dR}(X)$$ admits local sections for the \'etale topology. By \cite[Proposition~13.1.1]{Springer09} there is a finite separable field extension $L_0\supset\BF_q\dpl\zeta\dpr$ such that $G_{L_0}:=G\otimes_{\BF_q\dbl z\dbr,z\mapsto\zeta}L_0$ splits. Therefore the group $G\otimes_{\BF_q\dbl z\dbr}\BF_q\dpl\zeta\dpr\dpl z-\zeta\dpr$ over $\BF_q\dpl\zeta\dpr\dpl z-\zeta\dpr$ is unramified. Thus the inertia group of $\Gal\bigl(\BF_q\dpl\zeta\dpr\dpl z-\zeta\dpr^\sep\!/\BF_q\dpl\zeta\dpr\dpl z-\zeta\dpr\bigr)$ acts trivially on $\pi_1(G)$ and the connected components of $\Gr_G^{\bB_\dR}\wh\otimes_{\BF_q\dpl\zeta\dpr}\BF_q\dpl\zeta\dpr^\alg$ are in canonical bijection with $\pi_1(G)$ by \cite[Theorem~5.1]{PR2}. For every field extension $L\subset \BF_q\dpl\zeta\dpr^\alg$ of $\BF_q\dpl\zeta\dpr$ we then obtain from \cite[Lemma~2.2.6]{N16} that
\begin{equation}\label{EqPi0OfHG}
\pi_0\bigl(\Gr_G^{\bB_\dR}\wh\otimes_{\BF_q\dpl z\dpr}L\bigr) \;\cong \; \pi_1(G)/\Gal(L^\sep\!/L)
\end{equation}
where the quotient is the \emph{set} of $\Gal(L^\sep\!/L)$-orbits. It has a natural projection to the group of coinvariants $\pi_1(G)_{\Gamma_L}$. In particular $\pi_0\bigl(\Gr_G^{\bB_\dR}\wh\otimes_{\BF_q\dpl z\dpr} L_0\bigr)=\pi_1(G)$.

\begin{definition}\label{DefHPG-Structure}
Let $X$ be an $\BF_q\dpl\zeta\dpr$-scheme or a strictly $\BF_q\dpl\zeta\dpr$-analytic space. A \emph{Hodge-Pink $G$-structure} over $X$ is an element $\gamma$ in $\Gr_G^{\bB_\dR}(X)$. For a Hodge-Pink $G$-structure $\gamma\in\Gr_G^{\bB_\dR}(L)$ with values in a field $L$ we let $\gamma^\#\in\pi_1(G)_\Gamma$ be its image under the projection $\pi_0(\Gr_G^{\bB_\dR})\onto\pi_1(G)_\Gamma$ induced by \eqref{EqPi0OfHG}.

A \emph{Hodge-Pink structure of rank $r$ over $X$} is a sheaf $\Fq$ on $X$ of $\CO_X\dbl z-\zeta\dbr$-submodules of $\CO_X\dpl z-\zeta\dpr^r$ that is finitely generated as $\CO_X\dbl z-\zeta\dbr$-module, is a direct summand as $\CO_X$-module, and satisfies $\CO_X\dpl z-\zeta\dpr\cdot\Fq=\CO_X\dpl z-\zeta\dpr^r$. 
\end{definition}

\begin{remark}\label{RemHPStr}
By \cite[Proposition~2.2.5]{SchauchDiss} a Hodge-Pink structure $\Fq$ of rank $r$ over $X$ is Zariski-locally on $X$ free of rank $r$ as $\CO_X\dbl z-\zeta\dbr$-module. In particular, it is of the form $\Fq=\gamma\cdot\CO_X\dbl z-\zeta\dbr^r\subset\CO_X\dpl z-\zeta\dpr^r$ for a uniquely determined Hodge-Pink $\GL_r$-structure $\gamma\in\Gr_{\GL_r}^{\bB_\dR}(X)$ over $X$. This yields an equivalence between Hodge-Pink $\GL_r$-structures and Hodge-Pink structures of rank $r$ over $X$.
\end{remark}

\medskip

To define the notion of weak admissibility recall from \cite[Definitions~3.5.1 and 3.5.2]{HartlKim} that a \emph{$z$-isocrystal over $\BaseFld$} is a pair $(D,\tau_D)$ consisting of a finite-dimensional $\BaseFld\dpl z\dpr$-vector space $D$ and an $\BaseFld\dpl z\dpr$-isomorphism $\tau_D\colon\sigma^*D\isoto D$. A \emph{Hodge-Pink structure} on $(D,\tau_D)$ over a field extension $L$ of $\BaseFld\dpl\zeta\dpr$ is a free $L\dbl z-\zeta\dbr$-submodule $\Fq_D\subset D\otimes_{\BaseFld\dpl z\dpr}L\dpl z-\zeta\dpr$ of full rank. Here, as always, we use the homomorphism $\BaseFld\dpl z\dpr\into \BaseFld\dpl\zeta\dpr\dbl z-\zeta\dbr,\,z\mapsto z=\zeta+(z-\zeta)$. 

\begin{definition}\label{DefWA}
Assume that $\BaseFld\dpl\zeta\dpr\subset L$ and let $b\in LG(\BaseFld)$ and $\gamma\in\Gr_G^{\bB_\dR}(L)$.
\begin{enumerate}
\item 
Let $\Darst\colon G_{\BF_q\dpl z\dpr}\to\GL_{r,\BF_q\dpl z\dpr}$ be in $\Rep_{\BF_q\dpl z\dpr}G$ and set $V=\BF_q\dpl z\dpr^{\oplus r}$, the representation space. We consider the elements $\Darst(b)\in\GL_r(\BaseFld\dpl z\dpr)$ and $\Darst(\gamma)\in\Gr_{\GL_r}^{\bB_\dR}(L)=\GL_r\bigl(L\dpl z-\zeta\dpr\bigr)/\GL_r\bigl(L\dbl z-\zeta\dbr\bigr)$. Then we define the \emph{$z$-isocrystal} 
\[
\ulD_b(V,\Darst)\;:=\;(D,\tau_D)\;:=\;\bigl(V\otimes_{\BF_q\dpl z\dpr}\BaseFld\dpl z\dpr,\,\Darst(\s b)\s\bigr)
\]
over $\BaseFld$ and the \emph{Hodge-Pink structure}
\[
\Fq_D(V)\;:=\;\Darst(\gamma)\cdot V\otimes_{\BF_q\dpl z\dpr}L\dbl z-\zeta\dbr\;\subset\; V\otimes_{\BF_q\dpl z\dpr}L\dpl z-\zeta\dpr\;=\;D\otimes_{\BaseFld\dpl z\dpr}L\dpl z-\zeta\dpr
\]
on it over $L$. We set $\ulD_{b,\gamma}(V):=\bigl(V\otimes_{\BF_q\dpl z\dpr}\BaseFld\dpl z\dpr,\,\Darst(\s b)\s,\,\Fq_D(V)\bigr)$.
\item 
Let $\ulD=(D,\tau_D,\Fq_D)$ be a $z$-isocrystal over $\BaseFld$ with Hodge-Pink structure over $L$ and let $\det\tau_D$ be the determinant of the matrix representing $\tau_D$ with respect to an $\BaseFld\dpl z\dpr$-basis of $D$. The $z$-adic valuation $t_N(\ulD):=\ord_z(\det\tau_D)$ is independent of this basis and is called the \emph{Newton degree of $\ulD$}. The integer $t_H(\ulD)$ with $\wedge^r\Fq_D=(z-\zeta)^{-t_H(\ulD)}\wedge^r\Fp_D$ is called the \emph{Hodge degree of $\ulD$}, where $\Fp_D:=D\otimes_{\BaseFld\dpl z\dpr}L\dbl z-\zeta\dbr$.

In particular, we have $t_N(\ulD_{b,\gamma}(V))=\ord_z(\det\Darst(\s b))=\ord_z(\det\Darst(b))$ and $t_H(\ulD_{b,\gamma}(V))=-\ord_{z-\zeta}(\det\Darst(\gamma))$.
\item \label{DefWA_C}
We say that $\ulD$ is \emph{weakly admissible} if $t_H(\ulD)=t_N(\ulD)$ and the following equivalent conditions are satisfied (compare \cite[Definition~2.2.4]{HartlPSp})

\smallskip\noindent
$\bullet\es t_H(\ulD')\leq t_N(\ulD')$ for every strict subobject
\begin{equation}\label{EqUlD'}
\ulD'\,=\,\bigl(D'\!,\,\tau_D|_{\s D'},\,\Fq_D\cap D'\otimes_{\BaseFld\dpl z\dpr}L\dpl z-\zeta\dpr\bigr)
\end{equation}
of $\ulD$, where $D'\subset D$ is a $\tau_D$-stable $\BaseFld\dpl z\dpr$-subspace,

\smallskip\noindent
$\bullet\es t_H(\ulD'')\geq t_N(\ulD'')$ for every strict quotient object $\ulD''=(D''\!,\tau_{D''},\Fq_{D''})$ of $\ulD$, where $f\colon D\onto D''$ is a $\tau_D$-stable $\BaseFld\dpl z\dpr$-quotient space and $\Fq_{D''}=(f\otimes\id)(\Fq_D)$.
\end{enumerate}
\end{definition}

\begin{remark}\label{RemDIsTensorFunctor}
(a) The Hodge-Pink structure $\Fq_D$ on $(D,\tau_D)$ induces a decreasing \emph{Hodge-Pink filtration} on $D_L:=D\otimes_{\BaseFld\dpl z\dpr,z\mapsto\zeta}L=\Fp_D/(z-\zeta)\Fp_D$ given by
\[
Fil^i D_L\;:=\;\bigl((z-\zeta)^i\Fq_D\cap\Fp_D\bigr)\big/\bigl((z-\zeta)^i\Fq_D\cap(z-\zeta)\Fp_D\bigr)\;=\;\im\bigl((z-\zeta)^i\Fq_D\cap\Fp_D\to D_L\bigr).
\]
If $(D,\tau_D,Fil^\bullet D_L)$ is weakly admissible in the sense of Fontaine, see \cite[D\'efinition~4.1.4]{Fontaine79}, then $(D,\tau_D,\Fq_D)$ is weakly admissible, but the converse is false in general. An example is $D=\BaseFld\dpl z\dpr^2$, $\tau_D=z^{-1}$, $\Fq_D=\binom{z-\zeta}{1}L\dbl z-\zeta\dbr+(z-\zeta)^2\Fp_D$ and $D'=\binom{0}{1}\BaseFld\dpl z\dpr$. This is due to the facts that our Hodge slope $t_H(D,\tau_D,\Fq_D)$ equals Fontaine's Hodge slope $t_H(D,\tau_D,Fil^\bullet D_L):=$ $\sum_{i\in\BZ}i\cdot\dim_L Fil^i D_L/Fil^{i+1}D_L$, see \cite[p.~1290 before Definition~2.2.4]{HartlPSp}, and that the subspace $Fil^iD'_L$ of $D'_L$ induced by $\ulD'$ from \eqref{EqUlD'} is (in general strictly) contained in $D'_L\cap Fil^iD_L$.

\medskip\noindent
(b) We obtain an $\BF_q\dpl z\dpr$-linear tensor functor $\ulD_{b,\gamma}\colon V\mapsto\ulD_{b,\gamma}(V)$. Namely, if $f\colon (V,\Darst)\to (V'\!,\Darst')$ is a morphism in $\Rep_{\BF_q\dpl z\dpr}G$ and $(D,\tau_D,\Fq_D)=\ulD_{b,\gamma}(V)$ and $(D'\!,\tau_{D'},\Fq_{D'})=\ulD_{b,\gamma}(V')$, then $\sigma^*f=f$ and so $f\circ\tau_D=(f\circ\Darst(\s b))\s=(\Darst'(\s b)\circ f)\s=(\Darst'(\s b)\circ\s f)\s=\tau_{D'}\circ\sigma^*f$ and $f(\Fq_D)\subset\Fq_{D'}$. Furthermore, the compatibility with tensor products
\[
(D,\tau_D,\Fq_D)\otimes(D'\!,\tau_{D'},\Fq_{D'})\;:=\;(D\otimes_{\BaseFld\dpl z\dpr}D'\!,\,\tau_D\otimes\tau_{D'},\,\Fq_D\otimes_{L\dbl z-\zeta\dbr}\Fq_{D'})
\]
is clear.
\end{remark}

If $L$ is a \emph{finite} field extension of $\BaseFld\dpl\zeta\dpr$ then ``weakly admissible implies admissible'' by the analog \cite[Theorem~2.5.3]{HartlPSp} of the theorem of Colmez and Fontaine~\cite[Th\'eor\`eme~A]{CF}. More precisely, when $\ulD_{b,\gamma}(V)$ is weakly admissible then it is even \emph{admissible}, that is it arises from a local shtuka over $\CO_L$ via the analog $\BH$ of Fontaine's mysterious functor; see \cite[\S\,2.3]{HartlPSp} or Theorem~\ref{MainThm}. In contrast, if $L$ is algebraically closed, weakly admissible does \emph{not} imply admissible. In general there is a criterion for admissibility in terms of $\sigma$-bundles over the analog of the Robba ring as follows. 

We consider field extensions $L$ of $\BF_q\dpl\zeta\dpr$ equipped with an absolute value $|\,.\,|\colon L\to\BR_{\ge0}$ extending the $\zeta$-adic absolute value on $\BF_q\dpl\zeta\dpr$ such that $L$ is complete with respect to $|\,.\,|$. We call such an $L$ a \emph{complete valued field extension of $\BF_q\dpl\zeta\dpr$}, and we let $\ol L:=\wh{L^\alg}$ be the completion of an algebraic closure of $L$. For a rational number $s>0$ we define the ring
\[
L\ancon[s]\;:=\;\bigl\{\;{\TS\sum\limits_{i=-\infty}^\infty} b_i z^i\colon b_i\in L,\, |b_i|\,|\zeta|^{s'i}\to0\es(i\to\pm\infty)\text{ for all }s'\ge s\;\bigr\}\,.
\]
It equals the ring of rigid analytic functions on the punctured disc $\{0<|z|\le|\zeta|^s\}$ over $L$ of radius $|\zeta|^s$. The ring $L\ancon[s]$ is the function field analog of the Robba ring; see \cite[\S\,2.8]{HartlDict}. It contains the element
\begin{equation} \label{EqTMinus}
\TS\tminus\;:=\;\prod\limits_{i\in\BN_0}\bigl(1-{\tfrac{\zeta^{q^i}}{z}}\bigr)\,,\quad\text{that satisfies}\quad\tminus\;=\;(1-\tfrac{\zeta}{z})\cdot\s(\tminus)\,.
\end{equation}

\begin{definition}\label{DefSigmaModule}
Let $s\in\BQ$ satisfy $1>s>\frac{1}{q}$. A \emph{$\sigma$-bundle} (\emph{on $\{0<|z|\le|\zeta|^s\}$}) \emph{over $L$} is a pair $\ulCF=(\CF,\tau_\CF)$ consisting of a locally free $L\ancon[s]$-module $\CF$ of finite rank together with an isomorphism $\tau_\CF\colon\sigma^\ast \CF\isoto \iota^\ast \CF$ of $L\ancon[qs]$-modules, where $\sigma^*\CF:=\CF\otimes_{L\ancon[s],\,\sigma}\,L\ancon[qs]$ and $\iota^*\CF:=\CF\otimes_{L\ancon[s],\,\iota}\,L\ancon[qs]$ for the natural inclusion $\iota\colon L\ancon[s]\into L\ancon[qs],\,\sum_ib_iz^i\mapsto\sum_ib_iz^i$.
The abelian group $\Hom_\sigma(\ulCF,\ulCF')$ of \emph{morphisms} between two $\sigma$-bundles $\ulCF=(\CF,\tau_\CF)$ and $\ulCF'=(\CF'\!,\tau_{\CF'})$ consists of all $L\ancon[s]$-homomorphisms $f\colon \CF\to \CF'$ that satisfy $\tau_{\CF'}\circ\sigma^\ast f=\iota^\ast f\circ \tau_\CF$. 
\end{definition}

The category of $\sigma$-bundles over $L$ is an $\BF_q\dpl z\dpr$-linear rigid additive tensor category with unit object $\ul\CO(0):=(L\ancon[s],\tau_{\CO(0)}=\id)$.

\begin{example}\label{ExSigmaModule}
For $d\in\BZ$ we define the $\sigma$-bundle $\ul\CO(d)$ over $L$ as the pair $(L\ancon[s],\tau_{\CO(d)}=z^{-d}\cdot\id)$. For more examples let $d$ and $n$ be relatively prime integers with $n>0$. Consider the matrix
\begin{equation*}\label{EqMatrixA_dn}
A_{d,n} := \left( \raisebox{4.9ex}{$
\xymatrix @C=0pc @R=0pc {
0 \ar@{.}[rrr]\ar@{.}[drdrdrdr] & & & 0 & z^{-d}\\
1\ar@{.}[drdrdr]  & & &  & 0 \ar@{.}[ddd]\\
0 \ar@{.}[dd]\ar@{.}[drdr] & & & &  \\
& & & & \\
0 \ar@{.}[rr] & & 0 & 1 & 0\\
}$}
\right)\es\in\es\GL_n\bigl(L\ancon[s]\bigr).
\end{equation*}
Let $\ulCF_{d,n}=\bigl(L\ancon[s]^n,\tau_{\CF_{d,n}}=A_{d,n}\bigr)$. It is a $\sigma$-bundle of rank $n$ over $L$. As a special case if $n=1$ we obtain $\ulCF_{d,1}=\ul\CO(d)$. 
\end{example}

\begin{proposition}[{\cite[Theorem 11.1 and Corollary 11.8]{HP}}]\label{PropHP}
If $L$ is algebraically closed (and complete) every $\sigma$-bundle $\ulCF$ is isomorphic to a direct sum $\bigoplus_{i}\ulCF_{d_i,n_i}$ for uniquely determined pairs $(d_i,n_i)$ up to permutation with $\gcd(d_i,n_i)=1$. One has $\wedge^n\ulCF\cong\ulCF_{d,1}=\ul\CO(d)$ where $n=\rk\ulCF=\sum_i n_i$ and $d=\sum_i d_i$. One calls $d$ the \emph{degree of $\ulCF$}.
\end{proposition}

\begin{proposition}[{\cite[Proposition~8.5]{HP}}]\label{Prop0.7}
If $L$ is algebraically closed (and complete) then $\Hom_\sigma(\ulCF_{d,n},\,\ulCF_{d'\!,n'})\ne(0)$ if and only if $d/n\le d'/n'$.
\end{proposition}

Let $\ulD=(D,\tau_D,\Fq_D)$ be a $z$-isocrystal over $\BaseFld$ with a Hodge-Pink structure over a complete valued field extension $L$ of $\BaseFld\dpl\zeta\dpr$. To define the $\sigma$-bundles associated with $\ulD$ we first define the $\sigma$-bundle
\[
\CE\;:=\;\CE(\ulD)\;:=\;D\otimes_{\BaseFld\dpl z\dpr}L\ancon[s]\,,\quad \tau_\CE\;:=\;\tau_D\otimes\id\,,\quad\ulCE(\ulD)\;=\;(\CE,\tau_\CE)
\]
over $L$. Then $\CE(\ulD)\otimes L\dbl z-\zeta\dbr=\Fp_D:=D\otimes_{\BaseFld\dpl z\dpr}L\dbl z-\zeta\dbr$ and the Hodge-Pink structure $\Fq_D\subset\CE\otimes L\dbl z-\zeta\dbr[\frac{1}{z-\zeta}]$ defines a $\sigma$-bundle $\ulCF(\ulD)$ over $L$ that is a modification of $\ulCE(\ulD)$ at $z=\zeta^{q^i}$ for $i\in\BN_0$ as follows.
Consider the isomorphism $\eta_i:=\bigl(\tau_\CE\circ\ldots\circ(\sigma^{i-1})^\ast \tau_\CE\bigr)\otimes\id$
\[
\eta_i\colon\;(\sigma^i)^\ast\bigl(\Fp_D[{\TS\frac{1}{z-\zeta}}]\bigr)\;=\; \bigl((\sigma^i)^\ast\CE\bigr)\otimes L\dbl z-\zeta^{q^i}\dbr[{\TS\frac{1}{z-\zeta^{q^i}}}]\;\isoto\;\CE\otimes L\dbl z-\zeta^{q^i}\dbr[{\TS\frac{1}{z-\zeta^{q^i}}}]\,.
\]
We define $\CF=\CF(\ulD)$ as the $L\ancon[s]$-submodule of $\CE[\tfrac{1}{\tminus}]$ that coincides with $\CE$ outside $z=\zeta^{q^i}$ for $i\in\BN_0$ and at $z=\zeta^{q^i}$ satisfies $\CF\otimes L\dbl z-\zeta^{q^i}\dbr\;=\;\eta_i(\sigma^{i\ast}\Fq_D)$, that is
\begin{equation}\label{EqSheafCF}
\CF\;:=\;\bigl\{\,m\in\CE[\tfrac{1}{\tminus}]\colon\,\eta_i^{-1}(m)\in\sigma^{i*}\Fq_D\text{ for }i\in\BN_0\,\bigr\}\,. 
\end{equation}
This can equivalently be viewed as the global sections over $\{0<|z|\le|\zeta|^s\}$ of the sheaf $\wt\CF$ obtained as the modification of the sheaf associated with $\CE$ at the discrete set $\{z=\zeta^{q^i}:i\in\BN_0\}$ according to the rule given in \eqref{EqSheafCF}.

By construction $\tau_\CE$ induces on $\CF$ the structure of a $\sigma$-bundle $\ulCF(\ulD)=(\CF,\tau_\CF)$ over $L$, and $\ulCF(\ulD)$ is the unique $\sigma$-subbundle of $\ulCE[\tfrac{1}{\tminus}]$ that coincides with $\ulCE$ outside $z=\zeta^{q^i}$ for $i\in\BN_0$ and satisfies $\CF\otimes L\dbl z-\zeta\dbr\;=\;\Fq_D$. This characterization implies that the assignment $\ulD\mapsto\bigl(\ulCE(\ulD),\,\ulCF(\ulD)\bigr)$ is an $\BF_q\dpl z\dpr$-linear tensor functor.

\begin{definition}\label{DefSigmaModuleOfD}
The \emph{pair of $\sigma$-bundles associated with $\ulD$} is the pair $\bigl(\ulCE(\ulD),\,\ulCF(\ulD)\bigr)$ constructed above. 

The $z$-isocrystal with Hodge-Pink structure $\ulD$ is said to be \emph{admissible} if $\ulCF(\ulD)\otimes_{L\ancon[s]}\ol L\ancon[s]\cong (\ulCF_{0,1})^{\oplus\dim D}\otimes_{L\ancon[s]}\ol L\ancon[s]$.

In the notation of Definition~\ref{DefWA} let $\ulD=\ulD_{b,\gamma}(V)$ for a representation $\Darst\colon G\to\GL(V)$ in $\Rep_{\BF_q\dpl z\dpr}G$ and write $\bigl(\ulCE_{b,\gamma}(V),\,\ulCF_{b,\gamma}(V)\bigr):=\bigl(\ulCE(\ulD),\,\ulCF(\ulD)\bigr)$.
\end{definition}

As a motivation for this definition note that $\ulD$ is admissible if and only if it arises from a local $\GL_r$-shtuka over $\CO_L$ by \cite[Theorem~2.4.7 and Definition~2.3.3]{HartlPSp}.

\begin{proposition}[{\cite[Lemma~2.4.5]{HartlPSp}}]\label{PropDeg}
For every $z$-isocrystal with Hodge-Pink structure $\ulD$ over $L$ the degree (defined in Proposition~\ref{PropHP}) satisfies $\deg\ulCF(\ulD)\;=\;t_H(\ulD)-t_N(\ulD)$ and $\deg\ulCE(\ulD)\;=\;-t_N(\ulD)$.
\end{proposition}

\begin{corollary}\label{CorA=>WA}
If $\ulD_{b,\gamma}(V)$ is admissible then $\ulD_{b,\gamma}(V)$ is weakly admissible.
\end{corollary}

\begin{proof}
If $\ulD:=\ulD_{b,\gamma}(V)$ is admissible then $\ulCF(\ulD)\otimes_{L\ancon[s]}\ol L\ancon[s]\cong(\ulCF_{0,1})^{\oplus\dim V}$ and therefore $t_H(\ulD)-t_N(\ulD)=\deg(\ulCF(\ulD))=0$.  If $\ulD'\subset\ulD$ is a strict subobject then $\ulCF(\ulD')\subset\ulCF(\ulD)$ is a $\sigma$-sub-bundle. It satisfies $\ulCF(\ulD')\otimes_{L\ancon[s]}\ol L\ancon[s]\cong\bigoplus_i\ulCF_{d_i,n_i}$ for some $d_i,n_i$ by Proposition~\ref{PropHP}. By Proposition~\ref{Prop0.7} all $d_i\le0$ and hence $t_H(\ulD')-t_N(\ulD')\;=\;\deg\ulCF(\ulD')\;=\;\sum_i d_i\le0$.
\end{proof}

\begin{lemma}\label{LemmaKottwitzCond}
Let $\nu_b\colon \BD_{\BaseFld\dpl z\dpr}\to G_{\BaseFld\dpl z\dpr}$ be the Newton point associated with $b$; see Remark~\ref{RemKottwitz}(a). Let $\Darst\colon G\to\GL(V)$ be a representation in $\Rep_{\BF_q\dpl z\dpr}G$. Then under the canonical identifications $\pi_1\bigl(\GL(V)\bigr)_\Gamma=\pi_1\bigl(\GL(V)\bigr)=\BZ$ and $\Hom(\BD_{\BaseFld\dpl z\dpr},\BG_m)=\BQ$ we have
\[
\TS\Darst_*(\gamma^\#) \;=\;t_H\bigl(\ulD_{b,\gamma}(V)\bigr)\qquad\text{and}\qquad \det_V\circ\Darst\circ\nu_b\;=\;t_N\bigl(\ulD_{b,\gamma}(V)\bigr)\,.
\]
In particular, the images $[b]^{\#}$ and $\gamma^{\#}$ of $\nu_b$ and $\gamma$ in $\pi_1(G)_{{\Gamma},\BQ}:=\pi_1(G)_{\Gamma}\otimes_\BZ\BQ$ coincide if and only if $t_N\bigl(\ulD_{b,\gamma}(V)\bigr)=t_H\bigl(\ulD_{b,\gamma}(V)\bigr)$ for all $V\in\Rep_{\BF_q\dpl z\dpr}G$. 
\end{lemma}

\begin{proof}
Since $\s b=b(\s b) \s b^{-1}$, that is $\s b$ and $b$ are $\sigma$-conjugate via $b$, their Newton points $\nu_{\s b}$ and $\nu_b$ are conjugate via $b$. So it suffices to show that $\det_V\circ\Darst\circ\nu_{\s b}=t_N\bigl(\ulD_{b,\gamma}(V)\bigr)$. The latter follows from the construction of $\nu_{\s b}$ in \cite[\S\,4.2]{Kottwitz85}. The statement about $t_H$ follows from the fact that $\Darst_*(\gamma^\#)=\Darst(\gamma)^\#=-\ord_{z-\zeta}\bigl(\det\Darst(\gamma))$ under the identification $\pi_0(\Gr_{\GL(V)}^{\bB_\dR})=\pi_1(\GL(V))_{\Gamma}\cong\BZ$. If $[b]^{\#}=\gamma^{\#}$ holds in $\pi_1(G)_{\Gamma,\BQ}$ then $\Darst_*([b]^\#)=\Darst_*(\gamma^\#)$ in $\pi_1(\GL(V))_{\Gamma,\BQ}\cong\BQ$. Under the last isomorphism we have $\Darst_*([b]^\#)=(\Darst\circ\nu_b)^\#=\det_V\circ\Darst\circ\nu_b$. This proves one direction of the last assertion. For the other direction we use the isomorphism $\pi_1(G)_{\Gamma,\BQ}\cong  \pi_1(G_{{\rm ab }})_{\Gamma,\BQ}\cong  X_*(G_{{\rm ab }})_{\Gamma,\BQ}$ where $G_{{\rm ab }}$ denotes the maximal abelian quotient of $G$ (compare \cite[Theorem~1.15(ii)]{RapoportRichartz}). Now assume that $[b]^{\#}\neq\gamma^{\#}$. Then there is a homomorphism $\phi\colon\pi_1(G)_{{\Gamma},\BQ}\rightarrow \BQ$ of $\BQ$-vector spaces such that $\phi([b]^{\#})\neq\phi(\gamma^{\#})$. We have ${\rm Hom}(\pi_1(G)_{{\Gamma},\BQ}, \BQ)\cong X^*(G_{{\rm ab }})^{\Gamma}_{\BQ}$, thus a non-zero integral multiple of $\phi$ induces a morphism $\Darst\colon G\rightarrow G_{{\rm ab }}\rightarrow \BG_m$ over $\BF_q\dpl z\dpr^\sep$ that is ${\Gamma}$-invariant and therefore defined over $\BF_q\dpl z\dpr$. For this representation $\Darst$ we then have $t_N\bigl(\ulD_{b,\gamma}(V)\bigr)\neq t_H\bigl(\ulD_{b,\gamma}(V)\bigr)$.
\end{proof}

\begin{definition} \label{DefWAPair}
We say that the pair $(b,\gamma)\in LG(\BaseFld)\times\Gr_G^{\bB_\dR}(L)$ is (\emph{weakly}) \emph{admissible} if $[b]^{\#}=\gamma^{\#}$ in $\pi_1(G)_{\Gamma,\BQ}$ and one of the following equivalent conditions holds:
\begin{enumerate}
\item 
$\ulD_{b,\gamma}(V)$ is (weakly) admissible for every representation $V$ in $\Rep_{\BF_q\dpl z\dpr}G$,
\item\label{DefWAPair_b} 
$\ulD_{b,\gamma}(V)$ is (weakly) admissible for some faithful representation $V$ in $\Rep_{\BF_q\dpl z\dpr}G$.
\end{enumerate}
In addition, $(b,\gamma)$ is \emph{neutral} if $[b]^{\#}=\gamma^{\#}$ in $\pi_1(G)_\Gamma$ already without tensoring with $\BQ$.
\end{definition}

\begin{remark}
If $\ulD_{b,\gamma}(V)$ is (weakly) admissible for every representation $V$ in $\Rep_{\BF_q\dpl z\dpr}G$, then $[b]^{\#}=\gamma^{\#}$ automatically holds in $\pi_1(G)_{\Gamma,\BQ}$ by Lemma \ref{LemmaKottwitzCond}.

In the analogous situation in mixed characteristic, the condition $[b]^{\#}=\gamma^{\#}$ also follows from \ref{DefWAPair_b}, due to the fact that in that case every $W$ as in the beginning of the following proof is even a direct summand of $U$.
\end{remark}
\begin{proof}[Proof of the equivalence in Definition~\ref{DefWAPair}]
Clearly (a) implies (b). For the converse fix a faithful representation $V$. Then every $\BF_q\dpl z\dpr$-rational representation $W$ of $G_{\BF_q\dpl z\dpr}$ is a subquotient of $U:=\bigoplus_{i=1}^r V^{\otimes l_i}\otimes (V\dual)^{\otimes m_i}$ for suitable $r$, $l_i$ and $m_i$. If $\ulD_{b,\gamma}(V)$ is weakly admissible then this also holds for $\ulD_{b,\gamma}(U)$ by \cite[Theorem~2.2.5]{HartlPSp}. Likewise, if $\ulD_{b,\gamma}(V)$ is admissible we use the compatibility of the functor $V\mapsto\ulCF_{b,\gamma}(V)$ with direct sums, tensor products and duals to compute
\begin{eqnarray*}
\ulCF_{b,\gamma}(U)\otimes_{L\ancon[s]}\ol L\ancon[s] & \cong & \bigoplus_{i=1}^r \ulCF_{b,\gamma}(V)^{\otimes l_i}\otimes (\ulCF_{b,\gamma}(V)\dual)^{\otimes m_i}\otimes_{L\ancon[s]}\ol L\ancon[s] \\   
& \cong & \bigoplus_{i=1}^r (\ulCF_{0,1}{}^{\oplus\dim V})^{\otimes l_i}\otimes ((\ulCF_{0,1}{}^{\oplus\dim V})\dual)^{\otimes m_i}\otimes_{L\ancon[s]}\ol L\ancon[s] \\   
& \cong & \ulCF_{0,1}{}^{\oplus\dim U}\otimes_{L\ancon[s]}\ol L\ancon[s]\,.
\end{eqnarray*}
So if $\ulD_{b,\gamma}(V)$ is admissible also $\ulD_{b,\gamma}(U)$ is. Therefore it suffices to show that (weak) admissibility is preserved under passage to sub-representations and quotient representations.

By Lemma~\ref{LemmaKottwitzCond} the condition $[b]^{\#}=\gamma^{\#}$ in $\pi_1(G)_{\Gamma,\BQ}$ implies $t_N\bigl(\ulD_{b,\gamma}(W)\bigr)\,=\,t_H\bigl(\ulD_{b,\gamma}(W)\bigr)$ for all representations $W$. Now let $U$ be a representation such that $\ulD_{b,\gamma}(U)$ is weakly admissible. Then the equivalent conditions from Definition~\ref{DefWA}\ref{DefWA_C} show that for every sub-representation or quotient representation $W$ of $U$ also $\ulD_{b,\gamma}(W)$ is weakly admissible.

If $\ulD_{b,\gamma}(U)$ is actually admissible then $\ulCF_{b,\gamma}(U)\otimes_{L\ancon[s]}\ol L\ancon[s]\cong\ulCF_{0,1}{}^{\oplus\dim U}$. If $W\subset U$ is a sub-representation then the $\sigma$-sub-bundle $\ulCF_{b,\gamma}(W)\subset\ulCF_{b,\gamma}(U)$ satisfies $\ulCF_{b,\gamma}(W)\otimes_{L\ancon[s]}\ol L\ancon[s]\cong\bigoplus_i\ulCF_{d_i,n_i}$ for some $d_i,n_i$ by Proposition~\ref{PropHP}. By Proposition~\ref{Prop0.7} all $d_i\leq0$. Since
\[
\sum_i d_i\es=\es\deg\ulCF_{b,\gamma}(W)\es=\es t_H\bigl(\ulD_{b,\gamma}(W)\bigr)- t_N\bigl(\ulD_{b,\gamma}(W)\bigr)\es=\es0
\]
by Proposition~\ref{PropDeg}, all $d_i$ must be zero and $\ulCF_{b,\gamma}(W)$ is admissible.
Dually if $W$ is a quotient representation of $U$, the $\sigma$-quotient-bundle $\ulCF_{b,\gamma}(W)$ of $\ulCF_{b,\gamma}(U)$ satisfies $\ulCF_{b,\gamma}(W)\otimes_{L\ancon[s]}\ol L\ancon[s]\cong\bigoplus_i\ulCF_{d_i,n_i}$ for some $d_i,n_i$ with $d_i\ge0$ by Proposition~\ref{Prop0.7}. Again $\deg\ulCF_{b,\gamma}(W)=0$ implies $d_i=0$ and $\ulCF_{b,\gamma}(W)$ is admissible.
\end{proof}

\begin{remark}\label{RemAcceptable}
Let $b\in LG(\BaseFld)$. Let $L_0$ be a finite field extension of $\BF_q\dpl\zeta\dpr$ for which $G_{L_0}:=G\times_{\BF_q\dbl z\dbr,z\mapsto\zeta}L_0$ is split. Let $T$ be a maximal split torus of $G_{L_0}$ that contains the image of the Newton point $\nu_b\colon \BD_{\BaseFld\dpl z\dpr}\to G_{\BaseFld\dpl z\dpr}$; see Remark~\ref{RemKottwitz}(a). We may view $\nu_b$ as an element of $X_*(T)_\BQ:=X_*(T)\otimes_\BZ \BQ$. Let $L$ be a complete valued field extension of the completion of the maximal unramified extension $\breve L_0$ of $L_0$ and let $\gamma\in\Gr_G^{\bB_\dR}(L)$. By the Cartan decomposition there is a unique dominant cocharacter $\mu_{\gamma}\in X_*(T)$ called the \emph{Hodge point of $\gamma$} such that 
\[
\gamma\;\in\; G(L\dbl z-\zeta\dbr)\cdot\mu_\gamma(z-\zeta)\cdot G(L\dbl z-\zeta\dbr)/G(L\dbl z-\zeta\dbr)\;\subset\;\Gr_G^{\bB_\dR}(L).
\]
If $(b,\gamma)$ is weakly admissible then  $\nu_b\preceq\mu_{\gamma}$ (see \cite[Theorem~9.5.10]{DOR}, which is for the arithmetic context, but gives a proof that can directly be translated to our situation), in other words $([b],\{\mu_\gamma\})$ is acceptable in the sense of \cite[Definition~2.5]{RV}. The converse of this is not true. However, if $\mu\in X_*(T)$ is dominant with $\nu_b\preceq\mu$, then one can show that there exists a cocharacter $\Int_g\circ\mu\colon\BG_{m,L}\to G_L$ with $g\in G(L)$ for a finite extension $L\supset\breve L_0$, that induces a weakly admissible Hodge-Pink filtration on $\ulD_b(V)$ for all $V$. Indeed, this can be shown in the same way as the arithmetic counterpart, compare \cite[Theorem~9.5.10]{DOR}. Then by Remark~\ref{RemDIsTensorFunctor}(a), $\ulD_{b,\gamma}(V)$ is weakly admissible (and even admissible) for every Hodge-Pink $G$-structure $\gamma\in\Gr_G^{\bB_\dR}(L)$ that induces $\Int_g\circ\mu$, like for example $\gamma=g\cdot\mu(z-\zeta)$. For more details and references in the arithmetic context compare also the discussion in \cite[Section~2.2 and Proposition~3.1]{RV}.
\end{remark}

\medskip

We next want to define period spaces in the bounded situation. Let $\hat{Z}=[(R,\hat Z_R)]$ be a bound as in Definition~\ref{DefBDLocal} with reflex ring $R_{\hat Z}=\kappa\dbl\xi\dbr$ and set $E:=E_{\hat Z}=\kappa\dpl\xi\dpr$ and $\breve E:=\BaseFld\dpl\xi\dpr$. By Proposition~\ref{PropBound}\ref{PropBound_D} the associated strictly $R[\tfrac{1}{\zeta}]$-analytic spaces $\hat{Z}_R^\an$ arise by base change to $R[\tfrac{1}{\zeta}]$ from a strictly $E_{\hat Z}$-analytic space $\hat Z^\an$ associated with a projective scheme $\hat{Z}_E$ over $\Spec E_{\hat Z}$, which is a closed subscheme of the affine Grassmannian $\Gr_{G,E}^{\bB_\dR}:=\Gr_G^{\bB_\dR}\times_{\BF_q\dpl\zeta\dpr}\Spec E_{\hat Z}$. 

\begin{definition}\label{DefPeriodSp}
We call $\CH_{G,\hat{Z}}:=\hat{Z}_E$ the \emph{space of Hodge-Pink $G$-structures bounded by $\hat{Z}$} and set $\breve\CH_{G,\hat{Z}}:=\CH_{G,\hat{Z}}\times_{E_{\hat Z}}\Spec\breve E$. Let $\ul\BG_0=\bigl((L^+G)_\BaseFld,b\s\bigr)$ be a local $G$-shtuka over $\BaseFld$. We define the \emph{period spaces of} (\!\emph{weakly}) \emph{admissible Hodge-Pink $G$-structures on $\ul\BG_0$ bounded by $\hat{Z}$} as
\begin{eqnarray*}
\breve\CH_{G,\hat{Z},b}^{wa} & := & \bigl\{\,\gamma\in\breve\CH_{G,\hat{Z}}^\an\colon\text{ the pair $(b,\gamma)$ is weakly admissible }\bigr\}\,,\\[2mm]
\breve\CH_{G,\hat{Z},b}^{a} & := & \bigl\{\,\gamma\in\breve\CH_{G,\hat{Z}}^\an\colon\text{ the pair $(b,\gamma)$ is admissible }\bigr\}\quad\text{and}\\[2mm]
\breve\CH_{G,\hat{Z},b}^{na} & := & \bigl\{\,\gamma\in\breve\CH_{G,\hat{Z}}^\an\colon\text{ the pair $(b,\gamma)$ is admissible and neutral }\bigr\}\,.
\end{eqnarray*}
$\breve\CH_{G,\hat{Z},b}^{na}$ equals the intersection of $\breve\CH_{G,\hat{Z},b}^{a}$ with the union of those connected components of $\breve\CH_{G,\hat Z}$ that map to $[b]^\#\in\pi_1(G)_\Gamma$ under the map $\pi_0(\breve\CH_{G,\hat Z})\to\pi_0(\Gr_G^{\bB_\dR})\onto\pi_1(G)_\Gamma$ induced by \eqref{EqPi0OfHG}. In particular, $\breve\CH_{G,\hat{Z},b}^{na}$ is a union of connected components of $\breve\CH_{G,\hat{Z},b}^{a}$. The period spaces only depend on $b$ and on the generic fibers $G_{\BF_q\dpl z\dpr}$ and $\hat{Z}_E$ of $G$ and $\hat{Z}$.
\end{definition}

\begin{remark} \label{RemJActsOnPeriodSp}
If $\breve E\subset L$ then the homomorphism $\BaseFld\dpl z\dpr\to L\dbl z-\zeta\dbr,\,z\mapsto z=\zeta+(z-\zeta)$ induces a homomorphism $LG(\BaseFld)=G\bigl(\BaseFld\dpl z\dpr\bigr)\to G\bigl(L\dbl z-\zeta\dbr\bigr)$. Thus if $b'=g\,b\,\s(g^{-1})$ for some $g\in LG(\BaseFld)$ one can check that $\gamma\mapsto \s(g)\cdot\gamma=:\gamma'$ maps $\breve\CH_{G,\hat{Z},b}^{wa}$ isomorphically onto $\breve\CH_{G,\hat{Z},b'}^{wa}$ (and likewise for $\breve\CH_{G,\hat{Z},b'}^{a}$ and $\breve\CH_{G,\hat{Z},b'}^{na}$), because $\s(g)$ maps $\breve\CH_{G,\hat Z}^\an$ to itself by Lemma \ref{LemBoundConjInv}, and induces isomorphisms $\ulD_{b'\!,\gamma'}(V)\cong\ulD_{b,\gamma}(V)$ and $\ulCF_{b'\!,\gamma'}(V)\cong\ulCF_{b,\gamma}(V)$. In particular $\breve\CH_{G,\hat{Z},b}^{wa}$, $\breve\CH_{G,\hat{Z},b}^{a}$ and $\breve\CH_{G,\hat{Z},b}^{na}$ are invariant under the group $J_b\bigl(\BF_q\dpl z\dpr\bigr)\;=\;\QIsog_\BaseFld\bigl((L^+G)_\BaseFld,b\s\bigr)$ from \eqref{EqJ}.
\end{remark}
\begin{proposition}\label{PropAorWA}
The space $\breve\CH_{G,\hat{Z},b}^{a}$ is contained in $\breve\CH_{G,\hat{Z},b}^{wa}$ with $\breve\CH_{G,\hat{Z},b}^{a}(L)=\breve\CH_{G,\hat{Z},b}^{wa}(L)$ for all complete valued field extensions $L/\breve E$ satisfying the following condition: \enspace
Let $\ol L$ be the completion of an algebraic closure of $L$, let $\bar\ell\subset\ol L$ be a subfield isomorphic to the residue field of $\ol L$ under the residue map $\CO_{\ol L}\onto\CO_{\ol L}/\Fm_{\ol L}$, and let $\wt L$ be the closure of the compositum $\bar\ell L$ inside $\ol L$. (In particular, if the residue field of $L$ is perfect, then $\wt L$ is the completion of the maximal unramified extension of $L$.) The condition is that $\wt L$ does not contain an element $a$ with $0<|a|<1$ such that all the $q$-power roots of $a$ also lie in $\wt L$. 
\end{proposition}

\begin{proof}
The inclusion $\breve\CH_{G,\hat{Z},b}^{a}\subset\breve\CH_{G,\hat{Z},b}^{wa}$ follows from Corollary~\ref{CorA=>WA}. The equality $\breve\CH_{G,\hat{Z},b}^{a}(L)=\breve\CH_{G,\hat{Z},b}^{wa}(L)$ for the mentioned fields was proved in \cite[Theorem~2.5.3]{HartlPSp}.
\end{proof}

\begin{remark}\label{RemAorWA}
(a) The condition of the proposition, and hence $\breve\CH_{G,\hat{Z},b}^{a}(L)=\breve\CH_{G,\hat{Z},b}^{wa}(L)$ holds, if the value group of $L$ does not contain a non-zero element that is arbitrarily often divisible by $q$. This is due to the fact that the value groups of $L$ and $\wt L$ coincide. In particular, this is the case if $L$ is a finite field extension of $\breve E$, or more generally if $L$ is discretely valued, or even if the value group of $L$ is finitely generated. See \cite[Condition~(2.3) on page 1294]{HartlPSp} for further discussion of this condition.

\smallskip\noindent
(b) If $L$ violates the condition, for example if $L$ is algebraically closed and complete, it can happen that $\breve\CH_{G,\hat{Z},b}^{a}(L)\subsetneq\breve\CH_{G,\hat{Z},b}^{wa}(L)$. Examples in the case $G=\GL_r$ were given in \cite[Example~3.3.2]{HartlPSp}.
\end{remark}

\begin{theorem}\label{ThmWAOpen}
The period space $\breve\CH_{G,\hat{Z},b}^{wa}$ and the admissible locus $\breve\CH_{G,\hat{Z},b}^a$ are open paracompact strictly $\breve E$-analytic subspaces of $\breve\CH_{G,\hat{Z}}^\an$. The intersections of any connected component of $\breve\CH_{G,\hat{Z}}^\an$ with $\breve\CH_{G,\hat{Z},b}^{wa}$ and $\breve\CH_{G,\hat{Z},b}^a$ are arcwise connected. In the terminology of Remarks~\ref{DefBDLocal'}\ref{DefBDLocal_A6} and \ref{RemAcceptable} the spaces $\breve\CH_{G,\hat{Z},b}^{wa}$ and $\breve\CH_{G,\hat{Z},b}^a$ intersect after base change to $\breve L_0$ precisely those connected components of $\Gr_{G,\breve L_0}^{\bB_\dR}$ whose image in $\pi_1(G)=\pi_0(\Gr_{G,\breve L_0}^{\bB_\dR})$ (using \eqref{EqPi0OfHG}) is of the form $\mu^{\#}\in\pi_1(G)$ for a $\mu\in N_\an$ with $\nu_b\preceq\mu$.
\end{theorem}

\begin{proof}
Choose a faithful representation $\Darst\colon G\into\GL_r$ in $\Rep_{\BF_q\dpl z\dpr}G$ that factors through $\SL_r$, and let $n$ be an integer as in Proposition~\ref{PropBound}\ref{PropBound_A} for which $\Darst_*\colon\hat{Z}\to\wh\Flag_{\SL_r}$ factors through $\wh\Flag^{(n)}_{\SL_r}=\hat Z_{\GL_r,2n\rho\dual}$; see Examples~\ref{ExConstantG} and \ref{Ex-mu_dom}. By Proposition~\ref{PropBound} the $\breve E$-analytic space $\breve\CH_{G,\hat{Z}}^\an$ is a subspace of $\CH_{\GL_r,2n\rho\dual}^\an\wh\otimes_{\BF_q\dpl\zeta\dpr}\breve E$, where $\CH_{\GL_r,2n\rho\dual}^\an:=(\hat{Z}_{\GL_r,2n\rho\dual})^\an$. On the connected components of $\breve\CH_{G,\hat{Z}}^\an$ where $[b]^{\#}=\gamma^{\#}$ in $\pi_1(G)_{\Gamma,\BQ}$ we have by Definition~\ref{DefWAPair} 
\begin{eqnarray*}
\breve\CH_{G,\hat{Z},b}^{wa} & = & \breve\CH_{G,\hat{Z}}^\an\;\cap\;\breve\CH_{\GL_r,2n\rho\dual\!,\Darst(b)}^{wa}\wh\otimes_{\BaseFld\dpl\zeta\dpr}\breve E\qquad \text{and}\\[2mm]
\breve\CH_{G,\hat{Z},b}^a & = & \breve\CH_{G,\hat{Z}}^\an\;\cap\;\breve\CH_{\GL_r,2n\rho\dual\!,\Darst(b)}^a\wh\otimes_{\BaseFld\dpl\zeta\dpr}\breve E.
\end{eqnarray*}
The intersections of the other components with $\breve\CH_{G,\hat{Z},b}^{wa}$ and $\breve\CH_{G,\hat{Z},b}^a$ are empty.
Since every open subspace of the compact $\breve E$-analytic space $\breve\CH_{G,\hat{Z}}^\an$ is paracompact by \cite[Lemma~A.2.6]{HartlPSp}, it suffices to show that $\breve\CH_{\GL_r,2n\rho\dual\!,\Darst(b)}^{wa}$ and $\breve\CH_{\GL_r,2n\rho\dual\!,\Darst(b)}^a$ are open in $\breve\CH_{\GL_r,2n\rho\dual}^\an\wh\otimes_{\BaseFld\dpl\zeta\dpr}\breve E$. An analogous statement was proved in \cite[Theorems~3.2.2 and 3.2.4]{HartlPSp} for the quasi-projective Schubert cell 
\[
\breve\CC\;:=\;\GL_r\bigl(\fdot\dbl z-\zeta\dbr\bigr)\cdot (2n\rho\dual)(z-\zeta)\cdot\GL_r\bigl(\fdot\dbl z-\zeta\dbr\bigr)\big/\GL_r\bigl(\fdot\dbl z-\zeta\dbr\bigr)
\]
from \cite[Definition~3.1.6]{HartlPSp}, which is open and dense in the Schubert variety $\breve\CH_{\GL_r,2n\rho\dual}$. Let us explain how to modify that proof to obtain a proof of the assertion above. The Schubert cell is a homogeneous space $\breve\CC=\wt G/S$; see \cite[p.~1318]{HartlPSp}. The properties that were needed in the proofs of \cite[Theorems~3.2.2 and 3.2.4]{HartlPSp} were the following two. The morphism $\wt G^\an\to\breve\CC^\an$ is smooth, and therefore $\breve\CC^\an$ carries the quotient topology under the morphism $\wt G^\an\to\breve\CC^\an$. Secondly, the universal Hodge-Pink structure on $\breve\CC$ is given on $\wt G$ by a universal matrix $g$ in $\GL_r\bigl(\CO_{\wt G}\dbl z-\zeta\dbr/(z-\zeta)^{2n(r-1)}\bigr)$. 

For our purpose here we modify this as follows. Since $LG\to\Flag_G$ has local sections for the \'etale topology, there is an \'etale covering $X$ of $\breve\CH_{\GL_r,2n\rho\dual}^\an$ on which the universal Hodge-Pink $G$-structure is given by a universal element $h$ in $G\bigl(\CO_X\dpl z-\zeta\dpr\bigr)$ that satisfies $\Darst(h)\in M_r\bigl((z-\zeta)^{-n(r-1)}\CO_X\dbl z-\zeta\dbr\bigr)$. We replace $g$ by $(z-\zeta)^{n(r-1)}\Darst(h)\mod (z-\zeta)^{2n(r-1)}$ and use that $\breve\CH_{\GL_r,2n\rho\dual}^\an$ carries the quotient topology under the morphism $X\to\breve\CH_{\GL_r,2n\rho\dual}^\an$ by \cite[Corollary~3.7.4]{Berkovich2}. With these modifications the proofs of \cite[Theorems~3.2.2 and 3.2.4]{HartlPSp} carry over to our situation.

The connectedness of $\breve\CH_{G,\hat{Z},b}^{wa}$ and $\breve\CH_{G,\hat{Z},b}^a$ can be proved by the same arguments as in \cite[Theorem~3.2.5]{HartlPSp}.

It remains to compute which connected components of $\Gr_{G,\breve L_0}^{\bB_\dR}$ meet $\breve\CH_{G,\hat{Z},b}^a$ and $\breve\CH_{G,\hat{Z},b}^{wa}$. By Remark~\ref{RemAcceptable} for every point $\gamma\in\breve\CH_{G,\hat{Z},b}^{wa}\wh\otimes_{\breve E}\breve L_0$ the Hodge point $\mu_\gamma\in X_*(T)$ lies in $N_\an$ and $\nu_b\preceq\mu_\gamma$. In particular, $\gamma$ lies in the component with image $\mu_\gamma^\#\in\pi_1(G)$. Conversely, let $\mu\in N_\an$ with $\nu_b\preceq\mu$. Then Remark~\ref{RemAcceptable} implies that there is a point $\gamma\in\Gr_{G,\breve L_0}^{\bB_\dR}$ with Hodge point $\mu$ such that $(b,\gamma)$ is weakly admissible. By Remark~\ref{DefBDLocal'}\ref{DefBDLocal_A6} the point $\gamma$ lies in $\breve\CH_{G,\hat{Z},b}^a$ and $\breve\CH_{G,\hat{Z},b}^{wa}$, and moreover, it lies in the connected component of $\Gr_{G,\breve L_0}^{\bB_\dR}$ with image $\mu^\#$ in $\pi_0(\Gr_{G,\breve L_0}^{\bB_\dR})=\pi_1(G)$.
\end{proof}

\begin{remark}\label{RemFunctoriality2}
We keep the notation of Remark~\ref{RemFunctoriality1}. The morphism $\epsilon\colon G\to G'$ of parahoric group schemes over $\BF_q\dbl z\dbr$ induces a morphism $\epsilon\colon\Gr_{G}^{\bB_\dR}\to\Gr_{G'}^{\bB_\dR}$ of the affine Grassmannians from \eqref{EqH_G}. It maps Hodge-Pink $G$-structures $\gamma\in \Gr_{G}^{\bB_\dR}$ to Hodge-Pink $G'$-structures $\gamma':=\epsilon(\gamma)\in\Gr_{G'}^{\bB_\dR}$ and the corresponding element $\gamma^\#\in\pi_1(G)_\Gamma$ to $(\gamma')^\#=\epsilon(\gamma^\#)\in\pi_1(G')_\Gamma$.

The morphism $\epsilon$ also induces a tensor functor $\epsilon^*\colon\Rep_{\BF_q\dpl z\dpr}G'\to\Rep_{\BF_q\dpl z\dpr}G$, given by $(V'\!,\Darst')\mapsto\epsilon^*(V'\!,\Darst'):=(V'\!,\Darst'\circ\epsilon)$. It satisfies $\ulD_{b,\gamma}(\epsilon^*(V'\!,\Darst'))=\ulD_{b'\!,\gamma'}(V'\!,\Darst')$ for $b'=\epsilon(b)$. Therefore, $\ulCE_{b,\gamma}(\epsilon^*(V'\!,\Darst'))=\ulCE_{b'\!,\gamma'}(V'\!,\Darst')$ and $\ulCF_{b,\gamma}(\epsilon^*(V'\!,\Darst'))=\ulCF_{b'\!,\gamma'}(V'\!,\Darst')$. It follows that $(b,\gamma)$ is (weakly) admissible for $G$ if and only if $(b'\!,\gamma')$ is for $G'$. Moreover, if $(b,\gamma)$ is neutral, then also $(b'\!,\gamma')$ is neutral. In other words, if $\epsilon(\hat{Z})\subset\hat{Z}'$ then $\epsilon$ induces morphisms
\begin{equation}\label{EqEpsilon^*}
\epsilon\colon\breve\CH_{G,\hat{Z}}\;\longto\;\breve\CH_{G'\!,\hat{Z}'} \quad\text{and}\quad \epsilon\colon\breve\CH_{G,\hat{Z},b}^\bullet\;\longto\;\breve\CH_{G'\!,\hat{Z}'\!,b'}^\bullet,\quad \gamma\;\longmapsto\;\epsilon(\gamma)
\end{equation}
for $\bullet\in\{wa, a, na\}$. 
\end{remark}

\section{Local systems of $\BF_q\dpl z\dpr$-vector spaces} \label{SectLocSyst}
\setcounter{equation}{0}

\begin{definition}\label{DefForget}
For a ring $A$ we let $\FMod_A$ denote the category of finite locally free $A$-modules. If the ring $A$ is either $\BF_q\dbl z\dbr$ or $\BF_q\dpl z\dpr$ and $\Pi$ is a topological group we denote by $\Rep^\cont_A(\Pi)$ the category of continuous representations in finite free $A$-modules, and by
\begin{equation}\label{EqForget}
{\it forget}\colon\Rep^\cont_{A}\bigl(\pi_1^\et(X,\bar x)\bigr)\;\longto\;\FMod_A
\end{equation}
the forgetful fiber functor. Moreover, we let
\begin{equation}\label{EqOmega0}
\omega\open_A\colon\Rep_{A}G\;\longto\;\FMod_A
\end{equation}
be the forgetful fiber functor. We also write $\omega\open:=\omega\open_{\BF_q\dpl z\dpr}$. Then $\Aut^\otimes(\omega\open)=G_{\BF_q\dpl z\dpr}$ by \cite[Theorem~2.11]{DM82} and $\Aut^\otimes(\omega\open_{\BF_q\dbl z\dbr})=G$ by \cite[Corollary~5.20]{Wed}. 
\end{definition}

Let $X$ be a strictly $L$-analytic space, where $L$ is a field extensions of $\BF_q\dpl\zeta\dpr$ that is complete with respect to an absolute value $|\,.\,|\colon L\to\BR_{\ge0}$ extending the $\zeta$-adic absolute value on $\BF_q\dpl\zeta\dpr$. For any group or ring $A$ we denote by $\ul{A}$ the locally constant sheaf on the \'etale site $X_\et$ of $X$.

We recall de Jong's~\cite[\S\,2]{dJ95a} definition of the \'etale fundamental group of $X$. De Jong calls a morphism $f\colon Y\to X$ of $L$-analytic spaces an \emph{{\'e}tale covering space of $X$\/} if for every analytic point $x$ of $X$ there exists an open neighborhood $U\subset X$ such that $Y\times_XU$ is a disjoint union of $L$-analytic spaces $V_i$ each mapping finite \'etale to $U$. The {\'e}tale covering spaces of $X$ form a category $\ul{\Cov}_X^\et$. It contains the full subcategory $\ul{\Cov}_X^\alg$ of \emph{finite} \'etale covering spaces.

A geometric base point $\bar x$ of $X$ is a morphism $\bar x\colon\SpBerk(\ol L)\to X$ where we denote by $\SpBerk(\ol L)$ the Berkovich spectrum of an algebraically closed complete extension $\ol L$ of $L$. For a geometric base point $\bar x$ of $X$ define the fiber functors at $\bar x$ 
\begin{eqnarray}
F_{\bar x}^\et :=\!\!\!\! & F_{X,\bar x}^\et\colon & \!\ul{\Cov}_X^\et \;\longto \;\ul{\rm Sets}\;,\quad F_{\bar x}^\et\bigl(f\colon Y\to X\bigr) \;:=\;\{\,\bar{y}\colon\SpBerk(\ol L)\to Y \;\text{with} \; f\circ\bar{y} = \bar{x}\,\} \nonumber\\[2mm]
F_{\bar x}^\alg :=\!\!\!\! & F_{X,\bar x}^\alg\colon & \!\ul{\Cov}_X^\alg \;\longto \;\ul{\rm Sets}\;,\quad F_{\bar x}^\alg:=F_{\bar x}^\et|_{\ul{\Cov}_X^\alg}\;.\label{EqEtaleFiberFunctor}
\end{eqnarray}
The \emph{\'etale fundamental group} $\pi_1^\et(X,\bar{x})$ and the \emph{algebraic fundamental group} $\pi_1^\alg(X,\bar{x})$ of $X$ are the automorphism groups
\[ 
\pi_1^\et(X,\bar{x}) \; := \; \Aut(F_{\bar x}^\et)\qquad\text{and}\qquad\pi_1^\alg(X,\bar{x}) \; := \; \Aut(F_{\bar x}^\alg)\,. 
\]  
These fundamental groups classify (finite) \'etale covering spaces in the sense that $F_{\bar x}^\et$ induces an equivalence 
\begin{equation}\label{EqFundGp}
F_{\bar x}^\et\colon\;\{\text{disjoint unions of objects of }\ul{\Cov}_X^\et\} \;\longto\; \pi_1^\et(X,\bar{x})\text{\rm -\ul{Sets}}\,.
\end{equation}
Connected coverings correspond to $\pi_1^\et(X,\bar{x})$-orbits, and similarly for $F_{\bar x}^\alg$; see \cite[Theorem 2.10]{dJ95a}. Here $\pi_1^\et(X,\bar{x})$-\ul{Sets} (respectively $\pi_1^\alg(X,\bar{x})$-\ul{Sets}) is the category of discrete (respectively finite) sets endowed with a continuous left action of $\pi_1^\et(X,\bar{x})$ (respectively $\pi_1^\alg(X,\bar{x})$).

The natural continuous group homomorphism $\pi_1^\et(X,\bar{x}) \to\pi_1^\alg(X,\bar{x})$ has dense image. The \'etale fundamental group $\pi_1^\et(X,\bar{x})$ is Hausdorff and pro-discrete, $\pi_1^\alg(X,\bar{x})$ is pro-finite and every continuous homomorphism from $\pi_1^\et(X,\bar{x})$ to a pro-finite group factors through $\pi_1^\alg(X,\bar{x})$; see \cite[Lemma~2.7 and Theorem 2.10]{dJ95a}. In particular, $\Rep^\cont_{\BF_q\dbl z\dbr}\bigl(\pi_1^\et(X,\bar x)\bigr)=\Rep^\cont_{\BF_q\dbl z\dbr}\bigl(\pi_1^\alg(X,\bar x)\bigr)$, but this is not true for representations on $\BF_q\dpl z\dpr$-vector spaces.

For the following overview we follow \cite[Definition~A4.4]{HartlPSp}.
\begin{definition}\label{DefLocalSystem}
A \emph{local system of $\BF_q\dbl z\dbr$-lattices} on $X$ is a projective system $\CF=(\CF_n,i_n)$ of sheaves $\CF_n$ of $\BF_q\dbl z\dbr/(z^n)$-modules on $X_\et$, such that 
\begin{enumerate}
\item $\CF_n$ is \'etale locally a constant free $\ul{\BF_q\dbl z \dbr/(z^n)}$-module of finite rank 
\item $i_n\otimes\id\colon \CF_n\otimes_{\ul{\BF_q\dbl z \dbr/(z^n)}}\;\,\ul{\BF_q\dbl z \dbr/(z^{n-1})}\;\isoto\;\CF_{n-1}$ is an isomorphism of sheaves of $\ul{\BF_q\dbl z \dbr/(z^{n-1})}$-modules.
\end{enumerate}
The category $\BLoc_X$ of local systems of $\BF_q\dbl z\dbr$-lattices with the obvious morphisms is an additive $\BF_q\dbl z\dbr$-linear rigid tensor category.
If $\bar x$ is a geometric point of $X$
\[
\CF_{\bar x}\es:=\es\invlim(\CF_{n,\bar x},\,i_n)\,.
\]
is the \emph{stalk} $\CF_{\bar x}$ of $\CF$ at $\bar x$. It is a finite free $\BF_q\dbl z\dbr$-module. Starting from $\BF_q\dbl z\dbr$-lattices one defines local systems of $\BF_q\dpl z\dpr$-vector spaces and their stalks as in \cite[\S\,4]{dJ95a}, or \cite[Definition A4.4]{HartlPSp}.

Local systems of $\BF_q\dpl z\dpr$-vector spaces form a category $\PLoc_X$. It is an abelian $\BF_q\dpl z\dpr$-linear rigid tensor category. The theory of these local systems parallels the theory of local systems of $\BQ_\ell$-vector spaces developed in \cite{dJ95a}.
In particular there is the following description.
\end{definition}

\begin{proposition} \label{Prop2.13}
(Compare \cite[Corollary~4.2]{dJ95a}.) For any geometric point $\bar x$ of $X$ there is a natural $\BF_q\dbl z\dbr$-linear tensor functor
\[
\omega_{\bar x}\colon \BLoc_X\;\longto\; \Rep^\cont_{\BF_q\dbl z\dbr}\bigl(\pi_1^\alg(X,\bar x)\bigr) \; = \; \Rep^\cont_{\BF_q\dbl z\dbr}\bigl(\pi_1^\et(X,\bar x)\bigr)
\]
and a natural $\BF_q\dpl z\dpr$-linear tensor functor
\[
\omega_{\bar x}\colon \PLoc_X\;\longto\; \Rep^\cont_{\BF_q\dpl z\dpr}\bigl(\pi_1^\et(X,\bar x)\bigr)
\]
that assigns to a local system $\CF\in\BLoc_X$, respectively $\CV\in\PLoc_X$ the representation of $\pi_1^\alg(X,\bar x)$ on $\CF_{\bar x}$, respectively of $\pi_1^\et(X,\bar x)$ on $\CV_{\bar x}$. These tensor functors are equivalences if $X$ is connected.
\end{proposition}

\begin{proof}
De Jong \cite[Corollary~4.2]{dJ95a} proved this for $\BQ_\ell$ and the statement for $\BF_q\dpl z\dpr$ is proved verbatim. We indicate the (easier) argument for $\BF_q\dbl z\dbr$. Let $\CF=(\CF_n,i_n)\in\BLoc_X$. Then the $\CF_n$ are represented by finite \'etale covering spaces $Y_n$ of $X$. This yields an action of $\pi_1^\alg(X,\bar x)$ on $F^\alg_{\bar x}(Y_n)=\CF_{n,\bar x}$ and on $\CF_{\bar x}:=\invlim(\CF_{n,\bar x},\,i_n)$.
\end{proof}

Let $A=\BF_q\dbl z\dbr$ or $A=\BF_q\dpl z\dpr$ and recall the forgetful fiber functor $\omega\open_A\colon\Rep_{A}G\to\FMod_A$ from Definition~\ref{DefForget}. If $A=\BF_q\dpl z\dpr$ we let in addition $\wt\omega\colon\Rep_{\BF_q\dpl z\dpr}G\to\FMod_{\BF_q\dpl z\dpr}$ be another fiber functor, and we let $\wt G:=\Aut^\otimes(\wt\omega)$ be the group scheme over $\BF_q\dpl z\dpr$ of tensor automorphisms of $\wt\omega$; see \cite[Theorem~2.11]{DM82}. Then $\Isom^\otimes(\omega\open,\wt\omega)$ is a left $G$-torsor and a right $\wt G$-torsor over $\BF_q\dpl z\dpr$ and corresponds to a cohomology class $cl(\omega\open,\wt\omega)\in\Koh^1(\BF_q\dpl z\dpr,G)$ by \cite[Theorem 3.2]{DM82}. The group $\wt G$ is isomorphic to the inner form of $G$ defined by the image of $cl(\omega\open,\wt\omega)$ in $\Koh^1(\BF_q\dpl z\dpr,G^\ad)$, where $G\onto G^\ad$ is the adjoint quotient. If $A=\BF_q\dbl z\dbr$ we set $\wt\omega:=\omega\open_{\BF_q\dbl z\dbr}$ and $\wt G:=G$. This is no restriction because Lang's Theorem~\cite[Theorem~2]{Lang}, stating $\Koh^1(\BF_q\dbl z\dbr,G)=\Koh^1(\BF_q,G)=(1)$, implies that all fiber functors $\Rep_{\BF_q\dbl z\dbr}G\to\FMod_{\BF_q\dbl z\dbr}$ are isomorphic to $\omega\open_{\BF_q\dbl z\dbr}$.

Let $X$ be connected, let $\bar x$ be a geometric base point of $X$, and recall the forgetful fiber functor ${\it forget}\colon\Rep^\cont_{A}\bigl(\pi_1^\et(X,\bar x)\bigr)\to\FMod_A$ from \eqref{EqForget}.

\begin{corollary}\label{Cor1.7}
In the situation above, consider the set $\CT_A$ of isomorphism classes of pairs $(\ulCV,\beta)$ where $\ulCV\colon\Rep_{A}G\to\ALoc_X$ is a tensor functor and $\beta\in\Isom^\otimes(\wt\omega,\,{\it forget}\circ\omega_{\bar x}\circ\ulCV)(A)$ is an isomorphism of tensor functors. There is a canonical bijection between $\CT_A$ and the set of continuous group homomorphisms
\[
\pi_1^\et(X,\bar  x)\;\longto\;\wt G(A)\,.
\]
\end{corollary}

\begin{proof}
Let $(\ulCV,\beta)\in\CT_A$. By Proposition~\ref{Prop2.13} any element of $\pi_1^\et(X,\bar  x)$ yields a tensor automorphism of the fiber functor ${\it forget}\circ\omega_{\bar x}\circ\ulCV$. By $\beta$ it is transported to a tensor automorphism of $\wt\omega$, that is an element of $\wt G(A)$. This defines a group homomorphism $f:=f_{(\ulCV,\beta)}\colon \pi_1^\et(X,\bar  x)\to\wt G(A)$. Since for all $V\in\Rep_{A}G$ the induced homomorphism \mbox{$\pi_1^\et(X,\bar  x)\to\wt G(A)\to\GL\bigl(\wt\omega(V)\bigr)(A)$} is continuous, also $f$ is continuous.

Conversely let $f\colon \pi_1^\et(X,\bar  x)\to\wt G(A)$ be a continuous group homomorphism. Then we define a tensor functor $\Rep_{A}G \to\Rep^\cont_{A}\bigl(\pi_1^\et(X,\bar  x)\bigr)$ by sending a representation $V$ in $\Rep_{A}G$ to the representation $\bigl(\wt\omega(V),\rho'_V\bigr)$ given by
\[
\rho'_V\colon\;\pi_1^\et(X,\bar x)\;\longto\;\GL\bigl(\wt\omega(V)\bigr)(A)\;,\quad g\;\mapsto\;\wt\omega(V)(f(g))\,.
\]
Here $\wt\omega(V)(f(g))$ is the automorphism by which $f(g)\in\wt G(A)=\Aut^\otimes(\wt\omega)(A)$ acts on the $A$-module $\wt\omega(V)$. Note that $\rho'_V$ is continuous because $\wt G(A)\to\GL\bigl(\wt\omega(V)\bigr)(A)$ is continuous. Let $\ulCV_f(V)\in\ALoc_X$ be the local system on $X$ induced from $\rho'_V$ via Proposition~\ref{Prop2.13}. This defines a tensor functor $\ulCV_f\colon \Rep_{A}G \to\ALoc_{X}$ for which ${\it forget}\circ\omega_{\bar x}\circ\ulCV$ is identified with $\wt\omega$. We let $\beta_f\in\Isom^\otimes(\wt\omega,{\it forget}\circ\omega_{\bar x}\circ\ulCV)(A)$ be this identification.

Clearly the assignments $(\ulCV,\beta)\mapsto f_{(\ulCV,\beta)}$ and $f\mapsto(\ulCV_f,\beta_f)$ satisfy $f=f_{(\ulCV_f,\beta_f)}$. Conversely, if $(\ulCV,\beta)\in\CT_A$ and $f=f_{(\ulCV,\beta)}$ then $\beta$ provides an isomorphism $(\ulCV_f,\beta_f)\isoto(\ulCV,\beta)$, and so $(\ulCV,\beta)$ and $(\ulCV_f,\beta_f)$ coincide in $\CT_A$.
\end{proof}

\begin{remark}\label{Rem1.8}
(a) In the situation of Corollary~\ref{Cor1.7} any tensor functor $\ulCV\colon \Rep_{\BF_q\dpl z\dpr}G \to\PLoc_{X}$ induces a tower of \'etale covering spaces of $X$ with Hecke action, as described in \cite[Remark 2.7(a)]{HartlRZ}. To recall the construction, assume that $\Isom^\otimes(\wt\omega,{\it forget}\circ\omega_{\bar x}\circ\ulCV)\bigl(\BF_q\dpl z\dpr\bigr)$ is non-empty. The whole construction does not depend on the choice of the base point $\bar x$ by \cite[Theorem~2.9]{dJ95a} and hence also applies if $X$ is not connected. Let $\wt K\subset\wt G\bigl(\BF_q\dpl z\dpr\bigr)$ be a compact open subgroup. For a strictly $L$-analytic space $S$ over $X$ and a lift of $\bar x$ to a geometric base point $\bar s$ of $S$, a \emph{rational $\wt K$-level structure} on $\ulCV$ over $S$ is a residue class modulo $\wt K$ of $\BF_q\dpl z\dpr$-rational tensor isomorphisms 
\[
(\beta\colon \wt\omega\isoto{\it forget}\circ\omega_{\bar x}\circ\ulCV)\;\mod \wt K\es\in\es\Isom^\otimes(\wt\omega,{\it forget}\circ\omega_{\bar x}\circ\ulCV)\bigl(\BF_q\dpl z\dpr\bigr)/\wt K
\]
such that the class $\beta\wt K$ is invariant under the \'etale fundamental group $\pi_1^\et(S,\bar s)$. Here $\Isom^\otimes(\wt\omega,{\it forget}\circ\omega_{\bar x}\circ\ulCV)\bigl(\BF_q\dpl z\dpr\bigr)$ carries an action of $\wt K$ through the action of $\wt G$ on $\wt\omega$, and an action of $\pi_1^\et(S,\bar s)$ through its action on $\omega_{\bar x}\circ\ulCV$ via the map $\pi_1^\et(S,\bar s)\to\pi_1^\et(X,\bar x)$ and Proposition~\ref{Prop2.13}.

Let $\HeckeTower_{\wt K}$ be the \'etale covering space of $X$ corresponding to the discrete $\pi_1^\et(X,\bar x)$-set $\Isom^\otimes(\wt\omega,{\it forget}\circ\omega_{\bar x}\circ\ulCV)\bigl(\BF_q\dpl z\dpr\bigr)/\wt K$. Then $\HeckeTower_{\wt K}$ represents the rational $\wt K$-level structures on $\ulCV$. Any choice of a fixed tensor isomorphism $\beta_0\in\Isom^\otimes(\wt\omega,{\it forget}\circ\omega_{\bar x}\circ\ulCV)\bigl(\BF_q\dpl z\dpr\bigr)$ associates with $\ulCV$ a representation $\pi_1^\et(X,\bar x)\to\wt G\bigl(\BF_q\dpl z\dpr\bigr)$ as in Corollary~\ref{Cor1.7}, and induces an identification of the $\pi_1^\et(X,\bar x)$-sets $\Isom^\otimes(\wt\omega,{\it forget}\circ\omega_{\bar x}\circ\ulCV)\bigl(\BF_q\dpl z\dpr\bigr)/\wt K$ and $\wt G\bigl(\BF_q\dpl z\dpr\bigr)/\wt K$. Moreover, if $\wt K'\subset\wt K\subset\wt G\bigl(\BF_q\dpl z\dpr\bigr)$ are compact open subgroups there is a natural projection morphism $\breve\pi_{\wt K,\wt K'}\colon\HeckeTower_{\wt K'}\to\HeckeTower_{\wt K}$, $\beta\wt K'\mapsto\beta\wt K$.

On the tower $(\HeckeTower_{\wt K})_{\wt K\subset \wt G(\BF_q\dpl z\dpr)}$ the group $\wt G\bigl(\BF_q\dpl z\dpr\bigr)$ acts via Hecke correspondences: Let $g\in\wt G\bigl(\BF_q\dpl z\dpr\bigr)$ and let $\wt K\subset\wt G\bigl(\BF_q\dpl z\dpr\bigr)$ be a compact open subgroup. Then $g$ induces an isomorphism 
\begin{equation}\label{EqHeckeCor1}
\iota(g)_{\wt K}\colon\;\HeckeTower_{\wt K}\;\isoto\;\HeckeTower_{g^{-1}\wt Kg}\,,\quad \beta\wt K\;\longmapsto\;\beta \wt K g=\beta g(g^{-1}\wt Kg)\,.
\end{equation}

\medskip\noindent
(b) Assume moreover that a group $J$ acts on $X$ and let $\ulCV\colon \Rep_{\BF_q\dpl z\dpr}G \to\PLoc_{X}$ be a tensor functor that carries a \emph{$J$-linearization}, that is, for every $j\in J$ an isomorphism $\phi_j\colon j^*\ulCV\isoto\ulCV$ of tensor functors (where $j^*\ulCV$ is the pullback of $\ulCV$ under the morphism $j\colon X\to X$), satisfying a cocycle condition. Then the tower of \'etale covering spaces $\HeckeTower_{\wt K}$ inherits an action of $J$ over $X$ as in \cite[Remark 2.7(b)]{HartlRZ}.
\end{remark}

\medskip

We now apply these considerations to the period spaces of Hodge-Pink $G$-structures.

\begin{remark}\label{RemGlobalSigmaModule}
The construction of the $\sigma$-bundle $\ulCF_{b,\gamma}(V)$ from Definition~\ref{DefSigmaModuleOfD} not only works over a field extension $L$ of $\breve E$ but more generally over the entire $\breve E$-analytic space $\breve\CH_{G,\hat{Z}}^\an$ from Definition~\ref{DefPeriodSp}. There it produces a $\sigma$-bundle $\ulCF_b(V)$ over $\breve\CH_{G,\hat{Z}}^\an$ whose fiber at each point $\gamma$ is $\ulCF_{b,\gamma}(V)$; see \cite[\S\,2.4]{HartlPSp}. 

The restriction of $\ulCF_b(V)$ to $\breve\CH_{G,\hat{Z},b}^a$ induces by Theorem~\cite[Theorem~3.4.1]{HartlPSp} a canonical local system $\ulCV_b(V)$ of $\BF_q\dpl z\dpr$-vector spaces on $\breve\CH_{G,\hat{Z},b}^a$. It can be described as follows. On every connected component $Y$ of $\breve\CH_{G,\hat{Z},b}^a$ we choose a geometric base point $\bar\gamma\colon\SpBerk(\ol L)\to Y$. We consider the pullback $\bar\gamma^*\ulCF_b(V)$ and its $\tau$-invariants
\[
\bar\gamma^*\ulCF_b(V)^\tau\;:=\;\{\, f\in\bar\gamma^*\ulCF_b(V)\colon (\bar\gamma^*\tau_{\CF_b})(\s f)=f\,\}\,.
\]
Since $\bar\gamma^*\ulCF_b(V)\cong\ulCF_{0,1}{}^{\oplus\dim V}$ the $\tau$-invariants $\bar\gamma^*\ulCF_b(V)^\tau$ form an $\BF_q\dpl z\dpr$-vector space of dimension equal to $\dim V$ that is equipped with a continuous action of the \'etale fundamental group $\pi_1^\et(Y,\bar\gamma)$. This defines the local system $\ulCV_b(V)$ of $\BF_q\dpl z\dpr$-vector spaces on $Y$ under the correspondence of Proposition~\ref{Prop2.13}. It satisfies $\ulCV_b(V)_{\bar\gamma}=\bar\gamma^*\ulCF_b(V)^\tau$ and $\ulCV_b(V)_{\bar\gamma}\otimes_{\BF_q\dpl z\dpr}\ol L\ancon[s]=\bar\gamma^*\ulCF_b(V)^\tau\otimes_{\BF_q\dpl z\dpr}\ol L\ancon[s]=\bar\gamma^*\ulCF_b(V)=\ulCF_{b,\bar\gamma}(V)$.

If $\gamma\in\breve\CH_{G,\hat{Z},b}^a$ is the image of the geometric point $\bar\gamma$ and $L$ is the residue field of $\gamma$ then the fiber $\ulCV_b(V)_\gamma$ of $\ulCV_b(V)$ at $\gamma$ is a continuous representation of $\Gal(L^\sep\!/L)$ on the finite-dimensional $\BF_q\dpl z\dpr$-vector space $\bar\gamma^*\ulCF_b(V)$. By \cite[Remark~3.6.17]{HartlKim} and with the notation from \eqref{EqCrisRing}, this Galois-representation can be described as
\[
\ulCV_b(V)_\gamma \; = \; \bigl(\ulD\otimes_{k\dpl z\dpr}\CO_{\ol L}\dbl z,z^{-1}\}[\tplusminus^{-1}]\bigr)^\tau\,\cap\,\Fq_D\otimes_{L\dbl z-\zeta\dbr}\ol L\dbl z-\zeta\dbr\
\]
for $\ulD_{b,\gamma}(V)=\ulD=(D,\tau_D,\Fq_D)$. Here we use the notation from \eqref{EqCrisRing}, and in addition we let $\tplusminus_i\in \CO_{\BF_q\dpl\zeta\dpr^\sep}$ be solutions of the equations $\tplusminus_0^{q-1}=-\zeta$ and $\tplusminus_i^q+\zeta \tplusminus_i=\tplusminus_{i-1}$, and set $\tplus:=\sum_{i=0}^\infty \tplusminus_iz^i\in\CO_{\BF_q\dpl\zeta\dpr^\sep}\dbl z\dbr$ and $\tplusminus:=\tplus\tminus\in\CO_{\ol L}\dbl z,z^{-1}\}$, where $\tminus$ was defined in \eqref{EqTMinus}. In particular, if $L$ is discretely valued then $\ulCV_b(V)_\gamma$ is the (equal characteristic) crystalline Galois representation associated with the (weakly) admissible $z$-isocrystal with Hodge-Pink structure $\ulD_{b,\gamma}(V)$ in the sense of \cite[Definition~3.4.21 and Remark~3.6.17]{HartlKim}.
\end{remark}

\begin{theorem}\label{ThmLocSyst}
Let $b\in LG(\BaseFld)$ and let $\hat{Z}=[(R,\hat Z_R)]$ be a bound as in Definition~\ref{DefBDLocal} with reflex ring $R_{\hat Z}=\kappa\dbl\xi\dbr$. Then the assignments $V\mapsto \ulCF_b(V)\mapsto \bar\gamma^*\ulCF_b(V)^\tau\mapsto\ulCV_b(V)$ define an $\BF_q\dpl z\dpr$-linear tensor functor $\ulCV_b$ from $\Rep_{\BF_q\dpl z\dpr}G$ to the category $\BF_q\dpl z\dpr\mbox{-}\ul{\rm Loc}_{\breve\CH_{G,\hat{Z},b}^{a}}$ of local systems of $\BF_q\dpl z\dpr$-vector spaces on $\breve\CH_{G,\hat{Z},b}^{a}$. There is a canonical $J_b\bigl(\BF_q\dpl z\dpr\bigr)$-linearization on $\ulCV_b$.
\end{theorem}

\begin{proof}
First of all, the functors $V\mapsto\ulD_{b,\gamma}(V)$ from Remark~\ref{RemDIsTensorFunctor} and $\ulD\mapsto\bigl(\ulCE(\ulD),\,\ulCF(\ulD)\bigr)$ from Definition~\ref{DefSigmaModuleOfD} are $\BF_q\dpl z\dpr$-linear tensor functors, and this works equally for the entire families over $\breve\CH_{G,\hat{Z},b}^{a}$. In particular $V\mapsto \ulCF_b(V)$ is a tensor functor from $\Rep_{\BF_q\dpl z\dpr}G$ to the category of $\sigma$-bundles over $\breve\CH_{G,\hat{Z},b}^{a}$. Next, taking $\tau$-invariants is obviously compatible with morphisms. It is in general not compatible with tensor products, but since $\bar\gamma^*\ulCF_b(V)\cong\ulCF_{0,1}{}^{\oplus\dim V}$ we have $\bar\gamma^*\ulCF_b(V)\cong\bigl(\bar\gamma^*\ulCF_b(V)^\tau\bigr)\otimes_{\BF_q\dpl z\dpr}\ol L\ancon[s]$. Therefore 
\begin{eqnarray*}
\bar\gamma^*\ulCF_b(V\otimes V') & \cong &\bar\gamma^*\ulCF_b(V)\otimes_{\ol L\ancon[s]}\bar\gamma^*\ulCF_b(V') \\[2mm]
& \cong & \bigl(\bar\gamma^*\ulCF_b(V)^\tau\otimes_{\BF_q\dpl z\dpr}\bar\gamma^*\ulCF_b(V)^\tau\bigr)\otimes_{\BF_q\dpl z\dpr}\ol L\ancon[s],
\end{eqnarray*}
and hence $\bar\gamma^*\ulCF_b(V\otimes V')^\tau\cong\bar\gamma^*\ulCF_b(V)^\tau\otimes_{\BF_q\dpl z\dpr}\bar\gamma^*\ulCF_b(V)^\tau$. We now use Proposition~\ref{Prop2.13} to conclude that $\ulCV_b$ is an $\BF_q\dpl z\dpr$-linear tensor functor.

If $j\in J_b\bigl(\BF_q\dpl z\dpr\bigr)$ the isomorphism $$\bigl(j^*\ulCF_b(V)\bigr)_{\gamma}\;=\;\bigl(\ulCF_b(V)\bigr)_{j(\gamma)}\;=\;\ulCF_{b,j(\gamma)}(V)\;\cong\;\ulCF_{b,\gamma}(V)\;=\;\bigl(\ulCF_b(V)\bigr)_\gamma$$ from Remark~\ref{RemJActsOnPeriodSp} yields $J_b\bigl(\BF_q\dpl z\dpr\bigr)$-linearizations $\phi_j\colon j^*\ulCF_b(V)\isoto\ulCF_b(V)$ and $\phi_j\colon j^*\ulCV_b\isoto\ulCV_b$.
\end{proof}

Let us end this section by stating the significance of Theorem~\ref{ThmLocSyst} in terms of the \'etale fundamental group and in terms of \'etale covering spaces. Let $\bar\gamma$ be a geometric base point of $\breve\CH_{G,\hat{Z},b}^{a}$.

\begin{remark}\label{RemLocSyst}
Let $\omega\open$ and $\omega_{b,\bar\gamma}:={\it forget}\circ\omega_{\bar\gamma}\circ\ulCV_b$ be the fiber functors from $\Rep_{\BF_q\dpl z\dpr}G$ to $\FMod_{\BF_q\dpl z\dpr}$ with $\omega\open(V):=V$ and $\omega_{b,\bar\gamma}(V):=\ulCV_b(V)_{\bar\gamma}=\bar\gamma^*\ulCF_b(V)^\tau$. By \cite[Theorem~2.9]{dJ95a} the fiber functors $\omega_{b,\bar\gamma}$ and $\omega_{b,\bar\gamma'}$ are isomorphic for any two geometric base points $\bar\gamma$ and $\bar\gamma'$ that lie in the same connected component of $\breve\CH_{G,\hat{Z},b}^{a}$. Let $\wt G:=\Aut^\otimes(\omega_{b,\bar\gamma})$. Then $\Isom^\otimes(\omega\open,\omega_{b,\bar\gamma})$ is a left $G$-torsor and a right $\wt G$-torsor over $\BF_q\dpl z\dpr$ and corresponds to a cohomology class $cl(b,\bar\gamma)\in\Koh^1(\BF_q\dpl z\dpr,G)$ by \cite[Theorem~3.2]{DM82}.  $\wt G$ is the inner form of $G$ defined by the image  $cl(b,\bar\gamma)\in\Koh^1(\BF_q\dpl z\dpr,G^\ad)$.

In the analogous situation over $\BQ_p$, the torsor between the crystalline and the \'etale fiber functor is given by the cohomology class
\[
cl(b,\bar\gamma)\;=\;[b]^\#-\bar\gamma^\#\;\in\;\Koh^1(\BQ_p,G)\,.
\]
This was proved by Rapoport and Zink~\cite[1.20]{RZ} if $G_{\rm der}$ is simply connected and in general by Wintenberger~\cite[Corollary to Proposition 4.5.3]{Wintenberger97}; see also \cite[Proposition on page 4]{CF}. 

The analog of Wintenberger's theorem in our situation (which has not been proved yet) is the statement that the torsor $\Isom^\otimes(\omega\open,\omega_{b,\bar\gamma})$ is given by the cohomology class
\begin{equation}\label{EqWintenberger}
cl(b,\bar\gamma)\;=\;[b]^\#-\bar\gamma^\#\;\in\;\Koh^1(\BF_q\dpl z\dpr,G)\,.
\end{equation}
Note that also in the function field case considered here, the weak admissibility of the pair $(b,\bar\gamma)$ implies that the difference $[b]^\#-\bar\gamma^\#$ lies in $(\pi_1(G)_\Gamma)_{{\rm tors}}$ which is identified with $\Koh^1(\BF_q\dpl z\dpr,G)$; see \cite[Theorem~1.15]{RapoportRichartz} (and use \cite[8.6]{BorelSpringer2} instead of Steinberg's theorem). In particular, if $\bar\gamma\in\breve\CH_{G,\hat{Z},b}^{na}$ and the analog \eqref{EqWintenberger} of Wintenberger's theorem is established, then there is an $\BF_q\dpl z\dpr$-rational tensor isomorphism $\beta\colon\omega\open\isoto\omega_{b,\bar\gamma}$ and $\wt G:=\Aut^\otimes(\omega_{b,\bar\gamma})\cong G$.
\end{remark}

\begin{remark}\label{RemTowerLocSyst}
By Corollary~\ref{Cor1.7} the restriction of the tensor functor $\ulCV_b$ from Theorem~\ref{ThmLocSyst} to the connected component $Y$ of $\breve\CH_{G,\hat{Z},b}^{a}$ containing the geometric base point $\bar\gamma$ corresponds to a continuous group homomorphism
\begin{equation}\label{EqRepFundGpPeriodSp}
\pi_1^\et(Y,\bar\gamma)\;\longto\;\wt G\bigl(\BF_q\dpl z\dpr\bigr)
\end{equation}
(under the choice of $\beta:=\id_{\wt\omega}$ for $\wt\omega:=\wt\omega_{b,\bar\gamma}$). 

By Remark~\ref{Rem1.8}, the tensor functor $\ulCV_b$ defines a tower $(\HeckeTower_{\wt K})_{\wt K\subset \wt G(\BF_q\dpl z\dpr)}$ of \'etale covering spaces of $\breve\CH_{G,\hat{Z},b}^{a}$. However, the group $\wt G$ might vary on the different connected components of $\breve\CH_{G,\hat{Z},b}^{a}$. It is therefore more useful to fix a base point $\bar\gamma$, and to define $\breve\CH_{G,\hat{Z},b}^{a,\bar\gamma}$ as the union of those connected components consisting of elements $\bar\gamma'$ with $\wt\omega_{b,\bar\gamma'}\cong\wt\omega_{b,\bar\gamma}$. Then we obtain a tower $(\HeckeTower_{\wt K}^{\bar\gamma})_{\wt K\subset \wt G(\BF_q\dpl z\dpr)}$ of \'etale covering spaces of $\breve\CH_{G,\hat{Z},b}^{a,\bar\gamma}$ with commuting actions of $\wt G\bigl(\BF_q\dpl z\dpr\bigr)$ by Hecke correspondences and of the quasi-isogeny group $J_b\bigl(\BF_q\dpl z\dpr\bigr)$ of $\ul\BG_0=\bigl((L^+G)_\BaseFld,b\s\bigr)$ from \eqref{EqJ}.

Note that for $A=\BF_q\dpl z\dpr$ it can happen that $\wt\omega\not\cong\omega\open$ but $\wt G\cong G$, namely if $cl(\omega\open,\wt\omega)\in\Koh^1(\BF_q\dpl z\dpr,G)$ is non-trivial and lies in the image of $\Koh^1(\BF_q\dpl z\dpr,Z)$, where $Z\subset G_{\BF_q\dpl z\dpr}$ is the center. In this case Corollary~\ref{Cor1.7} could be applied to a continuous group homomorphism $\pi_1^\et(X,\bar  x)\to G\bigl(\BF_q\dpl z\dpr\bigr)$ both using $\omega\open$ and $\wt\omega$. One obtains two tensor functors $\ulCV,\wt\ulCV\colon\Rep_{A}G\to\ALoc_X$ and tensor isomorphisms $\beta\in\Isom^\otimes(\omega\open,{\it forget}\circ\omega_{\bar x}\circ\ulCV)\bigl(\BF_q\dpl z\dpr\bigr)$ and $\tilde\beta\in\Isom^\otimes(\wt\omega,{\it forget}\circ\omega_{\bar x}\circ\wt\ulCV)\bigl(\BF_q\dpl z\dpr\bigr)$. Since $cl(\omega\open,\wt\omega)$ is trivialized by a finite unramified extension $\BF_{q^n}\dpl z\dpr$ of $\BF_q\dpl z\dpr$ by \cite[8.6]{BorelSpringer2}, the pairs $(\ulCV,\beta)$ and $(\wt\ulCV,\tilde\beta)$ become isomorphic after tensoring with $\BF_{q^n}\dpl z\dpr$. 
\end{remark}

\section{The period morphisms} \label{SectPeriodMorph}
\setcounter{equation}{0}

In this section we fix a local $G$-shtuka $\ul\BG_0=\bigl((L^+G)_\BaseFld,b\s\bigr)$ over $\BaseFld$ and a bound $\hat{Z}$ with reflex ring $R_{\hat Z}=\kappa\dbl\xi\dbr$. We set $E_{\hat Z}=\kappa\dpl\xi\dpr$ and $\breve E:=\BaseFld\dpl\xi\dpr$. For any point $(\ul\CG,\bar\delta)\in\breveRZ(S)$ with values in $S\in\Nilp_{\breve R_{\hat Z}}$ the quasi-isogeny $\bar\delta\colon\ul\CG_\olS\to\ul\BG_{0,\olS}$ lifts to a quasi-isogeny $\delta\colon\ul\CG\to\ul\BG_{0,S}$ by rigidity \cite[Proposition~2.11]{AH_Local}. To construct period morphisms for local $G$-shtukas we need to lift the universal $\bar\delta$, which is defined over the zero locus $\Var(\zeta)$ of $\zeta$ in $\breveRZ$, to the entire formal scheme $\breveRZ$. This lift will no longer be a quasi-isogeny, because it acquires larger and larger powers of $z$ in the denominators by lifting successively modulo $\zeta^{q^n}$. To describe what the limit of this lifting procedure is, we need the following generalization of \cite[Lemma~2.3.1]{HartlPSp} and \cite[Lemmas~2.8 and 6.4]{GL}. For an $\BaseFld\dbl\zeta\dbr$-algebra $B$ that is complete and separated with respect to a bounded norm $|\,.\,|\colon B\to\{x\in\BR\colon0\le x\le1\}$ with $0<|\zeta|<1$, we define the $\BaseFld\dpl z\dpr$-algebra
\begin{equation}\label{EqCrisRing}
B\dbl z,z^{-1}\}\;:=\;\bigl\{\,\sum_{i\in\BZ} b_iz^i\colon b_i\in B\,,\,|b_i|\,|\zeta|^{ri}\to0\;(i\to-\infty) \text{ for all }r>0\,\bigr\}\,.
\end{equation}
Note that the convergence $|b_i|\,|\zeta|^{ri}\to0$ for $i\to-\infty$ implies that $|b_i|<1$ for $i\gg0$. The element $\tminus:=\prod_{i\in\BN_0}\bigl(1-{\tfrac{\zeta^{q^i}}{z}}\bigr)$ lies in $B\dbl z,z^{-1}\}$ and satisfies $z\,\tminus=(z-\zeta)\s(\tminus)$; see \eqref{EqTMinus}. Note that $B\dbl z,z^{-1}\}$ contains $B\dbl z\dbr$ and $z^{-1}$, and that $B\dbl z,z^{-1}\}/(\zeta)=\olB\dbl z\dbr[\tfrac{1}{z}]=:\olB\dpl z\dpr$, for $\olB:=B/\zeta B$. Moreover, note that $\sigma^*(\tminus)\in B[\tfrac{1}{\zeta}]\dbl z-\zeta\dbr\mal$, and $B\dbl z,z^{-1}\}\subset B[\tfrac{1}{\zeta}]\dbl z-\zeta\dbr$ by $\sum_i b_i z^i=\sum_i b_i \bigl(\zeta+(z-\zeta)\bigr)^i=\sum_{n\ge0}\sum_i b_i\binom{i}{n}\zeta^{i-n}(z-\zeta)^n$. Therefore
\begin{equation}\label{EqBmaxInBdR}
B\dbl z,z^{-1}\}\bigl[\tfrac{1}{\sigma^*(\tminus)}\bigr]\;\subset\; B[\tfrac{1}{\zeta}]\dbl z-\zeta\dbr\,.
\end{equation}

\begin{lemma}\label{LemmaGL}
Let $B$ be an $\BaseFld\dbl\zeta\dbr$-algebra as above. Let $b\in LG(\BaseFld)=G\bigl(\BaseFld\dpl z\dpr\bigr)$, $A\in G\bigl(B\dbl z\dbr[\tfrac{1}{z-\zeta}]\bigr)$ and $\olDelta\in LG(\olB)=G\bigl(\olB\dpl z\dpr\bigr)$ such that $\olDelta\cdot (A\mod\zeta)=b\cdot\s(\olDelta)$ in $G\bigl(\olB\dpl z\dpr\bigr)$. Then there is a unique $\Delta\in G\bigl(B\dbl z,z^{-1}\}[\tfrac{1}{\tminus}]\bigr)$ with $\Delta\mod\zeta=\olDelta$ and $\Delta\cdot A=b\cdot\s(\Delta)$ in $G\bigl(B\dbl z,z^{-1}\}[\tfrac{1}{\tminus}]\bigr)$.
\end{lemma}

\begin{proof}
We choose a faithful representation $\Darst\colon G\into\GL_{r,\BF_q\dbl z\dbr}$ over $\BF_q\dbl z\dbr$. There is a positive integer $N$ such that $\Darst(b),\Darst(b^{-1})\in M_r\bigl(z^{-N}\BaseFld\dbl z\dbr\bigr)$, as well as $\Darst(A),\Darst(A^{-1})\in M_r\bigl((z-\zeta)^{-N}B\dbl z\dbr\bigr)$, and $\Darst(\olDelta),\Darst(\olDelta^{-1})\in M_r\bigl(z^{-2N}\olB\dbl z\dbr\bigr)$. We choose $C_0,D_0\in M_r\bigl(z^{-2N}B\dbl z\dbr\bigr)$ with $C_0\equiv \Darst(\olDelta)\pmod{\zeta}$ and $D_0\equiv\Darst(\olDelta^{-1})\pmod{\zeta}$. For $m>0$ we inductively define
\begin{eqnarray*}
C_m&:=&(z^{-N}\Darst(b))\sigma^{*}C_{m-1}((z-\zeta)^N\Darst(A^{-1}))\\[2mm]
D_m&:=&((z-\zeta)^N\Darst(A))\sigma^{*}D_{m-1}( z^{-N}\Darst(b^{-1}))
\end{eqnarray*}
in $M_r\bigl(z^{-2N(m+1)}B\dbl z\dbr\bigr)$. The assumption on $\bar\Delta$ implies \[
C_1-C_0\;\equiv\; \rho\bigl(b\sigma^*(\bar\Delta) A^{-1}\bigr) -\rho(\bar\Delta)\;\equiv\; 0 \pmod \zeta\,.
\]
By induction on $m$, we obtain that 
\begin{eqnarray*}
C_{m+1}- C_m & = & (z^{-N}\Darst(b))\sigma^{*}(C_m-C_{m-1})((z-\zeta)^N\Darst(A^{-1})) \\[2mm]
& \equiv & 0 \pmod{\sigma^*(\sigma^{m-1})^{*}(\zeta)} \\[2mm]
& \equiv & 0 \pmod{(\sigma^{m})^{*}(\zeta)}\,. 
\end{eqnarray*}
Therefore the sequence $(C_m)_m$ converges to a matrix $C\in M_r\bigl(B\dbl z,z^{-1}\}\bigr)$ and the sequence $\bigl(\prod_{i=0}^m(1-\tfrac{\zeta^{q^i}}{z})^{-N}\cdot C_m\bigr)_m$ converges to the matrix $\tminus^{-N}\cdot C\in M_r\bigl(B\dbl z,z^{-1}\}[\tfrac{1}{\tminus}]\bigr)$ with $(\tminus^{-N}C)\equiv C\equiv \Darst(\olDelta)\pmod{\zeta}$. The equation 
\[
\prod_{i=0}^m(1-\tfrac{\zeta^{q^i}}{z})^{-N}\cdot C_m\cdot \Darst(A) \;=\; \Darst(b)\cdot\s\bigl(\prod_{i=0}^{m-1}(1-\tfrac{\zeta^{q^i}}{z})^{-N}\cdot C_{m-1}\bigr)
\]
implies $(\tminus^{-N}C)\Darst(A)=\Darst(b)\s(\tminus^{-N}C)$ in $M_r\bigl(B\dbl z,z^{-1}\}[\tfrac{1}{\tminus}]\bigr)$. In the same way one sees that $(D_m)_m$ converges to a matrix $D\in M_r\bigl(B\dbl z,z^{-1}\}\bigr)$ with $(\tminus^{-N}D)\Darst(b)=\Darst(A)\s(\tminus^{-N}D)$ in $M_r\bigl(B\dbl z,z^{-1}\}[\tfrac{1}{\tminus}]\bigr)$ and $(\tminus^{-N}D)\mod\zeta=\Darst(\olDelta^{-1})$. We obtain the congruences
\begin{eqnarray*}
\Id_r-\tminus^{-2N}CD & \equiv & 0\pmod{\zeta}\quad \text{and by induction}\\[2mm]
\Id_r-\tminus^{-2N}CD & = & \Darst(b)\cdot\s(\Id_r-\tminus^{-2N}CD)\Darst(b)^{-1} \es\equiv\es 0\pmod{\zeta^{q^m}} \quad\text{for all $m$}\,.
\end{eqnarray*}
By looking at power series expansions in $z$ of the matrix coefficients, these congruences imply by the separatedness of the norm $|\,.\,|$ that $(\tminus^{-N}C)(\tminus^{-N}D)=\Id_r$, that is $\tminus^{-N}C\in\GL_r\bigl(B\dbl z,z^{-1}\}[\tfrac{1}{\tminus}]\bigr)$. Since $(\sigma^{m*}C)\mod\zeta^{q^m}=\sigma^{m*}(C\mod\zeta)=\Darst(\sigma^{m*}\olDelta)$ it follows that
\begin{eqnarray*}
(\tminus^{-N}C)\mod\zeta^{q^m} & = & \bigl(\prod_{i=0}^{m-1}(1-\tfrac{\zeta^{q^i}}{z})^{-N}\cdot C_{m-1}\bigr)\mod\zeta^{q^m}\\[2mm]
& = & \Darst\Bigl(b\cdot\ldots\cdot\sigma^{(m-1)*}(b)\cdot\sigma^{m*}(\olDelta)\cdot\sigma^{(m-1)*}(A^{-1})\cdot\ldots\cdot A^{-1}\Bigr)\mod\zeta^{q^m}
\end{eqnarray*}
satisfies the equations that cut out the closed subgroup scheme $\Darst(G)$ of $\GL_r$. Since $B$ is separated with respect to the norm $|\,.\,|$ we must have $\tminus^{-N}C=\Darst(\Delta)$ for a matrix $\Delta\in G\bigl(B\dbl z,z^{-1}\}[\tfrac{1}{\tminus}]\bigr)$. 

Finally to prove the uniqueness of $\Delta$ we assume that $\Delta'$ also satisfies the assertions of the lemma. Then $U:=\Darst(\Delta)-\Darst(\Delta')$ satisfies $U\in M_r(\zeta\cdot B\dbl z,z^{-1}\}[\tfrac{1}{\tminus}]\bigr)$ and $U=\Darst(b)\cdot \s(U)\cdot\Darst(A^{-1})$. Since $B$ is separated with respect to the norm $|\,.\,|$ it follows that $U=0$ and $\Delta'=\Delta$.
\end{proof}

We can now define the period morphism as a morphism of $\breve E$-analytic spaces 
\[
\breve\pi\colon(\breveRZ)^\an\to\breve\CH_{G,\hat{Z}}^\an
\]
as follows. Let $S$ be an affinoid, strictly $\breve E$-analytic space and let $S\to(\breveRZ)^\an$ be a morphism of $\breve E$-analytic spaces. With it we have to associate a uniquely determined morphism $S\to\breve\CH_{G,\hat{Z}}^\an$. Then the period morphism is obtained by glueing when $S$ runs through an affinoid covering of $(\breveRZ)^\an$.

Before we proceed we recall that an algebra $B$ over $\breve R_{\hat Z}=\BF\dbl\xi\dbr$ is \emph{admissible} in the sense of Raynaud \cite{Raynaud} if it has no $\xi$-torsion and is a quotient of an $\breve R_{\hat Z}$-algebra of the form 
\begin{equation}\label{EqFormalTateAlgebra}
\breve R_{\hat Z}\langle X_1,\ldots,X_s\rangle \;:=\;\bigl\{\,\sum_{\ul n\in\BN_0^s}b_{\ul n}X_1^{n_1}\cdot\ldots\cdot X_s^{n_s}\colon b_{\ul n}\in\breve R_{\hat Z},\es \lim_{|\ul n|\to\infty}b_{\ul n}=0\,\bigr\}\,,
\end{equation}
where $\ul n=(n_1,\ldots,n_s)$ and $|\ul n|:=n_1+\ldots+n_s$. A formal $\breve R_{\hat Z}$-scheme $\scrS$ is \emph{admissible} if it is locally $\breve R_{\hat Z}$-isomorphic to $\Spf B$ for admissible $\breve R_{\hat Z}$-algebras $B$; see \cite[\S\,1]{FRG1}.

Recall that $(\breveRZ)^\an$ is constructed as the union of the strictly $\breve E$-analytic spaces associated with a family of admissible formal $\breve R_{\hat Z}$-schemes that are obtained by admissible formal blowing-ups of $\breveRZ$ in closed ideals; see \cite[Chapter~5]{RZ} or \cite[\S\,0.2]{Berthelot96}. By Raynaud's theorem the morphism $S\to(\breveRZ)^\an$ is induced by a morphism from a quasi-compact admissible formal $\breve R_{\hat Z}$-scheme $\scrS$ with $\scrS^\an=S$ to one of these admissible formal $\breve R_{\hat Z}$-schemes; see \cite[Theorem~4.1]{FRG1} or for example \cite[Theorem~A.2.5]{HartlPSp} for a formulation with Berkovich spaces. Composing with the map to $\breveRZ$ yields a morphism of formal schemes $\scrS\to\breveRZ$. The latter corresponds to $(\ul\CG,\bar\delta)\in\breveRZ(\scrS)$.

\begin{lemma}\label{LemmaTrivializing}
There is an \'etale covering $\scrS'=\Spf B'\to\scrS$ of admissible formal $\breve R_{\hat Z}$-schemes such that there is a trivialization $\alpha\colon\ul\CG_{\scrS'}\isoto\bigl((L^+G)_{\scrS'},A\s\bigr)$ for some $A\in G\bigl(B'\dbl z\dbr[\tfrac{1}{z-\zeta}]\bigr)$.
\end{lemma}

\begin{proof}
We may choose an \'etale covering $\olCS'$ of $\olCS:=\Var_\scrS(\zeta)$ together with a trivialization $\bar\alpha\colon\CG_{\olCS'}\isoto(L^+G)_{\olCS'}$. After refining the covering $\olCS'$ there is by \cite[Lemma~1.4(a)]{FRG2} an \'etale morphism $\scrS'\to\scrS$ of admissible formal $\breve R_{\hat Z}$-schemes lifting $\olCS'\to\olCS$. Since $\scrS$ is quasi-compact we may further assume that $\scrS'=\Spf B'$ is affine. By \cite[Proposition~2.2(c)]{HV1} there is a lift $\alpha\colon\CG_{\scrS'}\isoto(L^+G)_{\scrS'}$ of the trivialization $\bar\alpha$. Since $\ul\CG$ is bounded by $\hat{Z}^{-1}$ this lift induces an isomorphism $\alpha\colon\ul\CG_{\scrS'}\isoto\bigl((L^+G)_{\scrS'},A\s\bigr)$ for $A\in G\bigl(B'\dbl z\dbr[\tfrac{1}{z-\zeta}]\bigr)$; compare the proof of Proposition~\ref{PropBound}\ref{PropBound_A}. 
\end{proof}

In addition the quasi-isogeny $\bar\delta\colon\ul\CG_{\olCS'}\to\ul\BG_{0,\olCS'}$ corresponds under $\bar\alpha$ to an element $\olDelta\in LG(\olB')$. We apply Lemma~\ref{LemmaGL} to obtain a uniquely determined element $\Delta\in G\bigl(B'\dbl z,z^{-1}\}[\tfrac{1}{\tminus}]\bigr)$ lifting $\olDelta$ with $\Delta A=b\,\s(\Delta)$ for $A$ as in Lemma \ref{LemmaTrivializing}. We set 
\begin{eqnarray}\label{EqPeriodMorph}
\gamma & := & \s(\Delta)A^{-1}\cdot G\bigl(B'[\tfrac{1}{\zeta}]\dbl z-\zeta\dbr\bigr)\es=\es b^{-1}\Delta\cdot G\bigl(B'[\tfrac{1}{\zeta}]\dbl z-\zeta\dbr\bigr)\\[2mm]
& \in & G\bigl(B'[\tfrac{1}{\zeta}]\dpl z-\zeta\dpr\bigr)\big/G\bigl(B'[\tfrac{1}{\zeta}]\dbl z-\zeta\dbr\bigr)\,.\nonumber
\end{eqnarray}
Since $\ul\CG$ is bounded by $\hat{Z}^{-1}$ the inverse $A^{-1}$ yields a point of $\hat Z^\an(\scrS'{}^\an)=\breve\CH_{G,\hat{Z}}^\an(\scrS'{}^\an)$. Because $\s(\Delta)\in G\bigl(B'\dbl z,z^{-1}\}[\tfrac{1}{\s\tminus}]\bigr)$ and $\s\tminus\in B'[\tfrac{1}{\zeta}]\dbl z-\zeta\dbr\mal$, we have $\gamma\in G\bigl(B'[\tfrac{1}{\zeta}]\dbl z-\zeta\dbr\bigr)\cdot\breve\CH_{G,\hat{Z}}^\an=\breve\CH_{G,\hat{Z}}^\an$ using \eqref{EqBmaxInBdR} and Definition~\ref{DefBDLocal}\ref{DefBDLocal_A11}. If we replace our trivialization $\alpha$ by a different trivialization $\alpha'\colon\ul\CG_{\scrS'}\isoto\bigl((L^+G)_{\scrS'},A'\s\bigr)$, there is an $h\in L^+G(B')=G\bigl(B'\dbl z\dbr\bigr)\subset G(B'[\tfrac{1}{\zeta}]\dbl z-\zeta\dbr)$ with $\alpha'=h\cdot\alpha$ and $A'=hA\,\s(h)^{-1}$. Then the quasi-isogeny $\bar\delta$ corresponds to $\olDelta'=\olDelta h^{-1}\in LG(\olB')$ and $\Delta'=\Delta h^{-1}$ is the lift of $\olDelta'$ from Lemma~\ref{LemmaGL}. This yields
\[
\gamma'\;=\;b^{-1}\Delta'\cdot G\bigl(B'[\tfrac{1}{\zeta}]\dbl z-\zeta\dbr\bigr)\;=\;b^{-1}\Delta\cdot G\bigl(B'[\tfrac{1}{\zeta}]\dbl z-\zeta\dbr\bigr)\;=\;\gamma\,.
\]
Therefore $\gamma$ descends to an element $\gamma\in\breve\CH_{G,\hat{Z}}^\an(S)$ giving the desired morphism $S\to\breve\CH_{G,\hat{Z}}^\an$. This completes the construction of the period morphism.

\begin{definition}\label{DefPeriodMorph}
The morphism $\breve\pi\colon(\breveRZ)^\an\to\breve\CH_{G,\hat{Z}}^\an$ of $\breve E$-analytic spaces constructed above is called the \emph{period morphism} associated with $\ul\BG_0$ and $\hat Z$. 
\end{definition}

\begin{remark}\label{RemHPStrOfLocGSht}
If $G=\GL_r$ and $B$ is an admissible $\BaseFld\dbl\zeta\dbr$-algebra in the sense of Raynaud, there is an equivalence \cite[\S\,4]{HV1} between local $\GL_r$-shtukas and local shtukas $\ulM=(M,\tau_M)$ over $\scrS=\Spf B$ consisting of a locally free $B\dbl z\dbr$-module $M$ of rank $r$ and an isomorphism $\tau_M\colon\s M[\tfrac{1}{z-\zeta}]\isoto M[\tfrac{1}{z-\zeta}]$. The de Rham cohomology 
\[
\Fp^\ulM\;:=\;\Koh^1_\dR(\ulM,B[\tfrac{1}{\zeta}]\dbl z-\zeta\dbr)\;:=\;\s M\otimes_{B\dbl z\dbr}B[\tfrac{1}{\zeta}]\dbl z-\zeta\dbr
\]
of $\ulM$ over $\scrS^\an$ carries a natural Hodge-Pink structure 
\[
\Fq^\ulM\;:=\;\tau_M^{-1}\bigl( M\otimes_{B\dbl z\dbr}B[\tfrac{1}{\zeta}]\dbl z-\zeta\dbr\bigr)\;\subset\;\Fp^\ulM[\tfrac{1}{z-\zeta}]\,;
\]
see \cite[Definition~3.4.13]{HartlKim}. If moreover, there is a fixed local $\GL_r$-shtuka $\ul\BG_0$ over $\BaseFld$ with associated local shtuka $(\BM,\tau_\BM)=(\BaseFld\dbl z\dbr^r,b\s)$, and a quasi-isogeny $\bar\delta\colon  \ul{\CG}_{\Spec B/(\zeta)}\to \ul{\BG}_{0,\Spec B/(\zeta)}$ given by $\olDelta\in LG\bigl(B/(\zeta)\bigr)$, then the lift $\Delta$ from Lemma~\ref{LemmaGL} provides an isomorphism 
\[
\s(\Delta)\colon\Koh^1_\dR(\ulM,B[\tfrac{1}{\zeta}]\dbl z-\zeta\dbr)\;\isoto\;\Koh^1_\cris\bigl(\ul\BM,\BaseFld\dpl z\dpr\bigr)\otimes_{\BaseFld\dpl z\dpr}B[\tfrac{1}{\zeta}]\dbl z-\zeta\dbr
\]
that transports the Hodge-Pink structure $\Fq^\ulM$ to the family $\s(\Delta)\circ\tau_M^{-1}\bigl( M\otimes_{B\dbl z\dbr}B[\tfrac{1}{\zeta}]\dbl z-\zeta\dbr\bigr)$ of Hodge-Pink structures on the constant $z$-isocrystal $\Koh^1_\cris\bigl(\ul\BM,\BaseFld\dpl z\dpr\bigr)\,:=\,\s\ul\BM\otimes_{\BaseFld\dbl z\dbr}\BaseFld\dpl z\dpr$; see \cite[Definition~3.5.14]{HartlKim}. Our period morphism $\breve\pi$ associates this family of Hodge-Pink structures with the universal local $\GL_r$-shtuka over $\breveRZ$. More precisely, this family equals $\gamma\cdot B[\tfrac{1}{\zeta}]\dbl z-\zeta\dbr^r$ where the element $\gamma$ from \eqref{EqPeriodMorph} is the image under $\breve\pi$ of the local $\GL_r$-shtuka $(\ulM,\bar\delta)\in\breveRZ(\Spf B)$.
\end{remark}

\begin{remark}\label{RemPiAndJ}
The period morphism $\breve\pi$ is equivariant for the action of $J:=J_b\bigl(\BF_q\dpl z\dpr\bigr)$. Indeed $j\in J$ acts on $\breveRZ$ by $j\colon(\ul\CG,\bar\delta)\mapsto(\ul\CG,j\circ\bar\delta)$. In terms of \eqref{EqPeriodMorph} this means that $j\in J\subset G\bigl(\BaseFld\dpl z\dpr\bigr)$ sends $\Delta$ to $j\cdot\Delta$ and $\gamma$ to $\s(j)\cdot\gamma$. Thus it coincides with the action on $\breve\CH_{G,\hat{Z}}^\an$ defined in Remark~\ref{RemJActsOnPeriodSp}.
\end{remark}

\begin{remark}
As in the arithmetic case, these period morphisms are not compatible with Weil descent data of source and target. Here the source is equipped with the Weil descent datum induced by the one in Remark \ref{defweildes}. On the target we have the natural Weil descent datum given by the fact that $\CH_{G,\hat{Z}}^\an$ is defined over $E_{\hat{Z}}$. In order to ensure such a compatibility, one has to extend the period morphism by a second component mapping to $\pi_1(G)_{\Gamma}$. For a more detailed discussion we refer to \cite{RZ}, or \cite[Properties~4.27(iv)]{RV}. 
\end{remark}

\begin{remark}\label{RemEtaleLocSht}
In the above construction the bounded local $G$-shtuka $\bigl((L^+G)_{\scrS'},A\s\bigr)$ over $\scrS'$ with $A\in G\bigl(B'\dbl z\dbr[\tfrac{1}{z-\zeta}]\bigr)$ induces an \'etale local $G$-shtuka $\bigl((L^+G)_{S'},A\s\bigr)$ over $S':=(\scrS')^\an$ in the sense of Definition~\ref{DefEtaleLocSht}, because $(z-\zeta)^{-1}=-\sum_{i=0}^\infty\zeta^{-i-1}z^i\in\CO_{S'}(S')\dbl z\dbr$ implies $B'\dbl z\dbr[\tfrac{1}{z-\zeta}]\subset\CO_{S'}(S')\dbl z\dbr$. The isomorphism $\alpha\colon\ul\CG_{\scrS'}\isoto\bigl((L^+G)_{\scrS'},A\s\bigr)$ yields a descent datum $g:=pr_2^*\alpha\circ pr_1^*\alpha^{-1}\in L^+G(\scrS'')$ with $pr_2^*A\cdot\s(g)=g\cdot pr_1^*A$, where $\scrS'':=\scrS'\times_\scrS\scrS'$ and $pr_i\colon\scrS''\to\scrS'$ is the projection onto the $i$-th factor. Viewing $g\in L^+G(S'')$ where $S'':=(\scrS'')^\an=S'\times_S S'$ provides a descent datum $pr_1^*(L^+G)_{S'}\isoto pr_2^*(L^+G)_{S'}$ on the $L^+G$-torsor $(L^+G)_{S'}$ via multiplication by $g$ on the right. This allows to descend $\bigl((L^+G)_{S'},A\s\bigr)$ to an \'etale local $G$-shtuka over $S=\scrS^\an$ which by abuse of notation we denote again by $\ul\CG$. In this way we obtain the universal family of \'etale local $G$-shtukas over $(\breveRZ)^\an$.
\end{remark}

\begin{definition}\label{DefEtaleLocSht}
Let $S$ be an $\BF\dpl\zeta\dpr$-scheme or a strictly $\BF\dpl\zeta\dpr$-analytic space. An \emph{\'etale local $G$-shtuka} over $S$ is a pair $\ul\CG=(\CG,\tau_\CG)$ consisting of an $L^+G$-torsor $\CG$ on $S$ and an isomorphism $\tau_\CG\colon\s\CG\isoto\CG$ of $L^+G$-torsors.
\end{definition}

\begin{proposition}\label{PropPeriodMorph}
The period morphism factors through the open $\breve E$-analytic subspace $\breve\CH_{G,\hat{Z},b}^a$.
\end{proposition}

\begin{proof}
Let $x$ be a point of $(\breveRZ)^\an$ with values in a complete field extension $L$ of $\breve E$ and let $\gamma:=\breve\pi(x)\in\breve\CH_{G,\hat{Z}}^\an$ be its image under the period morphism. Then $x$ corresponds to a pair $(\ul\CG,\bar\delta)\in\breveRZ\!(\Spf\CO_L)$ where $\ul\CG$ is a local $G$-shtuka over the valuation ring $\CO_L$ of $L$. Choose a faithful $\BF_q\dbl z\dbr$-rational representation $(\Darst,V)$ of $G$, and under the equivalence between local $\GL(V)$-shtukas and local shtukas from Remark~\ref{RemHPStrOfLocGSht} let $\ulM:=M(\Darst_*\ul\CG)$ be the associated local shtuka over $\CO_L$, and let $M(\Darst_*\bar\delta)$ be the associated quasi-isogeny. By \cite[Proposition~2.4.4]{HartlPSp} the $\sigma$-bundle $\ulCF_{b,\gamma}(V)$ is isomorphic to $\ulM\otimes_{\CO_L\dbl z\dbr}L\ancon[s]$ and hence $\ulCF_{b,\gamma}(V)\otimes_{L\ancon[s]}\ol L\ancon[s]\cong \ulCF_{0,1}{}^{\oplus\dim\Darst}$; see \cite[(Proof of) Theorem~2.4.7]{HartlPSp}. In other words $\gamma\in\breve\CH_{G,\hat{Z},b}^a$.
\end{proof}

\begin{proposition}\label{PropPeriodMEtale}
The period morphism $\breve\pi\colon(\breveRZ)^\an\to\breve\CH_{G,\hat{Z},b}^a$ is \'etale.
\end{proposition}
For later use, we formulate one smaller step in the proof as a separate lemma.

\begin{lemma}\label{LemmaPartProper}
$(\breveRZ)^\ad$ is separated and partially proper over $\breve E$.
\end{lemma}

\begin{proof}
The irreducible components of its special fiber are proper by Theorem \ref{ThmRRZSp}. Thus the lemma follows from \cite[Remark~1.3.18]{Huber96}.
\end{proof}

\begin{proof}[Proof of Proposition \ref{PropPeriodMEtale}]
Let $\breve\pi^\rig\colon(\breveRZ)^\rig\rightarrow (\breve\CH_{G,\hat{Z},b}^a)^\rig$ be the associated morphism of rigid analytic spaces and let $\breve\pi^\ad\colon(\breveRZ)^\ad\rightarrow (\breve\CH_{G,\hat{Z},b}^a)^\ad$ be the associated morphism of adic spaces in the sense of Huber~\cite{Huber96}. By \cite[Assertion~(a) on p.~427]{Huber96} the morphism $\breve\pi$ is \'etale if and only if $\breve\pi^\ad$ is \'etale and partially proper. The subspace $(\breve\CH_{G,\hat{Z},b}^a)^\ad\subset(\breve\CH_{G,\hat{Z}})^\ad$ is open by Theorem~\ref{ThmWAOpen} and \cite[Assertion~(1) on p.~431]{Huber96}. Therefore $(\breve\CH_{G,\hat{Z},b}^a)^\ad$ is separated over $\breve E$ by \cite[Lemma~1.10.17]{Huber96}. So by \cite[Lemma~1.10.17(vi)]{Huber96} and the above lemma $\breve\pi^\ad$ is partially proper.

It remains to show that $\breve\pi^\ad$ is \'etale. By \cite[Proposition~1.7.11]{Huber96} this is equivalent to $\breve\pi^\rig$ being \'etale. So by \cite[Proposition~2.4]{FRG3} we must show that for any admissible $\breve R_{\hat Z}$-algebra $B$ in the sense of Raynaud and for any ideal $I\subset B$ with $I^2=0$ and any commutative diagram with solid arrows
\begin{equation}\label{EqDiagEtale}
\xymatrix @C+3pc {
S_0\;:=\;\Spm B/I[\tfrac{1}{\zeta}] \ar[d] \ar[r]^{\TS f_0} & (\breveRZ)^\rig \ar[d]^{\TS\breve\pi^\rig}\\
S\;:=\;\Spm B[\tfrac{1}{\zeta}] \ar[r]^{\TS\gamma} \ar@{-->}[ru]^{\TS\exists\,!\,f} & (\breve\CH_{G,\hat{Z},b}^a)^\rig
}
\end{equation}
there is a unique dashed arrow $f$ making the diagram commutative. We set $\scrS:=\Spf B$, as well as $B_0:=B/I$ and $\scrS_0:=\Spf B_0$. We will construct $f$ as a morphism $\scrS\to\breveRZ$ after replacing $\scrS$ by an admissible blowing-up. Note that every admissible blowing-up of $\scrS_0$ is induced by an admissible blowing-up of $\scrS$.  Moreover, in the course of the proof we may replace $S$ by a quasi-compact, quasi-separated, \'etale covering $S'\to S$. Namely, by \cite[Corollaries~5.10 and 5.4]{FRG2}, every such covering is obtained from a quasi-compact morphism $\scrS'\to\scrS$ of formal schemes that is faithfully flat after replacing $\scrS$ by an admissible blowing-up. By the uniqueness assertion for $f$ it suffices to construct $f$ over $\scrS'$ and descend it back to $\scrS$. 

After replacing $\scrS$ by an admissible blowing-up the morphism $f_0$ extends to $f_0\colon\scrS_0\to\breveRZ$ and corresponds to a pair $(\ul\CG_0,\bar\delta_0)\in\breveRZ(\scrS_0)$ where $\ul\CG_0$ is a local $G$-shtuka over $\scrS_0$. By Lemma~\ref{LemmaTrivializing} we may replace $\scrS_0$ by an \'etale covering such that $\ul\CG_0\cong\bigl((L^+G)_{\scrS_0},A_0\s\bigr)$ for some $A_0\in G\bigl(B_0\dbl z\dbr[\tfrac{1}{z-\zeta}]\bigr)$. By \cite[Th\'eor\`eme~I.8.3]{SGA1} this \'etale covering lifts uniquely to an \'etale covering of $\scrS$. The quasi-isogeny $\bar\delta_0\colon\ul\CG_{0,\olCS_0}\to\ul\BG_{0,\olCS_0}$ over $\olCS_0:=\Spec B_0/(\zeta)$ corresponds to an element $\olDelta_0\in LG\bigl(B_0/(\zeta)\bigr)$ that lifts by Lemma~\ref{LemmaGL} to a unique $\Delta_0\in G\bigl(B_0\dbl z,z^{-1}\}[\tfrac{1}{\tminus}]\bigr)$ with $\Delta_0\mod\zeta=\olDelta_0$ and $\Delta_0\cdot A_0=b\cdot\s(\Delta_0)$. If we let $\gamma_0\in\breve\CH_{G,\hat{Z},b}^a(S_0)$ be the pullback of $\gamma\in\breve\CH_{G,\hat{Z},b}^a(S)$ then $\gamma_0=\s(\Delta_0)A_0^{-1}\cdot G\bigl(B_0[\tfrac{1}{\zeta}]\dbl z-\zeta\dbr\bigr)$ and $\s(\Delta_0)^{-1}\gamma_0=A_0^{-1}\cdot G\bigl(B_0[\tfrac{1}{\zeta}]\dbl z-\zeta\dbr\bigr)$. 

We claim that $\s(\Delta_0)$ lifts to a uniquely determined element of $G\bigl(B\dbl z,z^{-1}\}[\tfrac{1}{\s\tminus}]\bigr)$. Indeed, after choosing a faithful $\BF_q\dbl z\dbr$-rational representation $\Darst\colon G\into\SL_r$ there is an integer $e$ such that the matrix $\tminus^e\cdot\Darst(\Delta_0)\in B_0\dbl z,z^{-1}\}^{r\times r}$. Since $\s(I)=0$ the morphism $\s\colon B\to B$, $x\mapsto x^q$ factors over $B\onto B_0\to B$, and so $(\s\tminus)^e\cdot\Darst(\s\Delta_0)=\s\bigl(\tminus^e\cdot\Darst(\Delta_0)\bigr)\in B\dbl z,z^{-1}\}^{r\times r}$. This implies $\Darst(\s\Delta_0)\in\SL_r\bigl(B\dbl z,z^{-1}\}[\tfrac{1}{\s\tminus}]\bigr)$ and $\s(\Delta_0)\in G\bigl(B\dbl z,z^{-1}\}[\tfrac{1}{\s\tminus}]\bigr)$. By \eqref{EqBmaxInBdR} it follows moreover, that $\s(\Delta_0)\in G\bigl(B[\tfrac{1}{\zeta}]\dbl z-\zeta\dbr\bigr)$.

We now replace $S$ by a quasi-compact, quasi-separated, \'etale covering over which $\s(\Delta_0)^{-1}\gamma$ is induced from an element $g\in G\bigl(B[\tfrac{1}{\zeta}]\dpl z-\zeta\dpr\bigr)$. Consider the element $c_0:=(g\mod I)^{-1}A_0^{-1}\in G(B_0[\tfrac{1}{\zeta}]\dbl z-\zeta\dbr)$. Since the kernel of $B[\tfrac{1}{\zeta}]\dbl z-\zeta\dbr\to B_0[\tfrac{1}{\zeta}]\dbl z-\zeta\dbr$ is a nilpotent ideal and $G$ is smooth, $c_0$ lifts to an element $c\in G(B[\tfrac{1}{\zeta}]\dbl z-\zeta\dbr)$ and replacing $g$ by $gc$ yields $g\mod I=A_0^{-1}$. Also $\s(\Delta_0)^{-1}\gamma\in(\breve\CH_{G,\hat Z})^\rig(S)$ by Definition~\ref{DefBDLocal}\ref{DefBDLocal_A11}, whence $g\cdot G(B[\tfrac{1}{\zeta}]\dbl z-\zeta\dbr)\;\in\;(\hat Z_E\wh\otimes_E\breve E)^\rig(S)$. We denote the corresponding morphism of rigid analytic spaces by $\alpha\colon S\to(\hat Z_E\wh\otimes_E\breve E)^\rig$. Let $R$ be an extension of $R_{\hat Z}$ over which a representative $\hat Z_R$ of $\hat Z$ exists and such that $\Quot(R)$ is a finite Galois extension of $E_{\hat Z}$. We let $\breve R$ be the ring of integers in the completion of the maximal unramified extension of $\Quot(R)$, and we set $\hat{Z}_{\breve R}:=\hat{Z}_R\whtimes_R\Spf\breve R$. By applying Galois descend with respect to the field extension $\Quot(\breve R)/\breve E$ in the end, we may restrict to the case where $\scrS$ is a formal scheme over $\Spf\breve R$ and not just over $\Spf\breve R_{\hat Z}$. Let $\scrS'\subset \hat Z_{\breve R}\whtimes_{\breve R}\scrS$ be the $\zeta$-adic completion of the scheme theoretic closure of the graph of $\alpha$. It is projective over $\scrS$ by Proposition~\ref{PropBound}\ref{PropBound_D} and therefore $(\scrS')^\rig=\scrS^\rig$. So $\scrS'\to\scrS$ is an admissible blowing-up by \cite[Corollary~5.4]{FRG2} and we replace $\scrS$ by $\scrS'$ to obtain an extension $\alpha\colon\scrS\to\hat Z_{\breve R}$. After replacing $\scrS$ by an \'etale covering this morphism $\alpha$ is of the form $A^{-1}\cdot G(B\dbl z\dbr)\in\hat Z_{\breve R}(\scrS)$ for an element $A^{-1}\in G(B\dbl z\dbr[\tfrac{1}{z-\zeta}])$; compare the proof of Proposition~\ref{PropBound}. Since $g$ and $A^{-1}$ both correspond to the morphism $\alpha\colon S\to(\breve\CH_{G,\hat Z})^\rig$, we have $g\cdot G(B[\tfrac{1}{\zeta}]\dbl z-\zeta\dbr)=A^{-1}\cdot G(B[\tfrac{1}{\zeta}]\dbl z-\zeta\dbr)$ in $(\breve\CH_{G,\hat Z})^\rig(S)$. We consider the element $a_0:=A_0(A^{-1}\mod I)\in G(B_0\dbl z\dbr[\tfrac{1}{z-\zeta}])$. Its image in $(\breve\CH_{G,\hat Z})^\rig(S_0)$ equals 
\[
(g^{-1}\mod I)(A^{-1}\mod I)\cdot G(B_0[\tfrac{1}{\zeta}]\dbl z-\zeta\dbr)\;=\;(g^{-1}g\mod I)\cdot G(B_0[\tfrac{1}{\zeta}]\dbl z-\zeta\dbr)\;=\;1\cdot G(B_0[\tfrac{1}{\zeta}]\dbl z-\zeta\dbr), 
\]
that is $a_0\in G(B_0[\tfrac{1}{\zeta}]\dbl z-\zeta\dbr)$. By the following Lemma~\ref{LemmaDenominators} this implies that $a_0\in G(B_0\dbl z\dbr)$. Again since the kernel of $B\dbl z\dbr\to B_0\dbl z\dbr$ is nilpotent and $G$ is smooth, $a_0$ lifts to an element $a\in G(B\dbl z\dbr)$, and replacing $A^{-1}$ by $A^{-1}a^{-1}$ yields $A\mod I=A_0$. We consider the local $G$-shtuka $\ul\CG:=\bigl((L^+G)_{\scrS},A\s\bigr)$ with $\ul\CG_{\scrS_0}=\ul\CG_0$. Then $\Delta_0$ lifts to $\Delta:=b\,\s(\Delta_0)\,A^{-1}\in G\bigl(B\dbl z,z^{-1}\}[\tfrac{1}{\tminus}]\bigr)$ and $\olDelta:=\Delta\mod(\zeta)\colon\ul\CG_{\olCS}\to\ul\BG_{0,\olCS}$ is the unique lift of the quasi-isogeny $\bar\delta_0=\olDelta_0$ by rigidity. We let $f\colon\scrS\to\breveRZ$ be the morphism given by $(\ul\CG,\olDelta)\in\breveRZ(\scrS)$. It makes the diagram \eqref{EqDiagEtale} commutative, because $\s(\Delta)A^{-1}\cdot G\bigl(B[\tfrac{1}{\zeta}]\dbl z-\zeta\dbr\bigr)\;=\; \s(\Delta_0)g\cdot G\bigl(B[\tfrac{1}{\zeta}]\dbl z-\zeta\dbr\bigr)\;=\;\gamma$.

To prove that $f$ is uniquely determined let $f'\colon\scrS\to\breveRZ$ be a second morphism making the diagram \eqref{EqDiagEtale} commutative. The corresponding point $(\ul\CG'\!,\bar\delta')\in\breveRZ(\scrS)$ is of the form $\ul\CG'=\bigl((L^+G)_{\scrS},A'\s\bigr)$ with $A'\mod I=A_0$ and $\Delta'=b\,\s(\Delta_0)\,A'{}^{-1}\in G\bigl(B\dbl z,z^{-1}\}[\tfrac{1}{\tminus}]\bigr)$. We assume that it is mapped under $\breve\pi^\rig$ also to $\gamma$. This means $\s(\Delta_0)A^{-1}\cdot G\bigl(B[\tfrac{1}{\zeta}]\dbl z-\zeta\dbr\bigr)\;=\;\gamma\;=\;\s(\Delta_0)A'{}^{-1}\cdot G\bigl(B[\tfrac{1}{\zeta}]\dbl z-\zeta\dbr\bigr)$, whence $\Phi:=A'A^{-1}\in G(B[\tfrac{1}{\zeta}]\dbl z-\zeta\dbr)\subset G(B[\tfrac{1}{\zeta}]\dpl z-\zeta\dpr)$. From Lemma~\ref{LemmaDenominators} it follows that $\Phi\in G(B\dbl z\dbr)$. Also $\s(\Phi)=\s(A'A^{-1}\mod I)=\s(1)=1$ implies $\Phi A=A'\s(\Phi)$ and $\Phi=\Delta'{}^{-1}\Delta$. We conclude that $\Phi$ is an isomorphism $(\ul\CG,\bar\delta)\isoto(\ul\CG'\!,\bar\delta')$. This means $f=f'$ and finishes the proof.
\end{proof}

\begin{remark}
When $G=\GL_r$ the proof starts in terms of Remark~\ref{RemHPStrOfLocGSht} with a local shtuka $\ulM_0$ over $\scrS_0$. Then it considers the de Rham cohomology of $\ulM_0$, which lifts to $\scrS$ by its crystalline nature. Next it produces from the Hodge-Pink structure $\gamma$ a Hodge-Pink structure on the de Rham cohomology of $\ulM_0$ over $\scrS$ which lifts the Hodge-Pink structure of $\ulM_0$. This lift of the Hodge-Pink structure corresponds to a unique lift of $\ulM_0$ to a local shtuka $\ulM$ over $\scrS$ by \cite[Proposition~6.3]{GL}. In that sense our proof is a direct translation of \cite[Proposition~5.17]{RZ}.
\end{remark}

\begin{lemma}\label{LemmaDenominators}
Let $B$ be an $\BF_q\dbl\zeta\dbr$-algebra without $\zeta$-torsion that is $\zeta$-adically complete, and let $a\in G(B\dbl z\dbr[\tfrac{1}{z-\zeta}])$ such that the image of $a$ in $G(B[\tfrac{1}{\zeta}]\dpl z-\zeta\dpr)$ lies in $G(B[\tfrac{1}{\zeta}]\dbl z-\zeta\dbr)$. Then $a\in G(B\dbl z\dbr)$.
\end{lemma}

\begin{proof}
Note that $B\dbl z\dbr$ has no $(z-\zeta)$-torsion, because $B$ has no $\zeta$-torsion. Let $\Darst\colon G\into\SL_r$ be a faithful representation over $\BF_q\dbl z\dbr$ and consider the matrix $\Darst(a)\in\SL_r(B\dbl z\dbr[\tfrac{1}{z-\zeta}])\subset B\dbl z\dbr[\tfrac{1}{z-\zeta}]^{r\times r}$.

It is enough to show that this matrix is in $B\dbl z\dbr^{r\times r}$ as $\SL_r(B\dbl z\dbr[\tfrac{1}{z-\zeta}])\cap B\dbl z\dbr^{r\times r}= \SL_r(B\dbl z\dbr)$ and  $\SL_r(B\dbl z\dbr)\cap \rho (G(B\dbl z\dbr[\tfrac{1}{z-\zeta}]))= \rho(G(B\dbl z\dbr))$ as $\rho$ is defined over $\BF_q\dbl z\dbr.$

After multiplying $\Darst(a)$ by $(z-\zeta)^n$ for sufficiently large $n$, its denominators disappear and its image in $B[\tfrac{1}{\zeta}]\dbl z-\zeta\dbr^{r\times r}$ is divisible by $(z-\zeta)^n$. Thus it suffices to show that an element $f$ of $B\dbl z\dbr$ whose image in $B[\tfrac{1}{\zeta}]\dbl z-\zeta\dbr$ is divisible by $z-\zeta$, is already divisible by $z-\zeta$ in $B\dbl z\dbr$. This follows as in Lemma~\ref{LemmaDivisibility}.
\end{proof}

We end this section with some examples.

\begin{example}{\bfseries (The Drinfeld period morphism.)}\label{ExDrinfeld}
This example is due to Drinfeld~\cite{Drinfeld76}. A good account is given by Genestier and Lafforgue~\cite{Genestier,GL08}. Let $d$ be a positive integer and let $D$ be the central division algebra over $\BF_q\dpl z\dpr$ of Hasse invariant $1/d$. Let $\CO_D$ be its maximal order. We may identify $D\cong\bigoplus_{i=0}^{d-1}\BF_{q^d}\dpl z\dpr\Pi^i$ and $\CO_D\cong\bigoplus_{i=0}^{d-1}\BF_{q^d}\dbl z\dbr\Pi^i$ with $\Pi^d=z$ and $\Pi a=\sigma(a)\Pi$ for $a\in\BF_{q^d}\dpl z\dpr$. We let $G:=\CO_D\mal$ be the group scheme over $\BF_q\dbl z\dbr$ with $G(A)=(\CO_D\otimes_{\BF_q\dbl z\dbr}A)\mal$ for $\BF_q\dbl z\dbr$-algebras $A$. Consider the matrices
\begin{equation}\label{EqDrinfeldPi}
T\;:=\;\left( \raisebox{4.6ex}{$
\xymatrix @C=0.2pc @R=0pc {
0 \ar@{.}[rrr]\ar@{.}[drdrdrdr] & & & 0 & z-\zeta\\
1\ar@{.}[drdrdr]  & & &  & 0 \ar@{.}[ddd]\\
0 \ar@{.}[dd]\ar@{.}[drdr] & & & &  \\
& & & & \\
0 \ar@{.}[rr] & & 0 & 1 & 0\\
}$}
\right) \qquad\text{and}\qquad
\olT\;:=\;\left( \raisebox{4.6ex}{$
\xymatrix @C=0.2pc @R=0pc {
0 \ar@{.}[rrr]\ar@{.}[drdrdrdr] & & & 0 & z\\
1\ar@{.}[drdrdr]  & & &  & 0 \ar@{.}[ddd]\\
0 \ar@{.}[dd]\ar@{.}[drdr] & & & &  \\
& & & & \\
0 \ar@{.}[rr] & & 0 & 1 & 0\\
}$}
\right)\;=\;T\mod\zeta.
\end{equation}
The field extension $\BF_{q^d}$ of $\BF_q$ splits the division algebra $D$ by the isomorphisms $D\otimes_{\BF_q}\BF_{q^d}\cong \BF_{q^d}\dpl z\dpr^{d\times d}$ and $\CO_D\otimes_{\BF_q}\BF_{q^d}\cong \{g\in\BF_{q^d}\dbl z\dbr^{d\times d}\colon g\mod z\text{ is lower triangular}\}$ sending $\Pi\otimes1$ to $\olT$ and $a\otimes1$ to $\diag\bigl(\sigma^{d-1}(a),\sigma^{d-2}(a),\ldots,a\bigr)$ for $a\in\BF_{q^d}\dpl z\dpr\subset D$. So $G\otimes_{\BF_q}\BF_{q^d}$ is the Iwahori group scheme $I:=\{g\in\GL_d\colon g\mod z\text{ is lower triangular}\}$. Let $b=\Pi\in LG(\BF_q)=D\mal$ and let the bound $\hat Z$ be represented over $R=\BF_{q^d}\dbl\zeta\dbr$ by 
\[
\hat Z_R\::=\:L^+I\cdot T^{-1}\cdot L^+I/L^+I\;\subset\;\wh\Flag_{I,R}\;\cong\;\wh\Flag_{G,R}\,. 
\]
Its reflex ring is $R_{\hat Z}=\BF_q\dbl\zeta\dbr$ and $\breve\CH_{G,\hat Z}=\BP^{d-1}_{\BaseFld\dpl\zeta\dpr}$. The quasi-isogeny group $J_b$ equals $\GL_d$. We are going to describe $\breveRZ$.

The category of $L^+G$-torsors over a scheme $S\in\Nilp_{\BF_q\dbl\zeta\dbr}$ is equivalent to the category of $\CO_S\dbl z\dbr$-modules with $\CO_{D^\opp}\wh\otimes_{\BF_q}\CO_S$-action, that are Zariski locally on $S$ of the form $\CO_D\wh\otimes_{\BF_q}\CO_S$, where $\CO_{D^\opp}\wh\otimes_{\BF_q}\CO_S$ acts by multiplication on the right. This equivalence sends an $L^+G$-torsor $\CG$ that is trivialized over an \'etale covering $S'\to S$ by $\alpha\colon\CG_{S'}\isoto(L^+G)_{S'}$ with $h:=p_1^*\alpha\circ p_2^*\alpha^{-1}\in L^+G(S'')=\bigl(\CO_D\otimes_{\BF_q\dbl z\dbr}\Gamma(S''\!,\CO_{S''})\dbl z\dbr\bigr)\mal$ where $p_i\colon S'':=S'\times_SS'\to S'$ is the projection onto the $i$-th factor, to the $\CO_S\dbl z\dbr$-module $M$ obtained by descent from $M':=\CO_D\wh\otimes_{\BF_q}\CO_{S'}$ with the descent datum $p_2^*M'\isoto p_1^*M'$, $m\mapsto hm$. Then $M$ is Zariski locally trivial by Hilbert 90; see \cite[Proposition~2.3]{HV1}. If $S'\in\Nilp_{\BF_{q^d}\dbl\zeta\dbr}$ then $M'=\CO_D\wh\otimes_{\BF_q}\CO_{S'}$ decomposes as a direct sum of eigenspaces $M'_i$ on which $a\in\BF_{q^d}\subset\CO_D$ acts as $a^{q^i}\in\CO_{S'}$ for $i\in\BZ/d\BZ$. Under the isomorphism $\CO_D\wh\otimes_{\BF_q}\CO_{S'}\cong\{g\in\CO_{S'}\dbl z\dbr^{d\times d}\colon g\mod z\text{ is lower triangular}\}$ the $i$-th eigenspace $M'_i$ is mapped to the $(d-i)$-th column in the matrix space (for $0\le i<d$). Multiplication with $\Pi$ on the right defines morphisms $\Pi\colon M'_i\to M'_{i+1}$. If $\ul\CG$ is a local $G$-shtuka over such an $S'$ then $\tau_\CG$ maps $\sigma^*M'_i$ to $M'_{i+1}[\tfrac{1}{z}]$. It is bounded by $\hat Z^{-1}$ if and only if for all $i$ the map $\tau_\CG$ is a morphism $\sigma^*M'_i\to M'_{i+1}$ with cokernel locally free of rank $1$ over $S'$. This means that $M$ is the local shtuka (called ``module de coordonn\'ees'' in \cite{Genestier,GL08}) of a special formal $\CO_D$-module of dimension $d$ and height $d^2$ in the sense of Drinfeld~\cite{Drinfeld76}. 

The formal scheme $\breveRZ=\coprod_\BZ\wh\Omega^d$ and the space $(\breveRZ)^\an=\coprod_\BZ\Omega^d$ are the disjoint unions indexed by the height of the quasi-isogeny $\bar\delta$, where
\[
\Omega^d\;:=\;\BP^{d-1}_{\BaseFld\dpl\zeta\dpr}\setminus \text{all $\BF_q\dpl\zeta\dpr$-rational hyperplanes}
\]
is Drinfeld's upper halfspace over $\breve E=\BaseFld\dpl\zeta\dpr$ and $\wh\Omega^d$ is its formal model over $\BaseFld\dbl\zeta\dbr$ constructed by Drinfeld, Deligne and Mumford. The representability and structure of $\breveRZ$ is described in detail in \cite[Chapitre~II]{Genestier}. The period space $\breve\CH_{G,\hat{Z},b}^{na}=\breve\CH_{G,\hat{Z},b}^{wa}$ also equals $\Omega^d$ and on each connected component of $(\breveRZ)^\an$ the period morphism is the identity of $\Omega^d$. The fibers of $\breve\pi$ are in bijection with $\BZ=D\mal/\CO_D\mal$; compare Proposition~\ref{PropFibersOfPiK} and Theorem~\ref{MainThm}\ref{MainThm_1}. Note that this example has a $\BQ_p$-analog also going back to Drinfeld that is discussed by \cite[1.44--1.46, 3.54--3.77 and 5.48--5.49]{RZ}. Our exposition differs from \cite{RZ} since they take covariant Dieudonn\'e modules of formal $\CO_D$-modules whereas the local shtuka functor \cite[\S\,2.1]{GL08} is contravariant.
\end{example}

\begin{example}{\bfseries (The Gross-Hopkins period morphism.)}\label{ExGrossHopkins}
Also this example is discussed in \cite{GL08}. Gross and Hopkins \cite{HG1,HG2} take $G=\GL_r$, with the upper triangular Borel subgroup and the diagonal torus. Let $b\in LG(\BF_q)$ be the matrix $\olT$ from \eqref{EqDrinfeldPi}, and let the bound $\hat Z=\hat Z_{\preceq\mu}$ be as in Example~\ref{exboundmu} for $\mu=(0,\ldots,0,-1) \in  \BZ^r\cong X_*(T)$ and with reflex field $E:=E_{\hat Z}=\BF_q\dpl\zeta\dpr$. Then $\hat Z^{-1}=\hat Z_{\preceq(-\mu)_\dom}$ with $(-\mu)_\dom=(1,0,\ldots,0)$, compare Example~\ref{exboundmu}. The quasi-isogeny group $J_b$ is the unit group of the central skew field over $\BF_q\dpl z\dpr$ with Hasse invariant $1/r$. The Rapoport-Zink space is the Lubin-Tate space
\[
\breveRZ\;=\;\coprod_\BZ\Spf\BaseFld\dbl\zeta,u_1,\ldots,u_{r-1}\dbr
\]
of $1$-dimensional formal $\BF_q\dbl z\dbr$-modules of height $r$. Its connected components are indexed by the \emph{height} of the quasi-isogeny $\bar\delta$, that is the image of $\bar\delta\in\Flag_{\GL_r}$ under the map $\Flag_{\GL_r}\to\pi_0(\Flag_{\GL_r})=\pi_1(\GL_r)=\BZ$. The period space is $\breve\CH_{G,\hat{Z},b}^{na}=\breve\CH_{G,\hat{Z}}^\an=\BP^{r-1}_{\BaseFld\dpl\zeta\dpr}$; compare \cite[Example~3.3.1]{HartlPSp}. To define the Hodge-Pink structure on the universal formal $\BF_q\dbl z\dbr$-module Gross and Hopkins \cite[\S\,11]{HG2} use the universal additive extension. See \cite[Remark~2.5.43]{HartlJuschka} for a comparison of this definition with our definition of the Hodge-Pink structure in Remark~\ref{RemHPStrOfLocGSht}. In \cite[\S\,23]{HG2} they construct the period morphism $\breve\pi$ and show that its image is $(\BP^{r-1}_{\breve E})^\an$; compare Theorem~\ref{MainThm}\ref{MainThm_1}. Note that Gross and Hopkins treat the $\BQ_p$-analog simultaneously; see also \cite[5.50]{RZ}.
\end{example}

\begin{example}{\bfseries(The $\boldsymbol{\zeta}$-adic Carlitz logarithm.)}\label{ExBreutmann}
The following example was computed by Breutmann~\cite{Breutmann15}. Let $G=\GL_2$, and let the Borel, the maximal torus, the bound $\hat Z=\hat Z_{\preceq\mu}$ and $\mu$ be as in the previous example. Let $b=\left(\begin{smallmatrix} z & 0 \\ 0 & 1 \end{smallmatrix}\right)$. Then $J_b$ is the diagonal torus in $\GL_2$. The Rapoport-Zink space descends to $\BF_q\dbl\zeta\dbr$ as the formal scheme
\[
\RZ\;=\;\coprod_{(i,j)\in\BZ^2}\Spf\BF_q\dbl\zeta,h\dbr\,.
\]
whose underlying affine Deligne-Lusztig variety $X_{Z^{-1}}(b)=\coprod_{\BZ^2}\Spec\BF_q$ is $0$-dimensional. Over the component $(i,j)\in\BZ^i$ the universal local $\GL_2$-shtuka $\ul\CG$ is given by the local shtuka $\ulM(\ul\CG)=(\BF_q\dbl\zeta,h\dbr\dbl z\dbr^2,\tau_M)$ with $\tau_M=\left(\begin{smallmatrix} z-\zeta & 0 \\ h & 1 \end{smallmatrix}\right)$; see Remark~\ref{RemHPStrOfLocGSht}. The universal quasi-isogeny 
\[
\bar\delta=\left(\begin{matrix} z^i & 0 \\[2mm] -z^j\sum_{\nu=0}^\infty h^{q^\nu}\!\!/z^{\nu+1}\; & z^j \end{matrix}\right)
\]
lifts to 
\[
\Delta\;=\;\left(\begin{array}{c@{\qquad}c} z^i\prod\limits_{\nu=0}^\infty\dfrac{z}{z-\zeta^{q^\nu}} & 0 \\[4mm] -z^j\sum\limits_{\nu=0}^\infty \dfrac{h^{q^\nu}}{(z-\zeta)\cdots(z-\zeta^{q^\nu})} & z^j \end{array}\right).
\]
The $E$-analytic space $(\breveRZ)^\an$ is the disjoint union indexed by $(i,j)\in\BZ^2$ of the open unit discs with coordinate $h$ and $\CH_{G,\hat{Z},b}^{na}=\CH_{G,\hat{Z},b}^{wa}=\BA^1_E=\BP^1_E\setminus\{(0:1)\}\subset\CH_{G,\hat{Z}}^\an=\BP^1_E$; compare \cite[Example~3.3.3]{HartlPSp}. On the component $(i,j)$ the period morphism $\breve\pi$ is given by $h\mapsto\zeta^{j-i}(\s\tminus)|_{z=\zeta}\log_{\text{Carlitz}}(h)$, where $\tminus$ was defined in \eqref{EqTMinus} and $\log_{\text{Carlitz}}(h):=\sum_{\nu=0}^\infty\tfrac{h^{q^\nu}}{(\zeta-\zeta^q)\cdots(\zeta-\zeta^{q^\nu})}$ is the $\zeta$-adic Carlitz logarithm; see \cite[\S\,3.4]{Goss}. In particular, $\breve\pi$ is surjective onto $\CH_{G,\hat{Z},b}^{na}$; compare Theorem~\ref{MainThm}\ref{MainThm_1}. This example is analogous to the period morphism for $p$-divisible groups \cite[5.51, 5.52]{RZ} given by the $p$-adic logarithm, which was constructed by Dwork; compare \cite[\S\S\,7,8]{KatzDwork}.
\end{example}

\section{The tower of \'etale coverings} \label{SectTower}
\setcounter{equation}{0}

In this section we fix a local $G$-shtuka $\ul\BG_0$ over $\BaseFld$ and a bound $\hat Z$ with reflex ring $R_{\hat Z}=\kappa\dbl\xi\dbr$. Let again $E_{\hat Z}=\kappa\dpl\xi\dpr$ and $\breve E=\BaseFld\dpl\xi\dpr$ and $\breve R_{\hat Z}=\BaseFld\dbl\xi\dbr$. We write ${\breve\CM}$ for the strictly $\breve E$-analytic space $(\breveRZ)^\an$. We shall construct a tower of finite \'etale coverings of ${\breve\CM}$ obtained by trivializing the Tate module of the universal \'etale local $G$-shtuka $\ul\CG$ over ${\breve\CM}$ from Remark~\ref{RemEtaleLocSht}.

We start more generally with a field extension $L/\breve E$ that is complete with respect to an absolute value extending the absolute value on $\breve E$, and with an \'etale local $G$-shtuka $\ul{\CG}$ over a connected strictly $L$-analytic space $\Test$ as in Definition~\ref{DefEtaleLocSht}. We choose a geometric base point $\bar x$ of $\Test$.

\begin{definition}
Let $\rho:G\rightarrow \GL_r$ be in $\Rep_{\BF_q\dbl z\dbr}G$. Let $\ulM=(M,\tau_M)$ be the \'etale local shtuka of rank $r$ associated with the \'etale local $\GL_r$-shtuka $\Darst_*\ul\CG$ obtained from $\ul{\CG}$ via $\rho$, see \cite[\S\,3]{AH_Local}, and let $\ulM_{\bar x}$ denote its fiber over $\bar x$. Then the (\emph{dual}) \emph{Tate module $\check T_{\ul{\CG},\bar x}(\rho)$ of $\ul{\CG}$ with respect to $\rho$} is defined as the (dual) Tate module of $\ulM_{\bar x}$
\[
\check T_{\ul{\CG},\bar x}(\rho)\;:=\;\check T_z\ulM_{\bar x}\;:=\;\{m\in \ulM_{\bar x} \colon \tau_M(\sigma^*m)=m\}\,.
\]
By \cite[Proposition~6.1]{TW} it is a free $\BF_q\dbl z\dbr $-module of rank $r$ with a continuous monodromy action of $\pi_1^\et(\Test,\bar x)$. This action factors through $\pi_1^\alg(\Test,\bar x)$.

Let now $\Darst:G\rightarrow \GL_r$ be in $\Rep_{\BF_q\dpl z\dpr}G$. Let $\ulN=(N,\tau_N)$ be the locally free $\CO_\Test\dpl z\dpr$-module of rank $r$ and the $\sigma$-linear isomorphism associated with  $\Darst_*L\ul\CG$ obtained from $L\ul{\CG}$ via $\Darst$. Let $\ulN_{\bar x}$ denote its fiber over $\bar x$. The \emph{rational} (\emph{dual}) \emph{Tate module} $\check V_{\ul{\CG},\bar x}(\Darst)$ of $\ul{\CG}$ with respect to $\Darst$ is 
$$\check V_{\ul{\CG},\bar x}(\Darst)\;:=\;\{n\in \ulN_{\bar x} \colon \tau_N(\sigma^*n)=n\},$$
a finite-dimensional $\BF_q\dpl z\dpr $-vector space with a continuous monodromy action of $\pi_1^\et(\Test,\bar x)$.
\end{definition}

\begin{remark}\label{RemNeupertTate}
As was pointed out to us by S.~Neupert, see also \cite[2.6]{N16}, one can also use the following direct way to define the Tate module of an \'etale local $G$-shtuka that does not use tensor functors. Let $\ul{\CG}=(\CG, \tau_\CG)$ be an \'etale local $G$-shtuka over a base scheme or a strictly $\BaseFld\dpl\zeta\dpr$-analytic space $S$, in other words $\tau_\CG$ induces an isomorphism $\tau_\CG:\sigma^*\CG\isoto \CG$. Consider for each $n\in\BN$ the $\tau$-invariants of the induced map $\sigma^*\CG_n\rightarrow \CG_n$. Here $\CG_n$ is the $G/G_n$-torsor induced by $\CG$ where $G_n$ is the kernel of the projection $G(\BF_q[[z]])\rightarrow G(\BF_q[[z]]/(z^n))$. These $\tau$-invariants form a $G(\BF_q[[z]]/(z^n))$-torsor which is trivialized by a finite \'etale covering. One can then define the Tate module of $\ul\CG$ as the inverse limit over $n$ of these torsors.
\end{remark}

\begin{remark}\label{RemTateModule}
\begin{enumerate}
\item\label{RemTateModule_A}
Let $(\rho'\!,V)\in \Rep_{\BF_q\dpl z\dpr}G$. Let $\Lambda_0$ be any $\BF_q\dbl z\dbr$-lattice in $V$. Then the stabilizer in $G(\BF_q\dbl z\dbr)$ of $\Lambda_0$ is open, and in particular of finite index in the compact group $G(\BF_q\dbl z\dbr)$. Therefore, 
$$\Lambda:=\bigcap_{g\in G(\BF_q\dbl z\dbr)}\rho'(g)(\Lambda_0)$$ is an intersection of finitely many translates of $\Lambda_0$, hence a lattice in $V$. By definition, $\Lambda$ is $G(\BF_q\dbl z\dbr)$-invariant. Thus $\rho'$ is induced by $(\rho,\Lambda):=(\rho'|_{\Lambda},\Lambda)\in \Rep_{\BF_q\dbl z\dbr}G$. From the definition above we obtain $$\check V_{\ul{\CG},\bar x}(\rho')=\check T_{\ul{\CG},\bar x}(\rho)\otimes_{\BF_q\dbl z\dbr }\BF_q\dpl z\dpr.$$
In particular, the vector space $\check V_{\ul{\CG},\bar x}(\rho')$ is of dimension $\dim V$.

\item \label{RemTateModule_C}
These definitions are independent of the chosen base point $\bar x$, because for any other geometric base point $\bar x'$ of $\Test$ there is an isomorphism of fiber functors $\check T_{\ul{\CG},\bar x}\cong\check T_{\ul{\CG},\bar x'}$ and  $\check V_{\ul{\CG},\bar x}\cong\check V_{\ul{\CG},\bar x'}$ by \cite[Theorem~2.9]{dJ95a} and Remark~\ref{RemNeupertTate}.

\item\label{RemTateModule_B}
From the definition one obtains that the Tate module and the rational Tate module are tensor functors 
\begin{eqnarray*}
& & \check T_{\ul{\CG},\bar x}\colon\Rep_{\BF_q\dbl z\dbr}G\;\longto\; \Rep^\cont_{\BF_q\dbl z\dbr }\bigl(\pi_1^\et(\Test,\bar x)\bigr)\;=\;\Rep^\cont_{\BF_q\dbl z\dbr }\bigl(\pi_1^\alg(\Test,\bar x)\bigr)\qquad\text{and}\\[2mm]
& & \check V_{\ul{\CG},\bar x}\colon\Rep_{\BF_q\dpl z\dpr}G\;\longto\; \Rep^\cont_{\BF_q\dpl z\dpr }\bigl(\pi_1^\et(\Test,\bar x)\bigr)\,.
\end{eqnarray*}
In terms of Definition~\ref{DefLocalSystem} we may view $\check T_{\ul{\CG},\bar x}$ and $\check V_{\ul{\CG},\bar x}$ as tensor functors
\begin{eqnarray}
\label{EqTateFunctor}
& & \check T_{\ul{\CG}}\colon\Rep_{\BF_q\dbl z\dbr}G\;\longto\; \BLoc_X\qquad\text{and}\\[2mm]
\label{EqRatTateFunctor}
& & \check V_{\ul{\CG}}\colon\Rep_{\BF_q\dpl z\dpr}G\;\longto\; \PLoc_X\,,
\end{eqnarray}
with $\check T_{\ul{\CG},\bar x}=F_{\bar x}^\et\circ\check T_{\ul\CG}$ and $\check V_{\ul{\CG},\bar x}=F_{\bar x}^\et\circ\check V_{\ul\CG}$. The tensor functors \eqref{EqTateFunctor} and \eqref{EqRatTateFunctor} also exist if $X$ is not connected.

Furthermore, $\check T_{\ul{\CG},\bar x}$ and $\check T_{\ul{\CG}}$ are functorial on the category of \'etale local $G$-shtukas $\ul{\CG}$ with isomorphisms as morphisms, and $\check V_{\ul{\CG},\bar x}$ and $\check V_{\ul{\CG}}$ are functorial on the category of \'etale local $G$-shtukas $\ul{\CG}$ with isogenies as morphisms. Indeed, an isomorphism, respectively an isogeny of \'etale local $G$-shtukas canonically induces an isomorphism between the corresponding $\ulM$, respectively $\ulN$.
\end{enumerate}
\end{remark}

Recall the forgetful functors $\omega^{\circ}_A:\Rep_A G\to\FMod_A$ and ${\it forget}\colon\Rep^\cont_{A}\bigl(\pi_1^\et(X,\bar x)\bigr)\to\FMod_A$ from Definition~\ref{DefForget}. For an \'etale local $G$-shtuka $\ul{\CG}$ over $\Test$ the sets 
\begin{eqnarray}\label{EqTensorIsom}
\Triv_{\ul\CG,\bar x}(\BF_q\dbl z\dbr) & := & \Isom^{\otimes}(\omega^{\circ}_{\BF_q\dbl z\dbr },{\it forget}\circ\check T_{\ul{\CG},\bar x})(\BF_q\dbl z\dbr)\qquad\text{and}\\[2mm]
\Triv_{\ul\CG,\bar x}\bigl(\BF_q\dpl z\dpr\bigr) & := & \Isom^{\otimes}(\omega^{\circ},{\it forget}\circ\check V_{\ul{\CG},\bar x})\bigl(\BF_q\dpl z\dpr\bigr) \nonumber
\end{eqnarray}
are non-empty; see \cite[after Definition~3.5]{AH_Local}. This is due to the fact that we assumed $G$ to have connected fibers. In \cite[5.32]{RZ}, where also non-connected orthogonal groups are allowed, the isomorphism class of the \'etale fiber functor analogous to ${\it forget}\circ\check T_{\ul{\CG},\bar x}$ can vary. By the definition of the Tate functor, $\Triv_{\ul{\CG},\bar x}(\BF_q\dbl z\dbr)$ carries an action of $G(\BF_q\dbl z\dbr )\times \pi_1^\alg(\Test,\bar x)$ where the first factor acts through $\omega^{\circ}_{\BF_q\dbl z\dbr }$ and the action of $\pi_1^\alg(\Test,\bar x)$ is induced by the action on the Tate functor. Similarly, $\Triv_{\ul{\CG},\bar x}\bigl(\BF_q\dpl z\dpr\bigr)$ admits an action of $G\bigl(\BF_q\dpl z\dpr\bigr)\times \pi_1^\et(\Test,\bar x)$. For every choice of an element $\eta\in\Triv_{\ul{\CG},\bar x}(\BF_q\dbl z\dbr)$ we obtain a $G(\BF_q\dbl z\dbr )$-equivariant bijection 
\begin{equation}\label{EqBijIsom}
G(\BF_q\dbl z\dbr )\;\isoto\;\Triv_{\ul{\CG},\bar x}(\BF_q\dbl z\dbr)\,,\quad g\longmapsto\eta\circ g\,,
\end{equation}
where $G(\BF_q\dbl z\dbr )$ acts on itself by multiplication on the right. Under this bijection the action of $\pi_1^\alg(\Test,\bar x)$ corresponds to a group homomorphism
\[
\pi_1^\alg(\Test,\bar x)\;\longto\;G(\BF_q\dbl z\dbr )\,,\quad h\longmapsto\eta^{-1}\circ h(\eta)
\]
that is independent of $\eta$ up to conjugation by elements of $G(\BF_q\dbl z\dbr )$. Similar statements hold for $\Triv_{\ul\CG,\bar x}\bigl(\BF_q\dpl z\dpr\bigr)$ with $G(\BF_q\dbl z\dbr )$ and $\pi_1^\alg(\Test,\bar x)$ replaced by $G\bigl(\BF_q\dpl z\dpr\bigr)$ and $\pi_1^\et(\Test,\bar x)$.

\begin{definition}\label{DefLevel}
Let $\ul{\CG}$ be an \'etale local $G$-shtuka over a connected $\BaseFld\dpl\zeta\dpr$-analytic space $\Test$, and let $K\subset G(\BF_q\dbl z\dbr)$ be an open subgroup. Then an \emph{integral $K$-level structure} on $\ul{\CG}$ is a $\pi_1^\alg(\Test,\bar x)$-invariant $K$-orbit in $\Triv_{\ul{\CG},\bar x}(\BF_q\dbl z\dbr)$.

If $K\subset G\bigl(\BF_q\dpl z\dpr\bigr)$ is an open compact subgroup a \emph{rational $K$-level structure} on $\ul{\CG}$ is a $\pi_1^\et(\Test,\bar x)$-invariant $K$-orbit in $\Triv_{\ul{\CG},\bar x}\bigl(\BF_q\dpl z\dpr\bigr)$. For non-connected $\Test$ we make a similar definition choosing a base point on each connected component and an integral, respectively rational $K$-level structure on the restriction to each connected component separately. Note that every integral $K$-level structure on $\ul\CG$ defines a rational $K$-level structure but not conversely.

For an open subgroup $K\subset G(\BF_q\dbl z\dbr)$ let $\Test^K$ be the functor on the category of $L$-analytic spaces over $\Test$ parametrizing integral $K$-level structures on the \'etale local $G$-shtuka $\ul\CG$ over $\Test$. 
\end{definition}

\begin{proposition}\label{PropCM^K}
\begin{enumerate}\item \label{PropCM^K_A}
$\Test^K$ is represented by the finite \'etale covering space of $\Test$ that corresponds to the finite $\pi_1^\alg(\Test,\bar x)$-set $\Triv_{\ul{\CG},\bar x}(\BF_q\dbl z\dbr)/K$ under the equivalence \eqref{EqFundGp}. In particular $\Test^K$ is a strictly $L$-analytic space.
\item \label{PropCM^K_B}
For $K_0=G(\BF_q\dbl z\dbr )$, the morphism assigning to $\ul{\CG}$ the $K_0$-orbit of all elements of $\Triv_{\ul{\CG},\bar x}(\BF_q\dbl z\dbr)$ induces an isomorphism $\Test\cong \Test^{K_0}$.
\item \label{PropCM^K_C}
For any inclusion of open subgroups $K'\subset K\subset G\bigl(\BF_q\dbl z\dbr\bigr)$, forgetting part of the level structure induces compatible finite \'etale surjective morphisms
$$\breve\pi_{K,K'}:\Test^{K'}\rightarrow\Test^K$$ that are Galois with Galois group $K/K'$ if $K'$ is normal in $K$.
\end{enumerate}
\end{proposition}

\begin{proof}
Denote by $\wt\Test^K$ the finite \'etale covering space of $\Test$ from \ref{PropCM^K_A}. Let $f\colon Y\to\Test$ be a connected $L$-analytic space over $\Test$ and let $\eta K$ be an integral $K$-level structure on $f^*\ul\CG$, that is $\eta\in\Triv_{f^*\ul{\CG},\bar y}(\BF_q\dbl z\dbr )$ and the $K$-orbit $\eta K$ is $\pi_1^\alg(Y,\bar y)$-invariant where $\bar y$ is a geometric base point of $Y$. We must show that $\eta K$ arises from a uniquely determined $\Test$-morphism $Y\to\wt\Test^K$. Moving $\bar x$ by Remark~\ref{RemTateModule}\ref{RemTateModule_C} we may assume that $f(\bar y)=\bar x$, and hence $\check T_{f^*\ul{\CG},\bar y}=\check T_{\ul{\CG},\bar x}$. Consider the finite \'etale covering space $\wt\Test^K\times_\Test Y\to Y$. Then $F^\et_{Y,\bar y}(\wt\Test^K\times_\Test Y)=F^\et_{\Test,\bar x}(\wt\Test^K)=\Triv_{\ul{\CG},\bar x}(\BF_q\dbl z\dbr)/K$ for the \'etale fiber functors from \eqref{EqEtaleFiberFunctor}. In particular, the element $\eta K$ defines a $\pi_1^\alg(Y,\bar y)$-equivariant map from the one-element set $\{\bar y\}=F^\et_{Y,\bar y}(Y)$ to $F^\et_{Y,\bar y}(\wt\Test^K\times_\Test Y)$. By \cite[Theorem~2.10]{dJ95a} this map corresponds to a uniquely determined $Y$-morphism $Y\to\wt\Test^K\times_\Test Y$. The projection $Y\to\wt\Test^K$ onto the first component is the desired $\Test$-morphism that induces the integral $K$-level structure $\eta K$ over $Y$.

\smallskip\noindent
\ref{PropCM^K_B} and \ref{PropCM^K_C} follow directly from \ref{PropCM^K_A} .
\end{proof}

For arbitrary $\Test$ and $\ul\CG$, Proposition~\ref{PropCM^K} is the best one can hope for. However, if $\Test=(\breveRZ)^\an$ one can even replace $\check T_{\ul{\CG},\bar x}$ and $\Triv_{\ul{\CG},\bar x}(\BF_q\dbl z\dbr)$ by $\check V_{\ul{\CG},\bar x}$ and $\Triv_{\ul{\CG},\bar x}\bigl(\BF_q\dpl z\dpr\bigr)$, and allow compact open subgroups $K\subset G\bigl(\BF_q\dpl z\dpr\bigr)$; see Corollaries~\ref{CorLevel} and \ref{CorLevel2}. To explain this (also as a preparation to define rational level structures in Definition~\ref{DefRatLevel}) we keep the field $L$ introduced at the beginning of this section and consider in the following an admissible $\CO_L$-algebra $B$ in the sense of Raynaud, that is $B$ is a quotient $\beta\colon\CO_L\langle X_1,\ldots,X_s\rangle\onto B$ that is $\zeta$-torsion free; see \eqref{EqFormalTateAlgebra} and \cite[p.~293]{FRG1}. Then $\CX:=\Spf B$ is an admissible formal $\CO_L$-scheme. Let $B[\tfrac{1}{\zeta}]$ be the associated strictly affinoid $L$-algebra. We equip $B[\tfrac{1}{\zeta}]$ with the quotient map $\beta\colon L\langle X_1,\ldots,X_s\rangle\onto B[\tfrac{1}{\zeta}]$ and the $L$-Banach norm $|b|:=\inf\{|f|_{\sup}\colon f\in\beta^{-1}(b)\}$, where $|f|_{\sup}$ denotes the Gau{\ss} norm on the Tate algebra $L\langle X_1,\ldots,X_s\rangle$. Then $B=\{b\in B[\tfrac{1}{\zeta}]\colon |b|\le1\}$. The Berkovich spectrum $X=\SpBerk B[\tfrac{1}{\zeta}]$ is the $L$-analytic space $\CX^\an$ associated with the formal scheme $\CX$.

\begin{lemma}\label{LemmaDivisibility}
Recall the notation from \eqref{EqCrisRing}. Let $f=\sum\limits_{i\in\BZ}b_i z^i\in B\dbl z,z^{-1}\}$ and $a\in B$ with $|a|<1$ and assume that $f\in B\dbl z\dbr$ or $a\in\CO_L\setminus\{0\}$. If $f(a)=\sum\limits_{i\in\BZ}b_i a^i=0$ in $B[\tfrac{1}{\zeta}]$, then $f=(z-a)\cdot g$ for a uniquely determined $g=\sum\limits_{n\in\BZ}c_n z^n\in B\dbl z,z^{-1}\}$ with $c_n=\sum_{i>n} b_i a^{i-n-1}\in B$. Moreover, if $f\in B\dbl z\dbr$ then also $g\in B\dbl z\dbr$. 
\end{lemma}

\begin{proof}
First of all, $b_i\in B$ and $|a|<1$ implies that the series $c_n:=\sum_{i>n} b_i a^{i-n-1}$ converge in $B$ for all $n\in\BZ$. One easily computes that $f=(z-a)\cdot g$ for $g:=\sum_{n\in\BZ}c_n z^n$. To prove uniqueness let $\tilde g=\sum_{n\in\BZ}\tilde c_n z^n\in B\dbl z,z^{-1}\}$ also satisfy $f=(z-a)\cdot\tilde g$. Setting $c'_n:=c_n-\tilde c_n$ yields $c'_{n-1}=ac'_n$, whence $c'_n=a^{m-n}c'_m$ for all $m\ge n$. Letting $m$ go to $\infty$ and using $c'_m\in B$ shows that $|c'_n|$ is arbitrarily small, and therefore $c'_n=0$ for all $n$. This proves the uniqueness of $g$. 

If $f\in B\dbl z\dbr$ and $n<0$ then $c_n=a^{-n-1}\cdot f(a)=0$ and therefore $g\in B\dbl z\dbr$. If $a\in\CO_L\setminus\{0\}$ we must verify the convergence condition $\lim_{n\to-\infty}|c_n|\,|a|^{rn}=0$ for all $r\ge1$. We compute $c_n=-\sum_{i\le n} b_i a^{i-n-1}$. If $i\le n$ then $|a|^{(r-1)n}\le|a|^{(r-1)i}$, and hence $$|c_n|\,|a|^{rn}\;\le\;\max_{i\le n}|b_i|\,|a|^{i-n-1+(r-1)i+n}\;=\;\max_{i\le n}|b_i|\,|a|^{ri-1}.$$ The latter goes to zero for $n\to-\infty$ because $f\in B\dbl z,z^{-1}\}$. Therefore $g\in B\dbl z,z^{-1}\}$. 
\end{proof}

\begin{remark}\label{RemUnivLocGSht}
In addition to the loop group $LG$ we consider over $\BF_q\dbl\zeta\dbr$ the loop groups defined as the \fppf-sheaves on $\BF_q\dbl\zeta\dbr$-algebras $R$ by
\[
L_{z-\zeta}G(R) \;:=\;G\bigl(R\dbl z\dbr[\tfrac{1}{z-\zeta}]\bigr)\qquad \text{and} \qquad L_{z(z-\zeta)}G(R) \;:=\;G\bigl(R\dbl z\dbr[\tfrac{1}{z(z-\zeta)}]\bigr)
\]
and the canonical maps of groups $L^+G\to LG\to L_{z(z-\zeta)}G$ and $L^+G\to L_{z-\zeta}G\to L_{z(z-\zeta)}G$ which coincide as homomorphisms $L^+G\to L_{z(z-\zeta)}G$. If $\zeta\in R\mal$ is a unit, note that $z-\zeta\in R\dbl z\dbr\mal$, and hence $L_{z-\zeta}G=L^+G$ and $L_{z(z-\zeta)}G=LG$. On the other hand, if $\zeta$ is nilpotent in $R$, then $L_{z-\zeta}G=L_{z(z-\zeta)}G=LG$.

Recall from \cite[Definition~4.22]{AH_Local} that a local $G$-shtuka over an admissible formal $\CO_L$-scheme $\CX$ can be viewed as a projective system $(\ul\CG_m)_{m\in\BN}$ of local $G$-shtukas $\ul\CG_m$ over $\CX_m=\Var(\zeta^m)\subset\CX$ with $\ul\CG_{m-1}\cong\ul\CG_m\otimes_{\CX_m}\CX_{m-1}$. On $\CX_m$ the element $\zeta$ is nilpotent.

Now let $B$ and $B[\tfrac{1}{\zeta}]$ be as before. If $(\ul\CG,\bar\delta)$ is a $\Spf B$-valued point in $\breveRZ(\Spf B)$, then the \'etale covering $\Spf B'\to\Spf B$ from Lemma~\ref{LemmaTrivializing} is given by a faithfully flat ring homomorphism $B\to B'$ by \cite[Lemma~1.6]{FRG1}, and this implies that $\ul\CG$ comes from an $L^+G$-torsor $\CG$ over $\Spec B$ together with an isomorphism of the associated $L_{z-\zeta}G$-torsors $\tau_\CG\colon L_{z-\zeta}\s\CG\isoto L_{z-\zeta}\CG$. We view $(\CG,\tau_\CG)$ as the \emph{bounded local $G$-shtuka over $\Spec B$} induced from the bounded local $G$-shtuka $\ul\CG$ over $\Spf B$. A \emph{quasi-isogeny} $u\colon(\CG'\!,\tau_{\CG'})\to(\CG,\tau_\CG)$ between two such bounded local $G$-shtukas $(\CG'\!,\tau_{\CG'})$ and $(\CG,\tau_\CG)$ over $\Spec B$ is an isomorphism of the associated $LG$-torsors $u\colon L\CG'\isoto L\CG$ which satisfies $u\circ\tau_{\CG'}=\tau_\CG\circ\s u$ as isomorphism of the associated $L_{z(z-\zeta)}G$-torsors.

In particular, $(\CG,\tau_\CG)$ induces an \'etale local $G$-shtuka on the $L$-analytic space $X=(\Spf B)^\an=\SpBerk(B[\tfrac{1}{\zeta}])$.
\end{remark}

The following proposition is a weaker analog of the fact that lifts of $p$-divisible groups and morphisms between them correspond uniquely to lifts of the Hodge-filtrations on the associated crystals (resp.\ morphisms between them). We do not dispose of the full analog of this assertion as in our (in general non-minuscule) context Griffiths transversality does not hold; compare the discussion of Genestier and Lafforgue in \cite[\S\,11]{GL}. 

\begin{proposition}\label{PropLiftOfIsog}
Let $\CX=\Spf B$ be an admissible formal $\CO_L$-scheme and let $X:=\CX^\an$ be its associated $L$-analytic space. Assume that $X$ is connected and choose a geometric base point $\bar x$ of $X$. Let $(\ul\CG,\bar\delta)$ and $(\ul\CG'\!,\bar\delta')$ be (representatives of) points in $\breveRZ(\CX)$. Then $\breve\pi(\ul\CG,\bar\delta)=\breve\pi(\ul\CG'\!,\bar\delta')$ in $\breve\CH_{G,\hat{Z}}^\an(X)$ if and only if the unique lift of the quasi-isogeny $\bar\delta^{-1}\circ\bar\delta'$ by rigidity \cite[Proposition~2.11]{AH_Local} is a quasi-isogeny $u\colon\ul\CG'\to\ul\CG$ over $\Spec B$ in the sense of Remark~\ref{RemUnivLocGSht}. In this case $u$ induces an isomorphism of the rational Tate module functors $\check V_{u,\bar x}\colon\check V_{\ul\CG'\!,\bar x}\isoto\check V_{\ul\CG,\bar x}$ over $X$ and the following assertions are equivalent
\begin{enumerate}
\item \label{PropLiftOfIsog_A}
$u\colon\ul\CG'\to\ul\CG$ is an isomorphism of local $G$-shtukas, that is $(\ul\CG,\bar\delta)=(\ul\CG'\!,\bar\delta')$ in $\breveRZ(\CX)$,
\item \label{PropLiftOfIsog_B}
$\check V_{u,\bar x}$ is an isomorphism $\check T_{\ul\CG'\!,\bar x}\isoto\check T_{\ul\CG,\bar x}$ of the integral Tate module functors,
\item \label{PropLiftOfIsog_C}
$\check V_{u,\bar x}(\Darst)$ is an isomorphism $\check T_{\ul\CG'\!,\bar x}(\Darst)\isoto\check T_{\ul\CG,\bar x}(\Darst)$ for some faithful 
$\Darst\in\Rep_{\BF_q\dbl z\dbr}G$.
\end{enumerate}
Moreover, for every rational $G(\BF_q\dbl z\dbr)$-level structure $\eta G(\BF_q\dbl z\dbr)$ on $\ul\CG$ with $\eta\in\Triv_{\ul\CG,\bar x}\bigl(\BF_q\dpl z\dpr\bigr)$ there is an admissible formal blowing-up $\CY\to\CX$ and a $(\ul\CG''\!,\bar\delta'')\in\breveRZ(\CY)$ with $\breve\pi(\ul\CG,\bar\delta)=\breve\pi(\ul\CG''\!,\bar\delta'')$ and $(\check V_{u''\!,\bar x})^{-1}\circ\eta\in\Triv_{\ul\CG''\!,\bar x}(\BF_q\dbl z\dbr)$, where $u''\colon\ul\CG''\to\ul\CG$ is the unique lift of $\bar\delta^{-1}\circ\bar\delta''$.
\end{proposition}

\noindent
{\it Remark.} The last assertion uses the ind-projectivity of the affine flag variety $\Flag_G$ and is in general false if $G$ is not parahoric; see Example~\ref{ExampleNotATorsor}.

\begin{proof}[{Proof of Proposition~\ref{PropLiftOfIsog}}]
By Lemma~\ref{LemmaTrivializing} there is an \'etale covering $\wt\CX=\Spf \wt B\to\CX$ of admissible formal $\CO_L$-schemes and trivializations $\alpha\colon\ul\CG\isoto\bigl((L^+G)_{\wt\CX},A\s\bigr)$ and $\alpha'\colon\ul\CG'\isoto\bigl((L^+G)_{\wt\CX},A'\s\bigr)$ with $A,A'\in G\bigl(\wt B\dbl z\dbr[\tfrac{1}{z-\zeta}]\bigr)$. Note that $\wt B\subset \wt B[\tfrac{1}{\zeta}]$ because $\wt B$ has no $\zeta$-torsion. In addition the quasi-isogenies $\bar\delta$ and $\bar\delta'$ correspond under $\alpha$ and $\alpha'$ to elements $\olDelta,\olDelta'\in LG\bigl(\wt B/(\zeta)\bigr)$, that lift by Lemma~\ref{LemmaGL} to uniquely determined elements $\Delta,\Delta'\in G\bigl(\wt B\dbl z,z^{-1}\}[\tfrac{1}{\tminus}]\bigr)$ with $\Delta A=b\,\s(\Delta)$ and $\Delta' A'=b\,\s(\Delta')$. In particular, the quasi-isogeny $\bar\delta^{-1}\circ\bar\delta'\colon\ul\CG\to\ul\CG'$ over $\Spec B/(\zeta)$ lifts to $U=\Delta^{-1}\Delta'\in G\bigl(\wt B\dbl z,z^{-1}\}[\tfrac{1}{\tminus}]\bigr)$ with $UA'=A\s(U)$. The morphism $\breve\pi$ sends $(\ul\CG,\bar\delta)$ and $(\ul\CG'\!,\bar\delta')$ to 
\begin{eqnarray*}
\gamma & := & \s(\Delta)A^{-1}\cdot G\bigl(\wt B[\tfrac{1}{\zeta}]\dbl z-\zeta\dbr\bigr)\qquad\text{and}\\[2mm]
\gamma' & := & \s(\Delta')(A')^{-1}\cdot G\bigl(\wt B[\tfrac{1}{\zeta}]\dbl z-\zeta\dbr\bigr)\\[2mm]
& = & \s(\Delta)A^{-1}U\cdot G\bigl(\wt B[\tfrac{1}{\zeta}]\dbl z-\zeta\dbr\bigr).
\end{eqnarray*}
If $u$ is a quasi-isogeny over $\Spec B$ then $U\in G(\wt B\dbl z\dbr[\tfrac{1}{z}])\subset G\bigl(\wt B[\tfrac{1}{\zeta}]\dbl z-\zeta\dbr\bigr)$ and hence $\breve\pi(\ul\CG,\bar\delta)=\gamma=\gamma'=\breve\pi(\ul\CG'\!,\bar\delta')$ in $\breve\CH_{G,\hat Z}^\an$. 

Conversely, the condition $\breve\pi(\ul\CG,\bar\delta)=\breve\pi(\ul\CG'\!,\bar\delta')$ yields $U\in G\bigl(\wt B[\tfrac{1}{\zeta}]\dbl z-\zeta\dbr\bigr)$. We claim that this implies $U\in G\bigl(\wt B\dbl z,z^{-1}\}[\tfrac{1}{\s(\tminus)}]\bigr)$. To prove the claim let $\Darst\colon G\into\GL_r$ be a faithful representation. Then the entries of $\Darst(U)$ and $\Darst(U)^{-1}$ are of the form $\tminus^{-e}f$ with $e\in\BN_0$ and $f\in \wt B\dbl z,z^{-1}\}$. We must show that there is a $g\in \wt B\dbl z,z^{-1}\}$ with $\tminus^{-e}f=\s(\tminus)^{-e}g$. Recall from \eqref{EqTMinus} that $\tminus^{-e}f=(1-\tfrac{\zeta}{z})^{-e}\s(\tminus)^{-e}f$. If $e>0$ then $U\in G\bigl(\wt B[\tfrac{1}{\zeta}]\dbl z-\zeta\dbr\bigr)$ resp.~$U^{-1}\in G\bigl(\wt B[\tfrac{1}{\zeta}]\dbl z-\zeta\dbr\bigr)$ implies that $f(\zeta)=0$ in $\wt B[\tfrac{1}{\zeta}]$. By Lemma~\ref{LemmaDivisibility} we find $f=(z-\zeta)g_1=(1-\tfrac{\zeta}{z})zg_1$ with $g_1\in \wt B\dbl z,z^{-1}\}$, and hence $\tminus^{-e}f=(1-\tfrac{\zeta}{z})^{1-e}\s(\tminus)^{-e}zg_1$. Continuing in this way for $\Darst(U)$ and $\Darst(U)^{-1}$, we obtain that $\Darst(U)\in \GL_r\bigl(\wt B\dbl z,z^{-1}\}[\tfrac{1}{\s(\tminus)}]\bigr)$. The claim follows.

This shows that $\s(U)\in G\bigl(\wt B\dbl z,z^{-1}\}[\tfrac{1}{\sigma^{2*}(\tminus)}]\bigr)\subset G\bigl(\wt B[\tfrac{1}{\zeta}]\dbl z-\zeta^q\dbr\bigr)$, where we use \eqref{EqBmaxInBdR}. Since $A,A'\in G\bigl(\wt B[\tfrac{1}{\zeta}]\dbl z-\zeta^{q^i}\dbr\bigr)$ for all $i>0$ we obtain $U=A\s(U)(A')^{-1}\in G\bigl(\wt B[\tfrac{1}{\zeta}]\dbl z-\zeta^q\dbr\bigr)$. Analogously to the previous paragraph this implies that $U\in G\bigl(\wt B\dbl z,z^{-1}\}[\tfrac{1}{\sigma^{2*}(\tminus)}]\bigr)$ and iteratively $U\in G\bigl(\wt B\dbl z,z^{-1}\}[\tfrac{1}{\sigma^{i*}(\tminus)}]\bigr)$ for all $i>0$. It follows that the entries of $\Darst(U)$ converge on all of $\{0<|z|<1\}$, whence lie in $\wt B\dbl z,z^{-1}\}$, and so $U\in G\bigl(\wt B\dbl z,z^{-1}\}\bigr)$.

Now let $\Fp\subset \wt B[\tfrac{1}{\zeta}]$ be a minimal prime ideal and let $x\in\wt\CX^\an=\SpBerk(\wt B[\tfrac{1}{\zeta}])$ be a point given by a multiplicative semi-norm $|\,.\,|_x\colon \wt B[\tfrac{1}{\zeta}]\to\BR_{\ge0}$ such that $\{b\in \wt B[\tfrac{1}{\zeta}]\colon|b|_x=0\}=\Fp$. Note that $|\,.\,|_x$ exists by \cite[Corollary~2.1.16]{Berkovich1} for example as the preimage of the multiplicative Gau{\ss} norm on $L\langle X_1,\ldots,X_d\rangle$ under a Noether normalization map $L\langle X_1,\ldots,X_d\rangle\into \wt B[\tfrac{1}{\zeta}]/\Fp$; see \cite[\S\,6.1.2, Theorem~1]{BGR}. Let $\Omega$ be the completion with respect to $|\,.\,|_x$ of an algebraic closure of $\wt B[\tfrac{1}{\zeta}]/\Fp$ and let $\CO_{\Omega}$ be the valuation ring of $\Omega$. Then the image of $\wt B$ in $\Omega$ lies in $\CO_{\Omega}$. We denote the image of $U$ in $G\bigl(\CO_{\Omega}\dbl z,z^{-1}\}\bigr)$ by $U_\Fp$. By \cite[Lemma~2.8]{AH_Local} there are elements $H_\Fp,H_\Fp'\in G(\Omega\dbl z\dbr)$ with $A\s(H_\Fp)=H_\Fp$ and $A'\s(H_\Fp')=H_\Fp'$ that provide a trivialization of the Tate module functors. By \cite[Remark~3.4.3]{HartlKim} we even have $\Darst(H_\Fp),\Darst(H_\Fp')\in\GL_r(\Omega\langle\tfrac{z}{\zeta^s}\rangle)$ for every $s>\frac{1}{q}$. We compute $\Darst(H_\Fp^{-1}U_\Fp H_\Fp')=\Darst(\s(H_\Fp^{-1})A^{-1}U_\Fp A'\s(H_\Fp'))=\s\Darst(H_\Fp^{-1}U_\Fp H_\Fp')\in \GL_r(\Omega\ancon[s])$ and this implies $\Darst(H_\Fp^{-1}U_\Fp H_\Fp')\in \GL_r\bigl(\BF_q\dpl z\dpr\bigr)$ because $\Omega\ancon[s]^\sigma=\BF_q\dpl z\dpr$. Let $N_\Fp\in\BN_0$ be minimal such that $z^{N_\Fp}\Darst(H_\Fp^{-1}U_\Fp H_\Fp'),\,z^{N_\Fp}\Darst(H_\Fp^{-1}U_\Fp H_\Fp')^{-1}\,\in\,\BF_q\dbl z\dbr^{r\times r}$. Then $N_\Fp=0$ if and only if $\Darst(H_\Fp^{-1}U_\Fp H_\Fp')\in\GL_r(\BF_q\dbl z\dbr)$, that is $H_\Fp^{-1}U_\Fp H_\Fp'\in G(\BF_q\dbl z\dbr)$. Moreover, $z^{N_\Fp}\Darst(U_\Fp),\,z^{N_\Fp}\Darst(U_\Fp)^{-1}\in \Omega\langle\tfrac{z}{\zeta^s}\rangle^{r\times r}\cap\CO_{\Omega}\dbl z,z^{-1}\}^{r\times r}=\CO_{\Omega}\dbl z\dbr^{r\times r}$. Because this holds for all of the finitely many minimal prime ideals of $\wt B[\tfrac{1}{\zeta}]$ and the intersection of these is the nil-radical $\CN$ of $\wt B[\tfrac{1}{\zeta}]$ we see that $U\in G\bigl(\wt B\dbl z,z^{-1}\}\bigr)$ implies $z^N\Darst(U),z^N\Darst(U^{-1})\in \wt B\dbl z\dbr^{r\times r}+\CN\dbl z^{-1}\dbr^{r\times r}$ for $N:=\max\{N_\Fp\colon\Fp\text{ minimal}\}$. Since $\wt B[\tfrac{1}{\zeta}]$ is noetherian the nil-radical is nilpotent and there is an integer $m$ such that $\CN^{q^m}=(0)$. In particular $z^N\Darst(\sigma^{m*}(U))\in \wt B\dbl z\dbr^{r\times r}$. Let $n$ be such that $(z-\zeta)^n\Darst(A),(z-\zeta)^n\Darst(A')^{-1}\in \wt B\dbl z\dbr^{r\times r}$. Then 
\begin{eqnarray*}
\Bigl(\prod_{i=0}^{m-1}(z-\zeta^{q^i})^{2n}\Bigr)\cdot z^N\Darst(U) & = & (z-\zeta)^n\Darst(A)\cdot\ldots\cdot\sigma^{(m-1)*}\bigl((z-\zeta)^n\Darst(A)\bigr)\cdot z^N\Darst(\sigma^{m*}(U))\cdot\\[-3mm]
& & \qquad\qquad \cdot\,\sigma^{(m-1)*}\bigl((z-\zeta)^n\Darst(A')^{-1}\bigr)\cdot\ldots\cdot(z-\zeta)^n\Darst(A')^{-1}\\
& \in & \wt B\dbl z\dbr^{r\times r}
\end{eqnarray*}
and applying Lemma~\ref{LemmaDivisibility} with $a=\zeta^{q^i}$ for $i=0,\ldots,m-1$ to the entries of this matrix yields $z^N\Darst(U)\in \wt B\dbl z\dbr^{r\times r}$. In the same way we see that $z^N\Darst(U^{-1})\in \wt B\dbl z\dbr^{r\times r}$. This implies $\Darst(U)\in\GL_r\bigl(\wt B\dbl z\dbr[\tfrac{1}{z}]\bigr)$ and $U\in G\bigl(\wt B\dbl z\dbr[\tfrac{1}{z}]\bigr)=LG(\wt B)$. We conclude that $U$ defines a quasi-isogeny $U\colon\ul\CG_{\wt B}\to\ul\CG'_{\wt B}$ that induces the isomorphism of the rational Tate module functors $\check V_{U,\bar x}=H_\Fp^{-1}U_\Fp H_\Fp'\colon\check V_{\ul\CG'\!,\bar x}\isoto\check V_{\ul\CG,\bar x}$ for the geometric base point $\bar x\colon\SpBerk(\Omega)\to X$. By uniqueness $U$ descends to a quasi-isogeny $u\colon\ul\CG'\to\ul\CG$ over $\Spec B$ as desired.

In this situation, clearly \ref{PropLiftOfIsog_A} implies \ref{PropLiftOfIsog_B} and \ref{PropLiftOfIsog_B} implies \ref{PropLiftOfIsog_C}. We further see that \ref{PropLiftOfIsog_C} for our representation $\Darst$ implies $\Darst(H_\Fp^{-1}U_\Fp H_\Fp')\in \GL_r\bigl(\BF_q\dbl z\dbr\bigr)$ for all minimal $\Fp\subset \wt B[\tfrac{1}{\zeta}]$, and hence the integer $N$ defined above is zero and $U\in G\bigl(\wt B\dbl z\dbr\bigr)$. In particular, \ref{PropLiftOfIsog_C} implies \ref{PropLiftOfIsog_A}, because $u\colon\ul\CG'\isoto\ul\CG$ is an isomorphism of local $G$-shtukas if and only if $U\in G\bigl(\wt B\dbl z\dbr\bigr)$.

\medskip
It remains to prove the last assertion about the rational $G(\BF_q\dbl z\dbr)$-level structure $\eta G(\BF_q\dbl z\dbr)$ on $\ul\CG$ where $\eta\in\Triv_{\ul{\CG},\bar x}\bigl(\BF_q\dpl z\dpr\bigr)$. Let $\rho''\colon\pi_1^\et(\Test,\bar x)\to G(\BF_q\dbl z\dbr)$ be the homomorphism by which the fundamental group acts on $\eta$, that is $g(\eta)=\eta\cdot\rho''(g)$ for $g\in\pi_1^\et(\Test,\bar x)$. Note that $\rho''(g)$ indeed lies in $G(\BF_q\dbl z\dbr)$ because $g$ fixes $\eta G(\BF_q\dbl z\dbr)$. In particular $\rho''$ factors through a representation $\pi_1^\alg(\Test,\bar x)\to G(\BF_q\dbl z\dbr)$. By Corollary~\ref{Cor1.7}, Proposition~\ref{Prop2.13} and \cite[Proposition~3.6]{AH_Local}, $\rho''$ comes from an \'etale local $G$-shtuka $\ul\CG''_L$ over $\Spec B[\tfrac{1}{\zeta}]$ together with a tensor isomorphism $\beta\in\Triv_{\ul{\CG}''\!,\bar x}(\BF_q\dbl z\dbr)$, and the tensor isomorphism $\eta\beta^{-1}\colon\check V_{\ul\CG''_L,\bar x}\isoto\check V_{\ul\CG,\bar x}$ is of the form $\check V_{u''_L,\bar x}$ for a quasi-isogeny $u''_L\colon\ul\CG''_L\to\ul\CG_{B[\frac{1}{\zeta}]}$ over $\Spec B[\tfrac{1}{\zeta}]$. This means that $\ul\CG''_L=(\CG''_L,\tau'')$ where $\CG''_L$ is an $L^+G$-torsor over $\Spec B[\tfrac{1}{\zeta}]$ and $\tau''\colon\s\CG''_L\isoto\CG''_L$ is an isomorphism of $L^+G$-torsors. Also $u''_L\colon L\CG''_L\isoto L\CG_{B[\frac{1}{\zeta}]}$ is an isomorphism of the associated $LG$-torsors with $u''_L\circ\tau''=\tau_\CG\circ\s u''_L$. Note that the assumption on the nilpotence of $\zeta$ in \cite[Proposition~3.6]{AH_Local} is not satisfied for $\Spec B[\tfrac{1}{\zeta}]$, but is also not used in the proof of loc.\ cit. 

We may thus apply the following Lemma~\ref{LemmaLiftOfIsog} by taking the $LG$-torsor associated with $\CG$ as the $LG$-torsor $\CG$ in Lemma~\ref{LemmaLiftOfIsog}. It provides an extension of the pair $(\ul\CG''_L,u''_L)$ to a local $G$-shtuka $\ul\CG''$ bounded by $\hat{Z}^{-1}$ and a quasi-isogeny $u''\colon\ul\CG''\to\ul\CG$ over a blowing-up $Y$ of $\Spec B$ in a finitely generated ideal $\Fb\subset B$ containing a power of $\zeta$. By \cite[Propositions~2.1 and 1.3]{FRG1} the $\zeta$-adic completion $\CY$ of $Y$ is the admissible formal blowing-up of $\CX=\Spf B$ in the ideal $\Fb$. In particular, $\CY^\an\to X$ is an isomorphism. We set $\bar\delta'':=\bar\delta\circ(u''\mod\zeta)$. Then $(\ul\CG''\!,\bar\delta'')\in\breveRZ(\CY)$, and $\breve\pi(\ul\CG,\bar\delta)=\breve\pi(\ul\CG''\!,\bar\delta'')$ by the first part of the proposition, and $(\check V_{u''\!,\bar x})^{-1}\circ\eta=\beta\in\Triv_{\ul\CG''\!,\bar x}(\BF_q\dbl z\dbr)$ by construction. 
\end{proof}

\begin{lemma}\label{LemmaLiftOfIsog}
Let $B$ be an admissible formal $\CO_L$-algebra. Let $\CG$ be an $LG$-torsor over $\Spec B$ and assume that there is an \'etale covering $\Spf\wt B\to\Spf B$ of admissible formal $\CO_L$-schemes such that $\CG$ admits a trivialization $\alpha\colon\CG\otimes_B\Spec\wt B\isoto LG_{\Spec\wt B}$. Let $\tau_\CG\colon L_{z(z-\zeta)}\s\CG\isoto L_{z(z-\zeta)}\CG$ be an isomorphism of the associated $L_{z(z-\zeta)}G$-torsors. Furthermore, let $\CG''_L$ be an $L^+G$-torsor over $\Spec B[\tfrac{1}{\zeta}]$, let $\tau''\colon\s\CG''_L\isoto\CG''_L$ be an isomorphism of $L^+G$-torsors, and let $u''_L\colon L\CG''_L\isoto\CG_{B[\frac{1}{\zeta}]}$ be an isomorphism of $LG$-torsors over $\Spec B[\frac{1}{\zeta}]$ satisfying $u''_L\circ\tau''=\tau_\CG\circ\s u''_L$.

Then there is a blowing-up $Y$ of $\Spec B$ in a finitely generated ideal $\Fb\subset B$ containing a power of $\zeta$, an $L^+G$-torsor $\CG''$ over $Y$, an isomorphism $\tau_{\CG''}\colon L_{z-\zeta}\s\CG''\isoto L_{z-\zeta}\CG''$ of the associated $L_{z-\zeta}G$-torsors over $Y$, and an isomorphism $u''\colon L\CG''\isoto \CG_Y$ of $LG$-torsors satisfying $u''\circ\tau_{\CG''}=\tau_\CG\circ\s u''$, such that the pullback of $(\CG''\!,\tau_{\CG''},u'')$ to $Y\times_B\Spec B[\tfrac{1}{\zeta}]=\Spec B[\tfrac{1}{\zeta}]$ is isomorphic to $(\CG''_L,\tau''\!,u''_L)$ via an isomorphism of $L^+G$-torsors $h\colon\CG''_{B[\frac{1}{\zeta}]}\isoto\CG''_L$ satisfying $h\circ\tau_{\CG''}=\tau''\circ\s h$ and $u''=u''_L\circ h$.

Moreover, if the element $\s\alpha\circ\tau_\CG^{-1}\circ\alpha^{-1}\in G\bigl(\wt B\dbl z\dbr[\tfrac{1}{z(z-\zeta)}]\bigr)\subset G\bigl(\wt B[\tfrac{1}{\zeta}]\dbl z\dbr[\tfrac{1}{z-\zeta}]\bigr)$ maps to a point in $\hat{Z}^\an(\SpBerk\wt B[\tfrac{1}{\zeta}])$, then $\tau_{\CG''}$ is bounded by $\hat{Z}^{-1}$ and $\ul\CG''=(\CG''\!,\tau_{\CG''})$ is a local shtuka over $Y$ bounded by $\hat{Z}^{-1}$ in the sense of Remark~\ref{RemUnivLocGSht}.
\end{lemma}

\begin{proof}
We consider the functor $\ul\CM_{\ul\CG}$ on $\Spec B$-schemes classifying ``quasi-isogenies of (unbounded) local $G$-shtukas to $\ul\CG:=(\CG,\tau_\CG)$\,''\!, that on affine $B$-schemes $S=\Spec R$ is defined by
\begin{eqnarray*}
\ul\CM_{\ul\CG}(S) & := & \bigl\{\text{Isomorphism classes of }(\CG''\!,\tau_{\CG''},u'')\colon\;\text{where }\CG''\text{ is an $L^+G$-torsor over } R, \\[1mm]
&&~~ \text{ where } \tau_{\CG''}\colon L_{z(z-\zeta)}\s\CG''\isoto L_{z(z-\zeta)}\CG'' \text{ is an isomorphism of the associated} \\[1mm]
&&~~ \text{ $L_{z(z-\zeta)}G$-torsors and }u''\colon L\CG''\isoto \CG\times_B S \text{ is an isomorphism of $LG$-torsors,} \\[1mm]
&&~~ \text{ with } \tau_\CG\circ\s u''=u''\circ\tau_{\CG''} \bigr\}. 
\end{eqnarray*}
Here $(\CG''\!,\tau_{\CG''},u'')$ and $(\CG'\!,\tau_{\CG'},u')$ are isomorphic if there is an isomorphism $h\colon\CG''\isoto\CG'$ of $L^+G$-torsors satisfying $h\circ\tau_{\CG''}=\tau_{\CG'}\circ\s h$ and $u''=u'\circ h$.

The functor $\ul\CM_{\ul\CG}$ is representable by an ind-projective ind-scheme over $\Spec B$ as follows. We consider its base change $\ul\CM_{\ul\CG}\otimes_B\wt B$ and fix a trivialization $\alpha\colon\CG\otimes_B\wt B\isoto LG_{\wt B}$. Over a $\wt B$-algebra $R$ the data $(\CG''\!,\alpha\circ u'')$ is represented by the ind-projective ind-scheme $\Flag_G\whtimes_{\BF_q}\Spec\wt B$ over $\Spec\wt B$. Indeed, over an \'etale covering $\Spec \wt R$ of $\Spec R$ the $L^+G$-torsor $\CG''$ can be trivialized by an isomorphism $\beta\colon\CG''_{\wt R}\isoto (L^+G)_{\wt R}$ and then $\alpha\circ u''\circ\beta^{-1}$ yields an $\wt R$-valued point of $\Flag_G$ that is independent of all choices and descends to an $R$-valued point of $\Flag_G$; see \cite[Theorem~6.2]{HV1} or \cite[Theorem~4.4]{AH_Local} for more details and for the inverse construction. Over $\Spec\wt R$, also 
\[
\beta\circ\tau_{\CG''}\circ\s\beta^{-1}\;=\;\beta\circ(u'')^{-1}\circ\tau_G\circ\s u''\circ\s\beta^{-1} \;\in\; L_{z(z-\zeta)}G(\wt R)\;=\;G\bigl(\wt R\dbl z\dbr[\tfrac{1}{z(z-\zeta)}]\bigr)
\]
is uniquely determined by $u''$. This shows that $\ul\CM_{\ul\CG}\otimes_B\Spec\wt B\isoto \Flag_G\whtimes_{\BF_q}\Spec\wt B$ is an ind-projective ind-scheme over $\Spec\wt B$. It descends to an ind-projective ind-scheme $\ul\CM_{\ul\CG}$ over $\Spec B$, because $B\to\wt B$ is faithfully flat by \cite[Lemma~1.6]{FRG1}; see \cite[Theorem~4.4]{AH_Local} for details. 

The triple $(\CG''_L,\tau''\!,u''_L)$ corresponds to a morphism $f\colon\Spec B[\tfrac{1}{\zeta}]\to\ul\CM_{\ul\CG}$. Since its source is quasi-compact, $f$ factors through a subscheme $\ul\CM_{\ul\CG}^{(N)}$ that is projective over $\Spec B$; see \cite[Lemma~5.4]{HV1}. The scheme theoretic closure $\Gamma$ of the graph of $f$ in $\ul\CM_{\ul\CG}^{(N)}$ is a projective scheme over $\Spec B$ and the projection $\Gamma\to\Spec B$ is an isomorphism over $\Spec B[\tfrac{1}{\zeta}]$. By the flattening technique of Raynaud and Gruson~\cite[Corollaire~5.7.12]{RaynaudGruson} there is a blowing-up $Y$ of $\Spec B$ in a finitely generated ideal $\Fb\subset B$ containing a power of $\zeta$, such that the strict transform of $\Gamma$, that is, the closed subscheme of $\Gamma\times_B Y$ defined by the sheaf of ideals of $\zeta$-torsion, is isomorphic to $Y$. The morphism $Y\to\Gamma\to\ul\CM_{\ul\CG}^{(N)}$ corresponds to an extension over $Y$ of the triple $(\CG''_L,\tau''\!,u''_L)$ from $\Spec B[\tfrac{1}{\zeta}]$. This means that over $Y$ there is an $L^+G$-torsor $\CG''$, an isomorphism of the associated $L_{z(z-\zeta)}G$-torsors $\tau_{\CG''}\colon L_{z(z-\zeta)}\s\CG''\isoto L_{z(z-\zeta)}\CG''$ and an isomorphism of the associated $LG$-torsors $u''\colon L\CG''\isoto \CG\times_B Y$ with $\tau_\CG\circ\s u''=u''\circ\tau_{\CG''}$ and over $Y\times_B\Spec B[\tfrac{1}{\zeta}]=\Spec B[\tfrac{1}{\zeta}]$ an isomorphism of $L^+G$-torsors $h\colon\CG''_{B[\frac{1}{\zeta}]}\isoto\CG''_L$ satisfying $h\circ\tau_{\CG''}=\tau''\circ\s h$ and $u''=u''_L\circ h$.

We claim that $\tau_{\CG''}$ comes from an isomorphism of $L_{z-\zeta}G$-torsors $\tau_{\CG''}\colon L_{z-\zeta}\s\CG''\isoto L_{z-\zeta}\CG''$. We choose a trivialization of $\ul\CG''$ over an \'etale covering $\wt Y$ of $Y\times_{\Spec B}\Spec\wt B$ and write the Frobenius $\tau_{\CG''}$ of $\ul\CG''$ as an element $\tau_{\CG''}\in G\bigl(\CO_{\wt Y}\dbl z\dbr[\tfrac{1}{z(z-\zeta)}]\bigr)$. Our claim means that $\tau_{\CG''}$ actually lies in $G\bigl(\CO_{\wt Y}\dbl z\dbr[\tfrac{1}{z-\zeta}]\bigr)$. To prove the claim, we choose a faithful representation $\Darst\colon G\into\SL_r$ and we consider the matrix entries $g_{ij}$ of $\Darst(\tau_{\CG''})$ which satisfy $(z-\zeta)^mg_{ij}\in\CO_{\wt Y}\dbl z\dbr[\tfrac{1}{z}]$ for an appropriate power $m$. Since $\tau_{\CG''}$ differs from $\tau''$ over $\wt Y_L:=\wt Y\otimes_{\CO_L}L$ by the isomorphism $h$ of $L^+G$-torsors, we see that $(z-\zeta)^mg_{ij}\in\CO_{\wt Y_L}\dbl z\dbr$. Since the intersection of $\CO_{\wt Y}\dbl z\dbr[\tfrac{1}{z}]$ and $\CO_{\wt Y_L}\dbl z\dbr$ in $\CO_{\wt Y_L}\dbl z\dbr[\tfrac{1}{z}]$ equals $\CO_{\wt Y}\dbl z\dbr$, this shows that $(z-\zeta)^mg_{ij}\in\CO_{\wt Y}\dbl z\dbr$. In particular, $g_{ij}\in\CO_{\wt Y}\dbl z\dbr[\tfrac{1}{z-\zeta}]$, and this proves our claim.

It remains to show that $\ul\CG''$ is bounded by $\hat Z^{-1}$ under the additional assumptions on $\tau_\CG$. By \cite[Propositions~2.1 and 1.3]{FRG1} the $\zeta$-adic completion $\CY$ of $Y$ is the admissible formal blowing-up of $\Spf B$ in the ideal $\Fb$. By Proposition~\ref{PropBound}\ref{PropBound_A} there is an integer $n$ such that for all representatives $(R,\hat Z_R)$ of $\hat Z$ the morphism $\hat{Z}_R\to\wh\Flag_{\SL_r,R}$ factors through $\wh\Flag^{(n)}_{\SL_r,R}$. By enlarging $n$, we may assume that the morphism $\tilde f\colon\wt\CY\to\wh\Flag_{G,R_{\hat Z}}\to\wh\Flag_{\SL_r,R_{\hat Z}}$ defined by $\tau_{\CG''}^{-1}$ factors through $\wh\Flag^{(n)}_{\SL_r,R_{\hat Z}}$. Let $(R,\hat Z_R)$ be such a representative and set $\wt\CY_R:=\wt\CY\whtimes_{R_{\hat Z}}\Spf R$. Thus $\tilde f\wh\otimes\id_R\colon\wt\CY_R\to\wh\Flag^{(n)}_{G,R}:=\wh\Flag_{G,R}\whtimes_{\wh\Flag_{\SL_r,R}}\wh\Flag^{(n)}_{\SL_r,R}$ and $\hat Z_R$ is a closed formal subscheme of $\wh\Flag^{(n)}_{G,R}$, defined by a sheaf of ideals $\Fa$ on $\wh\Flag^{(n)}_{G,R}$. We must show that $(\tilde f\wh\otimes\id_R)^*\Fa=(0)$. The associated morphism of $\Quot(R)$-analytic spaces $(\tilde f\wh\otimes\id_R)^\an\colon(\wt\CY_R)^\an\to(\wh\Flag^{(n)}_{G,R})^\an$ is given by $\tau_{\CG''}^{-1}=(\s u'')^{-1}\circ\tau_\CG^{-1}\circ u''$ and factors through $\hat Z_R^\an$, because $\s\alpha\circ\tau_\CG^{-1}\circ\alpha^{-1}\in\hat Z_R^\an(\wt\CY^\an)$, as well as $\alpha\circ u''\!,\s(\alpha\circ u'')^{-1}\in G(\CO_{\wt\CY^\an}\dbl z-\zeta\dbr)$, and $\hat Z_R^\an$ is invariant under multiplication with $G(\CO_{\wt\CY^\an}\dbl z-\zeta\dbr)$ on the left. This implies $(\tilde f\wh\otimes\id_R)^{\an*}\Fa=(0)$ on $(\wt\CY_R)^\an$ and since $\CO_{\wt\CY_R}\subset\CO_{(\wt\CY_R)^\an}$, we obtain $(\tilde f\wh\otimes\id_R)^*\Fa=(0)$ on $\wt\CY_R$. Therefore the morphism $\wt\CY_R\to\wh\Flag^{(n)}_{G,R}$ given by $\tau_{\CG''}^{-1}$ factors through the closed formal subscheme $\hat Z_R$. By Definition~\ref{DefBDLocal}\ref{DefBDLocal_C} and Remark~\ref{RemBdLocal}(c) this means that $\tau_{\CG''}^{-1}$ is bounded by $\hat Z$ and $\ul\CG''$ is bounded by $\hat{Z}^{-1}$.
\end{proof}

To define ${\breve\CM}^K$ for all compact open subgroups $K\subset G\bigl(\BF_q\dpl z\dpr\bigr)$ we proceed slightly differently than Rapoport and Zink~\cite[5.34]{RZ} and instead use the following definition and corollary. For a comparison with \cite[5.34]{RZ} see Remark~\ref{remdefhecke}.

\begin{definition}\label{DefRatLevel}
Let $X$ be a connected affinoid strictly $L$-analytic space with geometric base point $\bar x$, and let $K\subset G\bigl(\BF_q\dpl z\dpr\bigr)$ be a compact open subgroup. Consider the category of triples $(\ul\CG,\bar\delta,\eta K)$ where $(\ul\CG,\bar\delta)\in\breveRZ(\CX)$ for an admissible formal model $\CX$ of $X$ and $\eta K\in\Triv_{\ul\CG,\bar x}\bigl(\BF_q\dpl z\dpr\bigr)/K$ is a rational $K$-level structure on $\ul\CG$ over $X$. Morphisms between two such triples $(\ul\CG_1,\bar\delta_1,\eta_1 K)$ and $(\ul\CG_2,\bar\delta_2,\eta_2 K)$ over admissible formal models $\CX_1$, respectively $\CX_2$, are quasi-isogenies $u\colon\ul\CG_1\to\ul\CG_2$ as in Remark~\ref{RemUnivLocGSht} over a model $\wt\CX$ dominating both $\CX_i$ with $\bar\delta_2\circ (u\mod\zeta) =\bar\delta_1$ such that $\check V_{u,\bar x}\circ\eta_1 K=\eta_2 K$. In particular all morphisms are isomorphisms and by rigidity of quasi-isogenies \cite[Proposition~2.11]{AH_Local} all Hom sets contain at most one element.
\end{definition}

For $K\subset G(\BF_q\dbl z\dbr)$ we consider over ${\breve\CM}^K$ the triple $(\ul\CG^\univ,\bar\delta^\univ,\eta^\univ K)$, where $\eta^\univ K$ is the rational $K$-level structure on $\ul\CG^\univ$ induced from the universal integral $K$-level structure on $\ul\CG^\univ$ via the inclusion $\Triv_{\ul\CG,\bar x}(\BF_q\dbl z\dbr)\to\Triv_{\ul\CG,\bar x}\bigl(\BF_q\dpl z\dpr\bigr)$.

\begin{corollary}\label{CorLevel}
Let ${\breve\CM}=(\breveRZ)^\an$ and for every open subgroup $K\subset G(\BF_q\dbl z\dbr)$ let ${\breve\CM}^K$ be the finite \'etale covering space from Definition~\ref{DefLevel} parametrizing \emph{integral} $K$-level structures on the universal local $G$-shtuka $\ul\CG^\univ$ over ${\breve\CM}$ from Remark~\ref{RemUnivLocGSht}. Then ${\breve\CM}^K$ also parametrizes isomorphism classes of triples $(\ul\CG,\bar\delta,\eta K)$ in the sense of Definition~\ref{DefRatLevel}.
\end{corollary}

\begin{proof}
Let $X$ be a connected affinoid strictly $L$-analytic space with geometric base point $\bar x$, and let $(\ul\CG,\bar\delta,\eta K)$ be a triple over $X$, where $(\ul\CG,\bar\delta)\in\breveRZ(\CX)$ for an admissible formal model $\CX$ of $X$, and $\eta K$ is a rational $K$-level structure on $\ul\CG$ over $X$. We may assume that $\CX=\Spf B$ is affine. By Proposition~\ref{PropLiftOfIsog} there is an admissible formal blowing-up $\CY\to\CX$ and a $(\ul\CG''\!,\bar\delta'')\in\breveRZ(\CY)$ with $\breve\pi(\ul\CG,\bar\delta)=\breve\pi(\ul\CG''\!,\bar\delta'')$ and an isogeny $u''\colon\ul\CG\to\ul\CG''$ over $\CY$ lifting $(\bar\delta'')^{-1}\circ\bar\delta$ such that $\check V_{u''\!,\bar x}\circ\eta\in\Triv_{\ul\CG''\!,\bar x}(\BF_q\dbl z\dbr)$. The pair $(\ul\CG''\!,\bar\delta'')$ induces a morphism of $\breve E$-analytic spaces $X=\CY^\an\to{\breve\CM}$ which is independent of the admissible formal blowing-up $\CY$. By Proposition~\ref{PropCM^K}\ref{PropCM^K_A} the integral $K$-level structure $(\check V_{u''\!,\bar x}\circ\eta)K$ on $\ul\CG''$ defines a uniquely determined ${\breve\CM}$-morphism $f\colon X\to{\breve\CM}^K$ such that $(f^*\ul\CG^\univ,f^*\bar\delta^\univ,f^*\eta^\univ K)=\bigl(\ul\CG''\!,\bar\delta''\!,(\check V_{u''\!,\bar x}\circ\eta)K\bigr)\cong(\ul\CG,\bar\delta,\eta K)$.
\end{proof}

\begin{definition}\label{deflevelgen}
Let $K\subset G\bigl(\BF_q\dpl z\dpr\bigr)$ be a compact open subgroup. Let $K'\subset K$ be a normal subgroup of finite index with $K'\subset G(\BF_q\dbl z\dbr)$. (Such a subgroup exists as $K\cap  G(\BF_q\dbl z\dbr)\subset K$ is of finite index due to the openness of $G(\BF_q\dbl z\dbr)$ and the compactness of $K$.) Then we define ${\breve\CM}^K$ as the $\breve E$-analytic space that is the quotient of ${\breve\CM}^{K'}$ by the finite group $K/K'$; see \cite[Proposition~2.1.14(ii)]{Berkovich1} and \cite[Lemma~4]{BerkovichAutom}. Here $gK'\in K/K'$ acts on ${\breve\CM}^{K'}$ by sending the universal triple $(\ul\CG^\univ,\bar\delta^\univ,\eta^\univ K')$ over ${\breve\CM}^{K'}$ from Corollary~\ref{CorLevel} to the triple $(\ul\CG^\univ,\bar\delta^\univ,\eta^\univ gK')$. By Remark~\ref{remdefhecke} this means that $gK'\in K/K'$ acts on ${\breve\CM}^{K'}$ as the Hecke correspondence $\iota(g)_{K'}$. In particular $\breve\CM^{K_0}=(\breveRZ)^\an$ for $K_0=G(\BF_q\dbl z\dbr)$. We denote by $(\ul\CG^\univ,\bar\delta^\univ,\eta^\univ K)$ the triple over ${\breve\CM}^K$ induced by the universal triple $(\ul\CG^\univ,\bar\delta^\univ,\eta^\univ K')$ over ${\breve\CM}^{K'}$. It is universal by the following corollary.
\end{definition}

By Proposition~\ref{PropCM^K}\ref{PropCM^K_C} the definition of ${\breve\CM}^K$ is independent of the normal subgroup $K'\subset K$ and Proposition~\ref{PropCM^K}\ref{PropCM^K_C} continues to hold in this more general setting. We will see in Theorem~\ref{MainThm} that ${\breve\CM}^K$ always is a strictly $\breve E$-analytic space.

\begin{corollary}\label{CorLevel2}
If $K\subset G\bigl(\BF_q\dpl z\dpr\bigr)$ is any compact open subgroup then ${\breve\CM}^K$ is an $\breve E$-analytic space that is separated and partially proper over $\breve E$ and parametrizes isomorphism classes of triples $(\ul\CG,\bar\delta,\eta K)$ in the sense of Definition~\ref{DefRatLevel}.
\end{corollary}

\begin{proof}
That ${\breve\CM}^K$ is separated and partially proper over $\breve E$ follows from Lemma~\ref{LemmaPartProper} and \cite[Lemma~1.10.17 (iv), (vii)]{Huber96}. Let $X$ be a connected affinoid strictly $L$-analytic space with geometric base point $\bar x$, and let $(\ul\CG,\bar\delta,\eta K)$ be a triple over $X$ where $(\ul\CG,\bar\delta)\in\breveRZ(\CX)$ for an admissible formal model $\CX$ of $X$ and where $\eta K\in\Triv_{\ul\CG,\bar x}\bigl(\BF_q\dpl z\dpr\bigr)/K$ is a rational $K$-level structure on $\ul\CG$ over $X$. Let $K'\subset K$ be a normal subgroup of finite index with $K'\subset G(\BF_q\dbl z\dbr)$. Consider the \'etale covering space $X'\to X$ corresponding to the $\pi_1^\et(X,\bar x)$-set $\{\eta'K'\in\Triv_{\ul\CG,\bar x}\bigl(\BF_q\dpl z\dpr\bigr)/K'\colon \eta'K=\eta K\}$ that is isomorphic to $K/K'$ under the maps $\eta'K'\mapsto \eta^{-1}\eta'K'$ and $\eta gK'\leftmapsto gK'$. In particular $X'\to X$ is a $K/K'$-torsor. By Corollary~\ref{CorLevel} there is a uniquely determined $\breve E$-morphism $X'\to{\breve\CM}^{K'}$ that is equivariant for the action of $K/K'$ and therefore descends to a uniquely determined $\breve E$-morphism $X\to{\breve\CM}^K$.
\end{proof}

\begin{remark}\label{remdefhecke}
We have an action of $G\bigl(\BF_q\dpl z\dpr\bigr)$ on the tower $({\breve\CM}^K)_{K\subset G(\BF_q\dpl z\dpr)}$ by Hecke correspondences defined as follows. Let $g\in G\bigl(\BF_q\dpl z\dpr\bigr)$ and let $K$ be a compact open subgroup of $G\bigl(\BF_q\dpl z\dpr\bigr)$. Then $g$ induces an isomorphism 
\begin{equation}\label{EqHeckeCor2}
\iota(g)_K\colon{\breve\CM}^K\;\isoto\;{\breve\CM}^{g^{-1}Kg}
\end{equation}
by sending the universal triple $(\ul\CG^\univ,\bar\delta^\univ,\eta^\univ K)$ over ${\breve\CM}^K$ from Definition~\ref{deflevelgen} to the triple $\bigl(\ul\CG^\univ,\bar\delta^\univ,\eta^\univ Kg=\eta^\univ g(g^{-1}Kg)\bigr)$. The morphisms $\iota(g)$ are compatible with the group structure on $G\bigl(\BF_q\dpl z\dpr\bigr)$, that is they satisfy $\iota(gh)_K=\iota(g)_{h^{-1}Kh}\circ\iota(h)_K$ for all $g,h\in G\bigl(\BF_q\dpl z\dpr\bigr)$ and all $K$. They are also compatible with the projection maps $\breve\pi_{K,K'}$, that is $\iota(g)_K\circ\breve\pi_{K,K'}=\breve\pi_{g^{-1}Kg,\,g^{-1}K'g}\circ\iota(g)_{K'}$.

If both $K$ and $g^{-1}Kg$ are contained in $G(\BF_q\dbl z\dbr)$, we can translate the definition of $\iota(g)_K$ in terms of \emph{integral} $K$-level structures by inspecting the proof of Corollary~\ref{CorLevel}. Namely we start with the universal integral $K$-level structure $\eta^\univ K\in\Triv_{\ul\CG^\univ,\bar x}(\BF_q\dbl z\dbr)/K$ on $\ul\CG^\univ$ over ${\breve\CM}^K$. The rational $g^{-1}Kg$-level structure $\eta^\univ Kg=\eta^\univ g(g^{-1}Kg)$ yields a pair $(\ul\CG''\!,\bar\delta'')\in\breveRZ(\CY)$ for an admissible formal blowing-up $\CY\to\CX$ with $\breve\pi(\ul\CG^\univ,\bar\delta^\univ)=\breve\pi(\ul\CG''\!,\bar\delta'')$ and $\check V_{u''\!,\bar x}\circ\eta^\univ g\in\Triv_{\ul\CG''\!,\bar x}(\BF_q\dbl z\dbr)$, where $u''$ is the unique lift of $(\bar\delta'')^{-1}\circ\bar\delta^\univ$. Then the morphism $\iota(g)_K$ is given by
\begin{equation}\label{EqHeckeExplicit}
\iota(g)_K\colon (\ul\CG^\univ,\bar\delta^\univ,\eta^\univ K)\;\longmapsto\;(\ul\CG''\!,\bar\delta''\!,(\check V_{u''\!,\bar x}\circ\eta^\univ g)g^{-1}Kg)
\end{equation}
This shows that the definition of $\iota(g)_K$ in this case, and therefore the definitions of ${\breve\CM}^K$, $\breve\pi_{K,K'}$ and $\iota(g)_K$ for all $K$ coincide with the definitions analogous to \cite[5.34]{RZ}. 

Although it is not explicitly stated in \cite{RZ}, the analog of Corollary~\ref{CorLevel2} also holds in their setup, as it can be deduced from the existence of the $\iota(g)_K$ as follows. If $(\ul\CG,\bar\delta)\in\breveRZ(\CX)$ and $\eta K$ with $\eta\in\Triv_{\ul\CG,\bar x}\bigl(\BF_q\dpl z\dpr\bigr)$ is a rational $K$-level structure on $\ul\CG$ over $X=\CX^\an$, we can choose an element $\eta_0\in\Triv_{\ul\CG,\bar x}(\BF_q\dbl z\dbr)$ and set $g:=\eta_0^{-1}\eta\in\Aut^\otimes(\omega\open)=G\bigl(\BF_q\dpl z\dpr\bigr)$. Then $\eta K g^{-1}=\eta_0(gKg^{-1})$ is an integral $gKg^{-1}$-level structure on $\ul\CG$ and defines a uniquely determined $\breve E$-morphism $X\to{\breve\CM}^{gKg^{-1}}$. Composing with $\iota(g)_{gKg^{-1}}$ produces the desired $\breve E$-morphism $X\to{\breve\CM}^K$.
\end{remark}

\begin{proposition}\label{PropPiK}
The period morphism $\breve\pi$ induces compatible morphisms $\breve\pi_K\colon{\breve\CM}^K\to\breve\CH_{G,\hat Z}^\an$ for all compact open subgroups $K\subset G\bigl(\BF_q\dpl z\dpr\bigr)$. In terms of Corollary~\ref{CorLevel2} it has the form $\breve\pi_K\colon(\ul\CG^\univ,\bar\delta^\univ,\eta^\univ K)\mapsto\breve\pi(\ul\CG^\univ,\bar\delta^\univ)$ where $\eta^\univ K$ is the universal (integral or) rational $K$-level structure on $\ul\CG^\univ$. These morphisms commute with the Hecke correspondences in the sense that $\breve\pi_{g^{-1}Kg}\circ \iota(g)_K=\breve\pi_K$ for all $g$ and $K$.
\end{proposition}

\begin{proof}
To construct $\breve\pi_K$ we choose a normal subgroup $K'\subset K$ of finite index with $K'\subset G(\BF_q\dbl z\dbr)$. We let $\breve\pi_{K'}:=\breve\pi\circ\breve\pi_{G(\BF_q\dbl z\dbr),K'}\colon{\breve\CM}^{K'}\to\breve\CH_{G,\hat Z}^\an$. It has the given form. If $g\in G(\BF_q\dpl z\dpr)$ satisfies $g^{-1}K'g\subset G(\BF_q\dbl z\dbr)$, we see from the description of $\iota(g)_{K'}$ in \eqref{EqHeckeExplicit} that $\breve\pi_{g^{-1}K'g}\circ \iota(g)_{K'}=\breve\pi_{K'}$. In particular, if $g\in K$ then $\breve\pi_{K'}$ is $K/K'$-invariant. Therefore $\breve\pi_{K'}$ descends to a morphism $\breve\pi_K\colon{\breve\CM}^K\to\breve\CH_{G,\hat Z}^\an$ which has the given form. By Proposition~\ref{PropCM^K}\ref{PropCM^K_C} the definition of $\breve\pi_K$ is independent of the chosen $K'$ and satisfies $\breve\pi_K\circ\breve\pi_{K,\wt K}=\breve\pi_{\wt K}$ for all compact open subgroups $\wt K\subset K\subset G\bigl(\BF_q\dpl z\dpr\bigr)$. From this also $\breve\pi_{g^{-1}Kg}\circ \iota(g)_K=\breve\pi_K$ follows.
\end{proof}

The following result is the analog of \cite[Proposition 5.37]{RZ} and can be proved in the same way. However, we give a different proof using Corollary~\ref{CorLevel2}.

\begin{proposition}\label{PropFibersOfPiK}
Let $K_1,K_2\subset G\bigl(\BF_q\dpl z\dpr\bigr)$ be compact open subgroups and let $\Omega$ be an algebraically closed complete extension of $\breve E$. Then two points $x_1\in {\breve\CM}^{K_1}(\Omega)$ and $x_2\in {\breve\CM}^{K_2}(\Omega)$ satisfy $\breve\pi_{K_1}(x_1)=\breve\pi_{K_2}(x_2)$ if and only if they are mapped to each other under a Hecke correspondence, i.e. if there is a $g\in G\bigl(\BF_q\dpl z\dpr\bigr)$ and a $y\in{\breve\CM}^{K_1\cap gK_2g^{-1}}(\Omega)$ with $\breve\pi_{K_1,\,K_1\cap gK_2g^{-1}}(y)=x_1$ and $\breve\pi_{K_2,\,g^{-1}K_1g\cap K_2}\circ\iota(g)_{K_1\cap gK_2g^{-1}}(y)=x_2$. In particular, the geometric fibers of $\breve\pi_K$ are (non-canonically) isomorphic to $G\bigl(\BF_q\dpl z\dpr\bigr)/K$.
\end{proposition}

\begin{proof}
Since one direction was proved in Proposition~\ref{PropPiK}, we now assume $\breve\pi_{K_1}(x_1)=\breve\pi_{K_2}(x_2)$. In terms of Corollary~\ref{CorLevel2} let $x_i$ correspond to the triple $(\ul\CG_i,\bar\delta_i,\eta_iK_i)$. Since $\Omega$ is algebraically closed we may choose representatives $\eta_i\in\Triv_{\ul\CG_i,\SpBerk(\Omega)}\bigl(\BF_q\dpl z\dpr\bigr)$ of $\eta_iK_i$. By the description of $\breve\pi_K$ in Proposition~\ref{PropPiK} we have $\breve\pi(\ul\CG_1,\bar\delta_1)=\breve\pi(\ul\CG_2,\bar\delta_2)$ and so Proposition~\ref{PropLiftOfIsog} yields a quasi-isogeny $u\colon\ul\CG_2\to\ul\CG_1$ over $\Spec\CO_\Omega$ with $\bar\delta_1\circ (u\mod\zeta) =\bar\delta_2$. We set $\eta'_1:=\check V_{u,\SpBerk(\Omega)}\circ\eta_2\in\Triv_{\ul\CG_1,\SpBerk(\Omega)}\bigl(\BF_q\dpl z\dpr\bigr)$. Then $x_2=(\ul\CG_2,\bar\delta_2,\eta_2K_2)\cong(\ul\CG_1,\bar\delta_1,\eta'_1K_2)$. Therefore $g:=\eta_1^{-1}\eta_1'\in\Aut^\otimes(\omega\open)=G\bigl(\BF_q\dpl z\dpr\bigr)$ and $y=(\ul\CG_1,\bar\delta_1,\eta_1(K_1\cap gK_2g^{-1})\bigr)\in{\breve\CM}^{K_1\cap gK_2g^{-1}}(\Omega)$ solve the problem.

Thus after the choice of the representative $\eta_1$ of $\eta_1K_1$ the bijection between $G\bigl(\BF_q\dpl z\dpr\bigr)/K_1$ and the fiber of $\breve\pi_{K_1}$ over $\breve\pi_{K_1}(x_1)$ is given by $gK_1\mapsto(\ul\CG_1,\bar\delta_1,\eta_1gK_1)$ with inverse $\eta_1^{-1}\eta'_1K_1\leftmapsto(\ul\CG_1,\bar\delta_1,\eta'_1K_1)$.
\end{proof}

\begin{remark}\label{RemJActsOnMK}
Finally, the action of $j\in J_b\bigl(\BF_q\dpl z\dpr\bigr)$ on $\breveRZ$ from Definition~\ref{DefRZSpace} induces actions on each of the spaces ${\breve\CM}^K$ individually by 
\[
j\colon{\breve\CM}^K\;\longto\;{\breve\CM}^K,\quad(\ul\CG,\bar\delta,\eta K)\;\longmapsto\;(\ul\CG,j\circ\bar\delta,\eta K)\,. 
\]
In other words the pullback $j^*(\ul\CG^\univ,\bar\delta^\univ,\eta^\univ K)$ of the universal triple over $\breve\CM^K$ is isomorphic to $(\ul\CG^\univ,j\circ\bar\delta^\univ,\eta^\univ K)$ in the category of triples from Definition~\ref{DefRatLevel}. Using the universal property of $\breveRZ$ the condition $j^*(\ul\CG^\univ,\bar\delta^\univ)\cong(\ul\CG^\univ,j\circ\bar\delta^\univ)$ shows that there is an isomorphism $\Phi_j\colon j^*\ul\CG^\univ\isoto \ul\CG^\univ$ of local $G$-shtukas with $j^\ast\bar\delta^\univ=j\circ\bar\delta^\univ\circ(\Phi_j\mod\zeta)$. By rigidity of quasi-isogenies \cite[Proposition~2.11]{AH_Local} the isomorphism $\Phi_j$ is uniquely determined and a straight forward calculation shows that $\Phi_j\circ j^\ast\Phi_{j'}=\Phi_{j'j}$. It follows from Definition~\ref{DefRatLevel} that $j^*\eta^\univ K=(\check V_{\Phi_j,\bar x}^{-1}\circ\eta^\univ) K$. The $J_b\bigl(\BF_q\dpl z\dpr\bigr)$-action on $\breve\CM^K$ is compatible with the projection maps $\breve\pi_{K,K'}$ and the Hecke action.
\end{remark}

\begin{lemma}\label{lemjcont}
The action of $J:=J_b\bigl(\BF_q\dpl z\dpr\bigr)$ on $\breveRZ$ and the induced actions on each ${\breve\CM}^K$ are continuous.
\end{lemma}
\begin{proof}
For the first assertion we have to show the following claim. Let $S\in \Nilp_{\BF_q\dbl\xi\dbr}$ be a quasi-compact scheme and $(\ul\CG,\bar\delta)\in \breveRZ(S)$. We must show that there is a neighborhood $U$ of $1\in J$ such that for $j\in U$ we have $j(\ul\CG,\bar\delta)\cong(\ul\CG,\bar\delta)$, that is, $\ol{\delta}^{-1}\circ j\circ \ol \delta$ lifts to an automorphism of $\ul\CG$ over $S$. 

To prove this claim let $\delta\colon\ul\CG\to\ul\BG_{0,S}$ be the quasi-isogeny that lifts $\bar\delta$ by rigidity of quasi-isogenies \cite[Proposition~2.11]{AH_Local}. Let $S'\to S$ be an \'etale covering that trivializes $\ul\CG\cong\bigl((L^+G)_{S'},A\s\bigr)$. Since $S$ is quasi-compact there is a refinement of this covering such that $S'=\Spec B$ is affine. Then $\delta$ corresponds to an element $\Delta\in LG(B)$. We fix a faithful representation $\Darst\colon G\into\SL_r$ and consider the elements $\Darst(\Delta)$ and $\Darst(\Delta^{-1})$ of $L\SL_r(B)=\SL_r(B\dbl z\dbr[\tfrac{1}{z}])$. Let $N\in\BN_0$ be such that the matrices $z^N\cdot\Darst(\Delta)$ and $z^N\cdot\Darst(\Delta^{-1})$ have their entries in $B\dbl z\dbr$. If $\Darst(j)\equiv 1\mod z^{2N}$ then $\Darst(\Delta^{-1}\cdot j\cdot\Delta)\in\SL_r(B\dbl z\dbr)$, and hence $\Delta^{-1}\cdot j\cdot\Delta\in L^+G(B)$ and $\delta^{-1}\circ j\circ\delta\in\Aut(\ul\CG)$. Since the map $J\subset LG(\BaseFld)\xrightarrow{\;\Darst\,}L\SL_r(\BaseFld)$ is continuous with respect to the $z$-adic topology, the claim follows.

Because the action on $\breveRZ$ is continuous, the same holds for each ${\breve\CM}^K$ by \cite[Lemma~8.4]{Berkovich3}. Note that continuity of the $J$-action is defined as continuity of the homomorphism $J\to \scrG({\breve\CM}^K)$, where $\scrG({\breve\CM}^K)$ is the topological automorphism group defined by Berkovich~\cite[\S\,6]{Berkovich3}. The topology of $\scrG({\breve\CM}^K)$ is defined via compact subspaces of ${\breve\CM}^K$. Therefore the continuity of the $J$-action on $\breveRZ$ we proved above and \cite[Lemma~8.4]{Berkovich3} are applicable although $\breveRZ$ is not quasi-compact.
\end{proof}

\begin{remark}\label{RemFunctoriality3}
We keep the notation from Remarks~\ref{RemFunctoriality1} and \ref{RemFunctoriality2} and let $\epsilon\colon G\to G'$ be a morphism of parahoric group schemes over $\BF_q\dbl z\dbr$. If $\epsilon(\hat{Z})\subset\hat{Z}'$ and $\ul\BG'_0=\epsilon_*\ul\BG_0\cong\bigl((L^+G')_\BaseFld,b'\s\bigr)$ where $\ul\BG_0\cong\bigl((L^+G)_\BaseFld,b\s\bigr)$ and $b'=\epsilon(b)$, then there is a commutative diagram of $\breve E$-analytic spaces
\[
\xymatrix @C+2pc {
\breveRZ\ar[r]^-{\TS\epsilon_*} \ar[d]_{\TS\breve\pi} & {\breve\CM}_{\ul\BG'_0}^{\raisebox{-0.7ex}{$\SC\hat{Z}'{}^{-1}$}} \ar[d]^{\TS\breve\pi'} \\
\breve\CH_{G,\hat Z,b}^{a} \ar[r]^-{\TS\epsilon} & \breve\CH_{G'\!,\hat Z'\!,b'}^{a},
}
\]
where $\breve\pi$ and $\breve\pi'$ are the period morphisms from Definition~\ref{DefPeriodMorph} for $\ul\BG_0$ and $\hat{Z}$, respectively for $\ul\BG'_0$ and $\hat{Z}'$, and where the horizontal morphisms $\epsilon_*$ and $\epsilon$ were described in \eqref{EqFunctRZ} and \eqref{EqEpsilon^*}. This diagram is equivariant for the action of $J^G_b$ that acts on the right column via the morphism $J_b^G\to J_{b'}^{G'}$ from \eqref{EqFunctJ}. The diagram is indeed commutative, because in the notation of Lemma~\ref{LemmaTrivializing} and equation~\eqref{EqPeriodMorph}, the left morphism $\breve\pi$ sends $\ul\CG_{\scrS'}\cong\bigl((L^+G)_{\scrS'},A\s\bigr)$ to $\gamma:=\sigma^*(\Delta) A^{-1}=b^{-1}\Delta$ and the right morphism $\breve\pi'$ sends $\epsilon_*(\ul\CG)_{\scrS'}\cong\bigl((L^+G')_{\scrS'},\epsilon(A)\s\bigr)$ to $\gamma':=\sigma^*(\epsilon(\Delta)) \epsilon(A)^{-1}=\epsilon(b)^{-1}\epsilon(\Delta)=\epsilon(\gamma)$.

The tensor functors $\ulCV_b$ and $\ulCV_{b'}$ from Theorem~\ref{ThmLocSyst} are part of a commutative diagram of tensor functors
\[
\xymatrix @C+2pc {
\Rep_{\BF_q\dpl z\dpr} G \ar[r]^-{\TS\ulCV_b} & \PLoc_{\breve\CH_{G,\hat Z,b}^{a}} \\
\Rep_{\BF_q\dpl z\dpr} G' \ar[r]^-{\TS\ulCV_{b'}} \ar[u]^{\TS\epsilon^*} & \PLoc_{\breve\CH_{G'\!,\hat Z'\!,b'}^{a}} \ar[u]_{\TS\epsilon^*}
}
\]
where the right vertical arrow denotes pullback of local systems of $\BF_q\dpl z\dpr$-vector spaces along the morphism $\epsilon\colon\breve\CH_{G,\hat Z,b}^{a}\to\breve\CH_{G'\!,\hat Z'\!,b'}^{a}$ from \eqref{EqEpsilon^*}, and the left vertical functor sends $(V',\Darst')$ to $(V',\Darst'\circ\epsilon)$, see Remark~\ref{RemFunctoriality2}.

Moreover, the Tate module functor $\check{T}_{\ul\CG}$ from \eqref{EqTateFunctor} satisfies $\check{T}_{\epsilon_*\ul\CG}(\Darst')=\check{T}_{\ul\CG}(\Darst'\circ\epsilon)$, that is, $\check{T}_{\epsilon_*\ul\CG}=\check{T}_{\ul\CG}\circ\epsilon^*$, and likewise $\check{V}_{\epsilon_*\ul\CG}=\check{V}_{\ul\CG}\circ\epsilon^*$. This defines a map of the sets from \eqref{EqTensorIsom} 
\[
\epsilon^*\colon\Triv_{\ul\CG,\bar x}(A)\;\longto\;\Triv_{\epsilon_*\ul\CG,\bar x}(A)\,,\quad\eta\;\longmapsto\;\epsilon^*(\eta)
\]
for both arguments $A=\BF_q\dbl z\dbr$ and $A=\BF_q\dpl z\dpr$. In particular, if $K\subset G(A)$ and $K'\subset G'(A)$ are compact open subgroups with $\epsilon(K)\subset K'$ then the map $\epsilon_*\colon\eta K\mapsto\epsilon_*(\eta)K'$ sends $K$-level structures on $\ul\CG$ to $K'$-level structures on $\epsilon_*\ul\CG$. This yields a commutative diagram of $\breve E$-analytic spaces
\[
\xymatrix @C+2pc {
\breve\CM^K_G\ar[r]^-{\TS\epsilon_*} \ar[d]_{\TS\breve\pi_K} & \breve\CM^{K'}_{G'} \ar[d]^{\TS\breve\pi'_{K'}} \\
\breve\CH_{G,\hat Z,b}^{a} \ar[r]^-{\TS\epsilon} & \breve\CH_{G'\!,\hat Z'\!,b'}^{a},
}
\]
which is again $J^G_b$-equivariant. It is further equivariant for the action of $g\in G\bigl(\BF_q\dpl z\dpr\bigr)$ that acts on the left column by the Hecke correspondence $\iota(g)_K$ from Remark~\ref{remdefhecke} and on the right column  by the Hecke correspondence $\iota\bigl(\epsilon(g)\bigr)_{K'}$\,.
\end{remark}

\section{The image of the period morphism} \label{SectImage}
\setcounter{equation}{0}

In this section we fix a local $G$-shtuka $\ul\BG_0=\bigl((L^+G)_\BaseFld,b\s\bigr)$ over $\BaseFld$ and a bound $\hat{Z}$ with reflex ring $R_{\hat Z}=\kappa\dbl\xi\dbr$. We set $E_{\hat Z}=\kappa\dpl\xi\dpr$ and $\breve E:=\BaseFld\dpl\xi\dpr$.

We will determine the image of the period morphism $\breve\pi$ from Definition~\ref{DefPeriodMorph}. This is the function field analog of \cite[Theorem~8.4]{HartlRZ}, where the situation of PEL-Rapoport-Zink spaces for $p$-divisible groups was treated. Note however, that our proof here is entirely different from loc.\ cit., because we here already constructed a tensor functor to the category of local systems in Theorem~\ref{ThmLocSyst}, and will use this to determine the image of $\breve\pi$. In loc.\ cit.\ the proof proceeds in the opposite direction and first determines the image of the period morphism and then constructs the local systems.

\begin{theorem}\label{MainThm}
\begin{enumerate}
\item \label{MainThm_1}
The image $\breve\pi\bigl((\breveRZ)^\an\bigr)$ of the period morphism $\breve\pi$ equals the union of the connected components of $\breve\CH_{G,\hat{Z},b}^{a}$ on which there is an $\BF_q\dpl z\dpr$-rational isomorphism $\beta\colon\omega\open\isoto\omega_{b,\bar\gamma}$ between the tensor functors $\omega\open$ and $\omega_{b,\bar\gamma}$ from Remark~\ref{RemLocSyst}. 
\item  \label{MainThm_2}
The rational dual Tate module $\check V_{\ul\CG}$ of the universal local $G$-shtuka $\ul\CG$ over $(\breveRZ)^\an$ descends to a tensor functor $\check V_{\ul\CG}$ from $\Rep_{\BF_q\dpl z\dpr}G$ to the category of local systems of $\BF_q\dpl z\dpr$-vector spaces on $\breve\pi\bigl((\breveRZ)^\an\bigr)$. It carries a canonical $J_b\bigl(\BF_q\dpl z\dpr\bigr)$-linearization and is canonically $J_b\bigl(\BF_q\dpl z\dpr\bigr)$-equivariantly isomorphic to the tensor functor $\ulCV_b$ from Theorem~\ref{ThmLocSyst}.
\item \label{MainThm_3}
The tower of strictly $\breve E$-analytic spaces $({\breve\CM}^K)_{K\subset G(\BF_q\dpl z\dpr)}$ is canonically isomorphic over $\breve\pi\bigl((\breveRZ)^\an\bigr)$ in a Hecke and $J_b\bigl(\BF_q\dpl z\dpr\bigr)$-equivariant way to the tower of \'etale covering spaces $(\HeckeTower_K)_{K\subset G(\BF_q\dpl z\dpr)}$ of $\breve\pi\bigl((\breveRZ)^\an\bigr)$ associated with the tensor functor $\ulCV_b$ in Remark~\ref{RemTowerLocSyst}.
\end{enumerate}
\end{theorem}

\begin{remark}\label{RemWintenberger}
(a) Note that if the analog of Wintenberger's theorem \eqref{EqWintenberger} is established, the union of connected components in \ref{MainThm_1} is simply $\breve\CH_{G,\hat{Z},b}^{na}$. In particular, this would imply that $\breve\CM^K\ne\emptyset$ if and only if $[b]\in B(G,\hat{Z}_E)$. This is the analog of \cite[Conjecture~4.21]{RV}.

\medskip\noindent
(b) By Theorem~\ref{MainThm}\ref{MainThm_3} the tower $({\breve\CM}^K)_{K\subset G(\BF_q\dpl z\dpr)}$ only depends on the triple $(G_{\BF_q\dpl z\dpr},[b],\hat{Z}_E)$, where $[b]$ is the $\sigma$-conjugacy class of $b$ under $LG(\BaseFld)$. Indeed, the tower $(\HeckeTower_K)_{K\subset G(\BF_q\dpl z\dpr)}$ only depends on this triple; see Definition~\ref{DefPeriodSp} and Remark~\ref{RemTowerLocSyst}. In the arithmetic situation the analogs of the tower $({\breve\CM}^K)_{K\subset G(\BF_q\dpl z\dpr)}$ are called \emph{local Shimura varieties}; see \cite[\S\,5.2]{RV} and \cite[\S\,24]{ScholzeBerkeley}. For the same reason they also only depend on the group scheme over $\BQ_p$, the bounding cocharacter $\mu$ and the $\sigma$-conjugacy class $[b]$; see \cite[Proposition~23.2.1 and thereafter]{ScholzeBerkeley}.
\end{remark}

To prove Theorem~\ref{MainThm} we will take a Tannakian approach and make use of the following proposition.

\begin{proposition} \label{PropTannaka}
Let $G$ be a faithfully flat affine group scheme over a Dedekind domain $A$ and let $B$ be an $A$-algebra. Let $\Rep_AG$ be the Tannakian category of representations of $G$ on finite projective $A$-modules and let $\omega\open\colon\Rep_AG\to\FMod_A$ be the forgetful fiber functor. Let $\omega\colon\Rep_AG\to\FMod_B$ be a tensor functor to the category $\FMod_B$ of finite projective $B$-modules that sends morphisms in $\Rep_AG$ that are epimorphisms on the underlying $A$-modules to epimorphisms in $\FMod_B$. Then $\Isom^\otimes(\omega\open\otimes_AB,\omega)$ is representable by a $G$-torsor over $B$ (for the \fpqc{} topology).
\end{proposition}

\begin{proof}
Let $BG=[\Spec A/G]$ be the classifying stack of $G$ that parametrizes $G$-torsors over $A$-schemes. By \cite[Tags \href{http://stacks.math.columbia.edu/tag/0443}{0443} and \href{http://stacks.math.columbia.edu/tag/06WS}{06WS}]{StacksProject} the category of linear representations of $G$ on (arbitrary) $A$-modules is tensor equivalent to the category $\Mod_{BG}$ of quasi-coherent sheaves on (the big \fppf{}-site of) $BG$; see \cite[Tags \href{http://stacks.math.columbia.edu/tag/06NT}{06NT} and \href{http://stacks.math.columbia.edu/tag/03DL}{03DL}]{StacksProject}. This equivalence is given by the functor that sends a quasi-coherent sheaf $\CF$ on $BG$ to its pullback under the morphism $p_0\colon\Spec A\to BG$. The quasi-inverse functor sends a representation of $G$ on an $A$-module $F$ to its faithfully flat descent on $BG$. By faithfully flat descent \cite[IV$_2$, Proposition~2.5.2]{EGA}, $\CF$ is finite locally free if and only if the representation is a finite locally free $A$-module. In particular, the category $\Rep_AG$ is tensor equivalent to the category $\FMod_{BG}$ of finite locally free sheaves on $BG$.

If we write $G=\Spec\Gamma$ the multiplication of $G$ makes $\Gamma$ into a comodule, which is the filtered union of its finitely generated sub-comodules by \cite[\S\,1.5, Corollaire]{Serre68}. Since $\Gamma$ is a flat $A$-module and $A$ is Dedekind, these finitely generated sub-comodules are torsion free, whence finite projective and dualizable. The Hopf algebroid $(A,\Gamma)$ is therefore an Adams Hopf algebroid and the stack $BG$ is an \emph{Adams stack}; see \cite[\S\,1.3]{Schaeppi12}. Due to our assumption the tensor functor $\omega$ defines a tensor functor from $\FMod_{BG}$ to $\Mod_B$ which sends locally split epimorphisms $\phi$ to epimorphisms. Indeed, if $U$ is a scheme and $u\colon U\to BG$ is an \fpqc{} covering over that $u^*\phi$ splits, then $u^*\phi$ is an epimorphism of quasi-coherent sheaves on $U$. Since $U\times_{BG}\Spec A\to\Spec A$ is an \fpqc{} covering, \fpqc{} descent \cite[IV$_2$, Proposition~2.2.7]{EGA} shows that $p_0^*\phi$ is an epimorphism on the underlying $A$-modules, and our assumption applies. Thus by \cite[Corollary~3.4.3]{Schaeppi15} and \cite[Theorems~1.3.2 and 1.2.1]{Schaeppi12} the functor $\omega$ corresponds to a morphism $f\colon\Spec B\to BG$ and an isomorphism of tensor functors $\alpha\colon f^*\isoto\omega\circ p_0^*$. 

By definition of $BG$ this morphism corresponds to a $G$-torsor $p\colon X\to\Spec B$ that  is obtained as the base change
\[
\xymatrix @R=0.5pc @C=0pc {
X \ar[rr]^{\TS h} \ar[dd]_{\TS p} & & \Spec A \ar[dd]^{\TS p_0}\\
& \qed\qquad \\
\Spec B \ar[rr]^{\quad\TS f} & & BG\,.
}
\]
We will show that $X$ represents $\Isom^\otimes(\omega\open\otimes_AB,\omega)$. Since $p_0^*$ is an equivalence between $\FMod_{BG}$ and $\Rep_AG$, the tensor functor $p^*\alpha\colon h^*\circ p_0^*=p^*\circ f^*\isoto p^*\circ\omega\circ p_0^*$ induces a tensor isomorphism $p^*\circ(\omega\open\otimes_AB)=h^*\isoto p^*\circ\omega$ over $X$, and hence a morphism $X\to\Isom^\otimes(\omega\open\otimes_AB,\omega)$. To prove that the latter is an isomorphism let $s\colon S=\Spec R\to\Spec B$ and let $\beta\in\Isom^\otimes(\omega\open\otimes_AB,\omega)(S)$ which we view as a tensor isomorphism $\beta\colon s^*\circ(\omega\open\otimes_AB)\circ p_0^* \isoto   s^*\circ\omega\circ p_0^*$. Let $r\colon\Spec B\to\Spec A$ be the structure morphism. The two morphisms $f\circ s$ and $p_0\circ r\circ s$ from $S$ to $BG$ induce the pullback functors $s^*\circ f^*$ and $s^*\circ r^*\circ p_0^*=s^*\circ(\omega\open\otimes_AB)\circ p_0^*$ which are isomorphic by 
\[
\beta^{-1}\circ s^*\alpha\colon\; s^*\circ f^*\;\isoto\; s^*\circ\omega\circ p_0^*\;\isoto\; s^*\circ(\omega\open\otimes_AB)\circ p_0^*\,. 
\]
Again by \cite[Corollary~3.4.3]{Schaeppi15} and \cite[Theorems~1.3.2 and 1.2.1]{Schaeppi12} the latter isomorphism corresponds to an isomorphism $\eta$ between the morphisms $f\circ s$ and $p_0\circ r\circ s$ from $S$ to $BG$. The data $(s,r\circ s,\eta)$ defines a morphism $S\to X$ and this proves that $X\isoto\Isom^\otimes(\omega\open\otimes_AB,\omega)$ as desired.
\end{proof}

We also need the following easy

\begin{lemma}\label{LemmaOOmegaZ}
Let $\Omega$ be an algebraically closed field that is complete with respect to an absolute value $|\,.\,|\colon\Omega\to\BR_{\ge0}$ and let $\CO_\Omega$ be its valuation ring. Let $\Fm_\Omega\subset\CO_\Omega$ be the maximal ideal and $\kappa_\Omega=\CO_\Omega/\Fm_\Omega$. 
\begin{enumerate}
\item \label{LemmaOOmegaZ_A}
If $\Fp\subset \CO_\Omega\dbl z\dbr$ is a prime ideal then one of the following assertions holds
\begin{enumerate}
\item \label{LemmaOOmegaZ_A1}
$\Fp$ is contained in $\Fp_0\,:=\,\{\,\sum_{i=0}^\infty b_i z^i\colon b_i\in\Fm_\Omega\,\}\,=\,\ker(\CO_\Omega\dbl z\dbr\onto\kappa_\Omega\dbl z\dbr)$,
\item \label{LemmaOOmegaZ_A2}
$\Fp=(z-\alpha)$ for an $\alpha\in\Fm_\Omega$, or
\item \label{LemmaOOmegaZ_A3}
$\Fp=(z)+\Fm_\Omega$ is the maximal ideal of $\CO_\Omega\dbl z\dbr$.
\end{enumerate}
\item \label{LemmaOOmegaZ_B}
The maximal ideals $\Fm$ of the ring $\CO_\Omega\dpl z\dpr:=\CO_\Omega\dbl z\dbr[\tfrac{1}{z}]$ are all of the following form: $\Fm=(z-\alpha)$ for $\alpha\in\Fm_\Omega\setminus\{0\}$ with $\CO_\Omega\dpl z\dpr/\Fm=\Omega$, or $\Fm\,=\,\Fm_0\,:=\,\{\,\sum_i b_i z^i\colon b_i\in\Fm_\Omega\,\}\,=\,\ker\bigl(\CO_\Omega\dpl z\dpr\onto\kappa_\Omega\dpl z\dpr\bigr)$ with $\CO_\Omega\dpl z\dpr/\Fm=\kappa_\Omega\dpl z\dpr$. 
\end{enumerate}
\end{lemma}

\begin{proof}
\ref{LemmaOOmegaZ_A} If (i) fails, that is if $\Fp\not\subset\Fp_0$ then there is an element $f=\sum_{i=0}^\infty b_i z^i\in\Fp$ with $b_n\in\CO_\Omega\mal$ for some $n$. We may assume that $b_i\in\Fm_\Omega$ for all $i<n$, and hence the image $\bar f$ of $f$ in $\kappa_\Omega\dbl z\dbr$ has $\ord_z(\bar f)=n$, where $\ord_z$ is the valuation of the discrete valuation ring $\kappa_\Omega\dbl z\dbr$. If $n=0$ then $f$ would be a unit, which is not the case. So $|b_0|<|b_n|=1$ and the Newton polygon of $f$ has a negative slope. By \cite[Proposition~2]{Lazard} there is an element $\alpha\in\Fm_\Omega$ and a power series $g=\sum_{i=0}^\infty c_i z^i\in\Omega\dbl z\dbr$ with $f=(z-\alpha)\cdot g$, that is $b_i=c_{i-1}-\alpha c_i$, such that the region of convergence of $f$ is contained in the one of $g$. In particular, $g$ converges for all $z$ in $\Fm_\Omega$. 

We claim that $g\in\CO_\Omega\dbl z\dbr$, that is $|c_i|\le1$ for all $i$. Assume contrarily that there exists an $m>0$ with $|c_{m-1}|>1\ge|b_m|$. Then $\alpha c_m=c_{m-1}-b_m$ implies $|c_m|=|\alpha^{-1} c_{m-1}|>|c_{m-1}|$ and by induction $|c_i|=|\alpha|^{m-1-i}|c_{m-1}|$ for all $i\ge m-1$. This implies that $g$ does not converge at $z=\alpha$. So we obtain a contradiction and the claim is proved.

For the images in $\kappa_\Omega\dbl z\dbr$ we obtain $\bar f=z\cdot\bar g$ and $\ord_z(\bar g)=\ord_z(\bar f)-1$. Continuing in this way we find that $f=(z-\alpha_1)\cdot\ldots\cdot(z-\alpha_n)\cdot h$ for $\alpha_i\in\Fm_\Omega$ and a unit $h\in\CO_\Omega\dbl z\dbr\mal$. Since $\Fp$ is prime, it contains $z-\alpha_i$ for some $i$. Under the isomorphism $\CO_\Omega\dbl z\dbr/(z-\alpha_i)\isoto\CO_\Omega$, $z\mapsto\alpha_i$, the quotient $\Fp/(z-\alpha_i)$ is a prime ideal in $\CO_\Omega$. If $\Fp/(z-\alpha_i)=(0)$, then $\Fp=(z-\alpha_i)$ and we are in case (ii). If $\Fp/(z-\alpha_i)=\Fm_\Omega$, then $\Fp=(z-\alpha_i)+\Fm_\Omega=(z)+\Fm_\Omega$ and we are in case (iii).

\medskip\noindent
\ref{LemmaOOmegaZ_B} follows from \ref{LemmaOOmegaZ_A} via the identification of the prime ideals in $\CO_\Omega\dpl z\dpr$ with the prime ideals in $\CO_\Omega\dbl z\dbr$ not containing $z$.
\end{proof}

\begin{proof}[Proof of Theorem~\ref{MainThm}]
\ref{MainThm_1} 
Since the proof is quite long and involved, let us give a summary first. By \cite[Theorem~2.9]{dJ95a} the subset of points $\bar\gamma$ admitting an $\BF_q\dpl z\dpr$-rational isomorphism $\beta\colon\omega\open\isoto\omega_{b,\bar\gamma}$ between the tensor functors $\omega\open$ and $\omega_{b,\bar\gamma}$ is a union of connected components. Thus the task is to prove that it agrees with the image of the period morphism.

We start with a point $\bar\gamma\in\breve\CH_{G,\hat Z,b}^{a}(\Omega)$ with values in an algebraically closed, complete field $\Omega$ and a tensor isomorphism $\beta$. In Step~1 we construct from this data a tensor functor $\ulCT_b\colon\Rep_{\BF_q\dbl z\dbr}G\to\BLoc_{\HeckeTower_{K_0}}$ over an \'etale covering space $\HeckeTower_{K_0}$ of $\breve\CH_{G,\hat Z,b}^{a}$. By the admissibility of $\bar\gamma$, this tensor functor will in Step~2 induce a tensor functor $M\colon\Rep_{\BF_q\dbl z\dbr}G\to\FMod_{\CO_{\Omega}\dbl z\dbr},\,V\mapsto M_V$ that gives the underlying module $M_V$ of a local shtuka $\ulM_V$. Unfortunately, it is not clear that the functor $V\mapsto M_V$ is exact without using the hypothesis that $G$ is parahoric. This might even be false for groups $G$ with non-reductive generic fiber as Example~\ref{ExampleNotATorsor} shows. Nevertheless, we show in Step~3 that $V\mapsto M_V\otimes_{\CO_\Omega\dbl z\dbr}\CO_\Omega\dpl z\dpr$ is exact and corresponds to a $G$-torsor $\CG$ over $\CO_\Omega\dpl z\dpr$. From \cite[Proposition~11.5]{Anschuetz} we obtain that $\CG\cong G\otimes_{\BF_q\dbl z\dbr}\CO_\Omega\dpl z\dpr$ is trivial. This allows us in Step~4 to apply Lemma~\ref{LemmaLiftOfIsog} to produce a local $G$-shtuka $(\ul\CG''\!,\bar\Delta'')\in\breveRZ(\CO_\Omega)$ with $\breve\pi(\ul\CG''\!,\bar\Delta'')=\bar\gamma$.

\medskip\noindent
1. Let $\gamma\in\breve\CH_{G,\hat Z,b}^{a}$ and let $\bar\gamma$ be a geometric point lying above $\gamma$ with algebraically closed, complete residue field $\Omega$. We consider the tensor functor $\ulCV_b\colon\Rep_{\BF_q\dpl z\dpr}G\to\PLoc_{\breve\CH_{G,\hat Z,b}^{a}}$ from Theorem~\ref{ThmLocSyst} and the two fiber functors $\omega_{b,\bar\gamma}:={\it forget}\circ\omega_{\bar\gamma}\circ\ulCV_b$ and $\omega\open\colon\Rep_{\BF_q\dpl z\dpr}G\to\FMod_{\BF_q\dpl z\dpr}$ from Remark~\ref{RemLocSyst} and assume that there is an $\BF_q\dpl z\dpr$-rational isomorphism 
\begin{equation}\label{EqBeta}
\beta\colon\omega\open\isoto\omega_{b,\bar\gamma}
\end{equation}
of tensor functors which we fix. It induces an isomorphism of group schemes $\wt G:=\Aut^\otimes(\omega_{b,\bar\gamma})\isoto\Aut^\otimes(\omega_0)= G_{\BF_q\dpl z\dpr}$ over $\BF_q\dpl z\dpr$. Consider the compact open subgroup $K_0:=G(\BF_q\dbl z\dbr)\subset G\bigl(\BF_q\dpl z\dpr\bigr)$ and the \'etale covering space $\HeckeTower_{K_0}$ of $\breve\CH_{G,\hat Z,b}^{a}$ from Remark~\ref{Rem1.8} associated with the local system $\ulCV_b$ on $\breve\CH_{G,\hat Z,b}^{a}$. The point $\beta K_0\in\Isom^\otimes(\omega\open,\omega_{b,\bar\gamma})\bigl(\BF_q\dpl z\dpr\bigr)/K_0=F_{\bar\gamma}^\et(\HeckeTower_{K_0})$ corresponds to a lift of the base point $\bar\gamma$ to a geometric base point $\bar\gamma$ of $\HeckeTower_{K_0}$ such that the morphism $\pi_1^\et(\HeckeTower_{K_0},\bar\gamma)\to\pi_1^\et(\breve\CH_{G,\hat Z,b}^{a},\bar\gamma)\to G\bigl(\BF_q\dpl z\dpr\bigr)$ from \eqref{EqRepFundGpPeriodSp} factors through $K_0$. By Corollary~\ref{Cor1.7} the induced morphism $\pi_1^\et(\HeckeTower_{K_0},\bar\gamma)\to K_0$ yields a tensor functor $\ulCT_b\colon\Rep_{\BF_q\dbl z\dbr}G\to\BLoc_{\HeckeTower_{K_0}}$ and an isomorphism $\omega\open\isoto{\it forget}\circ\omega_{\bar\gamma}\circ\ulCT_b$, which equals the restriction of $\beta$ by construction. In particular, for each $V\in\Rep_{\BF_q\dbl z\dbr}G$ the fiber $\ulCT_b(V)_{\bar\gamma}$ is a free $\BF_q\dbl z\dbr$-module of rank equal to the rank of $V$.

\medskip\noindent
2. Let $(\Darst_V,V)\in\Rep_{\BF_q\dbl z\dbr}G$. Then the $z$-isocrystal with Hodge-Pink structure
\[
\ulD_{b,\bar\gamma}(V)=\bigl(V\otimes_{\BF_q\dbl z\dbr}\BaseFld\dpl z\dpr,\Darst_V(\s b)\sigma^*,\Fq_D(V)\bigr)
\]
over $\Omega$ is admissible and comes from a pair $(\ulM_V,\delta_V)$ where $\ulM_V=(M_V,\tau_{M_V})$ is a local shtuka over $\CO_{\Omega}$ with $M_V\cong\CO_{\Omega}\dbl z\dbr^r$ and $\tau_{M_V}\colon\s M_V[\tfrac{1}{z-\zeta}]\isoto M_V[\tfrac{1}{z-\zeta}]$, and with the notation \eqref{EqCrisRing},
\begin{equation}\label{EqDeltaV}
\delta_V\colon\ulM_V\otimes_{\CO_{\Omega}\dbl z\dbr}\CO_{\Omega}\dbl z,z^{-1}\}[\tfrac{1}{\tminus}]\;\isoto\;\bigl(V\otimes_{\BF_q\dbl z\dbr}\CO_{\Omega}\dbl z,z^{-1}\}[\tfrac{1}{\tminus}],\Darst_V(b)\sigma^*\bigr)
\end{equation}
is an isomorphism satisfying $\delta_V\circ\tau_{M_V}=\Darst_V(b)\circ\s\delta_V$, such that
\[
\Fq_D(V):=\Darst_V(\bar\gamma)\cdot V\otimes_{\BF_q\dbl z\dbr}\Omega\dbl z-\zeta\dbr=\sigma^*\delta_V\circ\tau_{M_V}^{-1}(M_V\otimes_{\CO_{\Omega}\dbl z\dbr}\Omega\dbl z-\zeta\dbr)
\]
and so
\begin{equation}\label{Eq_h_ohneNenner}
\tau_{M_V}\circ\s\delta_V^{-1}\circ\Darst_V(\bar\gamma)\colon V\otimes_{\BF_q\dbl z\dbr}\Omega\dbl z-\zeta\dbr\isoto M_V\otimes_{\CO_{\Omega}\dbl z\dbr}\Omega\dbl z-\zeta\dbr.
\end{equation}
This follows from \cite[Proposition~2.4.9]{HartlPSp} with $X_L=\Spm\Omega$ and $X=\Spf\CO_{\Omega}$, where our pair $(\ulM_V,\delta_V)$ is the rigidified local shtuka of loc.\ cit.. We may apply this proposition by taking $\CQ':=\ulCT_b(V)_{\bar\gamma}\otimes_{\BF_q\dbl z\dbr}\Omega\langle\tfrac{z}{\zeta^s}\rangle$ inside $\CQ:=\ulCF_{b,\bar\gamma}(V)=\CQ'\otimes_{\Omega\langle\frac{z}{\zeta^s}\rangle}\Omega\ancon[s]$ for an $s$ with $1>s>\frac{1}{q}$; see Remark~\ref{RemGlobalSigmaModule}. Note that our notation here of the underlying $z$-isocrystal deviates from \cite{HartlPSp}. Namely, our $\ulD_{b}(V)=\bigl(V\otimes_{\BF_q\dbl z\dbr}\BaseFld\dpl z\dpr,\Darst_V(\s b)\sigma^*\bigr)$ was called $\sigma^*(D,F_D)$ in \cite[Proposition~2.4.9]{HartlPSp}. So the $(D,F_D)$ from \cite[Proposition~2.4.9]{HartlPSp} is equal to our $\bigl(V\otimes_{\BF_q\dbl z\dbr}\BaseFld\dpl z\dpr,\Darst_V(b)\sigma^*\bigr)$. The latter equals the $z$-isocrystal associated with the local $\GL(V)$-shtuka $\Darst_{V,*}\ul\BG_0$ over $\BaseFld$ that was used in (the proof of) Lemma~\ref{LemmaGL}.

As can be seen from the proof of loc.\ cit., there is an isomorphism $\varepsilon_V\colon\ulM_V\otimes_{\CO_{\Omega}\dbl z\dbr}\Omega\langle\tfrac{z}{\zeta^s}\rangle\isoto\CQ'=\ulCT_b(V)_{\bar\gamma}\otimes_{\BF_q\dbl z\dbr}\Omega\langle\tfrac{z}{\zeta^s}\rangle$ yielding $\check T_z\ulM_V\cong (\CQ')^\tau=\ulCT_b(V)_{\bar\gamma}$. Alternatively,  \cite[Proposition~2.4.4]{HartlPSp} provides an isomorphism $\varepsilon_V\colon\ulM_V\otimes_{\CO_{\Omega}\dbl z\dbr}\Omega\ancon[s]\isoto\ulCF_{b,\bar\gamma}(V)$ yielding $\check T_z\ulM_V\otimes_{\BF_q\dbl z\dbr}\BF_q\dpl z\dpr\cong(\ulCF_{b,\bar\gamma}(V))^\tau=\ulCT_b(V)_{\bar\gamma}\otimes_{\BF_q\dbl z\dbr}\BF_q\dpl z\dpr=\ulCV_b(V)_{\bar\gamma}$. Now the existence of a uniquely determined pair $(\ulM_V,\delta_V)$ with $\varepsilon_V(\check T_z\ulM_V)=\ulCT_b(V)_{\bar\gamma}$ follows from Proposition~\ref{PropLiftOfIsog}. 

We claim that the underlying free $\CO_{\Omega}\dbl z\dbr$-module $M_V$ of the local shtuka $\ulM_V$ defines a tensor functor $M\colon\Rep_{\BF_q\dbl z\dbr}G\to\FMod_{\CO_{\Omega}\dbl z\dbr},\,V\mapsto M_V$. We use the isomorphisms 
\begin{eqnarray}
\label{EqTensorIsoms1}
\s\delta_V\circ\tau_{M_V}^{-1}\hspace{-1em} & \colon & \hspace{-1em} M_V\otimes_{\CO_{\Omega}\dbl z\dbr}\CO_{\Omega}\dbl z,z^{-1}\}[\tfrac{1}{\tminus}]\;\isoto\;V\otimes_{\BF_q\dbl z\dbr}\CO_{\Omega}\dbl z,z^{-1}\}[\tfrac{1}{\tminus}]\quad\text{and\qquad}\\[2mm]
\label{EqTensorIsoms2}
(\beta_V\otimes \id_{\Omega\langle\frac{z}{\zeta^s}\rangle})^{-1}\circ\varepsilon_V\hspace{-1em} & \colon & \hspace{-1em} M_V\otimes_{\CO_{\Omega}\dbl z\dbr}\Omega\langle\tfrac{z}{\zeta^s}\rangle\;\isoto\;\ulCT_b(V)_{\bar\gamma}\otimes_{\BF_q\dbl z\dbr}\Omega\langle\tfrac{z}{\zeta^s}\rangle\;\isoto\; V\otimes_{\BF_q\dbl z\dbr}\Omega\langle\tfrac{z}{\zeta^s}\rangle,
\end{eqnarray}
where $\beta_V\colon V=\omega(V)\open\isoto\omega_{b,\bar\gamma}(V)$ is induced from the isomorphism $\beta$ from \eqref{EqBeta}. Consider the identification $$h_V\colon \ulCT_b(V)_{\bar\gamma}\otimes_{\BF_q\dbl z\dbr}\Omega\ancon[s]{}[\tfrac{1}{\tminus}]=\CF_{b,\bar\gamma}(V)[\tfrac{1}{\tminus}]=\CE_{b,\bar\gamma}(V)[\tfrac{1}{\tminus}]:=V\otimes_{\BF_q\dbl z\dbr}\Omega\ancon[s]{}[\tfrac{1}{\tminus}].$$ Then the isomorphisms $\s\delta_V\circ\tau_{M_V}^{-1}$ and $(\beta_V\otimes \id_{\Omega\langle\frac{z}{\zeta^s}\rangle})^{-1}\circ\varepsilon_V$ satisfy \begin{multline*}h_V\circ(\beta_V\otimes \id_{\Omega\ancon[s]{}[\frac{1}{\tminus}]})\circ\bigl((\beta_V\otimes \id_{\Omega\langle\frac{z}{\zeta^s}\rangle})^{-1}\circ\varepsilon_V\bigr)\otimes\id_{\Omega\ancon[s]{}[\frac{1}{\tminus}]}\\ \,=\,(\s\delta_V\circ\tau_{M_V}^{-1})\otimes\id_{\Omega\ancon[s]{}[\frac{1}{\tminus}]}\colon M_V\otimes_{\CO_{\Omega}\dbl z\dbr}\Omega\ancon[s]{}[\tfrac{1}{\tminus}]\isoto \CE_{b,\bar\gamma}(V)[\tfrac{1}{\tminus}]\end{multline*} and $M_V$ is obtained from the intersection
\begin{equation}\label{EqIntersection}
\sigma^*\delta_V\circ\tau_{M_V}^{-1}(M_V)\;=\; V\otimes_{\BF_q\dbl z\dbr}\CO_\Omega\dbl z,z^{-1}\}[\tfrac{1}{\tminus}]\cap h_V\beta_V\bigl(V\otimes_{\BF_q\dbl z\dbr}\Omega\langle\tfrac{z}{\zeta^s}\rangle\bigr)
\end{equation}
inside $V\otimes_{\BF_q\dbl z\dbr}\Omega\ancon[s]{}[\tfrac{1}{\tminus}]$.

To prove compatibility with morphisms $f\colon V\to V'$ in $\Rep_{\BF_q\dbl z\dbr}G$, we consider the induced morphisms $\CE_{b,\bar\gamma}(f):=f\otimes\id_{\Omega\ancon[s]{}[\frac{1}{\tminus}]}\colon\CE_{b,\bar\gamma}(V)\to\CE_{b,\bar\gamma}(V')$ and $\ulCT_b(f)_{\bar\gamma}\colon\ulCT_b(V)_{\bar\gamma}\to\ulCT_b(V')_{\bar\gamma}$. Since $h_{V'}\circ\ulCT_b(f)_{\bar\gamma}=\CE_{b,\bar\gamma}(f)\circ h_V$ and $\ulCT_b(f)_{\bar\gamma}\circ\beta_V=\beta_{V'}\circ f$, the morphism $f$ induces via \eqref{EqIntersection} a morphism $M_f\colon M_V\to M_{V'}$.

Finally, if $V,V'\in\Rep_{\BF_q\dbl z\dbr}G$ then $(\ulM_V,\delta_V)\otimes(\ulM_{V'},\delta_{V'})$ has $\ulD_{b,\bar\gamma}(V)\otimes\ulD_{b,\bar\gamma}(V')=\ulD_{b,\bar\gamma}(V\otimes V')$ as $z$-isocrystal with Hodge-Pink structure and $\check T_z\ulM_V\otimes\check T_z\ulM_{V'}\cong\ulCT_b(V)_{\bar\gamma}\otimes\ulCT_b(V')_{\bar\gamma}=\ulCT_b(V\otimes V')_{\bar\gamma}$ as Tate-module. This implies that $(\ulM_{V\otimes V'},\delta_{V\otimes V'})=(\ulM_V,\delta_V)\otimes(\ulM_{V'},\delta_{V'})$ and so $M$ is indeed a tensor functor.

Unfortunately, it is not a priory clear that this tensor functor comes from a $G$-torsor over $\CO_\Omega\dbl z\dbr$. Namely, a necessary (and by Proposition~\ref{PropTannaka} also sufficient) condition is that for every sequence $0\to V'\to V\to V''\to 0$ in $\Rep_{\BF_q\dbl z\dbr}G$ that is exact on the underlying free $\BF_q\dbl z\dbr$-modules, also the sequence $0\to M_{V'}\to M_V\to M_{V''}\to 0$ of free $\CO_\Omega\dbl z\dbr$-modules is exact. This is not obvious without using the hypothesis that $G$ is parahoric, and may even be false if the generic fiber of $G$ is not reductive as Example~\ref{ExampleNotATorsor} shows.

\medskip\noindent
3. We first claim that the sequence of $\CO_\Omega\dbl z\dbr$-modules $0\to M_{V'}\to M_V\to M_{V''}\to 0$ becomes exact after tensoring with $\CO_\Omega\dpl z\dpr:=\CO_\Omega\dbl z\dbr[\tfrac{1}{z}]$. Namely, since the cokernel $C$ of $M_V[\tfrac{1}{z}]\to M_{V''}[\tfrac{1}{z}]$ is finitely generated, it suffices by Nakayama's lemma to prove that $C/\Fm C=(0)$ for every maximal ideal $\Fm\subset\CO_\Omega\dpl z\dpr$. By Lemma~\ref{LemmaOOmegaZ}\ref{LemmaOOmegaZ_B} there are two cases, namely $\Fm=\Fm_0=\ker\bigl(\CO_\Omega\dpl z\dpr\onto\kappa_\Omega\dpl z\dpr\bigr)$, and $\Fm=(z-\alpha)$  for $\alpha\in\Fm_\Omega\setminus\{0\}$. In the first case, the morphism $\CO_\Omega\dbl z\dbr\to\CO_\Omega\dpl z\dpr/\Fm_0=\kappa_\Omega\dpl z\dpr$ factors through $\CO_\Omega\dbl z,z^{-1}\}[\tfrac{1}{\tminus}]$, because $\tminus\equiv 1\pmod{\Fm_0}$, and hence $C/\Fm_0 C\cong\coker(V\onto V'')\otimes_{\BF_q\dbl z\dbr}\kappa_\Omega\dpl z\dpr=(0)$ by using the isomorphism \eqref{EqTensorIsoms1}. In the second case, if $|\alpha|>|\zeta|^s$, the morphism $\CO_\Omega\dbl z\dbr\to\CO_\Omega\dpl z\dpr/\Fm=\Omega$, $z\mapsto\alpha$ likewise factors through $\CO_\Omega\dbl z,z^{-1}\}[\tfrac{1}{\tminus}]$, because $\tminus(\alpha)=\prod_{i\in\BN_0}\bigl(1-{\tfrac{\zeta^{q^i}}{\alpha}}\bigr)\ne0$, and hence $C/\Fm_0 C\cong\coker(V\onto V'')\otimes_{\BF_q\dbl z\dbr}\Omega=(0)$ by using the isomorphism \eqref{EqTensorIsoms1} again. Finally in the second case, if $|\alpha|\le|\zeta|^s$, the morphism $\CO_\Omega\dbl z\dbr\to\CO_\Omega\dpl z\dpr/\Fm=\Omega$ factors through $\Omega\langle\frac{z}{\zeta^s}\rangle$, and hence $C/\Fm_0 C\cong\coker(V\onto V'')\otimes_{\BF_q\dbl z\dbr}\Omega=(0)$ by using the isomorphism \eqref{EqTensorIsoms2}. 

Now Proposition~\ref{PropTannaka} shows that $\CG:=\Isom^\otimes\bigl(\omega\open\otimes_{\BF_q\dpl z\dpr}\CO_\Omega\dpl z\dpr,\,M\otimes_{\CO_\Omega\dbl z\dbr}\CO_\Omega\dpl z\dpr\bigr)$ is a $G$-torsor over $\CO_\Omega\dpl z\dpr$ for the \fpqc-topology, and hence for the \'etale topology, because $G$ is smooth.
By \cite[Proposition~11.5]{Anschuetz} the $G$-torsor $\CG$ over $\CO_\Omega\dpl z\dpr$ is trivial.

\medskip\noindent
4. From the triviality of $\CG$ and from the isomorphism \eqref{EqTensorIsoms2} we obtain an $\CO_\Omega\dpl z\dpr$-rational, respectively $\Omega\langle\tfrac{z}{\zeta^s}\rangle$-rational, tensor isomorphism $\theta\colon\omega\open\otimes_{\BF_q\dpl z\dpr}\CO_\Omega\dpl z\dpr\isoto M\otimes_{\CO_\Omega\dbl z\dbr}\CO_\Omega\dpl z\dpr$, respectively $\theta_2\colon\omega_{\BF_q\dbl z\dbr}\open\otimes_{\BF_q\dbl z\dbr}\Omega\langle\tfrac{z}{\zeta^s}\rangle\isoto M\otimes_{\CO_\Omega\dbl z\dbr}\Omega\langle\tfrac{z}{\zeta^s}\rangle$, which we fix in the sequel. Over $\Omega\langle\tfrac{z}{\zeta^s}\rangle[z^{-1}]$ the composition $u''_\Omega:=\theta^{-1}\circ\theta_2$ corresponds to an element $u''_\Omega\in G\bigl(\Omega\langle\tfrac{z}{\zeta^s}\rangle[z^{-1}]\bigr)\subset G\bigl(\Omega\dpl z\dpr\bigr)=LG(\Omega)$. The isomorphisms 
\[
\theta_V^{-1}\circ\tau_{M_V}\circ\s\theta_V\colon\; V\otimes\CO_{\Omega}\dbl z\dbr[\tfrac{1}{z(z-\zeta)}]\;\isoto\;\s M_V[\tfrac{1}{z(z-\zeta)}]\;\isoto\; M_V[\tfrac{1}{z(z-\zeta)}]\;\isoto\; V\otimes\CO_{\Omega}\dbl z\dbr[\tfrac{1}{z(z-\zeta)}]
\]
provide an automorphism of the tensor functor $\omega\open\otimes_{\BF_q\dpl z\dpr}\CO_{\Omega}\dbl z\dbr[\tfrac{1}{z(z-\zeta)}]$, that is an element $A\in G\bigl(\CO_{\Omega}\dbl z\dbr[\tfrac{1}{z(z-\zeta)}]\bigr)=L_{z(z-\zeta)}G(\CO_\Omega)$. We consider $A$ as an isomorphism $\tau_\CG:=A\s\colon \s L_{z(z-\zeta)}G_{\CO_\Omega}\isoto L_{z(z-\zeta)}G_{\CO_\Omega}$ of trivial $L_{z(z-\zeta)}G$-torsors. Likewise the isomorphisms 
\[
\delta_V\circ\theta_V\colon\; V\otimes\CO_{\Omega}\dbl z,z^{-1}\}[\tfrac{1}{\tminus}]\;\isoto\; M_V\otimes_{\CO_{\Omega}\dbl z\dbr}\CO_{\Omega}\dbl z,z^{-1}\}[\tfrac{1}{\tminus}]\;\isoto\; V\otimes\CO_{\Omega}\dbl z,z^{-1}\}[\tfrac{1}{\tminus}]
\]
provide an automorphism of the tensor functor $\omega\open\otimes_{\BF_q\dpl z\dpr}\CO_{\Omega}\dbl z,z^{-1}\}[\tfrac{1}{\tminus}]$, that is an element $\Delta\in G\bigl(\CO_{\Omega}\dbl z,z^{-1}\}[\tfrac{1}{\tminus}]\bigr)$. The equalities $\Darst_V(b)\circ\s\delta_V=\delta_V\circ\tau_{M_V}$ yield the equality $b\cdot\s\Delta=\Delta\cdot A$ in $G\bigl(\CO_{\Omega}\dbl z,z^{-1}\}[\tfrac{1}{\tminus}]\bigr)$. Finally, the isomorphisms 
\[
\theta_{2,V}^{-1}\circ\tau_{M_V}\circ\s\theta_{2,V}\colon\; V\otimes\Omega\langle\tfrac{z}{\zeta^{qs}}\rangle\;\isoto\;\s M_V\otimes\Omega\langle\tfrac{z}{\zeta^{qs}}\rangle\;\isoto\; M_V\otimes\Omega\langle\tfrac{z}{\zeta^{qs}}\rangle\;\isoto\; V\otimes\Omega\langle\tfrac{z}{\zeta^{qs}}\rangle
\]
provide an automorphism of the tensor functor $\omega\open_{\BF_q\dbl z\dbr}\otimes_{\BF_q\dbl z\dbr}\Omega\langle\tfrac{z}{\zeta^{qs}}\rangle$, that is an element $A''\in G\bigl(\Omega\langle\tfrac{z}{\zeta^{qs}}\rangle\bigr)\subset G\bigl(\Omega\dbl z\dbr\bigr)=L^+G(\Omega)$. It satisfies $u''_\Omega\cdot A''=A\cdot\s u''_\Omega$. We view $A''$ as an isomorphism $\tau'':=A''\s\colon \s L^+G_\Omega\isoto L^+G_\Omega$  of trivial $L^+G$-torsors satisfying $u''_\Omega\circ\tau''=\tau_\CG\circ\s u''_\Omega$.

Moreover, choose a representative $\bar\gamma\in G\bigl(\Omega\dpl z-\zeta\dpr\bigr)$ of $\bar\gamma\in\CH_{G,\hat Z}^\an(\Omega)\subset\Gr_G^{\bB_\dR}(\Omega)$. Then the element $h:=A\cdot\sigma^*\Delta^{-1}\cdot\bar\gamma\in G\bigl(\Omega\dpl z-\zeta\dpr\bigr)$ satisfies $\Darst_V(h)=\theta_V^{-1}\circ\tau_{M_V}\circ\s\delta_V^{-1}\circ\Darst_V(\bar\gamma)$. In formula \eqref{Eq_h_ohneNenner} we computed that this is an automorphism of $V\otimes_{\BF_q\dbl z\dbr}\Omega\dbl z-\zeta\dbr$. Thus $h\in G(\Omega\dbl z-\zeta\dbr)$, and hence $A^{-1}=\sigma^*\Delta^{-1}\cdot\bar\gamma\cdot h^{-1}$ lies in $\hat{Z}^\an(\Omega)$, because $\sigma^*\Delta\in G(\Omega\dbl z-\zeta\dbr)$. Therefore, the pair $(LG_{\CO_\Omega},\tau_\CG)$ consisting of the trivial $LG$-torsor over $\CO_\Omega$ associated with $\CG$ and $\tau_\CG:=A\s$, as well as the pair $(\CG''_\Omega,\tau'')$ consisting of the trivial $L^+G$-torsor $\CG''_\Omega:=L^+G_\Omega$ over $\Omega$ with $\tau'':=A''\s$, together with $u''_\Omega$ satisfy the hypothesis of Lemma~\ref{LemmaLiftOfIsog} with $B=\CO_\Omega$ and $L=\Omega$. We apply this lemma. Since every finitely generated ideal of $\CO_\Omega$ is principal, $Y=\Spec\CO_\Omega$. We obtain a local $G$-shtuka $\ul\CG''=(\CG''\!,\tau_{\CG''})$ over $\CO_\Omega$ bounded by $\hat{Z}^{-1}$ and an isomorphism $u''\colon (L\CG'',\tau_{\CG''})\isoto (LG_{\CO_\Omega},\tau_\CG)$ of $LG$-torsors over $\CO_\Omega$ satisfying $u''\circ \tau_{\CG''}=\tau_\CG\circ\s u''$. The reduction modulo $\zeta$ of $\Delta'':=\Delta\circ u''$ provides a quasi-isogeny $\ul\CG''_{\CO_\Omega/(\zeta)}\to\ul\BG_{0,\CO_\Omega/(\zeta)}$, because $\Delta\circ u''\circ\tau_{\CG''}=\Delta\cdot A\circ\s u''=b\circ\s(\Delta\circ u'')$. This yields an $\CO_\Omega$-valued point $(\ul\CG''\!,\bar\Delta'')\in\breveRZ(\CO_\Omega)$. Its image under $\breve\pi$ is computed as  $\breve\pi(\ul\CG''\!,\bar\Delta'')=b^{-1}\Delta''\cdot G\bigl(\Omega\dbl z-\zeta\dbr\bigr)$ with $b^{-1}\Delta''=b^{-1}\Delta\circ u''=\s\Delta\cdot A^{-1}\cdot u''=\bar\gamma\cdot h^{-1}u''$. Since $h,u''\in G\bigl(\Omega\dbl z-\zeta\dbr\bigr)$ this shows that $\breve\pi(\ul\CG''\!,\bar\Delta'')=\bar\gamma$ and finally \ref{MainThm_1} is proved.

\bigskip\noindent
To prove \ref{MainThm_2} fix a representation $\Darst\in\Rep_{\BF_q\dpl z\dpr}G$. By Remark~\ref{RemTateModule} the rational Tate module $\check V_{\ul\CG}(\Darst)$ is a local system of $\BF_q\dpl z\dpr$-vector spaces on ${\breve\CM}:=(\breveRZ)^\an$. In order that it descends to a local system on $\breve\pi\bigl((\breveRZ)^\an\bigr)=:\im(\breve\pi)$ it suffices by \cite[Definition~4.1]{dJ95a} to show that
\begin{enumerate}
\item[(i)]
$\breve\pi\colon{\breve\CM}\to\im(\breve\pi)$ is a covering for the \'etale topology and that
\item[(ii)]
there is a descent datum $\psi:pr_1^\ast \check V_{\ul\CG}(\Darst)\isoto pr_2^\ast \check V_{\ul\CG}(\Darst)$ over ${\breve\CM}\times_{\im(\breve\pi)}{\breve\CM}$ where $pr_i:{\breve\CM}\times_{\im(\breve\pi)}{\breve\CM}\to{\breve\CM}$ is the projection onto the $i$-th factor, such that $\psi$ satisfies the cocycle condition on ${\breve\CM}\times_{\im(\breve\pi)}{\breve\CM}\times_{\im(\breve\pi)}{\breve\CM}$.
\end{enumerate}
Statement (i) follows from Proposition~\ref{PropPeriodMEtale}. To prove (ii) let $L$ and $B$ be as introduced before Lemma~\ref{LemmaDivisibility} and let $\CX=\Spf B$ and $X=\CX^\an=\SpBerk(B[\tfrac{1}{\zeta}])$. Consider an $X$-valued point of ${\breve\CM}\times_{\im(\breve\pi)}{\breve\CM}$ and its image in ${\breve\CM}\times_{\breve E}{\breve\CM}$. By \cite[Theorem~4.1]{FRG1} this $X$-valued point is induced by two $\Spf B$-valued points $(\ul\CG,\bar\delta)$ and $(\ul\CG'\!,\bar\delta')$ of $\breveRZ$ with $\breve\pi(\ul\CG,\bar\delta)=\breve\pi(\ul\CG'\!,\bar\delta')$ in $\breve\CH_{G,\hat{Z},b}^{a}(X)$ possibly after replacing $\Spf B$ by an affine covering of an admissible formal blowing-up; see the explanations before Lemma~\ref{LemmaTrivializing} for more details. Now Proposition~\ref{PropLiftOfIsog} (together with Proposition~\ref{Prop2.13}) yields a canonical isomorphism $\psi:pr_1^\ast \check V_{\ul\CG}(\Darst)\isoto pr_2^\ast \check V_{\ul\CG}(\Darst)$ of local systems of $\BF_q\dpl z\dpr$-vector spaces over $X$ which is functorial in $\Darst$ and satisfies the cocycle condition by canonicity. Therefore $\check V_{\ul\CG}(\Darst)$ descends to a local system of $\BF_q\dpl z\dpr$-vector spaces on $\breve\pi\bigl((\breveRZ)^\an\bigr)$. Clearly $\check V_{\ul\CG}\colon\Darst\mapsto\check V_{\ul\CG}(\Darst)$ is a tensor functor. The isomorphism $\Phi_j\colon j^*\ul\CG\isoto \ul\CG$ from Remark~\ref{RemJActsOnMK} yields a canonical $J_b\bigl(\BF_q\dpl z\dpr\bigr)$-linearization $\check V_{\Phi_j}\colon j^*\check V_{\ul\CG}=\check V_{j^*\ul\CG}\isoto\check V_{\ul\CG}$ on $\check V_{\ul\CG}$ over $\breveRZ$ which descends to $\breve\pi\bigl((\breveRZ)^\an\bigr)$ because the period morphism $\breve\pi$ is $J_b\bigl(\BF_q\dpl z\dpr\bigr)$-equivariant by Remark~\ref{RemPiAndJ}.

Let $\Darst\colon G\to\GL_r$ be in $\Rep_{\BF_q\dbl z\dbr}G$. By \cite[Proposition~2.4.4]{HartlPSp} the pullback to $\breveRZ$ under $\breve\pi$ of the $\sigma$-bundle $\ulCF_b(\Darst)$ over $\breve\pi\bigl((\breveRZ)^\an\bigr)$ from Remark~\ref{RemGlobalSigmaModule} is canonically isomorphic to $\ulM_\Darst\otimes\CO_{\breveRZ}\ancon[s]$ where $\ulM_\Darst$ is the local shtuka over $\breveRZ$ associated with the local $\GL_r$-shtuka $\Darst_*\ul\CG^\univ$ obtained from the universal local $G$-shtuka $\ul\CG^\univ$ over $\breveRZ$. This isomorphism is functorial in $\Darst$ and compatible with tensor products and pullback under the action of $j\in J_b\bigl(\BF_q\dpl z\dpr\bigr)$ because the period morphism $\breve\pi$ is $J_b\bigl(\BF_q\dpl z\dpr\bigr)$-equivariant. Descending it to $\breve\pi\bigl((\breveRZ)^\an\bigr)$ and taking $\tau$-invariants yields a canonical isomorphism of tensor functors $\alpha\colon\ulCV_b\isoto\check V_{\ul\CG}$, which satisfies $\alpha\circ\phi_j=\check V_{\Phi_j}\circ j^*\alpha$ where $\phi_j\colon j^*\ulCV_b\isoto\ulCV_b$ is the linearization from Theorem~\ref{ThmLocSyst}.

\bigskip\noindent
\ref{MainThm_3}
We fix a geometric base point $\bar\gamma$ of $\breve\pi\bigl((\breveRZ)^\an\bigr)$ and consider the canonical family of morphisms $(f_K\colon{\breve\CM}^K\to\HeckeTower_K)_{K\subset G(\BF_q\dpl z\dpr)}$ that sends an $\Test$-valued triple $(\ul\CG,\bar\delta,\eta K)$ over ${\breve\CM}^K$ from Corollary~\ref{CorLevel2} with
\[
\eta K\;\in\;\Triv_{\ul\CG,\bar\gamma}\bigl(\BF_q\dpl z\dpr\bigr)/K\;=\;\Isom^{\otimes}(\omega^{\circ},{\it forget}\circ\omega_{\bar\gamma}\circ\check V_{\ul{\CG}})\bigl(\BF_q\dpl z\dpr\bigr)/K
\]
to the $\Test$-valued point of $\HeckeTower_K$ given by the $K$-orbit $\beta K\in\Isom^\otimes(\omega\open,{\it forget}\circ\omega_{\bar\gamma}\circ\ulCV_b)\bigl(\BF_q\dpl z\dpr\bigr)/K$ of tensor isomorphisms where $\beta:=({\it forget}\circ\omega_{\bar\gamma})(\alpha)^{-1}\circ\eta$; see Remarks~\ref{RemLocSyst} and \ref{Rem1.8}(a). The map $f_K$ does not depend on the chosen base point $\bar\gamma$ by \cite[Theorem 2.9]{dJ95a} and is thus defined on all connected components of the ${\breve\CM}^K$. The family $(f_K)_K$ is equivariant for the Hecke action of $G\bigl(\BF_q\dpl z\dpr\bigr)$ on both towers defined in \eqref{EqHeckeCor1} and \eqref{EqHeckeCor2}. For any algebraically closed complete extension $\Omega$ of $\breve E$ the morphism $f_K$ is bijective on $\Omega$-valued points because the fibers of ${\breve\CM}^K(\Omega)$ and $\HeckeTower_K(\Omega)$ over a fixed $\Omega$-valued point of $\breve\pi\bigl((\breveRZ)^\an\bigr)$ are both isomorphic to the quotient $G\bigl(\BF_q\dpl z\dpr\bigr)/K$ by Remark~\ref{Rem1.8}(a) and Proposition~\ref{PropFibersOfPiK}. Hence $f_K$ is quasi-finite by \cite[Proposition 3.1.4]{Berkovich2}. Since ${\breve\CM}^K$ and $\HeckeTower_K$ are \'etale over $\breve\pi\bigl((\breveRZ)^\an\bigr)$ the morphisms $f_K$ are \'etale by \cite[Corollary 3.3.9]{Berkovich2} and hence isomorphisms by \cite[Proposition~A.4]{HartlRZ}.

To prove that the $f_K$ are $J_b\bigl(\BF_q\dpl z\dpr\bigr)$-equivariant we must show that the upper ``rectangle'' in the diagram
\[
\xymatrix @R=0.5pc {
\breve\CM^K \ar[rr]^{j_{\breve\CM^K}}\ar[dd]_{f_K} \ar@{-->}[dddr]^f 
& & \breve\CM^K \ar[ddd]^{f_K} 
\\
\\
\HeckeTower_K \ar[dr] \ar@/^/[drr]^{j_{\HeckeTower_K}} \ar[ddd] \\
& j^*\HeckeTower_K \ar[dd] \ar[r]_{pr} & \HeckeTower_K \ar[dd] \\
\\
\breve\pi\bigl((\breveRZ)^\an\bigr) \ar@{=}[r] & \breve\pi\bigl((\breveRZ)^\an\bigr) \ar[r]^j & \breve\pi\bigl((\breveRZ)^\an\bigr)
}
\]
is commutative, that is $f_K\circ j_{\breve\CM^K}=j_{\HeckeTower_K}\circ f_K$ holds. Recall from Remark~\ref{RemJActsOnMK} that the action $j_{\breve\CM^K}\colon\breve\CM^K\to\breve\CM^K$ of $j\in J_b\bigl(\BF_q\dpl z\dpr\bigr)$ is defined on the universal objects by 
\[
j^*(\ul\CG^\univ,\bar\delta^\univ,\eta^\univ K)\;=\;\Bigl(j^*\ul\CG^\univ,\,j\circ\bar\delta^\univ\circ(\Phi_j\mod\zeta),\,{\it forget}(\check V_{\Phi_j,\bar x}^{-1})\circ\eta^\univ K\Bigr)
\]
This triple on the source of the morphism $j_{\breve\CM^K}\colon\breve\CM^K\to\breve\CM^K$ is mapped under the dashed arrow $f$ to $({\it forget}\circ\omega_{\bar\gamma})(j^*\alpha)^{-1}\circ{\it forget}(\check V_{\Phi_j,\bar x}^{-1})\circ\eta K\;=\;({\it forget}\circ\omega_{\bar\gamma})(\phi_j^{-1}\circ\alpha^{-1})\circ\eta K$. Likewise the image $f_K(\ul\CG^\univ,\bar\delta^\univ,\eta^\univ K)= ({\it forget}\circ\omega_{\bar\gamma})(\alpha)^{-1}\circ\eta K$ on $\HeckeTower_K$ of the universal object on $\breve\CM^K$ is mapped by \cite[Formula~(2.11)]{HartlRZ} to $({\it forget}\circ\omega_{\bar\gamma})(\phi_j^{-1})\circ({\it forget}\circ\omega_{\bar\gamma})(\alpha)^{-1}\circ\eta K$ on $j^*\HeckeTower_K$. This proves the commutativity of the upper ``rectangle'' and the $J_b\bigl(\BF_q\dpl z\dpr\bigr)$-equivariance of the isomorphism $f_K$.
\end{proof}

\begin{corollary}\label{CorQuasiAlg}
Every point $x\in\breve\CM^K$ has an affinoid neighborhood $U$ that is finite \'etale over its image $\breve\pi(U)$ in $\breve\CH_{G,\hat{Z},b}^{a}$. This image $\breve\pi(U)$ is an affinoid subspace of the projective variety $\breve\CH_{G,\hat{Z}}$. In particular, $\breve\CM^K$ is quasi-algebraic over $\breve E$; compare \cite[D\'efinition~4.1.11]{Fargues}. 
\end{corollary}

\begin{proof}
Note that the affinoid neighborhoods of $\breve\pi(x)$ in $\breve\CH_{G,\hat{Z},b}^{a}$ form a basis of neighborhoods of $\breve\pi(x)$ by \cite[p.~48]{Berkovich1}. By Theorem~\ref{MainThm}\ref{MainThm_3} and by the definition of an \'etale covering space, the point $\breve\pi(x)$ therefore has an affinoid neighborhood $V$ in $\breve\CH_{G,\hat{Z},b}^{a}$ such that $\breve\pi^{-1}(V)$ is a disjoint union of $\breve E$-analytic spaces, each of which is mapping finitely \'etale to $V$. We can thus take $U$ as the connected component of $\breve\pi^{-1}(V)$ containing $x$.
\end{proof}

\begin{example}\label{ExampleNotATorsor}
We exhibit a case in which the tensor functor $M\colon\Rep_{\BF_q\dbl z\dbr}G\to\FMod_{\CO_{\Omega}\dbl z\dbr},\,V\mapsto M_V$ used in the proof of Theorem~\ref{MainThm}\ref{MainThm_1} does not come from a $G$-torsor over $\CO_\Omega\dbl z\dbr$. This is also the function field analog of \cite[Remark~5.2.9]{ScholzeWeinstein}. Consider the non-reductive group scheme $G=\BG_a\rtimes\BG_m$ over $\BF_q\dbl z\dbr$ and its representations $\rho\colon G\isoto\bigl\{\left(\begin{smallmatrix} 1 & 0 \\ * & * \end{smallmatrix}\right)\}\into\GL_2$ on $V=\BF_q\dbl z\dbr^2$ and $\rho'\colon G \onto\BG_m$ on $V'=\BF_q\dbl z\dbr$. They sit in an exact sequence in $\Rep_{\BF_q\dbl z\dbr}G$
\begin{equation}\label{EqExampleNotATorsor1}
\xymatrix @C+1pc {
0 \ar[r] & (V'\!,\rho') \ar[r]^{\TS  \left(\begin{smallmatrix} 0 \\ 1 \end{smallmatrix}\right)} & (V,\rho) \ar[r]^{(1,0)} & \BOne \ar[r] & 0\,,
}
\end{equation}
where $\BOne\colon G\to\{1\}\subset\BG_m$ is the trivial representation on $\BF_q\dbl z\dbr$. Let $b = \left(\begin{smallmatrix} 1 & 0 \\ 0 & -z \end{smallmatrix}\right)\in G\bigl(\BaseFld\dpl z\dpr\bigr)$ and $\bar\gamma = \left(\begin{smallmatrix} 1 & 0 \\ 0 & (z-\zeta)^{-1} \end{smallmatrix}\right)\in G\bigl(\Omega\dpl z-\zeta\dpr\bigr)\big/G\bigl(\Omega\dbl z-\zeta\dbr\bigr)$.

Consider the local $G$-shtuka $\wt{\ul\CG}=\bigl((L^+G)_{\CO_\Omega},\tau_\CG=\left(\begin{smallmatrix} 1 & 0 \\ 0 & \zeta-z \end{smallmatrix}\right)\bigr)$ over $\CO_\Omega$. Recall the functor from $\Rep_{\BF_q\dbl z\dbr}G$ to the category of local shtukas that assigns to $(V,\rho)$ the local shtuka $\ulM_V$ associated with the local $\GL(V)$-shtuka $\rho_*\wt{\ul\CG}$ from Remark~\ref{RemHPStrOfLocGSht}. The underlying $\CO_\Omega\dbl z\dbr$-module of $\ulM_V$ equals $V\otimes_{\BF_q\dbl z\dbr}\CO_\Omega\dbl z\dbr$. Applied to $\wt{\ul\CG}$ this functor yields the exact sequence of local shtukas
\[
\xymatrix @C-1pc {
0 \ar[r] & \ulM'=(\CO_\Omega\dbl z\dbr, \tau_{M'}=\zeta-z) \ar[r] & (M''\oplus M'\!,\tau={\left(\begin{smallmatrix} 1 & 0 \\ 0 & \zeta-z \end{smallmatrix}\right)}) \ar[r] & \ulM''=(\CO_\Omega\dbl z\dbr, \tau_{M''}=1) \ar[r] & 0\,.
}
\]
Let $\Trivializer:=\sqrt[q-1]{-1}\cdot\tplus$. Then the rational Tate module $\CV_{b,\bar\gamma}=\check{V}_z\wt{\ul\CG}$ is ``generated by'' $\left(\begin{smallmatrix} 1 & 0 \\ 0 & \Trivializer \end{smallmatrix}\right)=\left(\begin{smallmatrix} 1 & 0 \\ 0 & \zeta-z \end{smallmatrix}\right)\cdot\sigma^*\left(\begin{smallmatrix} 1 & 0 \\ 0 & \Trivializer \end{smallmatrix}\right)$, that is, the tensor isomorphism $\tilde\beta\colon\omega\open\isoto\omega_{b,\bar\gamma}$ is given by multiplication with
\[
\rho\left(\begin{smallmatrix} 1 & 0 \\ 0 & \Trivializer \end{smallmatrix}\right)\colon\omega\open(V,\rho)=V\isoto\omega_{b,\bar\gamma}(V,\rho)=\check{V}_z\rho_*\wt{\ul\CG}\,.
\]

Now we let $\pi\in\CO_\Omega$ satisfy $\pi^{q-1}=\zeta$ and we choose a different tensor isomorphism $\beta\colon\omega\open\isoto\omega_{b,\bar\gamma}$ that is given by multiplication with $\rho\left(\begin{smallmatrix} 1 & 0 \\ \Trivializer & z\Trivializer \end{smallmatrix}\right)=\rho(\left(\begin{smallmatrix} 1 & 0 \\ 0 & \Trivializer \end{smallmatrix}\right)\cdot\left(\begin{smallmatrix} 1 & 0 \\ 1 & z \end{smallmatrix}\right))$, where $\left(\begin{smallmatrix} 1 & 0 \\ 1 & z \end{smallmatrix}\right)\in G\bigl(\BF_q\dpl z\dpr\bigr)$. The construction in step 2 of the proof of Theorem~\ref{MainThm}\ref{MainThm_1} with $\beta$ instead of $\tilde\beta$ replaces every $\ulM_V$ by a quasi-isogenous one. We claim that for the sequence \eqref{EqExampleNotATorsor1} it yields the upper row in the commutative diagram of local shtukas
\begin{equation}\label{EqExampleNotATorsor2}
\xymatrix @C+1pc  {
0 \ar[r] & \ulM' \ar[r]^{\TS\left(\begin{smallmatrix} -\pi \\ z \end{smallmatrix}\right)\qquad\qquad\qquad} \ar@{^{ (}->}[d]_{\TS z\cdot} & \ulM = (\CO_\Omega\dbl z\dbr^2,\tau_M={\left(\begin{smallmatrix} 1 & \pi \\ 0 & \zeta-z \end{smallmatrix}\right)}) \ar[r]^{\qquad\qquad\quad(z,\pi)} \ar@{^{ (}->}[d]_{\TS\left(\begin{smallmatrix} z & \pi \\ 0 & 1 \end{smallmatrix}\right)} & \ulM'' \ar[d]^\cong_{\TS 1\cdot} \\
0 \ar[r] & \ulM' \ar[r]^{\TS  \left(\begin{smallmatrix} 0 \\ 1 \end{smallmatrix}\right)} & \ulM'' \oplus \ulM' \ar[r]^{(1,0)} & \ulM'' \ar[r] & 0\,.
}
\end{equation}
To prove our claim, we note that \eqref{EqExampleNotATorsor2} yields over $R:=\CO_\Omega/(\zeta)$ the diagram
\[
\xymatrix @C+0.2pc {
0 \ar[r] & \ulM'\mod\zeta \ar[r]^{\TS\left(\begin{smallmatrix} -\pi \\ z \end{smallmatrix}\right)} \ar[d]_{\TS \bar\delta'=z} & \ulM\mod\zeta \ar[r]^{(z,\pi)} \ar[d]_{\TS\bar\delta=\left(\begin{smallmatrix} z & \pi \\ 0 & 1 \end{smallmatrix}\right)} & \ulM''\mod\zeta \ar[d]_{\TS \bar\delta''=1} \\
0 \ar[r] & (R\dbl z\dbr, \rho'(b)=-z) \ar[r]_{\TS  \left(\begin{smallmatrix} 0 \\ 1 \end{smallmatrix}\right)\quad} & (R\dbl z\dbr^2, \rho(b)={\left(\begin{smallmatrix} 1 & 0 \\ 0 & -z \end{smallmatrix}\right)}) \ar[r]_{\quad(1,0)} & (R\dbl z\dbr, \rho''(b)=1) \ar[r] & 0\,.
}
\]
Also, when we consider the element 
\[
\Omega\dbl z\dbr\mal\;\ni\; y\;:=\;\frac{1-\pi\Trivializer}{z}\;=\;\frac{1+(z-\zeta)\pi\,\sigma(\Trivializer)}{z}\;=\;\pi\,\sigma(\Trivializer)+\frac{1-\sigma(\pi\Trivializer)}{z}\;=\;\pi\,\sigma(\Trivializer)+\sigma(y)\,,
\]
the upper row of \eqref{EqExampleNotATorsor2} is right exact on Tate modules, because it induces the following commutative diagram
\[
\xymatrix @C+3pc {
0 \ar[r] & \BF_q\dbl z\dbr \ar[r]^{\TS\left(\begin{smallmatrix} 0 \\ 1 \end{smallmatrix}\right)} \ar[d]_{\TS\Trivializer\cdot}^\cong & \BF_q\dbl z\dbr^2 \ar[r] \ar[d]^\cong_{\TS\left(\begin{smallmatrix} y & \,-\pi\Trivializer \\ \Trivializer & z\Trivializer \end{smallmatrix}\right)} \ar[r]^{(1,0)} & \BF_q\dbl z\dbr \ar[d]_\cong^{\TS 1\cdot } \ar[r] & 0\,\; \\
0 \ar[r] & \check{T}_z\ulM' \ar[r]_{\TS\left(\begin{smallmatrix} -\pi \\ z \end{smallmatrix}\right)} & \check{T}_z\ulM \ar[r]_{(z,\pi)} & \check{T}_z\ulM'' \ar[r] & 0\,.
}
\]
Finally, the vertical quasi-isogeny in the middle of \eqref{EqExampleNotATorsor2} induces the following commutative diagram on Tate modules
\[
\xymatrix @C+3pc @R+1pc{
0 \ar[r] & \BF_q\dbl z\dbr^2 \ar[r]^{\TS\left(\begin{smallmatrix} 1 & 0 \\ 1 & z \end{smallmatrix}\right)} \ar[d]^\cong_{\TS\left(\begin{smallmatrix} y & \,-\pi\Trivializer \\ \Trivializer & z\Trivializer \end{smallmatrix}\right)} \ar[dr]^{\TS\left(\begin{smallmatrix} 1 & 0 \\ \Trivializer & z\Trivializer \end{smallmatrix}\right)} & \BF_q\dbl z\dbr^2 \ar[r]^{(1,-1)\mod z} \ar[d]_\cong^{\TS\left(\begin{smallmatrix} 1 & 0 \\ 0 & \Trivializer \end{smallmatrix}\right)} & \BF_q \ar[r] & 0\,. \\
& \check{T}_z\ulM \ar[r]_{\TS\left(\begin{smallmatrix} z & \pi \\ 0 & 1 \end{smallmatrix}\right)\qquad} & \check{T}_z(\ulM'' \oplus \ulM')
}
\]
This proves the claim.

Now, in diagram \eqref{EqExampleNotATorsor2} the map $(z,\pi)$ in the upper row is not surjective and this provides an example where the tensor functor $(V,\rho)\mapsto M_V$ is not exact, and hence does not come from a $G$-torsor $\CG$ over $\CO_\Omega\dbl z\dbr$. One can also check, although 
\[
\left(\begin{matrix} 1 & 0 \\ 0 & \Trivializer \end{matrix}\right)\cdot\left(\begin{matrix} 1 & 0 \\ 1 & z \end{matrix}\right) \;=\; \left(\begin{matrix} 1 & 0 \\ \Trivializer & z\Trivializer \end{matrix}\right) \;=\; \left(\begin{matrix} z & \pi \\ 0 & 1 \end{matrix}\right) \cdot \left(\begin{matrix} y & \,-\pi\Trivializer \\ \Trivializer & z\Trivializer \end{matrix}\right) \;\in\;\GL_2\bigl(\CO_\Omega\dpl z\dpr\bigr)\cdot\GL_2(\Omega\langle\tfrac{z}{\zeta}\rangle)\,,
\]
it is not possible to write it as a product in $G\bigl(\CO_\Omega\dpl z\dpr\bigr)\cdot G(\Omega\langle\tfrac{z}{\zeta}\rangle)$. This corresponds to the fact that the quasi-isogeny $\ulM\to\ulM''\oplus\ulM'$ does not come from a quasi-isogeny $\ul\CG\to\wt{\ul\CG}$ of local $G$-shtukas over $\CO_\Omega$, because $\ul\CG$ does not exist. Note also, that in terms of the proofs of Proposition~\ref{PropLiftOfIsog} and Lemma~\ref{LemmaLiftOfIsog} it is not possible to extend the \'etale local $G$-shtuka $\ul\CG''$ over $\Omega$ to $\CO_\Omega$, because $G$ is not parahoric and $\Flag_G$ and $\ul\CM_{\wt{\ul\CG}}$ are not ind-projective.
\end{example}

\section{Cohomology} \label{SectCohom}
\setcounter{equation}{0}

In this section we provide basic properties of the cohomology of the towers of moduli spaces. This theory and all of the proofs parallel the one for Rapoport-Zink spaces in \cite{Fargues} to which we also refer for some arguments that go over to our case without modification. Note that instead of the $\breve E$-analytic space $\breve\CM^K$ in the sense of Berkovich we may work with the associated adic space in the sense of Huber by \cite[Proposition~8.2.12 and Theorem~8.3.5]{Huber96}.

Let $\ell$ be a prime different from the characteristic of $\BF_q$ and let $\olE$ be the completion of an algebraic closure of $\breve E$.
\begin{definition}
We denote by $H_c^{\bullet}({\breve\CM}^K\hat{\otimes}_{\breve E}\olE,\BQ_{\ell})$ the $\ell$-adic cohomology with compact support of the analytic space ${\breve\CM}^K$. For short we denote it by $H_c^{\bullet}({\breve\CM}^K,\BQ_{\ell})$ and write $H_c^{\bullet}({\breve\CM}^K,\overline{\BQ}_{\ell}):=H_c^{\bullet}({\breve\CM}^K,\BQ_{\ell})\otimes_{\BQ_\ell} \overline{\BQ}_{\ell}$.
\end{definition}
Because the spaces ${\breve\CM}^K$ are in general only quasi-algebraic, the definition of the cohomology needs some explanation for which we refer to Fargues, \cite[\S\,4.1]{Fargues}.

These cohomology groups are equipped with the following group actions. The action of $J:=J_b\bigl(\BF_q\dpl z\dpr\bigr)$ on ${\breve\CM}^K$ induces an action on $H_c^{\bullet}({\breve\CM}^K,\BQ_{\ell})$ for each $K$. Furthermore we obtain an action of the Weil group $W_E$ of $E$. Indeed, the inertia subgroup $\Gal(\breve E^\sep\!/\breve E)$ acts on the coefficients $\olE$ inducing an action on the cohomology. The action of Frobenius $\sigma\in W_E$ is induced by the Weil descent datum of Remark \ref{defweildes}. As in \cite[Remarque 4.4.3]{Fargues},  one can show that the induced morphism on cohomology is invertible, and thus induces an action of $W_E$. Furthermore, for varying $K$ the action by Hecke correspondences induces an action of $G\bigl(\BF_q\dpl z\dpr \bigr)$ on the cohomology groups of the whole tower.

If $\epsilon\colon G\to G'$ is a morphism of parahoric group schemes over $\BF_q\dbl z\dbr$ as in Remarks~\ref{RemFunctoriality1}, \ref{RemFunctoriality2} and \ref{RemFunctoriality3} with $\epsilon(K)\subset K'$, we obtain a morphism
\[
\epsilon^*\colon H_c^{\bullet}({\breve\CM}_{G'}^{K'},\BQ_{\ell}) \;\longto\;H_c^{\bullet}({\breve\CM}_G^K,\BQ_{\ell})
\]
that is compatible with the actions of the Weil group $W_E$, the Hecke action of $G\bigl(\BF_q\dpl z\dpr\bigr)$ which acts on the source via the morphism $G\bigl(\BF_q\dpl z\dpr\bigr)\to G'\bigl(\BF_q\dpl z\dpr\bigr)$, and the action of the group $J^G_b$ which acts on the source via the morphism $J^G_b\to J^{G'}_{\epsilon(b)}$.

\begin{lemma}
For each $K$, the $J\times W_E$-representation $H_c^{\bullet}({\breve\CM}^K,\BQ_{\ell})$ is smooth for the action of $J$ and continuous for the action of $W_E$.
\end{lemma}
\begin{proof} 
As in the arithmetic context (compare \cite[Corollaire 4.4.7]{Fargues}) this follows from the facts that the ${\breve\CM}^K$ are quasi-algebraic by Corollary~\ref{CorQuasiAlg}, and that $J$ acts continuously on ${\breve\CM}^K$ by Lemma \ref{lemjcont}, using \cite[Corollaires~4.1.19, 4.1.20]{Fargues}, two general assertions on the cohomology of Berkovich spaces. Note for this that $J$ has an open pro-$p$ subgroup, namely $\{j\in J\cap L^+G(\BaseFld)\colon j\equiv 1\mod z\}$.
\end{proof}

Next we are interested in finiteness and vanishing properties of cohomology groups. We need the following finiteness statement for the set of irreducible components.
\begin{lemma}\label{lemfinirr}
The action of $J$ on the set of irreducible components of $\breveRZ$ has only finitely many orbits.
\end{lemma} 
\begin{proof}
This is a statement about the underlying reduced subscheme, i.e.~on the set of irreducible components of the affine Deligne-Lusztig variety $X_{Z^{-1}}(b)$ from \eqref{defadlv}. By \cite[Theorem 1.4 and Subsection 2.1]{RZ2} there is a closed subscheme $Y\subset \Flag_G$ of finite type such that for each $g\in X_{Z^{-1}}(b)$ there is a $j\in J$ with $j^{-1}g\in Y$. In other words, $g$ has a representative satisfying $g^{-1}b\s(g)=h^{-1}b\s(h)$ for some (representative of an element) $h=j^{-1}g\in Y\cap X_{Z^{-1}}(b)$. In particular, it is enough to show that $Y\cap X_{Z^{-1}}(b)$ has only finitely many irreducible components. This follows as  $Y\cap X_{Z^{-1}}(b)$ is of finite type.
\end{proof}

\begin{proposition}
For each compact open subgroup $K\subset G\bigl(\BF_q\dbl  z\dbr\bigr)$ the $J$-representation $H_c^{\bullet}({\breve\CM}^K,\BQ_{\ell})$ is of finite type.
\end{proposition}
\begin{proof}
Let $X_1,\dotsc,X_t$ be representatives of the finitely many orbits of the action of $J$ on the set of irreducible components of $\breveRZ$ (compare Lemma \ref{lemfinirr}). Let $K_0=G(\BF_q\dbl z\dbr)$ and let $U\subset (\breveRZ)^\an=\breve\CM={\breve\CM}^{K_0}$ be the tube over $X:=X_1\cup\ldots\cup X_t$, that is the preimage of $X$ under the specialization map $sp$ from $(\breveRZ)^\an$ to the underlying topological space of $\breveRZ$; see \cite[\S\,1]{Berkovich4}. If $V_0$ is a quasi-compact open subset of $\breveRZ$ containing $X$, then the complement $Y:=V_0\setminus X$ is open and quasi-compact, because $V_0$ is noetherian. Therefore $V:=sp^{-1}(V_0)\subset\breve\CM^{K_0}$ is a compact neighborhood of $U$ and $V\setminus U=sp^{-1}(Y)$ is compact, whence $U$ is open in ${\breve\CM}^{K_0}$. Let $U_K:=\breve\pi_{K_0,K}^{-1}(U)\subset {\breve\CM}^K$. 

Under the fully faithful functor \cite[\S\,1.6]{Berkovich2} from strictly $\breve E$-analytic spaces to rigid analytic spaces, $U$ and $U_K$ correspond to $U^\rig=sp^{-1}(X)$ and $U_K^\rig=(\breve\pi_{K_0,K}^\rig)^{-1}(U^\rig)$, where $sp\colon(\breveRZ)^\rig\to\breveRZ$ is the specialization map \cite[(0.2.2.1)]{Berthelot96}. These are admissible open subspaces of $(\breveRZ)^\rig$ and $(\breve\CM^K)^\rig$. Under the fully faithful functor \cite[(1.1.11)]{Huber96} from rigid analytic spaces over $\breve E$ to Huber's adic spaces, $U^\rig$ and $U_K^\rig$ correspond to $U^\ad=sp^{-1}(X)\open$ and $U_K^\ad=(\breve\pi_{K_0,K}^\ad)^{-1}(U^\ad)$, where $sp\colon(\breveRZ)^\ad\to\breveRZ$ is the specialization map \cite[Proposition~1.9.1]{Huber96} and $sp^{-1}(X)\open$ denotes the open interior. These are open subspaces of $(\breveRZ)^\ad$ and $(\breve\CM^K)^\ad$. By definition \cite[formula (*) on p.~315]{Huber96}
\[
H^q_c(U_K^\rig,\BQ_{\ell}) \;:=\;H^q_c(U_K^\ad,\BQ_{\ell})
\]
and this is a finite dimensional $\BQ_\ell$-vector space by \cite[Corollaries~5.8 and 5.4]{Huber07}. Since $X$ is proper over $\BaseFld$ by Theorem~\ref{ThmRRZSp}, $U_K^\rig$ is partially proper over $\breve E$ by \cite[Remark~1.3.18]{Huber96}. So $H^q_c(U_K,\BQ_{\ell})=H^q_c(U_K^\rig,\BQ_{\ell})$ by  \cite[Proposition~1.5]{Huber98}. Note that the proof of loc.\ cit.\ only uses that $U_K^\rig$ is partially proper over $\breve E$ and not the stated assumption that $U_K$ is closed. (We thank Roland Huber for explaining these arguments to us.)

Let $J'\subset J$ be the stabilizer of $U_K$, a compact open subgroup. Then the $g\cdot U_K$ for $\bar g\in J/J'$ are a covering of ${\breve\CM}^K$. We consider the associated spectral sequence for Cech cohomology of \cite[Proposition~4.2.2]{Fargues},
$$E_1^{pq}=\bigoplus_{\alpha\subset J/J'}H^q_c(U_K(\alpha),\BQ_{\ell})\;\Longrightarrow\; H_c^{p+q}({\breve\CM}^K,\BQ_{\ell})$$ 
with the sum being over subsets $\alpha$ with $-p+1$ elements and where $U_K(\alpha)=\bigcap_{\bar g\in\alpha}g\cdot U_K$. It is concentrated in degrees $p\leq 0$ and $0\leq q\leq \dim (\breveRZ)^\an$. Furthermore, it is $J$-equivariant where $g\in J$ acts via
$$g_{!}\colon H_c^q(U_K(\alpha),\BQ_{\ell})\rightarrow H_c^q(g\cdot U_K(\alpha),\BQ_{\ell}).$$
For $\alpha\subset J/J'$ let $J'_{\alpha}=\bigcap_{\bar g\in \alpha}gJ'g^{-1}$. By Lemma \ref{lemjcont} $J'_{\alpha}\subset J$ acts continuously on $U_K(\alpha)$. Hence the $H_c^q(U_K(\alpha),\BQ_{\ell})$ are smooth $J'_{\alpha}$-modules. We can rewrite $E_1^{pq}$ as compact induction
$$E_1^{pq}=\bigoplus \text{c-Ind}_{J'_{\alpha}}^JH^q_c(U_K(\alpha),\BQ_{\ell})$$
where the sum is now over equivalence classes $\bar\alpha$ of subsets $\alpha\subset J/J'$ with $-p+1$ elements up to the action of $J$ diagonally on $(J/J')^{-p+1}$. We claim that there are only finitely many such $\bar\alpha$ with $U_K(\alpha)\neq\emptyset$.

To show this claim note that if $A$ is a finite union of irreducible components of $\breveRZ$, then the set $\{g\in J\colon g\cdot A\cap A\neq \emptyset\}$ is compact. In particular $J''=\{g\in J\colon g\cdot U_K\cap U_K\neq\emptyset\}$ is compact and contains $J'$. Thus if $\bar\alpha=\{g_1,\dotsc,g_{-p+1}\}$ is as above with $U_K(\alpha)\neq\emptyset$ then for $i\neq j$ we have $g_i^{-1}g_j\in J''$. Modulo the left action of $J$ on the index set we may assume that $g_{-p+1}\in J''/J'$, hence all $g_i$ are in $J''/J'$, a finite set. In particular the index set is finite.

Altogether we obtain that $E_1^{pq}$ is a finite sum of compact inductions of finite-dimensional representations, and hence a representation of $J$ of finite type. By \cite[Remarque~3.12]{Bernstein} the category of smooth $J$-modules is locally noetherian. Because $H_c^{p+q}({\breve\CM}^K,\BQ_{\ell})$ has a finite filtration with all subquotients of finite type, it is itself of finite type. 
\end{proof}

\begin{corollary}
Let $\Pi$ be an admissible representation of $J$. Then for all $K$, $p$ and $q$
$$\dim_{\overline{\BQ}_{\ell}}\Ext^p_{J\text{\rm-smooth}}(H_c^q({\breve\CM}^K,\overline{\BQ}_{\ell}),\Pi)<\infty.$$
\end{corollary}
\begin{proof}
This follows from the preceding proposition together with the following fact. Let $H$ be a reductive group over $\BF_q\dpl z\dpr$, let $\Pi_1$ be a smooth representation of $H$ of finite type and $\Pi_2$ an admissible representation. Then $\dim({\rm Ext}_{H\text{\rm-smooth}}^i(\Pi_1,\Pi_2))<\infty$. This fact can be shown in the same way as for $p$-adic groups, compare \cite[Lemme~4.4.15]{Fargues}.
\end{proof}

%
%

Competing interests: The authors declare none.

\vfill

\begin{minipage}[t]{0.5\linewidth}
\noindent
Urs Hartl\\
Universit\"at M\"unster\\
Mathematisches Institut \\
Einsteinstr.~62\\
D -- 48149 M\"unster
\\ Germany
\\[1mm]
\href{http://www.math.uni-muenster.de/u/urs.hartl/index.html.en}{www.math.uni-muenster.de/u/urs.hartl/}
\end{minipage}
\begin{minipage}[t]{0.45\linewidth}
\noindent
Eva Viehmann\\
Technische Universit\"at M\"unchen\\
Fakult\"at f\"ur Mathematik -- M11\\
Boltzmannstr. 3 \\
D -- 85748 Garching b.~M\"unchen\\
Germany
\\[1mm]
\end{minipage}

\end{document}